%% file: Thesis.tex
\documentclass[french,12pt,a4paper]{book}
\usepackage{gensymb}
\usepackage[utf8]{inputenc}
\usepackage[T1]{fontenc}

\usepackage{moresize}

\usepackage{imakeidx}
\makeindex[columns=2, title=Index, options= -s includes/new_style.ist, intoc]
\usepackage[nottoc,notlof,notlot]{tocbibind}
\makeatletter
  \renewcommand{\@pnumwidth}{2em}
  \renewcommand{\@tocrmarg}{5em}
\makeatother

\usepackage{etoolbox}
\apptocmd{\sloppy}{\hbadness10000\relax}{}{}

\usepackage[english,french]{babel}

\usepackage{lmodern}

\usepackage{enumerate}

\usepackage{amsmath, amsfonts, amssymb, amsthm}
\usepackage{ dsfont }
\usepackage{caption}
\usepackage{subcaption}
\usepackage{mleftright}

\usepackage[linesnumbered,ruled,lined]{algorithm2e}

\usepackage{xcolor} 
\definecolor{Prune}{RGB}{99,0,60} 

\usepackage{color}
\definecolor{corange}{rgb}{0.898, 0.621, 0.0}
\definecolor{cskyblue}{rgb}{0.336, 0.703, 0.910}
\definecolor{cbluishgreen}{rgb}{0, 0.617, 0.449}
\definecolor{cyellow}{rgb}{0.937, 0.890, 0.258}
\definecolor{cblue}{rgb}{0, 0.445, 0.695}
\definecolor{cred}{rgb}{1, 0, 0}
\definecolor{cpurple}{rgb}{0.797, 0.473, 0.652}

\usepackage{mdframed}
\usepackage{multirow} 
\usepackage{multicol} 
\usepackage{scrextend} 

\usepackage{tikz}
\usepackage{graphicx}
\usepackage[absolute]{textpos} 
\usepackage{colortbl}
\usepackage{array}
\usepackage[margin=2cm]{geometry}
\usepackage{titlesec}
\usepackage[colorlinks=true,
linkcolor=cbluishgreen,
urlcolor=cpurple,
citecolor=cbluishgreen,
backref=page,
linktocpage]{hyperref}

\usepackage{fancyhdr}
\usepackage{todonotes}

\theoremstyle{definition}
\newtheorem{theorem}{Theorem}[section] 
\newtheorem{proposition}[theorem]{Proposition} 
\newtheorem{lemma}[theorem]{Lemma} 
\newtheorem{corollary}[theorem]{Corollary}
\newtheorem{conjecture}[theorem]{Conjecture}

\newtheorem{definition}[theorem]{Definition}

\newtheorem{perspective}[theorem]{Perspective}
\newtheorem{example}[theorem]{Example}
\newtheorem{remark}[theorem]{Remark}

\def\AA{\mathbb{A}}
\def\DD{\mathbb{D}}
\def\NN{\mathbb{N}}
\def\PP{\mathbb{P}}
\def\RR{\mathbb{R}}
\def\TT{\mathbb{T}}
\def\UU{\mathbb{U}}
\def\ZZ{\mathbb{Z}}

\def\bfa{\mathbf{a}}
\def\bfd{\mathbf{d}}
\def\bfi{\mathbf{i}}

\def\cA{\mathcal{A}}
\def\cB{\mathcal{B}}
\def\cBT{\mathcal{BT}}
\def\cC{\mathcal{C}}
\def\cD{\mathcal{D}}
\def\cF{\mathcal{F}}
\def\cH{\mathcal{H}}
\def\cI{\mathcal{I}}
\def\cL{\mathcal{L}}
\def\cN{\mathcal{N}}
\def\cO{\mathcal{O}}
\def\cP{\mathcal{P}}
\def\cPT{\mathcal{PT}}
\def\cT{\mathcal{T}}
\def\cR{\mathcal{R}}
\def\cS{\mathcal{S}}

\def\cV{\mathcal{V}}
\def\cW{\mathcal{W}}

\def\fS{\mathfrak{S}}

\DeclareMathOperator{\supp}{supp}
\DeclareMathOperator{\inv}{inv}
\DeclareMathOperator{\ver}{ver}
\DeclareMathOperator{\moveU}{moveU}
\DeclareMathOperator{\moveD}{moveD}
\DeclareMathOperator{\rank}{rank}
\DeclareMathOperator{\conv}{conv}
\DeclareMathOperator{\aff}{aff}
\DeclareMathOperator{\cone}{cone}
\DeclareMathOperator{\dimension}{dim}
\DeclareMathOperator{\vol}{vol}
\DeclareMathOperator{\indeg}{indeg}
\DeclareMathOperator{\outdeg}{outdeg}
\DeclareMathOperator{\inEdge}{in}
\DeclareMathOperator{\outEdge}{out}
\DeclareMathOperator{\car}{Car}
\DeclareMathOperator{\oru}{Oru}
\DeclareMathOperator{\mar}{Mar}
\DeclareMathOperator{\bic}{Bic}
\DeclareMathOperator{\PPerm}{Perm}
\DeclareMathOperator{\PAsoc}{Assoc}
\DeclareMathOperator{\PCube}{Cube}
\DeclareMathOperator{\PPT}{PT}

\newcommand{\PSPerm}[1][s]{\text{Perm}_{#1}}
\newcommand{\fpol}[1][G]{\cF_{#1}} 
\newcommand{\oruga}[1][n]{\oru_{#1}} 
\DeclareMathOperator{\height}{h}
\newcommand{\rpre}[1]{R[#1]} 
\newcommand{\prefix}[2]{{#1}_{[{#2}]}} 
\newcommand{\triangDKK}[1][G]{\text{Triang}_{DKK}({#1}, \preceq\nolinebreak)} 
\newcommand{\cliques}[1][G]{\text{Cliques}(#1, \preceq)} 
\newcommand{\maxcliques}[1][G]{\text{MaxCliques}(#1, \preceq)} 
\newcommand{\subdivCay}[1][$s$] {\text{Subdiv}{\square}(#1)} 
\newcommand{\gbinom}[2]{{\mleft(\genfrac..{0pt}{}{#1}{#2}\mright)}}
\newcommand\bbinom[2]%
{\mathchoice{\left(\kern-0.48em{\binom{#1}{#2}}\kern-0.48em\right)}
            {\bigl(\kern-0.3em{\binom{#1}{#2}}\kern-0.3em\bigr)}
            {\bigl(\kern-0.3em{\binom{#1}{#2}}\kern-0.3em\bigr)}
            {\bigl(\kern-0.3em{\binom{#1}{#2}}\kern-0.3em\bigr)}}

\newlength\myheight{}
\newlength\mydepth{}
\settototalheight\myheight{Xygp}
\DeclareRobustCommand*\nonee{\raisebox{-\mydepth}{\includegraphics[height=\myheight]{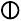}}}
\DeclareRobustCommand*\upp{\raisebox{-\mydepth}{\includegraphics[height=\myheight]{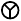}}}
\DeclareRobustCommand*\downn{\raisebox{-\mydepth}{\includegraphics[height=\myheight]{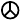}}}
\DeclareRobustCommand*\uppdownn{\raisebox{-\mydepth}{\includegraphics[height=\myheight]{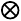}}}

\newcommand{\defn}[1]{\emph{\color{cblue} #1}} 

\begin{document}

\pagenumbering{gobble}

\pagestyle{fancy}

\include{includes/Cover}

\include{includes/Abstracts}


\titleformat{\chapter}[display]
{\normalfont\HUGE\color{black}}
{\chaptertitlename\,\thechapter\enskip\titlerule[2pt]}
{1.8pt}
{\centering\color{Prune}\huge\textbf\sffamily}[{\vspace{5pt}\color{black}\titlerule[2pt]}]

\titleformat{\section}[hang]
{\bfseries\LARGE\color{Prune}}
{\thesection}
{15pt}
{\vspace{0.1ex}}

\titleformat{\subsection}[hang]
{\Large\color{Prune}}
{\thesubsection}
{15pt}
{\vspace{0.1ex}}

\newgeometry{top=4cm, bottom=4cm, left=3cm, right=3cm}

\cleardoublepage{}
\include{includes/Quotes}
\cleardoublepage{}

\newgeometry{top=4cm, bottom=4cm, left=2cm, right=2cm}

\include{includes/Acknowledgements}

\cleardoublepage{}

\tableofcontents

\listoffigures

\listoftables

\include{includes/IntroductionEN}
\include{includes/IntroductionFR}
\cleardoublepage{}

\pagestyle{fancy}
\pagenumbering{arabic}


\part{Preliminaries}\label{part:prelim}
\include{includes/contenu/chap_preliminaires_structures}
\include{includes/contenu/chap_preliminaires_weak_order}

\part{Permutrees}\label{part:permutrees}
\include{includes/contenu/chap_permutrees_vectors}
\include{includes/contenu/chap_permutrees_sorting}

\part{Flow Polytopes}\label{part:sorder}

\include{includes/contenu/chap_sorder_flows}
\include{includes/contenu/chap_sorder_realizations}
\include{includes/contenu/chap_sorder_quotients}


\bibliographystyle{amsalpha}
\bibliography{bibliography.bib}

\printindex
\end{document}

%% file: includes/Cover.tex


\begin{titlepage}

	\newgeometry{left=6cm,bottom=2cm, top=1cm, right=1cm}

	\tikz[remember picture,overlay] \node[opacity=1,inner sep=0pt] at (-13mm,-135mm){\includegraphics{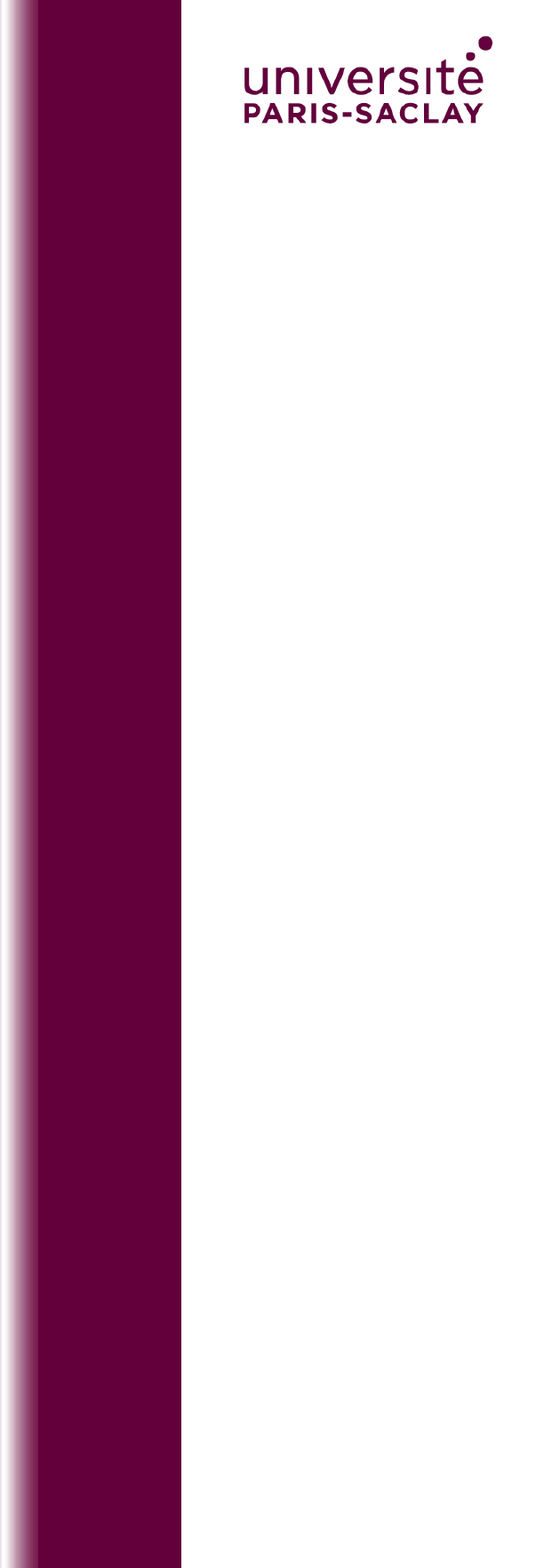}};


	\color{white}

	\begin{picture}(0,0)
		\put(-152,-743){\rotatebox{90}{\Large \textsc{THESE DE DOCTORAT}}} \\
		\put(-120,-743){\rotatebox{90}{NNT : 2023UPASG066}}
	\end{picture}


	\flushright{}
	\vspace{10mm} 
	\color{Prune}

	\fontsize{26}{26}\selectfont
	\mbox{Combinatorics of Permutreehedra} \mbox{and Geometry of $s$-Permutahedra}\\

	\enskip

	\normalsize
	\color{black}
	\Large{\textit{Combinatoire des permusylv\`{e}dres \\et g\'{e}om\'{e}trie des $s$-permuta\`{e}dres}} \\

	\fontsize{8}{12}\selectfont

	\vspace{1.5cm}

	\normalsize
	\textbf{Thèse de doctorat de l'université Paris-Saclay} \\

	\vspace{6mm}

	\small École doctorale n°580:\\
	\small Sciences et technologies de l’information et de la communication (STIC)\\
	\small Spécialité de doctorat: Informatique mathématique\\
	\small Graduate School: Informatique et sciences du numérique\\
	\small Référent: Faculté des sciences d’Orsay\\
	\vspace{6mm}

	\footnotesize Thèse préparée dans l'unité de recherche \textbf{Laboratoire interdisciplinaire des sciences du numérique (CNRS, Université Paris-Saclay)}, sous la direction de \\ \textbf{Viviane PONS}, Maîtresse de conférences et le coencadrement de \\ \textbf{Vincent PILAUD}, Directeur de recherche\\
	\vspace{15mm}

	\textbf{Thèse soutenue à Paris-Saclay, le 17 Octobre 2023, par}\\
	\bigskip
	\Large {\color{Prune} \textbf{Daniel TAMAYO JIMÉNEZ}} 

	\vspace{\fill} 

	\bigskip

	\flushleft{}
	\small {\color{Prune} \textbf{Composition du jury}}\\
	{\color{Prune} \scriptsize {Membres du jury avec voix délibérative}} \\
	\vspace{2mm}
	\scriptsize
	\begin{tabular}{|p{8.25cm}l}
		\arrayrulecolor{Prune}
		\textbf{Jean-Christophe NOVELLI}                      & Président               \\
		Professeur, Université Paris-Est Marne-la-Vallée (IGM)       &                           \\

		\textbf{Samuele GIRAUDO}                              & Rapporteur \& Examinateur \\
		Professeur,  Université du Québec à Montréal (LACIM) &                           \\

		\textbf{Torsten MÜTZE}                                & Rapporteur \& Examinateur \\
		Professeur assistant, University of Warwick                                    &                           \\

		\textbf{Nathalie AUBRUN}                              & Examinatrice              \\
		Chargée de recherche, CNRS, Université Paris-Saclay (LISN)        &                           \\

		\textbf{Mathilde BOUVEL}                              & Examinatrice               \\
		Chargée de recherche, CNRS,Université de Lorraine (LORIA)                                 &                           \\

		\textbf{Francisco SANTOS LEAL}                               & Examinateur      \\
		Professeur, Universidad de Cantabria                    &                           \\

	\end{tabular}

\end{titlepage}

%% file: includes/Abstracts.tex

\newgeometry{top=1.5cm, bottom=1.25cm, left=2cm, right=2cm}

\noindent 
\includegraphics[height=2.45cm]{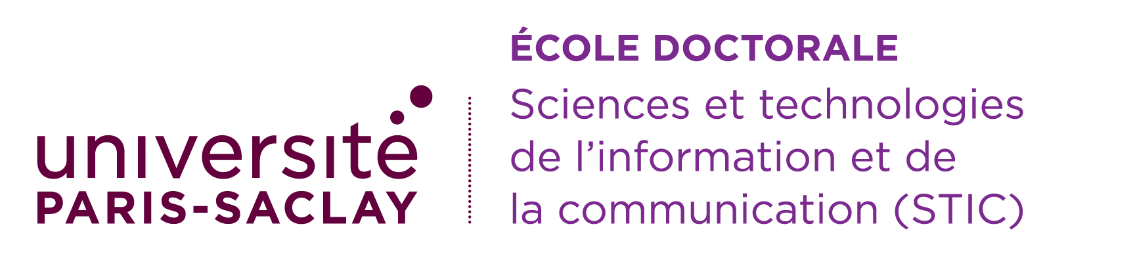}
\vspace{1cm}

\thispagestyle{empty}

\hspace{0pt}
\vfill
\begin{mdframed}[linecolor=Prune,linewidth=1]

    \textbf{Titre:} Combinatoire des permusylv\`{e}dres et g\'{e}om\'{e}trie des~$s$-permuta\`{e}dres
    
    \noindent \textbf{Mots cl\'{e}s:} ordres partiels, combinatoire alg\'{e}brique, g\'{e}om\'{e}trie discr\`{e}te, polytopes, quotients de treillis, permutations.
    
    \vspace{-.5cm}
    \begin{multicols}{2}
    \noindent \textbf{R\'{e}sum\'{e}:} En combinatoire alg\'{e}brique, les treillis sont des ensembles partiellement ordonn\'{e}s qui poss\`{e}dent \`{a} la fois des op\'{e}rations inf et sup. L'ordre faible sur les permutations est un exemple classique d'un treillis qui poss\`{e}de une riche structure combinatoire. Cela en a fait un point de d\'{e}part \`{a} partir duquel d'autres objets combinatoires ont \'{e}t\'{e} d\'{e}finis. Pour cette th\`{e}se, nous nous concentrons sur l'\'{e}tude de deux familles diff\'{e}rentes de treillis en relation avec l'ordre faible: les treillis des permutarbres et le~$s$-ordre faible.
    
    La premi\`{e}re partie de la th\`{e}se concerne la th\'{e}orie des quotients de treillis de l'ordre faible en s'appuyant sur le travail de N. Reading. On se concentre sp\'{e}cifiquement sur la famille des quotients des permutarbres de l'ordre faible. En les consid\'{e}rant comme des permutarbres, comme dans le travail de V. Pilaud et V. Pons, nous \'{e}tendons la technologie des vecteurs de crochet des arbres binaires en d\'{e}finissant les vecteurs d'inversion et les vecteurs cubiques. Le vecteur d'inversion capture l'op\'{e}ration de meet de ces treillis tandis que le vecteur cubique permet de les r\'{e}aliser g\'{e}om\'{e}triquement via un plongement cubique. En changeant de point de vue et en \'{e}tudiant ces quotients \`{a} travers les \'{e}l\'{e}ments minimaux de leurs classes de congruence, nous utilisons la description de Coxeter de type A des permutations pour caract\'{e}riser les permutarbres avec l'aide d'automates. Ces automates capturent l'\'{e}vitement de motifs~$ijk$ et/ou~$kij$ impliqu\'{e}s par ces quotients et nous permettent de d\'{e}finir des algorithmes qui g\'{e}n\'{e}ralisent le tri par pile. Dans le cas o\`{u} le quotient correspond \`{a} un treillis cambrien, nous relions nos automates au tri de Coxeter. Nous donnons quelques indications sur le m\^{e}me ph\'{e}nom\`{e}ne pour les groupes de Coxeter de types B et D.
    
    La deuxi\`{e}me partie de cette th\`{e}se d\'{e}coule du travail de V. Pons et C. Ceballos qui ont d\'{e}fini le~$s$-ordre faible sur les arbres~$s$-d\'{e}croissants o\`{u}~$s$ est une s\'{e}quence d'entiers positifs. Dans le cas de~$s=(1,\ldots,1)$, cette d\'{e}finition r\'{e}cup\`{e}re l'ordre faible. Dans leur premier article, les auteurs ont conjectur\'{e} que le~$s$-permuta\`{e}dre pouvait \^{e}tre r\'{e}alis\'{e} dans l'espace comme une subdivision poly\'{e}drale d'un zonotope. Nous donnons une r\'{e}ponse positive \`{a} leur conjecture lorsque~$s$ est une s\'{e}quence d'entiers positifs en d\'{e}finissant un graphe dont leur polytope de flot nous permettent de r\'{e}cup\'{e}rer le~$s$-ordre faible. Nous utilisons des techniques de flots sur les graphes, de g\'{e}om\'{e}trie discr\`{e}te et de g\'{e}om\'{e}trie tropicale pour obtenir des r\'{e}alisations du~$s$-permuta\`{e}dre avec diff\'{e}rentes propri\'{e}t\'{e}s. Finalement, nous introduisons une op\'{e}ration sur les graphes pour d\'{e}crire les permutarbres et leurs treillis \`{a} travers les polytopes de flot.
    
    \end{multicols}
\end{mdframed}
\vfill
\hspace{0pt}
\pagebreak

\newpage

\thispagestyle{empty}

\small

\hspace{0pt}
\vfill
\begin{mdframed}[linecolor=Prune,linewidth=1]

    \textbf{Title:} Combinatorics of Permutreehedra and Geometry of~$s$-Permutahedra


\vspace{-.5cm}
\begin{multicols}{2}
\noindent \textbf{Abstract:} In algebraic combinatorics, lattices are partially ordered sets which possess both meet and join operations. The weak order on permutations is a classical example of a lattice that has a rich combinatorial structure. This has made it a starting point from which other combinatorial objects have been defined. For this thesis, we focus on studying two different families of lattices in relation to the weak order: the permutree lattices and the~$s$-weak order.

The first part of the thesis involves the theory of lattice quotients of the weak order building upon the work of N. Reading, specifically focusing on the family of permutree quotients of the weak order. Considering them as permutrees, as done by V. Pilaud and V. Pons, we extend the technology of bracket vectors from binary trees by defining inversion and cubic vectors. The inversion vector captures the meet operation of these lattices while the cubic vector helps realize them geometrically via a cubical embedding. Changing our point of view and studying these quotients through the minimal elements of their congruence classes, we use the Coxeter Type A description of permutations to characterize permutrees using automata. These automata capture the pattern avoidance of the patterns~$ijk$ and/or~$kij$ implied by these quotients and allow us to define algorithms which generalize stack sorting. In the case where the quotient corresponds to a Cambrian lattice we relate our automata with Coxeter sorting. We give some insight about the same phenomenon for Coxeter groups of types B and D.

The second part of this thesis stems from the work of V. Pons and C. Ceballos who defined the~$s$-weak order on~$s$-decreasing trees where~$s$ is a sequence of non-negative integers. In the case of~$s=(1,\ldots,1)$ this definition recovers the weak order. In their first article, the authors conjectured that the~$s$-permutahedron could be realized in space as a polyhedral subdivision of a zonotope. We give a positive answer to their conjecture when~$s$ is a sequence of positive integers by defining a graph whose flow polytope allows us to recover the~$s$-weak order. We use techniques from flows on graphs, discrete geometry, and tropical geometry to obtain realizations of the~$s$-permutahedron with different properties. Finally, we introduce a graph operation to describe permutrees and their lattices through flow polytopes.
\end{multicols}
    
    \end{mdframed}
\vfill
\hspace{0pt}
\pagebreak

%% file: includes/Quotes.tex

\thispagestyle{empty}

\enskip{}

\enskip{}

\begin{center}
- ¿Sabe?, soy un extranjero.
    
- Esta ciudad está llena de nosotros, ¿no? Yo mismo soy uno.

- Buscando algo que falta. Echando de menos algo dejado atrás.
    
- Quizá, con buena suerte, encontraremos lo que nos eludió en los lugares que una vez llamamos hogar.
\end{center}
\vspace{-0.7cm}
\hrulefill{}
\vspace{-0.4cm}
\begin{flushright}
    The French Dispatch. Wes Anderson, 2021.
\end{flushright}

\enskip{}

\enskip{}

\enskip{}

\enskip{}

\enskip{}

\enskip{}

\enskip{}

\enskip{}

\begin{center}
- I’m a foreigner, you know.

- This city is full of us, isn’t it? I’m one, myself.

- Seeking something missing. Missing something left behind.

- Maybe, with good luck, we’ll find what eluded us in the places we once called home.
\end{center}
\vspace{-0.7cm}
\hrulefill{}
\vspace{-0.4cm}
\begin{flushright}
    The French Dispatch. Wes Anderson, 2021.
\end{flushright}

\enskip{}

\enskip{}

\enskip{}

\enskip{}

\enskip{}

\enskip{}

\enskip{}

\enskip{}

\begin{center}
- Vous savez, je suis un étranger.

- Cette ville est pleine de nous, n’est-ce pas? Je suis un moi-même.

- Cherchant quelque chose qui manque. Manquant quelque chose laissé derrière.

- Peut-être, avec de la chance, nous trouverons ce que nous a échappé dans les endroits que nous avons appelés autrefois chez nous.
\end{center}
\vspace{-0.7cm}
\hrulefill{}
\vspace{-0.4cm}
\begin{flushright}
    The French Dispatch. Wes Anderson, 2021.
\end{flushright}

%% file: includes/Acknowledgements.tex
\begin{small}
\chapter*{Acknowledgements}
\markboth{Acknowledgements}{Acknowledgements}

\pagestyle{fancy}

Hi there! It has been quite a ride in which 4 years have passed since I left Colombia and started living on the other side of the Atlantic. It is incredible that this journey is coming to an end, time went too slow and too fast somehow. I have had the immense fortune of sharing these years with plenty of amazing people that gave this adventure much of its meaning  and plenty of experiences that allowed me to grow and enjoy life more than I ever imagined. I dedicate this thesis to you all, for all your laughs, smiles, support, and hugs during these times. As tradition dictates, I'll try to find the impossible words to give my gratitude to you all as concisely as possible. To those that I forget: I'm terribly sorry, my spirit is willing, but my memory is weak! To those that think we shared more than what I will write: you are probably right, and we have then to repeat those times!

Firstly, I am most grateful to my advisors Viviane Pons and Vincent Pilaud for betting on me, helping me hop to a new continent, and providing all their support throughout all our mathematical discussions, administrative complications, and the professional and French culture teachings you gave me. Their dedication and passion to mathematics together with their humane spirit motivated me plenty during my time as a PhD student. It was an honor and I only wish I had done more. Thank you as well for your guidance and patience following all my questions, failures, and successes during our research. 

Special thanks Samuele Giraudo and Torsten Mütze for having accepted taking on the immense task of being my thesis reviewers. It is thanks to them and their pinpoint accurate comments that my thesis is as you see it right now. I extend these thanks to Nathalie Aubrun, Mathilde Bouvel, Jean-Christophe Novelli, Francisco Santos Leal, and Florent Hivert for being part of my thesis jury and traveling all the way to far away lands of Gif-sur-Yvette. I am also incredibly grateful to Eva Philippe and Loïc Le Mogne that accepted to proofread previous versions of my thesis.

An infinite amount of thanks go to the power couple Balthazar Charles and Monica Garcia for all the friendship, crazy adventures, support, and fun of these 4 years. You made me feel welcome and understood throughout all this time in this unknown land. The PhD was certainly brighter thanks to you guys. In the same way, eternal thanks to the best flatmate Agustin Borgna, who shared a good portion of this journey with me. The colocation with you was something that transcended a shared space, it felt like home. Thank you for all the crazy and amazing times we had together and teaching me about cooking, music, myself (although I still think you are a lake), programming, bikes, and many other things. I'll remember the times with you and our friendship wherever I go. A colossal bunch of thanks to Katherine Morales who was by my side during all the process of preparing and writing this thesis. Your support, companionship, and love made these months a complete joy. I'm glad to have you by my side, to what we have shared, and for what is to come. Also, for the opacity trick.

From the bottom of my heart I thank my mother Luz Marina Jiménez that has been there since the beginning of this journey and has always supported me with the same love no matter how far away I was from Bogotá. I always felt at home whenever we talked no matter the country I was in. Thank you for your wisdom, enthusiasm, and love. To my father Gerardo Tamayo whose memory guided and helped me in more than one occasion. Thank you both for everything. To my family and in particular to my brother Gerardo E. Tamayo and my aunts and uncles Myriam, Nina, and José Vicente Jiménez and Maria Ester, Chepe, and Javier Tamayo whose support, messages, and meetups in Colombia helped me not forget the warmness of a Colombian family. To Ricardo Ojeda for his delicious meals that remind me of Colombia and his joyful and wise personality.

Thanks to the many friends and colleagues I had the fortune to share with around the world. To Eva Phillipe for our friendship together with this last year of conferences, papers, talks, dancing, and all else, I'll miss those Tuesdays/Thursdays! To Doriann Albertin for his crazy and happy attitude and showing me all the fun to be had in France and even Colombia. To my little academic brother Loïc Le Mogne for the laughs and fun we had in all the countries we visited. To Hugo Mlodecki for all his invitations to plans, parties, and fun discussions about french life. To Justine Falque for her comradeship and friendship throughout our PhD, unemployed, and reconversion lives. To Chiara Mantovani for her joyous smiling personality and making our trip to Bertinoro and Forlimpimpompoli incredible! To Germain Poullot for his wise words, camaraderie, and patience with my stupid ideas. To Noémie Cartier for all our discussions and all the enriching stories/experiences you always shared with us. To Clément Cheneviere for being amazing all around and sharing his happy and curious spirit whenever present, it was a joy seeing you in all the events we crossed each other! To Nicolás Bitar, for bringing more latin life to our lab and his friendship. 

To Jonathan Niño, Santiago Estupiñan, and Catalina Largacha for all your support from across the ocean and the video calls full of laughter, jokes, and care, it was a blast. Also thanks Jonathan for having shown me the joys of dancing and climbing 4 years ago, they were crucial for me during these years. To Carolina Benedetti for her support and advice while I prepared to leave Colombia for my internship and then my PhD in France. To David Jaramillo with whom we grew so much during our masters' years and prepared me for this experience. To Sylvi Hassan and Pierre Jeandenand for an incredible Colombian friendship in Paris and all the support and help you've given me these years, Paris would not have been the same without you. To David Cardozo for a friendship across time and space, I'm really happy to have found you again in Montreal. To Jose Bastidas for the full Montreal experience that made me fell in love with the city, the amazing parties we had together, and the critical eye he had when we discussed mathematics. To Carlo Buccisano and Paulo Soto for all the experiences we have shared living together this last year. Also thanks Carlo, Sofía Lopez, David Llerena, and Micaela Manosalvas for all the climbing, boardgames, and parties we had. To Juan Carlos Restrepo for all we recovered in Colombia. To Jerónimo Valencia, Danai Deligeorgaki, and Sergio Fernandez for all the crazy and stupid adventures we had during our conferences! I extend these thanks to Nicolás Cuervo, Sergio Cristancho, Jorge Olarte, Hans Schmedling, Simón Bolivar (Moncho), and all of my colleagues at I saw in ECCO and other conferences for making Colombia and Colombians still feel like home. I extend this extension to Ángela Mesa, Gloria Buriticá, Julián Chitiva, Juan Numpaque, and Maria Camila Archila for their Colombian friendship in Paris. To Elías Masquil and Nicolas Violante for reminding me how cool Uruguayan people are! To Alex De la Concha, Any Escobar, Majo Castellano, Quentin Desgrousilliers, Eli Zuñiga and François Rudelle for all the laughs, support, parties, and meals we have shared this last year.

To Rafael Gonzáles D'León for his friendship, incredible mathematical optimism, and plenty of wisdom ever since that conference in Uruguay. To Martha Yip for her incredible drawings and what was possibly the best research trip I'll do in my life. To Alejandro Morales for his heartwarming way of doing mathematics, support, and showing me the wonderful program of ipe allowing me to make this thesis' images! To Yannic Vargas and his amazing bibliographic memory, every single chat with you was enriching be it mathematical or not. To Cesar Ceballos and Jean-Philippe Labbé for their warm welcome and help in Germany and Austria at the start of my thesis. I'll never forget your generosity and support. To Cesar also thanks for his professional and personal wisdom and the Weissensee workshop. It was the most beautiful academic trip I could have imagined.

Special thanks to Taha Halal and Anis Ghaoui, I had a great time hanging out with you crazy people be it during CARaDOC or anytime after it! I'm grateful also to the rest of the CARaDOC crew and to the JDD 2020 crew, I'm glad I was part of the organization of these events with you all. Thanks as well to Julian Ritter, Jonas Sénizergues, Pierre Béaur and Valentin Dardilhac for all the laughs and good times in our lab. Also thanks to Philippe Houdbine, Victor, Aleix, Anna, Agustín, Caroline Lucas, Magin Ferrer, Dongho Lee, Dora Novak, Carlo and David together with the people of ScienceAccueil, those french classes were always a blast.

I want to thank Florent Hivert, Nicolas Thierry, Benjamin Hellouin, Nathalie Aubrun, François Pirot, Johanne Cohen, Yannis Manoussakis, and the rest of the GALaC team for making me feel welcome at the LRI/LISN and all the french training I had listening to and interacting with you all! I extend these thanks to the LIX team for all the seminars and cafés and the LIGM team for their seminars and FPSAC revisions which were always plenty of fun. Thanks as well to Christophe Hohlweg, François Bergeron, and the LACIM team for welcoming me during my stay in Montreal and many interesting discussions. I am grateful also to Florent, Nicolas, François, Adeline Pierrot, Joffroy Beauquier, and Nicole Bidoit under which I gave courses during my PhD and demi-ATER. I basically did a minor area in informatics during my time here thanks to them. 

Thank you all. Although in the future I might not remember all of what I lived during these 4 years, I will remember I was happy.

\end{small}

\begin{small}

\chapter*{Agradecimientos}

\markboth{Agradecimientos}{Agradecimientos}

\thispagestyle{plain}

Qué tal! Ha sido un viaje bien largo en el cual 4 años han pasado desde que salí de Colombia y empecé a vivir al otro lado del Atlántico. Me parece increíble que esta aventura esté llegando a su fin y que el tiempo haya pasado a la vez tan lento y tan rápido. Afortunadamente he tenido la inmensa oportunidad de compartir estos años con un montón de gente increíble que le dio a esta odisea mucho de su significado y a la vez, un montón de experiencias que me permitieron crecer y disfrutar la vida como nunca hubiera imaginado. Les dedico esta tesis a todos ustedes, por todas sus carcajadas, sonrisas, apoyo y abrazos durante estos tiempos. Como dicta la tradición, trataré de hacer lo imposible y encontrar las palabras para darles mi gratitud de la manera más concisa. A aquellos que olvido: ¡lo siento mucho, el espíritu está dispuesto, pero la memoria es débil! A aquellos que piensan que compartimos y nos debemos más de lo que escribiré: les doy toda la razón, ¡y tendremos que repetir esos momentos entonces!

En primer lugar, estoy sumamente agradecido con mis asesores Viviane Pons y Vincent Pilaud por apostar por mí ayudándome a saltar a un nuevo continente y brindándome todo su apoyo a lo largo de nuestras discusiones matemáticas, complicaciones administrativas y enseñanzas profesionales y culturales sobre Francia. Su dedicación y pasión por las matemáticas junto con su espíritu humano me motivaron mucho durante mi tiempo como su estudiante. Fue un honor y solo desearía haber hecho más. Gracias también por su orientación y paciencia con todas mis preguntas, fracasos y éxitos durante nuestra investigación.

Estoy especialmente agradecido con Samuele Giraudo y Torsten Mütze por haber aceptado asumir la inmensa tarea de ser los revisores de mi tesis. Es gracias a ellos y sus comentarios precisos que pueden ver mi tesis tal cual como está ahora. Extiendo estos agradecimientos a Nathalie Aubrun, Mathilde Bouvel, Jean-Christophe Novelli, Francisco Santos Leal y Florent Hivert por haber formado parte de mi jurado de tesis y haber viajado hasta los confines lejanos de Gif-sur-Yvette. Estoy inmensamente agradecido con Eva Philippe y Loïc Le Mogne por haber leido y revisado versiones anteriores de mi tesis.

¡Un agradecimiento infinito a la poderosa pareja Balthazar Charles y Monica Garcia por toda la amistad, aventuras, apoyo y locuras que compartimos durante estos 4 años. Ustedes me hicieron sentir bienvenido y comprendido a lo largo de todo este tiempo en esta tierra desconocida. Sin lugar a dudas el doctorado fue más alegre gracias a ustedes. De la misma forma le day gracias infinitas al mejor compañero de piso Agustín Borgna que compartió una buena parte de este camino conmigo. La colocación contigo fue algo que trascendió compartir un espacio, se sintió como un hogar. Gracias por todo el tiempo increíble lleno de experiencias que pasamos juntos y por enseñarme sobre cocina, música, yo mismo (aunque todavía creo que eres un lago), programación, bicicletas y muchas otras cosas. Recordaré los tiempos contigo y llevaré nuestra amistad a donde quiera que vaya. Le doy una cantidad colosal de gracias a Katherine Morales que estuvo a mi lado durante todo el proceso de preparación y escritura de esta tesis. Su apoyo, alegría y amor hicieron que estos meses fueran una dicha completa. Me alegro de tenerte a mi lado, por lo compartido y por lo que vendrá! También, por el truco de la opacidad.

Del fondo de mi corazón le doy gracias a mi madre Luz Marina Jiménez que ha estado ahí desde el comienzo de este viaje y siempre me ha apoyado con el mismo amor sin importar cuán lejos yo estuviese de Bogotá. Siempre me sentí en casa cada vez que hablábamos sin importar el país en el que estuviera. Gracias por tu sabiduría, entusiasmo y amor. A mi padre Gerardo Tamayo cuya memoria me guió y ayudó en más de una ocasión. Gracias a ambos por todo. A mi familia y en particular a mi hermano Gerardo E. Tamayo y mis tías y tíos Myriam, Nina y José Vicente Jiménez y María Ester, Chepe y Javier Tamayo cuyo apoyo, mensajes y reuniones en Colombia me ayudaron a no olvidar el calor de una familia colombiana. A Ricardo Ojeda por sus comidas deliciosas que me recuerdan a Colombia y a su personalidad alegre y llena de sabiduria.

Gracias a los muchos amigos y colegas con los que tuve la fortuna de compartir alrededor del mundo! A Eva Phillipe por nuestra amistad y en particular este último año de conferencias, artículos, charlas, bailes y todo lo demás, ¡extrañaré esos martes/jueves! A Doriann Albertin por su personalidad loca y alegre y por mostrarme toda la diversión que se puede tener en Francia e incluso en Colombia. A mi hermano menor académico Loïc Le Mogne por las risas y la diversión que tuvimos en todos los países que visitamos. A Hugo Mlodecki por todas sus invitaciones a planes, fiestas y discusiones divertidas sobre la vida francesa. A Justine Falque por su camaradería y amistad a lo largo de nuestro doctorado, desempleo y ahora reconversión. A Chiara Mantovani por su personalidad alegre y sonriente y por hacer nuestro viaje a Bertinoro y Forlimpimpompoli increíble! A Germain Poullot por sus sabias palabras, camaradería y paciencia frente a mis ideas estúpidas. A Noémie Cartier por todas nuestras discusiones y todas las historias y experiencias enriquecedoras que siempre compartía con nosotros. A Clément Cheneviere por ser increíble en todo sentido y compartir su espíritu jovial y curioso siempre que estaba presente! Fue una alegría verte en todos los eventos en los que nos cruzamos. A Nicolás Bitar, por traer más vida latina a nuestro laboratorio y su amistad!

A Jonathan Niño, Santiago Estupiñan y Catalina Largacha por todo su apoyo desde el otro lado del océano y las videollamadas llenas de risas, chistes y cariño, fue increíble. También gracias Jonathan por haberme mostrado las alegrías del baile y la escalada hace 4 años, fueron cruciales para mí en esta odisea. A Carolina Benedetti por su apoyo y consejos mientras me preparaba para salir de Colombia para mi pasantía y luego mi doctorado en Francia. A David Jaramillo con a quien durante nuestro años de maestría crecimos tanto y me preparó para esta experiencia. A Sylvi Hassan y Pierre Jeandenand por una increíble amistad colombiana en París y todo el apoyo y ayuda que me han brindado estos años. París no hubiera sido lo mismo sin ustedes. A David Cardozo por una amistad a través del tiempo y el espacio, estoy muy feliz de haberte reencontrado en Montreal. A Jose Bastidas por la experiencia completa de Montreal que me hizo enamorarme de la ciudad, las increíbles fiestas que tuvimos juntos y el ojo crítico que tenía cuando discutíamos matemáticas! A Carlo Buccisano y Paulo Soto por todas las experiencias que compartimos viviendo juntos este último año. También gracias a Carlo, Sofía López, David Llerena y Micaela Manosalvas por toda la escalada, juegos de mesa y fiestas que compartimos. A Juan Carlos Restrepo por todo lo que recuperamos en Colombia. A Jerónimo Valencia, Danai Deligeorgaki y Sergio Fernández por todas las aventuras locas y estúpidas que tuvimos durante nuestras conferencias. Extiendo estos agradecimientos a Nicolás Cuervo, Sergio Cristancho, Jorge Olarte, Hans Schmedling, Simón Bolivar (Moncho) y todos mis colegas en ECCO y otras conferencias por hacer que Colombia y los colombianos todavía se sientan como un hogar. Extiendo esta extensión a Ángela Mesa, Gloria Buriticá, Julián Chitiva, Juan Numpaque y María Camila Archila por su amistad colombiana en París. A Elías Masquil y Nicolás Violante por recordarme lo genial que es la gente Uruguaya! A Alex De la Concha, Any Escobar, Majo Castellano, Quentin Desgrousilliers, Eli Zuñiga y François Rudelle por todas las risas, apoyo, fiestas y comidas de este ultimo año.

A Rafael Gonzáles D'León por su amistad, su increíble optimismo matemático y tanta sabiduría que me impartió desde aquella conferencia en Uruguay. A Martha Yip por sus increíbles dibujos y lo que posiblemente sea el mejor viaje de investigación que haré en mi vida. A Alejandro Morales por su manera tan sensata y pura de hacer matemáticas, su apoyo y haberme mostrado el increible programa ipe que me permitió hacer las imagenes de esta tesis! A Yannic Vargas y su increíble memoria bibliográfica, cada charla contigo matemática o no fue muy enriquecedora. A Cesar Ceballos y Jean-Philippe Labbé por su cálida bienvenida y ayuda en Alemania y Austria al comienzo de mi tesis. Nunca olvidaré su generosidad y apoyo. A Cesar gracias también por su sabiduría profesional y personal y el taller de Weissensee, fue el viaje académico más hermoso que pude haber imaginado.

Un agradecimiento especial a Taha Halal y Anis Ghaoui, ¡fue una locura parchar con ustedes ya sea durante CARaDOC o en cualquier momento después! También estoy agradecido con el resto del equipo de CARaDOC y con el equipo de JDD 2020, me alegra haber sido parte de la organización de estos eventos con todos ustedes. Gracias también a Julian Ritter, Jonas Sénizergues, Pierre Béaur y Valentin Dardilhac por todas las risas y buenos momentos en nuestro laboratorio. De igual forma gracias a Philippe Houdbine, Victor, Aleix, Anna, Agustín, Caroline Lucas, Magin Ferrer, Dongho Lee, Dora Novak, Carlo y David junto con la gente de ScienceAccueil, esas clases de francés siempre fueron espectaculares.

Quiero agradecerle a Florent Hivert, Nicolas Thierry, Benjamin Hellouin, Nathalie Aubrun, François Pirot, Johanne Cohen, Yannis Manoussakis y el resto del equipo GALaC por hacerme sentir bienvenido en el LRI/LISN y todo el entrenamiento en francés que tuve escuchando e interactuando con todos ustedes! Extiendo estos agradecimientos al equipo del LIX por todos los seminarios y cafés y al equipo del LIGM por sus seminarios y revisiones para FPSAC que siempre fueron muy entretenidas. Gracias también a Christophe Hohlweg, François Bergeron y el equipo de LACIM por recibirme durante mi estadía en Montreal y muchas discusiones interesantes. También estoy agradecido con Florent, Nicolas, François, Adeline Pierrot, Joﬀroy Beauquier y Nicole Bidoit bajo quienes di cursos durante mi doctorado y demi-ATER. Básicamente hice un área menor en informática durante mi tiempo aquí gracias a ellos.

Gracias a todos. Aunque en el futuro puede que no recuerde todo lo que viví durante estos 4 años, recordaré que fui feliz.

\end{small}

%% file: includes/IntroductionEN.tex


\pagenumbering{roman}

\chapter*{Introduction}\label{part:IntroductionEN}
\addcontentsline{toc}{part}{Introduction (English)}
\addcontentsline{lof}{part}{Introduction (English)}
\markboth{Introduction}{Introduction}

\pagestyle{fancy}

\section*{Context}

This thesis finds its place in the domain of combinatorics, somewhere in the interplay between algebraic combinatorics and geometric combinatorics. That is, it relies on connections between families endowed with algebraic operations such as finite groups or partially ordered sets and structures from discrete geometry such as polytopes.

Combinatorics in it of itself is interested in taking discrete objects and studying patterns within them. Most classically, it can be seen as an area interested in counting objects and phenomena or conversely, finding an object or property that describes a sequence of numbers. As such, it pops up in many distinct areas of mathematics either as a principal or secondary actor. Appearances of combinatorics can be found in algebraic topology~\cite{MK46}~\cite{K08}, number theory~\cite{TV06}, and even theoretical and statistical physics~\cite{T20}~\cite{R69} just to mention a few. Due to its nature of studying discrete structures, combinatorics has found a strong connection with computer science. For example, the study on algorithms and their complexities or  optimizations is inherently combinatorial. Reciprocally, many combinatorial results come from algorithmic analysis and computer exploration. More concretely, areas like integer programming and optimization, graph theory, and sorting algorithms find themselves using ideas from combinatorics and informatics symbiotically. Our work is no stranger to this, as much if not all of it, has been influenced on experiments implemented in the open source software SageMath~\cite{SAGE}.

\subsection*{Weak Order}

The combinatorial family at the very core of our work is that of \defn{permutations}. These structures make part of some of the most simple objects in combinatorics. Namely, a permutation of size~$n$ is a way of taking~$n$ objects in order and rearranging them in a new order. As such, they can also be thought as bijections from~$[n]:=\{1,2,\ldots,n\}$ to itself. This point of view endows permutations with an algebraic structure where the multiplication of permutations is simply the composition of their corresponding bijections. This forms an algebraic structure called the \defn{symmetric group} where each permutation is presented by a rearrangement of the numbers~$1$ through~$n$. A quick way of distinguishing permutations is by verifying which pairs~$(i,j)$ have been inversed. Any pair that finds itself in this situation is said to be an \defn{inversion}. The set of such pairs is called the \defn{inversion set} of a permutation and defines a way to give a partial order on permutations through the containment of their inversion sets. This is called the \defn{weak order} on permutations~\cite{B22}.

Adjacent permutations under the weak order can be expressed through the inversion of a single pair~$(i,i+1)$. We say that the permutations that correspond to these rearrangements are the \defn{adjacent transpositions} of the weak order. This is a crucial step in showing that given two permutations within the weak order, it is always possible to find a unique maximal (resp.\ minimal) permutation that is smaller (resp.\ larger) than both of them. These operations turn the weak order from being a partially ordered set (i.e.\ a \defn{poset}) into an algebraic structure called a \defn{lattice}~\cite{GR63}.

Each permutation~$\sigma$ of~$[n]$ can be associated with the point~$(\sigma(1),\ldots,\sigma(n))$ in the space~$\RR^n$. If one takes the convex hull of this configuration of points, the resulting polytope is called the \defn{permutahedron} whose properties reflect phenomena of the weak order~\cite{S911}~\cite{GG77}. For example, orienting the permutahedron in a particular direction, the directed graph consisting on its~\defn{$1$-skeleton} (i.e.\ vertices and edges) corresponds precisely with the order described by the weak order. Moreover, this process of obtaining the permutahedron shows that its faces are indexed by the ordered partitions of~$[n]$. Another geometrical way of presenting the permutahedron and the weak order structure is by partitioning the inversion sets of permutations via the smallest element that was inversed and taking the cardinalities of these sets. The resulting structure is an \defn{embedding} of the permutahedron onto a cube~\cite{BF71}~\cite{RR02}. 

\subsection*{Tamari Lattice}

Another family of combinatorial objects from which much of our work is based on are \defn{binary trees}. These consist of rooted trees where each node can be said to have one parent and two children. Binary trees are counted by one, if not the most prolific sequence of numbers, called the \defn{Catalan numbers}~\cite[A000108]{OEIS}, and thus are in bijection with a myriad of combinatorial objects~\cite{S15}.

Given a binary tree with~$n$ nodes, we can label its vertices via an anti-clockwise walk on the graph by labeling with~$i$ the~$i$-th vertex that we visit for a second time in our traversal. This is called the \defn{inorder} of binary trees and the resulting labeling has the characteristic that for any vertex, its label is greater than the labels in its left subtree and smaller than the labels in its right subtree. This allows us to define the \defn{rotation} of an edge~$i\to j$ to an edge~$j\to i$ where the structure of the subtrees below~$i$ and~$j$ is maintained. Rotations are a classical operation used for balancing binary search trees (i.e.\ making their height as small as possible) to obtain efficient sorting algorithms~\cite{AL62}. These rotations define an order on all binary trees called the \defn{Tamari lattice}~\cite{T62}.

To get a geometrical structure out of binary trees, one needs to consider vectors where each coordinate~$i$ is the product of the number of leaves in the left subtree times the leaves in the right subtree for a vertex labeled~$i$. The resulting convex hull of these vertices is called the \defn{associahedron}~\cite{L04} and its faces are indexed by Schr\"oder trees~\cite{S11}. As before, orienting the associahedron in a particular direction lets us find the Tamari lattice through its~$1$-skeleton. At this point a relation between binary trees and permutations starts to appear as the associahedron is a \defn{removahedron}. That is, the associahedron can be obtained simply by removing certain facets of the permutahedron~\cite{SS93}. As with permutations, changing the coordinates corresponding to binary trees to quantities derived from which and how many rotations it has endured, one obtains a new set of vectors called \defn{bracket vectors}. These vectors allow for a constructive proof for the lattice property of the Tamari lattice~\cite{HT72} and also a cubical embedding of the associahedron~\cite{K93}. These techniques have seen use in generalizations and structures related to the Tamari lattice~\cite{CPS20}~\cite{C21}~\cite{FMN21}~\cite{C22}~\cite{CG22}~\cite{PP23}.

Still, a more direct relation between binary trees and permutations can be found through the insertion algorithm of~\cite{T97}~\cite{HNT05}. As such, the fibers of binary trees under this algorithm are intervals of permutations whose minimal elements avoids the pattern~$312$. These fibers coincide with the fibers of the Stack-sorting algorithm~\cite{K73} and form a congruence relation in the weak order that respects the operations of meets and joins. Such congruences receive the name of \defn{lattice congruences} and in this case lead to the Tamari lattice being found as the order induced by the minimal permutations in these fibers. This lattice congruence is known as the \defn{sylvester congruence}. 

\subsection*{Permutrees}

The combinatorial family that motivates this thesis and that we study from different perspectives is that of \defn{permutrees}~\cite{PP18}. This family is general enough to encode permutations, binary trees, \defn{Cambrian trees}~\cite{LP13}~\cite{CP17}, and binary sequences while still defining new types of trees. A permutree consists of an unrooted directed tree with labeled nodes in~$[n]$ where each node can be said to have one or two parents and one or two children while each label of a vertex satisfies a relation similar to the one of binary trees via the inorder. As such, permutrees are characterized by how many parents and children each node has which is called the \defn{decoration} of the node. As the vertices are labeled, this groups permutrees into~$\delta$-permutrees where~$\delta$ is a vector of decorations. 

In a similar vein to~\cite{HNT05}, the insertion algorithm of~\cite{PP18} gives a surjection from permutations onto~$\delta$-permutrees for each possible decoration. These fibers describe a lattice congruence called the~\defn{$\delta$-permutree congruence}. The fibers of~$\delta$-permutrees under this algorithm are intervals of permutations whose minimal element avoids the pattern~$kij$ and/or~$jki$ for each~$j\in\{2,\ldots,n-1\}$ and $1\leq i<j<k\leq n$ depending on the decoration~$\delta_j$.

Like for binary trees,~$\delta$-permutrees have rotations that change their local structure at the level of a single edge while maintaining the rest of the tree intact. These rotations define the~$\delta$-permutree rotation posets. These posets are actually shown to be lattices in~\cite{PP18}, but the proof uses lattice quotients to show that this poset is a sublattice of the weak order. In general, they are called~\defn{$\delta$-permutree lattices} and generalize the type~$A$ Cambrian lattices of~\cite{R06}.

No matter the decoration, permutrees can be assigned to a vector whose coordinates correspond to a manipulation of the number of nodes at their right and left subtrees. This gives rise to a polytope called the~\defn{$\delta$-permutreehedron}. Orienting these polytopes in a particular direction recovers their corresponding~$\delta$-permutree lattice. Although~$\delta$-permutrees and their permutree congruences may be different, for certain subsets of decorations their lattices are isomorphic and for others their permutreehedra are combinatorially equivalent.

\subsection*{Coxeter Groups}

Our ideas of permutations and the weak order are only but a part of a bigger scheme which is described by \defn{Coxeter groups}. Introduced in~\cite{C34} and later completely classified for the finite case in~\cite{C35}, Coxeter groups describe groups generated by \defn{simple reflections} coming from hyperplane arrangements~\cite{H90}. As such, elements are sequences of simple reflections called \defn{words} and reflections correspond to \defn{inversions}. Said inversions form inversion sets which in turn define the \defn{weak order} of a Coxeter group. Thus, the permutations are just a particular case of a Coxeter group called Coxeter groups of type~$A$. Similar combinatorial families can be found for other types such as \defn{signed permutations} and \defn{even signed permutations} for Coxeter groups of types~$B$ and~$D$ respectively~\cite{BB06}. Other groups do not have such combinatorial descriptions to our knowledge. Still, the nature of Coxeter groups allows for them to be studied via algebraic, geometric, combinatorial or language theoretic means. This has lead for properties of the language of reduced words in Coxeter groups being described via automata~\cite{BH93}~\cite{HNW16}.

Going further than generalizing permutations and their properties, the context of Coxeter groups gives a broader space in which lattice congruences can be defined. This was done considering Coxeter groups as the poset of regions of hyperplane arrangements in~\cite{R04}. Afterwards, the concept of~\defn{$c$-sorting} (also known as \defn{Coxeter sorting}) was introduced in~\cite{R07a} giving~\defn{$c$-sortable elements} which were shown in~\cite{R07b} to be the minimal elements of \defn{Cambrian congruences}. These results were unified for all finite of Coxeter groups regardless of type in~\cite{RS11} and then compended in~\cite{R12}. In this situation the sylvester congruence is a Cambrian congruence in type~$A$ and its~$c$-sortable elements coincide with the stack-sortable permutations. Moreover, all Cambrian congruences of type~$A$ are permutree congruences.

\subsection*{\texorpdfstring{$s$}{}-Weak Order}

Other possible generalizations of the weak order come from taking multipermutations instead of permutations. That is, given a sequence of positive integers~$r=(r_1,\ldots,r_n)$ an~\defn{$r$-permutation} is a rearrangement of the word~$1^{r_1}\cdots n^{r_n}$ where~$i^{r_i}$ is the repetition of the letter~$i$ a total of~$r_i$ times. Coming from a geometrical setting in~\cite{RR02}, they were used to describe embeddings of the~\defn{combinohedron} which was known to come from a lattice structure called the \defn{multinomial lattice}~\cite{BB94}. Independently in a more algebraic setting, the case where all~$r_i=m$ for an~$m\geq 1$ was introduced in~\cite{NT20} to study the~\defn{$m$-sylvester congruence}. Renaming~$r$ as~$k$,~$k$-permutations that avoid the pattern~$121$ are called~\defn{$k$-Stirling permutations}. These permutations have seen ample interest and many of their properties have been determined through their statistics and bijections with other combinatorial families~\cite{P94a}~\cite{P94b}~\cite{P94c}~\cite{KP11}~\cite{JKP11}~\cite{RW15}~\cite{G19}.

Still, apart from all these constructions there is a more recent one linked to them that we are interested in, that is, the~\defn{$s$-weak order}~\cite{CP19}~\cite{CP22}. Taking a weak composition~$s=(s_1,\ldots,s_n)$ (i.e.~$s_i\in\ZZ_{\geq 0}$),~\defn{$s$-decreasing trees} consist of labeled rooted trees with~$n$ nodes such that each node has exactly one parent and~$s_i+1$ children and all descendants have smaller labels. These trees have \defn{inversions} defined from the relative position between nodes. As before, these inversions define an order on~$s$-decreasing trees called the~\defn{$s$-weak order} which is shown constructively to be a lattice and have a geometrical structure called the~\defn{$s$-permutahedron}. Moreover, there is an underlying order called the~\defn{$s$-Tamari lattice} with a geometrical counterpart called the~\defn{$s$-associahedron}. These correspond to the~\defn{$\nu$-Tamari lattice} of~\cite{PV17} and~\defn{$\nu$-associahedron} of~\cite{CPS19} for certain values of~$\nu$. Whenever~$s$ is a composition,~$s$-decreasing trees are in bijection with~\defn{$s$-Stirling permutations}, that is,~$s$-permutations avoiding the pattern~$121$. The~$s$-weak order coincides with the weak order of permutations when~$s_i=1$ for all coordinates of~$s$ and with the \defn{metasylvester lattice} of~\cite{P15} when~$s_i=m$ with~$m\geq 1$ for all coordinates of~$s$. 

\subsection*{Flow Polytopes}

A family of polytopes that we consider for a good part of our work because of their versatility is that of flow polytopes. They come from a loopless directed graph where each vertex is equipped with an integer that denotes the \defn{netflow} passing through it. That is, the difference between the incoming and the outgoing flow determined respectively by the incoming and outgoing edges of each vertex must equal this netflow. This limits the possible \defn{flows} that can be assigned to the edges of the graph while being coherent with the netflow. Taking flows as points in the space of edges of the graph forms a polytope called the \defn{flow polytope}. As flows and netflows are models of networks, flow polytopes have seen ample use in optimization problems. As such, there has been much research on the combinatorics of these polytopes~\cite{RH70}~\cite{FRD71}~\cite{CG78}~\cite{H03}~\cite{MM19}~\cite{GHMY21}. In particular, the \defn{normalized volume} of the flow polytope decomposes nicely using the \defn{Kostant partition formula} via the \defn{Baldoni–Vergne–Lidskii formulas}~\cite{BV08}. This has given new identities showing of many known numbers via product decompositions as exemplified in~\cite{BGHHKMY19}.

We are interested in flow polytopes specially because of their possible subdivisions. In particular, endowing each vertex of a graph with independent total orders for their incoming edges and outgoing edges (i.e.\ a \defn{framing}), the cliques of coherent routes of the graph give us a triangulation of the flow polytope when the netflow is~$\bfi:=(1,0,\ldots,0,-1)$. This triangulation is called the \defn{DKK triangulation} and comes equipped with a height function that makes it regular~\cite{DKK12}. Another possible subdivision consists on taking a series of \defn{reductions} on the graph which translate as cutting the flow polytope into smaller pieces that are combinatorially equivalent to other flow polytopes. This process gives the \defn{Postnikov–Stanley subdivision}~\cite{P1014}~\cite{S00}. The usefulness of these two subdivisions comes from the fact that when the graph is framed, \defn{framed Postnikov–Stanley subdivisions} become DKK triangulations~\cite{MMS19}. This gives a bijection between simplices of the DKK triangulation of netflow~$\bfi$ and integer points in the flow polytope of netflow~$\bfd$ where~$\bfd_i$ is the shifted indegree of the~$i$-th vertex. This bijection between two different flow polytopes has allowed for the recovering of certain posets as the dual of the interior faces of the DKK triangulation such the~$\nu$-Tamari lattice and principal order ideals of Young's lattice~\cite{BGMY23}.

\subsection*{Tropical Geometry}

\defn{Tropical geometry} comes from changing the usual operations of sums and multiplications respectively for minima and sums and the inclusion of infinity as an element~\cite{J21}. Although it has interesting links with economics~\cite{S15E} and mechanism design~\cite{CT18}, we are more interested in how it relates with classical discrete geometry. The change of operations and base set allows for new definitions of geometrical structures such as \defn{tropical polynomials}, \defn{tropical hypersurfaces}, and \defn{tropical varieties}. These objects have been studied in their own right and have also shown to have strong connections with convex geometry~\cite{J17}. For example, all \defn{point configurations} are in correspondence with tropical polynomials. Also, classical techniques from discrete geometry like the Cayley trick~\cite{S94}~\cite{HRS00}~\cite{DRS10} have seen repeated uses in the tropical context~\cite{DS04}~\cite{FR15}~\cite{J16}~\cite{JL16}~\cite{MS21}.

\section*{Contributions}

With the given context this thesis finds itself as a part of a bigger project that aims to answer the following question. Can we study permutree congruences in all Coxeter types? Although we do not give a complete answer to this question, we give three new points of view from which type~$A$ permutrees can be studied together with  other linked results we achieved using some of these tools for the~$s$-weak order. These contributions are contained in the following articles. \begin{itemize}
    \itemsep0em
    \item D. Tamayo Jiménez. Inversion and Cubic Vectors for Permutrees, 2023. arXiv:2308.05099,
    
    where we gave a direct constructive proof of the lattice property of permutree rotation lattices and a cubical embedding of permutreehedra. 
    \item V. Pilaud, V. Pons, and D. Tamayo Jiménez. Permutree Sorting. Algebraic Combinatorics,
    6(1):53–74, 2023,
    
    where we characterized minimal elements of the permutree classes in type~$A$ through their reduced words using automata and found leads for other Coxeter types ($B$ and~$D$).
    \item R.S. González D’León, A.H. Morales, E. Philippe, D. Tamayo Jiménez, and M. Yip.
    Realizing the~$s$-Permutahedron via Flow Polytopes, 2023. arXiv:2307.03474,
    
    where we gave a positive answer to a conjecture of Ceballos and Pons on the geometric realization of the s-weak order (when~$s$ has no zeros) using polytopal subdivisions of flow polytopes, sums of hypercubes, and tropical geometry.
    \item R.S. González D’León, A.H. Morales, E. Philippe, D. Tamayo Jiménez, and M. Yip. Flow
    Polytopes and Permutree Lattice Quotients of the~$s$-Weak Order. In preparation, 2023+,
    
    where we found a description of permutrees and their rotation lattices using subdivisions of flow polytopes.
\end{itemize}  The first three works have been presented in several international workshops, seminars, and conferences either as posters or presentations. The work in these articles uses ideas from the bracket vectors of binary trees, the language theoretic approach to Coxeter groups together with automata, and the combinatorics of flow polytopes supported by tropical geometry.

\section*{Thesis Outline}

The work presented in this thesis is divided into three parts. Part~\ref{part:prelim} describes the main combinatorial actors and the permutree problematic at the core of our work. Afterwards, Part~\ref{part:permutrees} presents two of our answers to this problematic in Chapters~\ref{chap:permutree_vectors} and~\ref{chap:permutree_sorting}. Finally, Part~\ref{part:sorder} deals with the use of Flow Polytopes in our context. Our contributions in this part are contained in Chapters~\ref{chap:sorder_realizations} and~\ref{chap:sorder_quotients}. Figure~\ref{fig:thesis_outline_EN} shows the dependencies between the contents of the thesis and describes the recommended reading order. The reader is invited to skip Chapters~\ref{chap:prelim_structures},~\ref{chap:prelim_weak_order}, and/or~\ref{chap:sorder_Flows}, if they are already familiar with the corresponding material. We have tried to make this thesis as self-contained as possible and thus the only real requirement to read it is knowledge of linear algebra. Experience with discrete geometry is not required but strongly recommended for intuition purposes. Still, throughout our work we have been invested on producing useful figures to better transmit our ideas, results, and in general, points of view from which we have approached the problems we have studied. We hope they aid the reader in times of need.

\begin{figure}[h!]
    \centering
    \includegraphics[scale=0.882]{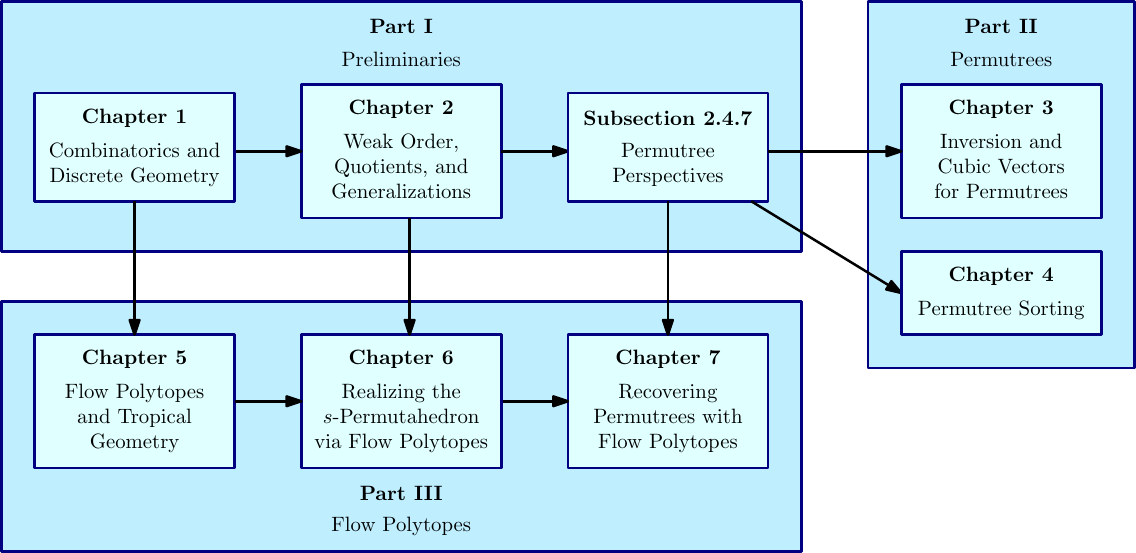}
    \caption{Thesis outline.}
    \label{fig:thesis_outline_EN}
\end{figure}

\subsection*{Preliminaries}

Part~\ref{part:prelim} is formed by two chapters introducing in different generality the main combinatorial concepts we use. Chapter~\ref{chap:prelim_structures} introduces first the concepts of \defn{partial orders}, \defn{lattices}, and \defn{lattice congruences} which are at the core of all parts of our work. Afterwards, we describe the bases of convex geometry including \defn{polytopes}, \defn{fans}, \defn{complexes}, and several classical operations and techniques involving them. Lastly we recall the bases of \defn{automata theory} and its capability of recognizing patterns in words. The convex geometry and automata theory setting gives us an advantageous point of view for the problems we tackle in the thesis. In Chapter~\ref{chap:prelim_weak_order} we present first the \defn{weak order on permutations} together with \defn{binary trees} and \defn{permutrees} while describing their similarities in terms of combinatorics, orders, polytopes, and how they relate to each other through lattice congruences. Afterwards, we recall the generalizations of the weak order onto the \defn{weak order of Coxeter groups} and the~\defn{$s$-weak order}. The generalization to Coxeter groups helps us describe the core problematic of this thesis which revolves around finding new combinatorial families through which we can study permutrees and their lattices (Perspective~\ref{pers:Coxeter_permutrees}).

\subsection*{Permutrees}

In Part~\ref{part:permutrees} we provide two answers to Perspective~\ref{pers:Coxeter_permutrees}. The first, presented in Chapter~\ref{chap:permutree_vectors} is based on the article~\cite{T23} and consists of two generalizations of the bracket vectors of binary trees to permutrees. We begin by defining \defn{inversion sets} for a permutree together with how they can be partitioned to give \defn{inversion vectors}. Like inversion sets for permutations, we characterize these sets in terms of transitivity, cotransitivity, and two additional conditions depending on the permutree decoration at hand (Lemma~\ref{lem:permutree_inversion_sets}). Following the steps of~\cite{HT72} with binary trees, we show that the permutree rotation order can be interpreted via the inclusion of inversion sets (Lemma~\ref{lem:inversion_set_inclusion}) and that the intersection of these sets together with a special condition defines a meet operation on the poset of permutrees (Theorem~\ref{thm:permutree_meet}). This allows us to recover in a constructive manner the result of~\cite{PP18} about the rotation order on permutrees having a lattice structure (Corollary~\ref{cor:permutrees_are_lattices}). Afterwards, we define \defn{cubic sets} and \defn{cubic vectors} for permutrees. This allows us to generalize the cubic embeddings of the permutahedron given in~\cite{BF71} and~\cite{RR02} and the associahedron as in~\cite{K93} for any permutreehedron (Theorem~\ref{thm:cubic_property_embedding}).

The second answer given in Chapter~\ref{chap:permutree_sorting} comes from the article~\cite{PPT23} where we characterize minimal elements of the permutree congruence classes using finite state automata that read reduced words. We start considering the case where the permutree decoration is given by a single oriented edge of the Coxeter diagram of type A. In such setting we define the automata~\defn{$\UU(j)$} and~\defn{$\DD(j)$} that read reduced words corresponding to permutations. These automata are relevant to our permutree setting because the property of a permutation having a reduced word accepted by them is shown to be equivalent to the pattern avoidance characterized by permutrees (Theorem~\ref{thm:pattern_avoidance_single}). After having this result, we devise an algorithm such that given a permutation, it returns a reduced word  accepted by the automaton at play (Algorithm~\ref{algo:permutree_sorting_simple}). We show that the property of this reduced word being a reduced word of the original permutation is equivalent to the permutation being minimal in its permutree congruence class (Corollary~\ref{coro:algorithm_single}). We also present how the set of accepted reduced words generates a tree structure within the Hasse diagram of the weak order (Theorem~\ref{thm:automata_generating_trees_simple}). 

Passing from the single orientations, we move on to the case where any edge of the Coxeter diagram can have at most one orientation.  In this situation we devise the automaton~\defn{$\PP(U,D)$} as the intersection of our previous automata and through it, we extend our previous results to this larger case of permutree congruences. That is, a reduced word being accepted by this automaton is shown to be equivalent to the corresponding permutation avoiding the patterns dictated by the permutree congruence (Corollary~\ref{coro:pattern_avoidance_product}). We define an algorithm (Algorithm~\ref{algo:permutree_sorting_multiple}) transforming a permutation to a reduced word with the property that the fact of the reduced word describes the permutation if and only if the permutation is minimal in its permutree congruence class (Theorem~\ref{coro:algorithm_multiple}). It is also shown that the set of accepted reduced words generates a tree structure within the Hasse diagram of the weak order (Theorem~\ref{thm:automata_generating_trees_multiple}).

With these results at hand, we study the maximal case where all the orientations of the Coxeter diagram have exactly one orientation. As this case coincides with the Cambrian congruences, we show that the event of a permutation being minimal in its permutree congruence class (or any of its equivalent events with pattern avoidance or acceptance in~$\PP(U,D)$) is equivalent to that permutation being Coxeter sortable and that the corresponding~$c$-sorting word is accepted by~$\PP(U,D)$ (Theorem~\ref{thm:coxter_sorting_permutrees}).

Finally, we propose a set of automata that have been checked computationally to characterize permutree minimal elements for Coxeter groups of types~$B$ and~$D$. In type~$B$ our proposal covers all possible ranks while in type~$D$ is covers only up to certain cases of~$n=5$. Based on the definition of~$c$-singletons from~\cite{HLT11}, we support our proposal by defining~\defn{$\delta$-singletons},~\defn{$\delta$-accepters},~\defn{$\delta$-maccrs} (minimal accepters), and~\defn{$\delta$-smaccrs} (shortest minimal accepters) which we conjecture allow with slight modifications for the definition of automata in any finite Coxeter group (Conjecture~\ref{conj:smaccr_automata}). We also present Conjectures~\ref{conj:unique_smaccr} and~\ref{conj:maccrs_from_smaccr} about the relation between maccrs and smaccrs together with computational evidence and examples in types~$D$ and~$H$.

\subsection*{Flow Polytopes}

Part~\ref{part:sorder} is dedicated to describe how triangulations of flow polytopes can realize distinct families of posets as~$1$-skeleta of simplicial complexes and how this can be used in our context together with other polytopal techniques. Chapter~\ref{chap:sorder_Flows} gives the necessary background on flow polytopes including the definitions of \defn{flows}, \defn{flow polytopes}, the \defn{Baldoni–Vergne–Lidskii formulas} for their volume, and \defn{Kostant partition function}. We recall the constructions of the \defn{Danilov–Karzanov–Koshevoy triangulations} and \defn{Postnikov–Stanley subdivisions} together with their relations when the PS subdivisions are framed. For DKK triangulations we refine Proposition~\ref{prop:DKKlem2_original} from~\cite{DKK12} that states a sufficient condition on routes for a height function to be admissible for the triangulation. We define the \defn{resolvent} for a conflict between two routes and introduce the concept of \defn{minimal conflicts}. This allows us to turn Proposition~\ref{prop:DKKlem2_original} into a necessary and sufficient condition (Lemma~\ref{lem:DKKlem2_us}) and then relax the necessary condition while keeping the if and only if (Lemma~\ref{lem:DKKlem2_us_pro}). Afterwards, we present the basic definitions of tropical geometry we require by concentrating ourselves in the geometrical constructions of \defn{tropical surfaces}, \defn{domes}, \defn{Newton polytopes}, and \defn{dual subdivisions} on point configurations. This chapter finishes with a note about how taking several point configurations, doing their \defn{Cayley embedding}, and applying \defn{the Cayley trick} relates their corresponding tropical constructions (Proposition~\ref{prop:arrangement_tropical_hypersurfaces}). While most of this background comes from~\cite{J21}, we refine Proposition~\ref{prop:bijections_dome_newton_polytope} that relates these definitions and defines a dimension reversing bijection between cells of tropical hypersurfaces and cells of dual subdivisions. We do this by showing that this bijection restricts to the bounded cells of the tropical hypersurface and the interior cells of the dual subdivision (Lemma~\ref{lem:tropical_dual_interior}).

Chapter~\ref{chap:sorder_realizations} is based completely on the article~\cite{GMPTY23} and focuses on answering Conjecture~\ref{conj:s-permutahedron} of realizing the~$s$-permutahedron as a polyhedral subdivision of a polytope combinatorially isomorphic to a zonotope for the case when~$s$ has no zeros.  For this we define two framed graphs called the \defn{oruga graph~$\oru_n$} and the~\defn{$s$-oruga graph~$\oru(s)$}. We prove that for the shifted indegree netflow~$\bfd$, the integer~$d$-flows in~$\oru(s)$ are in bijection with~$s$-Stirling permutations (Theorem~\ref{thm:bij_simplices_permutations}) and even in the case when~$s$ contains zeros, in bijection with~$s$-decreasing trees (Remark~\ref{rem:bij_simplices_trees}). Continuing with~$s$ being a composition, after using the bijection between integer~$\bfd$-flows of~$\oru(s)$ and maximal simplices in the DKK triangulation of the flow polytope of~$\fpol[\oru(s)](\bfi)$ with the basic netflow~$\bfi:=(1,0,\ldots,0,-1)$ from~\cite{MMS19}, we show that these maximal simplices are in bijection with Stirling~$s$-permutations (Lemma~\ref{lem:max_clique}) and that their adjacency encodes the~$1$-skeleton of the~$s$-weak order (Theorem~\ref{thm:cover_relations}). From here, we define a simplex associated to a face of the~$s$-permutahedron and show that this mapping defines an inclusion reversing isomorphism between the faces of the~$s$-permutahedron and simplices of the DKK triangulation of the flow polytope~$\fpol[\oru(s)](\bfi)$ (Corollary~\ref{cor:maximal_faces_perm}). This realizes the~$s$-permutahedron as a high dimensional polytope so to reduce its dimension we use \defn{the Cayley trick} as described in~\cite{S05} to obtain an inclusion reversing bijection between the~$s$-permutahedron and the interior cells of a mixed subdivision of a sum of hypercubes of varying dimension (Theorem~\ref{thm:bij_mixed_subdiv}). This realization has the desired dimension of the conjecture (Remark~\ref{rem:dim_n_minus_1}) but lacks explicit coordinates. To mend this we use the DKK formula to obtain an explicit a height function on the routes on~$\oru(s)$ (Lemma~\ref{lem:epsilonheight}). Consequently, we prove that there is an arrangement of tropical hypersurfaces whose tropical dual is the mixed subdivision of hypercubes we obtain from applying the Cayley trick to~$\oru(s)$ (Theorem~\ref{thm:arr_trop_hypersurfaces_s-perm}). This allows us to obtain the~$s$-permutahedron through the polyhedral complex of bounded cells of this tropical arrangement (Theorem~\ref{thm:bij_trop_arr}). As immediate consequences we describe several properties of this realization including its vertices (Theorem~\ref{thm:vertices}), containing hyperplane (Corollary~\ref{cor:s_containging_hyperplane}), edge directions showing that it is a generalized permutahedron (Theorem~\ref{thm:edges}), supporting vertices and hyperplanes (Lemma~\ref{lem:support}), and the fact that it is indeed the translation of a zonotope that is isomorphic to a permutahedron (Theorem~\ref{thm:s_realization_zonotope}). The chapter finishes with certain results and remarks on how the enumeration techniques in flow polytopes decompose the number of~$s$-combinatorial objects (Corollary~\ref{cor:identitise s-trees}).

Finally, in Chapter~\ref{chap:sorder_quotients} we use the flow polytope techniques to give a third answer to Perspective~\ref{pers:Coxeter_permutrees} for type~$A$ permutrees. This answer consists on the description of permutree rotation lattices through DKK triangulations of flow polytopes. We begin by defining \defn{M-moves} on~$\oru_n$ which create new framed graphs called the~\defn{$\delta$-bicho graphs~$\bic_\delta$} according to the permutree decoration at play. Among them, we have our \defn{oruga graph~$\oru_n$}, the \defn{caracol graph~$\car_n$} of~\cite{BGHHKMY19}, and a new \defn{mariposa graph~$\mar_n$} (Remark~\ref{rem:all_bichos_graphs}). We show that integer~$\bfd$-flows in~$\bic_\delta$ and~$\delta$-permutrees have the same cardinalities (Theorem~\ref{cor:vol_bic_bfi_permutrees}) and that the refinement order on permutree decorations determines the structure of maximal cliques of coherent routes between distinct~$\delta$-bicho graphs (Lemma~\ref{lem:cliques_through_M_moves}). With this in hand we prove that the simplices of the DKK triangulation of~$\bic_\delta$ are in bijection with~$\delta$-permutrees (Theorem~\ref{thm:permutree_to_clique}) and that the permutree rotation lattice is encoded by the adjacencies of the simplices in the DKK triangulation (Corollary~\ref{cor:rotation_to_adjacency}).

\subsection*{Conjectures and Perspectives}

Throughout the work of this thesis we made extensive use of the open source software SageMath~\cite{SAGE} for implementations and computations on the combinatorial objects we studied. This allowed us to obtain an intuition around the problems we worked on and define concretely our results. In particular, we found computational evidence for several  combinatorial phenomena that we leave here as Conjectures~\ref{conj:unique_smaccr},~\ref{conj:smaccr_automata},~\ref{conj:maccrs_from_smaccr},~\ref{conj:nonee_downn_flows_2}, and~\ref{conj:bicho_recursion}.

In a more general way, we also present several directions for future work that we could take following the problematics treated in this thesis. For some of them we have partial results or a strong intuition while others just state the natural next step to follow. We give them in Perspectives~\ref{pers:Coxeter_permutrees},~\ref{pers:sorting_network},~\ref{pers:s_associahedra},~\ref{pers:other_framings},~\ref{pers:graphs_for_zeroes},~\ref{pers:other_realizations},~\ref{pers:s_m_moves},~\ref{pers:s_permutrees_join_irreds},~\ref{pers:s_permutrees_combinatorics}, and~\ref{pers:other_types}.

%% file: includes/IntroductionFR.tex



\chapter*{Introduction}\label{part:IntroductionFR}
\addcontentsline{toc}{part}{Introduction (Fran\c{c}ais)}
\addcontentsline{lof}{part}{Introduction (Fran\c{c}ais)}
\markboth{Introduction}{Introduction}

\pagestyle{fancy}

\section*{Contexte}

Ce mémoire se trouve dans le domaine de la combinatoire, à l'intersection de la combinatoire algébrique et la combinatoire géométrique. Autrement dit, il repose sur les connexions entre familles dotées d'opérations algébriques comme les groupes finis et les ensembles partiellement ordonnés, et des structures de la géométrie discrète comme les polytopes.

La combinatoire s'intéresse à prendre des objets discrets et à étudier leurs motifs. De manière classique, on peut la considérer comme un domaine qui s'intéresse au comptage d'objets et de phénomènes, ou inversement, à la recherche d'un objet ou d'une propriété qui décrit une séquence de nombres. Par nature, la combinatoire apparaît dans nombreuses branches des mathématiques, que ce soit en tant qu'acteur principal ou secondaire. Des applications de la combinatoire peuvent être trouvées en topologie algébrique~\cite{MK46}~\cite{K08}, en théorie des nombres~\cite{TV06} et même en physique théorique et statistique~\cite{T20}~\cite{R69}. Comme elle étudie des structures discrètes, la combinatoire a établi une forte connexion avec l'informatique. Par exemple, l'étude des algorithmes, leurs complexités et optimisations est intrinsèquement combinatoire. Réciproquement, de nombreux résultats combinatoires proviennent de l'analyse algorithmique et de l'exploration informatique. Plus concrètement, des domaines tels que l'optimisation linéaire en nombres entiers, la théorie des graphes et les algorithmes de tri font appel de manière symbiotique aux idées provenant de la combinatoire et de l'informatique. Notre travail n'échappe pas à cette réalité, car la majeure partie, voire la totalité de nos idées, ont été influencée par des expériences faites dans le logiciel libre SageMath~\cite{SAGE}.

\subsection*{Ordre faible}

La famille combinatoire au cœur de notre travail est celle des \defn{permutations}. Ces objets font partie des plus simples de la combinatoire. Plus précisément, une permutation de taille~$n$ consiste à prendre~$n$ objets dans un certain ordre et à les réarranger dans un nouvel ordre. Dit autrement, on peut les considérer comme les bijections de~$[n]:=\{1,2,\ldots,n\}$ sur lui-même. Ce point de vue confère aux permutations une structure algébrique où la multiplication de permutations correspond simplement à la composition de leurs bijections correspondantes. Cela forme une structure algébrique appelée le \defn{groupe symétrique}, où chaque permutation est présentée par une réarrangement des nombres~$1$ à~$n$. Une façon rapide de distinguer les permutations est de vérifier quelles paires~$(i,j)$ ont été inversées. Toute paire qui se trouve dans cette situation est appelée une \defn{inversion}. L'ensemble de telles paires est appelé l'\defn{ensemble d'inversions} d'une permutation et définit une relation d'ordre partiel sur les permutations en fonction de l'inclusion de leurs ensembles d'inversions. Cela s'appelle l'\defn{ordre faible} sur les permutations~\cite{B22}.

Les permutations adjacentes dans l'ordre faible peuvent être exprimées par l'inversion d'une seule paire~$(i,i+1)$. Nous disons que les permutations correspondant à ces réarrangements sont les \defn{transpositions simples} de l'ordre faible. Avec ça on peut montrer que, étant donné deux permutations dans l'ordre faible, il est toujours possible de trouver une permutation maximale (resp. minimale) unique qui est plus petite (resp. plus grande) que les deux. Ces opérations transforment l'ordre faible en un ensemble partiellement ordonné (i.e.\ un \defn{poset}) avec une structure algébrique appelée \defn{treillis}~\cite{GR63}.

Chaque permutation~$\sigma$ de~$[n]$ peut être associée au point~$(\sigma(1),\ldots,\sigma(n))$ dans l'espace~$\RR^n$. Si on prend l'enveloppe convexe de cette configuration de points, le polytope résultant est appelé le \defn{permutaèdre}, dont les propriétés reflètent les phénomènes de l'ordre faible~\cite{S911}~\cite{GG77}. Par exemple, en orientant le permutaèdre dans une direction particulière, le graphe orienté constitué par son~\defn{$1$-squelette} (i.e.\ sommets et arêtes) correspond précisément à l'ordre faible. Par ailleurs, ce processus pour obtenir le permutaèdre montre que ses faces sont indexées par les partitions ordonnées de~$[n]$. Une autre façon géométrique de présenter le permutaèdre et la structure de l'ordre faible consiste à partitionner les ensembles d'inversions des permutations en fonction du plus petit élément dans une inversion et prendre les cardinalités de ces ensembles. La structure résultante est un \defn{plongement} du permutaèdre dans un cube~\cite{BF71}~\cite{RR02}.

\subsection*{Treillis de Tamari}

Une autre famille d'objets combinatoires sur laquelle repose une grande partie de notre travail est celle des \defn{arbres binaires}. Il s'agit d'arbres enracinés où chaque nœud a un parent et deux enfants. Les arbres binaires sont dénombrés par une des séquences de nombres les plus prolifiques, appelée les \defn{nombres de Catalan}~\cite[A000108]{OEIS}, et sont donc en bijection avec une myriade d'objets combinatoires~\cite{S15}.

Étant donné un arbre binaire avec~$n$ nœuds, nous pouvons étiqueter ses sommets en effectuant une marche dans le sens contraire des aiguilles d'une montre sur le graphe, en étiquetant le~$i$-ème sommet que nous visitons une deuxième fois lors de notre parcours. Cela s'appelle l'\defn{ordre infixe} des arbres binaires, et l'étiquetage résultant a la propriété que pour chaque sommet, son étiquette est supérieure aux étiquettes de son sous-arbre gauche et inférieure aux étiquettes de son sous-arbre droit. Cela nous permet de définir la \defn{rotation} d'une arête~$i\to j$ en une arête~$j\to i$, tout en maintenant la structure des sous-arbres de~$i$ et~$j$. Les rotations sont une opération classique utilisée pour équilibrer les arbres binaires de recherche (c'est-à-dire réduire leur hauteur au minimum) afin d'obtenir des algorithmes de tri efficaces~\cite{AL62}. Ces rotations définissent un ordre sur tous les arbres binaires appelé le \defn{treillis de Tamari}~\cite{T62}.

Pour obtenir une structure géométrique à partir des arbres binaires, il faut considérer des vecteurs où chaque coordonnée~$i$ est le produit des feuilles dans le sous-arbre gauche et les feuilles dans le sous-arbre droit du nœud étiqueté~$i$. L'enveloppe convexe de ces sommets est appelée l'\defn{associaèdre}~\cite{L04} et ses faces sont indexées par les arbres de Schröder~\cite{S11}. Comme précédemment, en orientant l'associaèdre dans une direction particulière, nous pouvons trouver le treillis de Tamari à travers son~$1$-squelette. À ce stade, une relation entre les arbres binaires et les permutations commence à apparaître, car l'associaèdre est un \defn{enlevoèdre}. Autrement dit, l'associaèdre peut être obtenu en enlevant certaines facettes du permutaèdre~\cite{SS93}. Comme pour les permutations, en changeant les coordonnées correspondant aux arbres binaires par des quantités dérivées des rotations subies et de leur nombre, on obtient un nouvel ensemble de vecteurs appelés \defn{vecteurs de crochet}. Ces vecteurs permettent une preuve constructive de la propriété du treillis de l'ordre de Tamari~\cite{HT72} et également un plongement cubique de l'associaèdre~\cite{K93}. Ces techniques ont été utilisées avec des généralisations et des structures liées au treillis de Tamari~\cite{CPS20}~\cite{C21}~\cite{FMN21}~\cite{C22}~\cite{CG22}~\cite{PP23}.

Une relation plus directe entre les arbres binaires et les permutations peut être trouvée grâce à l'algorithme d'insertion de~\cite{T97}~\cite{HNT05}. Ainsi, les fibres des arbres binaires sous cet algorithme sont des intervalles de permutations dont les élément minimaux évitent le motif~$312$. Ces fibres correspondent aux les fibres de l'algorithme du tri par pile~\cite{K73} et forment une relation de congruence dans l'ordre faible qui respecte les opérations de infimum et de supremum. On appelle ces congruences des \defn{congruences de treillis} et dans ce cas, elles conduisent au treillis de Tamari qui est l'ordre induit par les permutations minimales dans ces fibres. Cette congruence de treillis est connue sous le nom de \defn{congruence congruence sylvestre}.

\subsection*{Permutarbres}

La famille combinatoire qui motive cette thèse et que nous étudions sous différents angles est celle des \defn{permutarbres}~\cite{PP18}. Cette famille est suffisamment générale pour encoder les permutations, les arbres binaires, les \defn{arbres cambriens}~\cite{LP13}~\cite{CP17}, et les séquences binaires, tout en permettant de définir de nouveaux types d'arbres. Un permutarbre est constitué d'un arbre dirigé non enraciné avec des nœuds étiquetés par~$[n]$, où chaque nœud peut avoir un ou deux parents et un ou deux enfants, tandis que chaque étiquette d'un sommet satisfait une relation similaire à celle des arbres binaires via l'ordre infixe. Ainsi, les permutarbres sont caractérisés par le nombre de parents et d'enfants de chaque nœud, ce qui est appelé la \defn{décoration} du nœud. Comme les sommets sont étiquetés, cela regroupe les permutarbres en~$\delta$-permutarbres où~$\delta$ est un vecteur de décorations.

Dans la même lignée que~\cite{HNT05}, l'algorithme d'insertion de~\cite{PP18} établit une surjection des permutations sur les~$\delta$-permutarbres pour chaque décoration possible. Ses fibres décrivent une congruence de treillis appelée \defn{congruence des~$\delta$-permutarbres}. Les fibres des~$\delta$-permutarbres par cet algorithme sont des intervalles de permutations dont l'élément minimal évite le motif~$kij$ et/ou~$jki$ pour chaque~$j\in\{2,\ldots,n-1\}$ et $1\leq i<j<k\leq n$ en fonction de la décoration~$\delta_j$.

Comme pour les arbres binaires, les~$\delta$-permutarbres ont des rotations qui modifient leur structure locale au niveau d'une seule arête tout en maintenant le reste de l'arbre intact. Ces rotations définissent les posets de rotations des~$\delta$-permutarbres. Dans~\cite{PP18}, il est démontré démontré que ces posets sont des treillis mais la démonstration utilise des quotients de treillis pour montrer que ce poset est un sous-treillis de l'ordre faible. En général, ils sont appelés \defn{treillis de $\delta$-permutarbres} et généralisent les treillis cambriens en type~$A$ de~\cite{R06}.

Peu importe la décoration, les permutarbres peuvent être associés à un vecteur dont les coordonnées correspondent à une manipulation du nombre de nœuds dans leurs sous-arbres droit et gauche. Cela donne naissance à un polytope appelé le \defn{$\delta$-permusylvèdre}. En orientant ce polytope dans une direction particulière, on retrouve leur treillis correspondant. Bien que les~$\delta$-permutarbres et leurs congruences de permutarbres puissent être différents, pour certains sous-ensembles de décorations, leurs treillis sont isomorphes et pour d'autres, leurs permusylvèdres sont combinatoirement équivalents.

\subsection*{Groupes de Coxeter}

Nos idées sur les permutations et l'ordre faible ne constituent qu'une partie d'un schéma plus vaste décrit par les \defn{groupes de Coxeter}. Introduits dans~\cite{C34} et ensuite entièrement classifiés pour le cas fini dans~\cite{C35}, les groupes de Coxeter décrivent des groupes engendrés par des \defn{réflexions simples} provenant d'arrangements d'hyperplans~\cite{H90}. Ainsi, les éléments sont des séquences de réflexions simples appelées \defn{mots}, et les réflexions correspondent aux \defn{inversions}. Ces inversions définissent des ensembles d'inversions qui, à leur tour, définissent l'\defn{ordre faible} d'un groupe de Coxeter. Ainsi, les permutations ne sont qu'un cas particulier des groupes de Coxeter appelé les groupes de Coxeter de type~$A$. Il y a des familles combinatoires similaires qui peuvent être trouvées pour d'autres types tels que les \defn{permutations signées} et les \defn{permutations signées paires} pour les groupes de Coxeter de types~$B$ et~$D$ respectivement~\cite{BB06}. À notre connaissance, les autres groupes n'ont pas de descriptions combinatoires similaires. Néanmoins, la nature des groupes de Coxeter permet de les étudier par des moyens algébriques, géométriques, combinatoires ou par la théorie des langages. Cela a conduit à décrire les propriétés du langage des mots réduits dans les groupes de Coxeter à l'aide d'automates~\cite{BH93}~\cite{HNW16}.

En allant plus loin que la généralisation des permutations et de leurs propriétés, le contexte des groupes de Coxeter offre un espace plus large dans lequel les congruences de treillis peuvent être définies. Cela a été réalisé en considérant les groupes de Coxeter comme l'ensemble des régions d'arrangements d'hyperplans dans~\cite{R04}. Ensuite, le concept de~\defn{$c$-triage} (également connu sous le nom de \defn{triage de Coxeter}) a été introduit dans~\cite{R07a}. Cela a donné les \defn{éléments~$c$-triables} dont Reading a démontré plus tard qu'ils étaient les éléments minimaux des \defn{congruences cambriennes}~\cite{R07b}. Ces résultats ont été unifiés pour tous les groupes de Coxeter finis, indépendamment de leur type, dans~\cite{RS11}, puis résumés dans~\cite{R12}. Dans ce contexte, la congruence sylvestre est une congruence cambrienne de type~$A$, et ses éléments~$c$-triables coïncident avec les permutations triables par pile. De plus, toutes les congruences cambriennes de type~$A$ sont des congruences de permutarbres.

\subsection*{\texorpdfstring{$s$}{}-ordre faible}

D'autres généralisations possibles de l'ordre faible proviennent des multipermutations plutôt que des permutations. Cela signifie qu'étant donné une séquence d'entiers positifs~$r=(r_1,\ldots,r_n)$, une~\defn{$r$-permutation} est un réarrangement du mot~$1^{r_1}\cdots n^{r_n}$ où~$i^{r_i}$ représente la répétition de la lettre~$i$ un total de~$r_i$ fois. Venant d'un cadre géométrique dans~\cite{RR02}, elles ont été utilisées pour décrire quelques plongements du~\defn{combinohèdre}, qui était connu pour provenir d'une structure de treillis appelée le \defn{treillis multinomial}~\cite{BB94}. Indépendamment, dans un cadre plus algébrique, le cas où tous les~$r_i=m$ pour un~$m\geq 1$ a été introduit dans~\cite{NT20} pour étudier la \defn{congruence $m$-sylvestre}. En renommant~$r$ en~$k$, les~$k$-permutations évitant le motif~$121$ sont appelées \defn{$k$-permutations de Stirling}. Ces permutations ont suscité un vif intérêt et de nombreuses de leurs propriétés ont été déterminées grâce à leurs statistiques et leurs bijections avec d'autres familles combinatoires~\cite{P94a}~\cite{P94b}~\cite{P94c}~\cite{KP11}~\cite{JKP11}~\cite{RW15}~\cite{G19}.

Cependant, en dehors de toutes ces constructions, il en existe une plus récente liée à ces idées et qui nous intéresse appelée le~\defn{$s$-ordre faible}~\cite{CP19}~\cite{CP22}. En prenant une composition faible~$s=(s_1,\ldots,s_n)$ (c'est-à-dire~$s_i\in\ZZ_{\geq 0}$), les~\defn{$s$-arbres décroissants} sont des arbres enracinés étiquetés avec~$n$ nœuds tels que chaque nœud a exactement un parent et~$s_i+1$ enfants, et tous les descendants ont des étiquettes plus petites. Ces arbres ont des \defn{inversions} définies à partir de la position relative entre les nœuds. Comme précédemment, ces inversions définissent un ordre sur les~$s$-arbres décroissants appelé le~\defn{$s$-ordre faible}. Pour cet ordre il a été montré de manière constructive qu'il a une structure de treillis et a une structure géométrique appelée le \defn{$s$-permutaèdre}. De plus, il existe un ordre sous-jacent appelé le~\defn{$s$-treillis de Tamari} avec une contrepartie géométrique appelée le~\defn{$s$-associaèdre}. Ils correspondent au~\defn{$\nu$-treillis de Tamari} de~\cite{PV17} et au~\defn{$\nu$-associaèdre} de~\cite{CPS19} pour certaines valeurs de~$\nu$. Chaque fois que~$s$ est une composition, les~$s$-arbres décroissants sont en bijection avec les~\defn{$s$-permutations de Stirling}, c'est-à-dire les $s$-permutations évitant le motif~$121$. Le~$s$-ordre faible coïncide avec l'ordre faible des permutations lorsque~$s_i=1$ pour toutes les coordonnées de~$s$, et avec le \defn{treillis métasylvestre} de~\cite{P15} lorsque~$s_i=m$ avec~$m\geq 1$ pour toutes les coordonnées de~$s$.

\subsection*{Polytopes de flot}

Une famille de polytopes que nous considérons pour une bonne partie de notre travail en raison de leur polyvalence est celle des polytopes de flot. Ils proviennent d'un graphe orienté sans boucles où chaque sommet est équipé d'un entier qui représente le \defn{flot net} traversant. Autrement dit, la différence entre le flot entrant et le flot sortant déterminée respectivement par les arêtes entrantes et sortantes de chaque sommet doit être égale à ce flot net. Cela limite les \defn{flots} possibles qu'on peut assigner aux arêtes du graphe tout en étant cohérents avec le flot net. En considérant les flots comme des points dans l'espace des arêtes du graphe, on obtain un polytope appelé le \defn{polytope de flot}. Étant donné que les flots et les flots nets sont des modèles de réseaux, les polytopes de flot ont été largement utilisés dans les problèmes d'optimisation. Par conséquent, de nombreuses recherches ont été faites sur la combinatoire de ces polytopes~\cite{RH70}~\cite{FRD71}~\cite{CG78}~\cite{H03}~\cite{MM19}~\cite{GHMY21}. En particulier, le \defn{volume normalisé} du polytope de flot se décompose de manière agréable avec la \defn{formule de partition de Kostant} à travers des \defn{formules de Baldoni-Vergne-Lidskii}~\cite{BV08}. Cela a permis de mettre en évidence de nouvelles identités démontrant de nombreux nombres connus via des décompositions en produits, comme cela est illustré dans~\cite{BGHHKMY19}.

Nous nous intéressons aux polytopes de flot en particulier en raison de leurs possibles subdivisions. En particulier, en dotant chaque sommet d'un graphe de relations d'ordre totales indépendantes pour ses arêtes entrantes et sortantes (c'est-à-dire un \defn{cadre}), les cliques de routes cohérentes du graphe nous donnent une triangulation du polytope de flot lorsque le flux net est~$\bfi:=(1,0,\ldots,0,-1)$. Cette triangulation est appelée la \defn{triangulation DKK} et est équipée d'une fonction de hauteur qui la rend régulière~\cite{DKK12}. Une autre subdivision possible consiste à effectuer une série de \defn{réductions} sur le graphe, ce qui se traduit par la découpe du polytope de flot en plusieurs morceaux qui sont combinatoirement équivalents à d'autres polytopes de flot. Ce processus donne la \defn{subdivision de Postnikov-Stanley}~\cite{P1014}~\cite{S00}. L'utilité de ces deux subdivisions réside dans le fait que lorsque le graphe est encadré, les \defn{subdivisions encadrées de Postnikov-Stanley} deviennent des triangulations DKK~\cite{MMS19}. Cela établit une bijection entre les simplexes de la triangulation DKK du flot net~$\bfi$ et les points entiers du polytope de flot du flot net~$\bfd$, où~$\bfd_i$ est le degré d'entrée décalé du~$i$-ème sommet. Cette bijection entre deux polytopes de flot différents a permis de retrouver certains treillis comme le dual des faces intérieures de la triangulation DKK, tels que le treillis de Tamari~$\nu$ et les idéaux d'ordre principal du treillis de Young~\cite{BGMY23}.

\subsection*{Géométrie tropicale}

La \defn{géométrie tropicale} provient du remplacement des opérations habituelles de l'addition et de la multiplication respectivement par les opérations du minimum et de la somme, et l'inclusion de l'infini comme élément~\cite{J21}. Bien qu'elle ait des liens intéressants avec l'économie~\cite{S15E} et la conception de mécanismes~\cite{CT18}, nous sommes intéressés par sa relation avec la géométrie discrète classique. Le changement des opérations et de l'ensemble de base permet de définir de nouvelles structures géométriques telles que les \defn{polynômes tropicaux}, les \defn{hypersurfaces tropicales} et les \defn{variétés tropicales}. Ces objets ont été étudiés en tant que tels et ont également montré des liens avec la géométrie convexe~\cite{J17}. Par exemple, toutes les \defn{configurations de points} sont en correspondance avec des polynômes tropicaux. De plus, les techniques classiques de la géométrie discrète telles que la méthode de Cayley~\cite{S94}~\cite{HRS00}~\cite{DRS10} ont été utilisées à plusieurs reprises dans le contexte tropical~\cite{DS04}~\cite{FR15}~\cite{J16}~\cite{JL16}~\cite{MS21}.

\section*{Contributions}

Avec ce contexte donné, cette thèse s'inscrit dans le cadre d'un projet plus vaste visant à répondre à la question suivante: pouvons-nous étudier les congruences des permutarbres dans tous les types de Coxeter? Bien que nous ne donnions pas une réponse complète à cette question, nous proposons trois nouveaux points de vue à partir desquels les permutarbres de type A peuvent être étudiés, ainsi que d'autres résultats liés que nous avons obtenus en utilisant certains de ces outils pour le~$s$-ordre faible. Ces contributions sont contenues dans les articles suivants: \begin{itemize}
    \itemsep0em
    \item D. Tamayo Jiménez. Inversion and Cubic Vectors for Permutrees, 2023. arXiv:2308.05099,
    
    où nous avons donné une preuve constructive de la propriété de treillis pour les treillis des rotations des permutarbres et une incorporation cubique des permusylvèdres.
    \item V. Pilaud, V. Pons, et D. Tamayo Jiménez. Permutree Sorting. Algebraic Combinatorics, 6(1):53-74, 2023, 
    
    où nous avons caractérisé les éléments minimaux des classes des permutarbres de type~$A$ à l'aide de leurs mots réduits en utilisant des automates, et trouvé des pistes pour d'autres types de Coxeter ($B$ et~$D$).
    \item R.S. González D’León, A.H. Morales, E. Philippe, D. Tamayo Jiménez, and M. Yip.
    Realizing the~$s$-Permutahedron via Flow Polytopes, 2023. arXiv:2307.03474,

    où nous avons donné une réponse positive à une conjecture de Ceballos et Pons sur la réalisation géométrique de l'ordre faible s (lorsque~$s$ ne contient pas de zéros) en utilisant des subdivisions polytopales de flow polytopes, des sommes d'hypercubes et de la géométrie tropicale.
    \item R.S. González D’León, A.H. Morales, E. Philippe, D. Tamayo Jiménez, and M. Yip. Flow
    Polytopes and Permutree Lattice Quotients of the~$s$-Weak Order. In preparation, 2023+,

    où nous avons trouvé une description des permutarbres et de leurs treillis de rotations en utilisant les subdivisions de flow polytopes. 
\end{itemize}  Les trois premiers travaux ont été présentés dans plusieurs ateliers, séminaires et conférences internationales, sous forme de posters ou de présentations. Les travaux contenus dans ces articles s'appuient sur des idées telles que les vecteurs de parenthèses des arbres binaires, la théorie du langage dans les groupes de Coxeter avec automates, et la combinatoire des flow polytopes soutenue par la géométrie tropicale.

\section*{Plan de la thèse}

Le travail présenté dans cette thèse est divisé en trois parties. La Partie~\ref{part:prelim} décrit les principaux outils combinatoires et la problématique des permutarbres qui est au cœur de notre travail. Ensuite, la Partie~\ref{part:permutrees} présente deux de nos réponses à cette problématique dans les chapitres~\ref{chap:permutree_vectors} et~\ref{chap:permutree_sorting}. Enfin, la Partie~\ref{part:sorder} traite sur l'utilisation des polytopes de flot dans notre contexte. Nos contributions dans cette partie sont contenues dans les chapitres~\ref{chap:sorder_realizations} et~\ref{chap:sorder_quotients}. La figure~\ref{fig:thesis_outline_FR} montre les dépendances entre les contenus de la thèse et décrit l'ordre de lecture recommandé. Le lecteur est invité à sauter les chapitres~\ref{chap:prelim_structures},~\ref{chap:prelim_weak_order} et/ou~\ref{chap:sorder_Flows} s'il est déjà familier avec le matériel correspondant. Nous avons essayé de rendre cette thèse aussi autonome que possible, de sorte que la seule exigence réelle pour la lire est la connaissance de l'algèbre linéaire. De l'expérience en géométrie discrète n'est pas nécessaire, mais fortement recommandée pour avoir une meilleure intuition de notre contexte. Cependant, tout au long de notre travail, nous avons nous avons pris soin de proposer des figures utiles afin de mieux communiquer nos idées, nos résultats et, en général, les points de vue à partir desquels nous avons abordé les problèmes que nous avons étudiés. Nous espérons que le lecteur les trouvera utiles.

\begin{figure}[h!]
    \centering
    \includegraphics[scale=0.882]{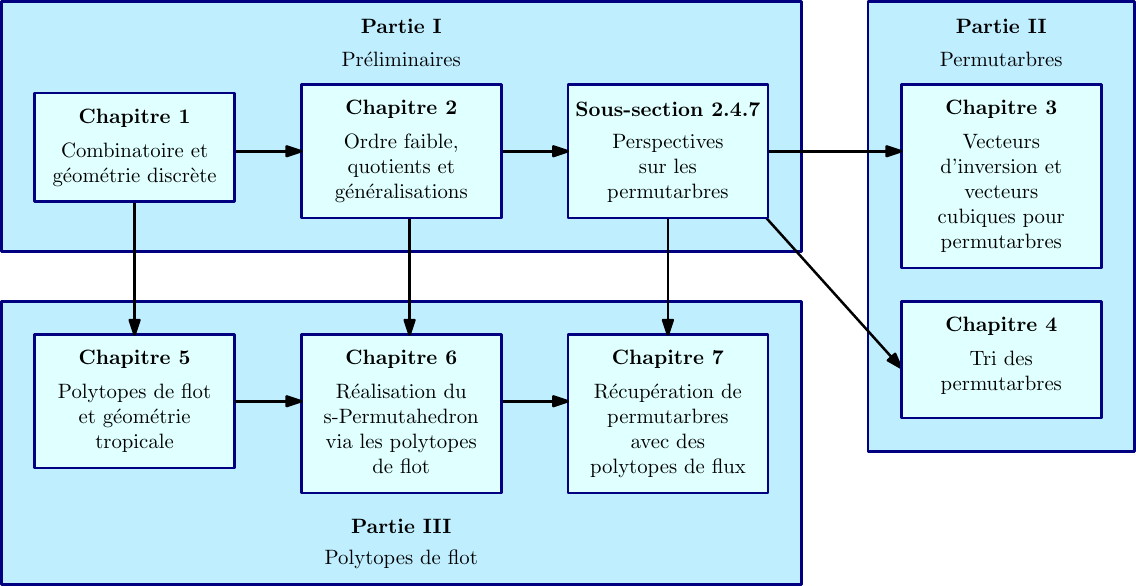}
    \caption{Plan de la thèse.}
    \label{fig:thesis_outline_FR}
\end{figure}

\subsection*{Préliminaires}

La Partie~\ref{part:prelim} se compose de deux chapitres qui introduisent, à différents niveaux de généralité, les principaux concepts combinatoires que nous utilisons. Le chapitre~\ref{chap:prelim_structures} commence par introduire les concepts d'\defn{ordres partiels}, de \defn{treillis} et de \defn{congruences de treillis} qui sont au cœur de toutes les parties de notre travail. Ensuite, nous décrivons les bases de la géométrie convexe, y compris les \defn{polytopes}, les \defn{cônes}, les \defn{complexes}, ainsi que plusieurs opérations et techniques classiques qui les concernent. Enfin, nous rappelons les bases de la \defn{théorie des automates} et sa capacité à reconnaître des motifs dans les mots. Le cadre de la géométrie convexe et de la théorie des automates nous offre un point de vue avantageux pour aborder les problèmes traités dans cette thèse.

Dans le chapitre~\ref{chap:prelim_weak_order}, nous présentons d'abord l'\defn{ordre faible sur les permutations}, ainsi que les \defn{arbres binaires} et les \defn{permutarbres}, en décrivant leurs similarités en termes de combinatoire, d'ordres et de polytopes, ainsi que leurs relations mutuelles à travers des congruences de treillis. Ensuite, nous rappelons les généralisations de l'ordre faible à l'\defn{ordre faible des groupes de Coxeter} et à le~\defn{$s$-ordre faible}. La généralisation aux groupes de Coxeter nous aide à décrire le problème central de cette thèse, qui consiste de trouver nouvelles familles combinatoires à travers lesquelles nous pouvons étudier les permutarbres et leurs treillis (Perspective~\ref{pers:Coxeter_permutrees}).

\subsection*{Permutarbres}

Dans la Partie~\ref{part:permutrees}, nous fournissons deux réponses à la Perspective~\ref{pers:Coxeter_permutrees}. La première, présentée dans le chapitre~\ref{chap:permutree_vectors}, est basée sur l'article~\cite{T23} et consiste en deux généralisations des vecteurs de parenthèses des arbres binaires aux permutarbres. Nous commençons par définir les \defn{ensembles d'inversions} pour un permutarbre, ainsi que la façon dont ils peuvent être partitionnés pour donner des \defn{vecteurs d'inversions}. Comme les ensembles d'inversions pour les permutations, nous caractérisons ces ensembles en termes de transitivité, de cotransitivité et de deux conditions supplémentaires en fonction de la décoration du permutarbre étudié (Lemme~\ref{lem:permutree_inversion_sets}). En suivant les étapes de~\cite{HT72} avec les arbres binaires, nous montrons que l'ordre de rotation des permutarbres peut être interprété via l'inclusion des ensembles d'inversions (Lemme~\ref{lem:inversion_set_inclusion}) et que l'intersection de ces ensembles, accompagnée d'une condition spéciale, définit une opération d'infimum sur le poset des permutarbres (Théorème~\ref{thm:permutree_meet}). Cela nous permet de retrouver de manière constructive le résultat de~\cite{PP18} selon lequel l'ordre de rotation sur les permutarbres a une structure de treillis (Corollaire~\ref{cor:permutrees_are_lattices}). Ensuite, nous définissons les \defn{ensembles cubiques} et les \defn{vecteurs cubiques} pour les permutarbres. Cela nous permet de généraliser les plongements cubiques du permutaèdre donnés dans~\cite{BF71} et~\cite{RR02}, ainsi que l'associaèdre comme dans~\cite{K93}, pour n'importe quel permusylvèdre (Théorème~\ref{thm:cubic_property_embedding}).

La deuxième réponse, présentée dans le chapitre~\ref{chap:permutree_sorting}, provient de l'article~\cite{PPT23} où nous caractérisons les éléments minimaux des classes de congruence des permutarbres à l'aide d'automates finis qui lisent des mots réduits. Nous commençons par considérer le cas où la décoration des permutarbres est donnée par une seule arête orientée du diagramme de Coxeter de type A. Dans ce cadre, nous définissons les automates \defn{$\UU(j)$} et \defn{$\DD(j)$} qui lisent des mots réduits correspondant à des permutations. Ces automates sont pertinents dans notre cadre de permutarbres car il est démontré que la propriété d'une permutation ayant un mot réduit accepté par ces automates est équivalente à l'évitement de motifs caractérisé par les permutarbres (Théorème~\ref{thm:pattern_avoidance_single}). Après avoir obtenu ce résultat, nous concevons un algorithme qui, étant donnée une permutation, renvoie un mot réduit accepté par l'automate en jeu (Algorithme~\ref{algo:permutree_sorting_simple}). Nous montrons que la propriété selon laquelle ce mot réduit est un mot réduit de la permutation d'origine est équivalente à ce que la permutation soit minimale dans sa classe de congruence des permutarbres (Corollaire~\ref{coro:algorithm_single}). Nous montrons également comment l'ensemble des mots réduits acceptés génère une structure d'arbre dans le diagramme de Hasse de l'ordre faible (Théorème~\ref{thm:automata_generating_trees_simple}).

Nous passons ensuite au cas où chaque arête du diagramme de Coxeter peut avoir au plus une orientation. Dans cette situation, nous concevons l'automate \defn{$\PP(U,D)$} comme l'intersection de nos automates précédents, et à travers celui-ci, nous étendons nos résultats précédents à ce cas plus large des congruences des permutarbres. Autrement dit, il est montré que le fait qu'un mot réduit soit accepté par cet automate est équivalent à la propriété que la permutation correspondante évite les motifs dictés par la congruence des permutarbres (Corollaire~\ref{coro:pattern_avoidance_product}). Nous définissons un algorithme (Algorithme~\ref{algo:permutree_sorting_multiple}) qui transforme une permutation en un mot réduit ayant la propriété que ce mot réduit décrit la permutation si et seulement si la permutation est minimale dans sa classe de congruence des permutarbres (Théorème~\ref{coro:algorithm_multiple}). Il est également montré que l'ensemble des mots réduits acceptés génère une structure d'arbre dans le diagramme de Hasse de l'ordre faible (Théorème~\ref{thm:automata_generating_trees_multiple}).

Avec ces résultats en main, nous étudions le cas maximal où toutes les orientations du diagramme de Coxeter ont exactement une orientation. Comme ce cas coïncide avec les congruences Cambriennes, nous montrons que le fait qu'une permutation soit minimale dans sa classe de congruence des permutarbres (ou tout autre événement équivalent avec l'évitement de motifs ou l'acceptation dans~$\PP(U,D)$) est équivalent à ce que cette permutation soit Coxeter triable et que le~$c$-sorting mot correspondant est accepté par~$\PP(U,D)$ (Théorème~\ref{thm:coxter_sorting_permutrees}).

Enfin, nous proposons un ensemble d'automates dont on a vérifié de manière calculatoire qu'ils caractérisent les éléments minimaux des permutarbres pour des groupes de Coxeter de types~$B$ et~$D$. Dans le cas de type~$B$, notre proposition couvre tous les rangs possibles, tandis que dans le cas de type~$D$, elle ne couvre que certains cas jusqu'à~$n=5$. En nous basant sur la définition des~$c$-singletons de~\cite{HLT11}, nous définissons par ailleurs les~\defn{$\delta$-singletons}, les~\defn{$\delta$-accepteurs}, les~\defn{$\delta$-maccrs} (accepteurs minimaux) et les~\defn{$\delta$-smaccrs} (accepteurs minimaux les plus courts) et nous conjecturons qu'ils permettent, avec de légères modifications, la définition d'automates dans n'importe quel groupe de Coxeter fini (Conjecture~\ref{conj:smaccr_automata}). Nous présentons également les conjectures~\ref{conj:unique_smaccr} et~\ref{conj:maccrs_from_smaccr} sur la relation entre les maccrs et les smaccrs, accompagnées d'explorations informatiques et d'exemples dans les types~$D$ et~$H$.

\subsection*{Polytopes de flot}

La partie~\ref{part:sorder} est consacrée à décrire comment les triangulations des polytopes de flot peuvent réaliser des familles distinctes de posets en tant que~$1$-squelettes de complexes simpliciaux, et comment cela peut être utilisé dans notre contexte en combinaison avec d'autres techniques polytopales. Le Chapitre~\ref{chap:sorder_Flows} donne les bases nécessaires sur les polytopes de flot, y compris les définitions de \defn{flots}, de \defn{polytopes de flot}, des \defn{formules de Baldoni–Vergne–Lidskii} pour leur volume et de la \defn{fonction de partition de Kostant}. Nous rappelons les constructions des \defn{triangulations de Danilov–Karzanov–Koshevoy} et des \defn{sous-divisions de Postnikov–Stanley} ainsi que leurs relations lorsque les sous-divisions PS sont encadrées. Pour les triangulations DKK, nous affinons la Proposition~\ref{prop:DKKlem2_original} de~\cite{DKK12} qui énonce une condition suffisante sur les routes pour qu'une fonction de hauteur soit admissible pour la triangulation. Nous définissons la \defn{résolvante} pour un conflit entre deux routes et introduisons le concept de \defn{conflits minimaux}. Cela nous permet de transformer la Proposition~\ref{prop:DKKlem2_original} en une condition nécessaire et suffisante (Lemma~\ref{lem:DKKlem2_us}) et ensuite de relâcher la condition nécessaire tout en conservant le si et seulement si (Lemma~\ref{lem:DKKlem2_us_pro}). Ensuite, nous présentons les définitions de base de la géométrie tropicale dont nous avons besoin en nous concentrant sur les constructions géométriques de \defn{surfaces tropicales}, de \defn{domes}, de \defn{polytopes de Newton} et de \defn{sous-divisions duales} sur des configurations de points. Ce chapitre se termine par une note sur la manière de prendre plusieurs configurations de points, de faire leur \defn{plongement de Cayley} et d'appliquer \defn{la méthode de Cayley} pour relier leurs constructions tropicales correspondantes (Proposition~\ref{prop:arrangement_tropical_hypersurfaces}). Bien que la plupart de ces éléments de base proviennent de~\cite{J21}, nous affinons la Proposition~\ref{prop:bijections_dome_newton_polytope} qui relie ces définitions et définit une bijection qui inverse les dimensions entre les cellules des hypersurfaces tropicales et les cellules des sous-divisions duales. Nous le faisons en montrant que cette bijection se restreint aux cellules bornées de l'hypersurface tropicale et aux cellules intérieures de la sous-division duale (Lemme~\ref{lem:tropical_dual_interior}).

Le chapitre~\ref{chap:sorder_realizations} est entièrement basé sur l'article~\cite{GMPTY23} et se concentre sur la réponse à la Conjecture~\ref{conj:s-permutahedron} concernant la réalisation du~$s$-permutahèdre en tant que subdivision polyédrale d'un polytope qui est combinatorialement isomorphe à un zonotope, dans le cas où~$s$ ne contient pas de zéros. Pour cela, nous définissons deux graphes encadrés appelés le \defn{graphe oruga~$\oru_n$} et le \defn{graphe~$\oru(s)$} pour le cas général. Nous démontrons que pour le flot net de degré d décalé~$\bfd$, les~$\bfd$-flots entiers dans~$\oru(s)$ sont en bijection avec les~$s$-permutations de Stirling (Théorème~\ref{thm:bij_simplices_permutations}), et même dans le cas où~$s$ contient des zéros, en bijection avec les~$s$-arbres décroissants (Remarque~\ref{rem:bij_simplices_trees}). En poursuivant avec~$s$ étant une composition, après avoir utilisé la bijection entre les~$\bfd$-flots entiers de~$\oru(s)$ et les simplexes maximaux dans la triangulation DKK du polytope de flot~$\fpol[\oru(s)](\bfi)$ avec le flot net de base~$\bfi:=(1,0,\ldots,0,-1)$ provenant de~\cite{MMS19}, nous montrons que ces simplexes maximaux sont en bijection avec les~$s$-permutations de Stirling (Lemme~\ref{lem:max_clique}) et que leur adjacence encode le 1-squelette de l'ordre~$s$-faible (Théorème~\ref{thm:cover_relations}). À partir de là, nous définissons un simplexe associé à une face du~$s$-permutaèdre et montrons que cette correspondance définit un isomorphisme inversant l'inclusion entre les faces du~$s$-permutaèdre et les simplexes de la triangulation DKK du polytope de flot~$\fpol[\oru(s)](\bfi)$ (Corollaire~\ref{cor:maximal_faces_perm}). Cela réalise le~$s$-permutaèdre en tant que polytope de grande dimension, donc pour réduire sa dimension, nous utilisons \defn{la méthode de Cayley} telle que décrite dans~\cite{S05} pour obtenir une bijection inversant l'inclusion entre le~$s$-permutaèdre et les cellules intérieures d'une subdivision mixte d'une somme d'hypercubes de dimensions variables (Théorème~\ref{thm:bij_mixed_subdiv}). Cette réalisation a la dimension souhaitée de la conjecture (Remarque~\ref{rem:dim_n_minus_1}), mais elle manque de coordonnées explicites. Pour remédier à ça, nous utilisons le résultat de DKK pour obtenir une fonction de hauteur explicite sur les routes de~$\oru(s)$ (Lemma~\ref{lem:epsilonheight}). En consequence, nous prouvons qu'il existe un arrangement d'hypersurfaces tropicales dont le dual tropical est la subdivision mixte d'hypercubes que nous obtenons en appliquant la méthode de Cayley à~$\oru(s)$ (Théorème~\ref{thm:arr_trop_hypersurfaces_s-perm}). Cela nous permet d'obtenir le~$s$-permutaèdre à travers du complexe polyédrale des cellules bornées de cet arrangement tropical (Théorème~\ref{thm:bij_trop_arr}). Comme conséquences immédiates, nous décrivons plusieurs propriétés de cette réalisation telles que ses sommets (Théorème~\ref{thm:vertices}), son hyperplan contenant (Corollaire~\ref{cor:s_containging_hyperplane}), les directions de ses arêtes montrant qu'il s'agit d'un permutaèdre généralisé (Théorème~\ref{thm:edges}), ses sommets et ses hyperplans de support (Lemme~\ref{lem:support}), et le fait qu'il est effectivement la translation d'un zonotope qui est isomorphe à un permutaèdre (Théorème~\ref{thm:s_realization_zonotope}). Le chapitre finit avec certains résultats et remarques sur la façon dont les techniques d'énumération des polytopes de flot décomposent le nombre de~$s$-objets combinatoires (Corollaire~\ref{cor:identitise s-trees}).

Enfin, dans le chapitre~\ref{chap:sorder_quotients}, nous utilisons les techniques des polytopes de flot pour donner une troisième réponse à la Perspective~\ref{pers:Coxeter_permutrees} pour les permutarbres de type~$A$. Cette réponse consiste en la description des treillis de rotation de permutarbres à travers des triangulations DKK des polytopes de flot. Nous commençons par définir les \defn{M-mouvements} sur~$\oru_n$ qui créent de nouveaux graphes encadrés appelés les \defn{~$\delta$-bicho graphes~$\bic_\delta$} en fonction de la décoration des permutarbres en jeu. Parmi eux, nous avons notre \defn{oruga graphe~$\oru_n$}, le \defn{caracol graphe~$\car_n$} de~\cite{BGHHKMY19}, et le nouveau \defn{mariposa graphe~$\mar_n$} (Remarque~\ref{rem:all_bichos_graphs}). Nous montrons que les~$\bfd$-flots entiers de~$\bic_\delta$ et les~$\delta$-permutarbres ont les mêmes cardinalités (Théorème~\ref{cor:vol_bic_bfi_permutrees}) et que l'ordre de raffinement sur les décorations des permutarbres détermine la structure des cliques maximales des routes cohérents entre les distincts~$\delta$-bicho graphes (Lemme~\ref{lem:cliques_through_M_moves}). Avec cela en main, nous prouvons que les simplexes de la triangulation DKK de~$\bic_\delta$ sont en bijection avec les~$\delta$-permutarbres (Théorème~\ref{thm:permutree_to_clique}) et que le treillis de rotation des permutarbres est encodé par les adjacences des simplexes dans la triangulation DKK (Corollaire~\ref{cor:rotation_to_adjacency}).

\subsection*{Conjectures et perspectives}

Tout au long de cette thèse, nous avons fait un usage intensif du logiciel libre SageMath~\cite{SAGE} pour les implémentations et les calculs sur les objets combinatoires que nous avons étudiés. Cela nous a permis d'obtenir une intuition sur les problèmes sur lesquels nous avons travaillé et de définir concrètement nos résultats. En particulier, nous avons obtenu des vérifications calculatoires pour plusieurs phénomènes combinatoires que nous laissons comme les Conjectures~\ref{conj:unique_smaccr},~\ref{conj:smaccr_automata},~\ref{conj:maccrs_from_smaccr},~\ref{conj:nonee_downn_flows_2} et~\ref{conj:bicho_recursion}.

De manière plus générale, nous présentons également plusieurs pistes de recherche pour des travaux futurs que nous pourrions entreprendre en suivant les problématiques traitées dans cette thèse. Pour certains d'entre eux, nous avons des résultats partiels ou une forte intuition, tandis que d'autres indiquent simplement la prochaine étape naturelle à suivre. Nous les présentons dans les Perspectives~\ref{pers:Coxeter_permutrees},~\ref{pers:sorting_network},~\ref{pers:s_associahedra},~\ref{pers:other_framings},~\ref{pers:graphs_for_zeroes},~\ref{pers:other_realizations},~\ref{pers:s_m_moves},~\ref{pers:s_permutrees_join_irreds},~\ref{pers:s_permutrees_combinatorics} et~\ref{pers:other_types}.

%% file: includes/contenu/chap_preliminaires_structures.tex

\chapter{Combinatorics and Discrete Geometry}\label{chap:prelim_structures}

\addcontentsline{lof}{part}{\protect\numberline{\thepart}Preliminaries}
\addcontentsline{lot}{part}{\protect\numberline{\thepart}Preliminaries}

\addcontentsline{lof}{chapter}{\protect\numberline{\thechapter}Combinatorics and Discrete Geometry}

For our work we denote by~$\RR^n$ the standard Euclidean space,~$\ZZ^n$ the point lattice of vectors of~$\RR^n$ with all entries integers. The standard basis of~$\RR^n$ is denoted by~$\mathbf{e_1},\ldots,\mathbf{e_n}$ and the vectors of all~$0$'s as~$\mathbf{0}$ and all~$1$'s as~$\mathbf{1}$. In general all vectors are bolded like~$\mathbf{x},\mathbf{y},\mathbf{z}$.

As notation, the sets of consecutive numbers are denoted~$[n]:=\{1,\ldots,n\}$ and in more generality,~$[i,j]:=\{i,i+1,\ldots,j-1,j\}$. For a finite set~$X$ we denote by~$|X|$ its cardinality and~$X^c$ the complement of~$X$ in its appropriate context. For subsets,~$\gbinom{[n]}{k}:=\{A\subseteq [n]\,:\, |A|=k\}$.

\begin{definition}\label{def:relations}
	Let~$R\subseteq{[n]}^2$ be a relation on~$[n]$. We say that~$R$ is \begin{itemize}
		\itemsep0em
		\item \defn{reflexive}\index{relation!reflexive} if~$(x,x)\in R$ for all~$x\in[n]$,
		\item \defn{antisymmetric}\index{relation!antisymmetric} if~$(x,y)\in R$ and~$(y,x)\in R$ implies~$x=y$, for all~$x,y\in[n]$,
		\item \defn{transitive}\index{relation!transitive} if~$(x,y)\in R$ and~$(y,z)\in R$ implies~$(x,z)\in R$, for all~$x,y,z\in[n]$.
		\item \defn{cotransitive}\index{relation!cotransitive} if~$(x,y)\notin R$ and~$(y,z)\notin R$ implies~$(x,z)\notin R$, for all~$x,y,z\in[n]$.
	\end{itemize}
	We denote by~$R^{tc}$ the \defn{transitive closure}\index{relation!transitive closure} of~$R$ (i.e.\ the smallest transitive relation containing~$R$).

\end{definition}

\section{Partial Orders and Lattices}\label{sec:orders}

This thesis revolves around objects with an associated notion of order. Therefore, we start by introducing several definitions and constructions on sets and orders on them. Most on this section is based on~\cite{S11} and~\cite{R16}.

\subsection{Partially Ordered Sets}\label{subsec:posets}

\begin{definition}\label{def:posets}
	A partially ordered set (\defn{poset}\index{poset}) consists of a discrete set~$X$ with a binary relation~$\leq$ that is reflexive, antisymmetric, and transitive. If for a pair of elements~$x,y\in X$ we have that~$x\leq y$ or~$x\geq y$ we say that they are \defn{comparable}\index{poset!comparable elements}, otherwise they are \defn{incomparable}\index{poset!incomparable elements}. Whenever all elements of~$(X,\leq)$ are comparable, we say that~$\leq$ is a \defn{total order}\index{poset!total order} on~$X$.
\end{definition}

\begin{example}Some examples of posets include:
	\begin{itemize}
		\itemsep0em
		\item Any subset of integers with the usual order.
		\item The partitions of the set~$[n]$ ordered by refinement.
		\item The set of subsets of a finite set ordered by inclusion.
		\item The divisors of an integer~$n$ ordered by divisibility.
	\end{itemize}
	Figure~\ref{fig:posets} shows examples of some of these posets.
	\begin{figure}[h!]
		\centering
		\includegraphics[scale=1.2]{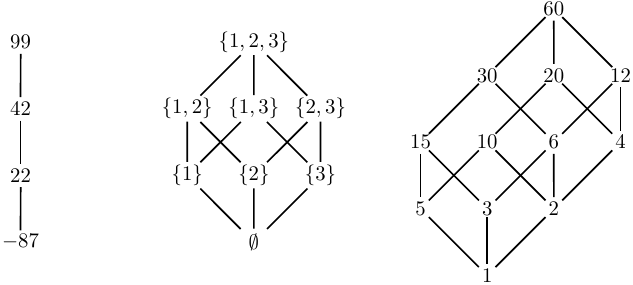}
		\caption{A poset of integers as a chain, the poset of subsets of~$[3]$, and the divisors of~$60$.}
		\label{fig:posets}
	\end{figure}
\end{example}

The theory of posets comes with a plethora of useful concepts. We only introduce a ``few'' of them here. For a more complete treatment we refer the reader to~\cite{S11}.

\begin{definition}\label{def:poset_concepts}
	Let~$(X,\leq)$ be a poset and consider elements~$x,y\in X$. The set~$[x,y]:=\{z\in X \, :\, x\leq z\leq y\}$ is called the \defn{interval}\index{poset!interval} of~$x$ and~$y$. For an interval~$[x,y]$ of cardinality 2, the pair~$(x,y)$ is called a \defn{covering relation}\index{poset!covering relation}, and we say that~$y$ \defn{covers}~$x$ denoted by \defn{$x\lessdot y$}. An element is \defn{minimal}\index{poset!minimal/maximal element} (resp.\ \defn{maximal}) if it does not cover (resp.\ is not covered by) any other element. A poset is said to be \defn{bounded}\index{poset!bounded} if it has a unique minimal element denoted by \defn{$\hat{0}$} and a unique maximal element denoted by \defn{$\hat{1}$}. The \defn{atoms}\index{poset!atom/coatom} (resp.\ \defn{coatoms}) of~$P$ are the elements that cover~$\hat{0}$ (resp.\ covered by~$\hat{1}$). The \defn{Hasse diagram}\index{Hasse diagram} of~$(X,\leq)$ is the directed graph on~$X$ where~$x\rightarrow y$ if and only if~$x\lessdot y$. Unless stated otherwise our figures of posets consist of their Hasse diagrams drawn with minimal elements at the bottom and maximal elements at the top.
\end{definition}

As is the case of many combinatorial structures, it is possible to construct other posets from an initial one. We describe several such constructions now.

\begin{definition}\label{def:poset_constructions}
	Given a poset~$(X,\leq)$ and~$Y\subseteq X$, we say that~$(Y,\leq)$ is a(n) \defn{(induced) subposet}\index{poset!subposet} of~$X$ if for~$a,b\in Y$ we have that~$a\leq b$ in~$Y$ if and only if~$a\leq b$ in~$X$. Likewise in this case~$(X,\leq)$ is a \defn{supposet}\index{poset!supposet} of~$Y$. If~$\leq'$ is a partial order over~$X$ such that~$x\leq y$ implies~$x\leq'y$ we say that \defn{${\leq} \subseteq {\leq'}$}.

	A \defn{linear extension}\index{poset!linear extension} of~$(X,\leq)$ is a poset~$(X,\leq_{tot})$ where~$\leq_{tot}$ is a total order and~${\leq} \subseteq {\leq_{tot}}$. Figure~\ref{fig:linearExtensions} shows a poset together with its linear extensions. The \defn{dual}\index{poset!dual} of~$(X,\leq)$ is the poset~$(X,\leq^*)$ where~$y\leq^*x$ if and only if~$x\leq y$.
\end{definition}

\begin{figure}[h!]
	\centering
	\includegraphics[scale=1.2]{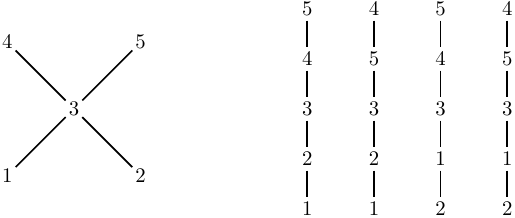}
	\caption{A poset on~$[5]$ and its linear extensions.}
	\label{fig:linearExtensions}
\end{figure}

\begin{definition}\label{def:poset_antichain_ideal_filter}
	Given a poset~$(X,\leq)$, a \defn{chain}\index{poset!chain} (resp.\ \defn{antichain}\index{poset!antichain}) is a subposet~$Y\subseteq X$ where all (resp.\ none) elements are pairwise comparable. An \defn{ideal}\index{poset!ideal} (resp.\ \defn{filter}\index{poset!filter}) of~$(X,\leq)$ is a subposet~$Y\subseteq X$ such that if~$y\in Y$ (resp.~$x\in Y$) and~$x\leq y$ then~$x\in Y$ (resp.~$y\in Y$).
\end{definition}

\begin{remark}\label{rem:bij_antichain_ideal_filter}
	Notice that whenever~$X$ is finite, antichains, ideals, and filters are all in bijection. The antichain corresponding to an ideal (resp.\ filter) is the set of its maximal (resp.\ minimal) elements in the induced subposet. Conversely, the ideal (resp.\ filter) corresponding to an antichain is the set of all elements under (resp.\ above) an element of the antichain. Ideals and filters correspond by being complements of each other.
\end{remark}

\begin{definition}\label{def:more_poset_concepts}
	A poset is \defn{finite}\index{poset!finite} if it has a finite number of elements. For a finite chain of~$X$, its \defn{length}\index{poset!length of a chain} is its number of elements minus one. The \defn{length}\index{poset!length} of a poset is the length of its longest chain. If~$X$ is bounded and all of its \defn{maximal chains}\index{poset!maximal chain} (i.e.\ chains between~$\hat{0}$ and~$\hat{1}$) have the same length, then~$X$ is said to be \defn{graded}\index{poset!graded}. The \defn{rank}\index{poset!ranked} of a graded poset is the length of its maximal chains and the \defn{rank}\index{poset!rank} of an element~$x$ is the length of the chains in~$[\hat{0},x]$.
\end{definition}

To distinguish posets from each other we consider them up to isomorphism as follows.

\begin{definition}\label{def:def:poset_isomorphism_dual}
	Let~$(X,\leq)$ and~$(Y,\leq')$ be posets. A \defn{poset isomorphism}\index{poset!isomorphism} from~$(X,\leq)$ to~$(Y,\leq')$ is a bijection~$\phi:X\rightarrow Y$ such that~$\phi(x)\leq'\phi(y)$ if and only if~$x\leq y$. That is,~$\phi$ and its inverse are both order preserving.

	If~$(X,\leq)$ and~$(X,\leq^*)$ are isomorphic then~$X$ is said to be \defn{self-dual}\index{poset!self-dual}.
\end{definition}

\subsection{Lattices}\label{subsec:lattices}

The posets we work with possess a pair of operations where given a family of elements, one can find a unique element that is minimal (resp.\ maximal) and above (resp.\ below) all the family. To define these operations we need to consider the relation of bounds within posets.

\begin{definition}\label{def:bound_meet_join}
	Let~$(X,\leq)$ be a poset with~$x,y\in X$. An \defn{upper bound}\index{poset!upper/lower bound} (resp.\ \defn{lower bound}) of~$x$ and~$y$ is an element~$z\in X$ such that~$x\leq z$ and~$y\leq z$ (resp.~$z\leq x$ and~$z\leq y$).

	The minimal upper bound (resp.\ maximal lower bound) of~$x$ and~$y$ is the least (resp.\ greatest) element in the set of upper bounds (resp.\ lower bounds) of~$x$ and~$y$, and we call it the \defn{join}\index{poset!join/meet} (resp.\ \defn{meet}) of~$x$ and~$y$ if it exists. We write \defn{$x\vee y$} for the join of~$x$ and~$y$ and \defn{$x\wedge y$} for the meet of~$x$ and~$y$. For a nonempty subset~$S\subset X$ of a lattice, we can denote by \defn{$\bigwedge S$} (resp.\ \defn{$\bigvee S$}) the meet (resp.\ join) of all elements of~$S$. The meet (resp.\ join) of a single element is the element itself.
\end{definition}

Given these operations one can think about expressing elements of a lattice as meets (resp.\ joins) of other elements. Of particular interest are those that cannot be expressed in such a way.

\begin{definition}\label{def:joinmeet_irreducibles}
	Given a lattice~$(L,\leq)$, an element~$x\in L$ is said to be \defn{meet-irreducible}\index{lattice!meet-irreducible} (resp.\ \defn{join-irreducible}\index{lattice!join-irreducible}) if it covers (resp.\ is covered by) exactly one element. That is, an element~$w$ is a meet-irreducible (resp.\ join-irreducible) if there does not exist a set~$S\subset X$ of elements such that~$w=\bigwedge S$ (resp.~$w=\bigvee S$). 
\end{definition}

\begin{definition}\label{def:lattice}
	A \defn{lattice}\index{lattice} is a poset~$(X,\leq)$ such that for every subset~$S\subseteq X$ the elements~$\bigvee S$ and~$\bigwedge S$ exist. If~$(X,\leq)$ only has a meet (resp.\ join) operation it is called a \defn{meet-semilattice}\index{lattice!meet-semilattice} (resp.\ \defn{join-semilattice}\index{lattice!join-semilattice}).
\end{definition}

\begin{figure}[h!]
	\centering
	\includegraphics[scale=1.2]{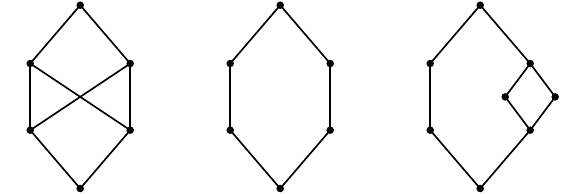}
	\caption[A poset, a polygonal lattice, and a non-polygonal lattice.]{ A poset that is not a lattice (left), a lattice that is polygonal (middle), and a lattice that is not polygonal (right).}
	\label{fig:latticeNonLattice}
\end{figure}

Notice that all finite lattices have~$\hat{0}$ and~$\hat{1}$ since we can calculate (very inefficiently) the meet or join of all elements of the lattice. Figure~\ref{fig:latticeNonLattice} presents a poset that is not a lattice and two lattices. In the case of the converse direction we have the following theorem.

\begin{proposition}[{\cite[Prop.3.3.1]{S11}},{\cite[Lem.9-2.1]{R16}}]\label{prop:semilattice_to_lattice}
	If~$X$ is a finite meet-semilattice with~$\hat{1}$ then~$X$ is a lattice. Dually, a finite join-semilattice~$X$ with~$\hat{0}$ is a lattice.
\end{proposition}

\begin{definition}\label{def:polygons}
	A \defn{polygon}\index{lattice!polygon} inside a lattice is an interval~$[x,y]$ that is the union of two finite maximal chains between~$x$ and~$y$ whose intersection is only~$x$ and~$y$. We say that a lattice is \defn{polygonal}\index{lattice!polygonal} if the following occurs:
	\begin{enumerate}
		\itemsep0em
		\item if~$y_1$ and~$y_2$ are different elements that cover~$x$, then~$[x,y_1\vee y_2]$ is a polygon,
		\item if~$x_1$ and~$x_2$ are different elements that are covered by~$y$, then~$[x_1\wedge x_2,y]$ is a polygon.
	\end{enumerate}

	For a polygon~$[x,y]$, the incident edges to~$x$ (resp.~$y$) are called the \defn{bottom edges} (resp.\ \defn{top edges}) of~$[x,y]$ and the others are called \defn{side edges}. See Figure~\ref{fig:latticeNonLattice} for examples of when a lattice is polygonal and when it is not.
\end{definition}

\subsection{Lattice Congruences}\label{subsec:congruences}

Just as with normal sets, we can define equivalence relations over posets. Of special interest for us are equivalence relations on lattices that are compatible with the meet and join operations.

\begin{definition}\label{def:Congruences}
	Given a lattice~$(L,\leq)$, a \defn{congruence}\index{lattice!congruence} on~$L$ is an equivalence relation~$\equiv$ on~$L$ such that for all~$x_1,x_2,y_1,y_2\in L$ if~$x_1\equiv x_2$ and~$y_1\equiv y_2$ then~$x_1\vee y_1\equiv x_2\vee y_2$ and~$x_1\wedge y_1\equiv x_2\wedge y_2$.

	If~$x\lessdot y$ and~$x\equiv y$ we say that~$\equiv$ \defn{contracts}\index{lattice!contraction} the edge~$x\lessdot y$.
\end{definition}

The following characterization of when an equivalence relation is a lattice congruence is a key element for our work.

\begin{proposition}[{\cite[Prop.9-5.2]{R16}}]\label{prop:latti_cong_minimal_elements_equivalence}
	Let~$\equiv$ be an equivalence relation over a lattice~$L$. Then~$\equiv$ is a lattice congruence if and only if all the following conditions hold:
	\begin{enumerate}
		\itemsep0em
		\item each equivalence class of~$\equiv$ is an interval of~$L$,
		\item the mappings~$\pi_\uparrow^\equiv$ and~$\pi_\downarrow^\equiv$ that respectively send an element to the maximal and minimal representative of its class are order preserving.
	\end{enumerate}
\end{proposition}

\begin{definition}\label{def:lattice_quotient}
	Given a lattice congruence~$\equiv$ on a lattice~$(L,\leq)$, the \defn{lattice quotient}\index{lattice!quotient} of~$L$ by~$\equiv$ is the lattice~$L / {\equiv}$ whose elements are equivalence classes of~$\equiv$ with relations~$A\preceq B$ if and only if there exists~$x\in A$ and~$y\in B$ such that~$x<y$. The meet~$A\wedge B$ (resp.\ join~$A\vee B$) consists of the only equivalence class that contains all~$x\wedge x'$ (resp.~$x\vee x'$) for~$x\in A$ and~$x'\in B$.
\end{definition}

It is possible to give a second nice characterization of lattice quotients using lattice homomorphisms in the following way.

\begin{definition}\label{def:lattice_homomorphism}
	Given lattices~$L$ and~$L'$, a map~$f:L\to L'$ is a \defn{lattice homomorphism} if for all~$x,y\in L$ we have that~$f(x\wedge_{L} y)= f(x)\wedge_{L'} f(y)$ and~$f(x\vee_{L} y)= f(x)\vee_{L'} f(y)$. 
\end{definition}

\begin{proposition}\label{prop:lattice_homomorphism_interval_characterization}
	A surjective map~$f:L\to L'$ is a lattice homomorphism if and only if \begin{enumerate}
		\itemsep0em
		\item $f$ is order preserving,
		\item for every interval~$[x,y]$ of~$L'$ the fiber~$f^{-1}([x,y])$ is an interval.
	\end{enumerate} 
\end{proposition}

\begin{definition}\label{def:down_up_operators}
	For a lattice~$L$ with a lattice congruence~$\equiv$, the set~\defn{$\pi_\downarrow^\equiv L$}~$:=\{\pi_\downarrow^\equiv (x)\,:\,x\in L\}$ is a poset as an induced poset of~$L$. Similar for~\defn{$\pi_\uparrow^\equiv L$}.
\end{definition}

\begin{proposition}[{\cite[Prop.9-5.5]{R16}}]\label{prop:down_up_operators_give_lattices}
	Let~$L$ be a lattice and~$\equiv$ a lattice congruence on~$L$. Then~$\pi_\downarrow^\equiv L$ (resp.~$\pi_\uparrow^\equiv L$) is a lattice and it is isomorphic to the quotient lattice~${L}/{\equiv}$. Moreover,~$\pi_\downarrow^\equiv$ (resp.~$\pi_\uparrow^\equiv$) is a lattice homomorphism. 
\end{proposition}

Out of the characteristics of lattice quotients, we are interested in how they preserve the property of polygonality.

\begin{proposition}[{\cite[Prop.9-6.2]{R16}}]\label{prop:quotients_of_polygons}
	Let~$L$ be a finite polygonal lattice and~$\equiv$ a congruence on~$L$. Then the lattice quotient~${L}/{\equiv}$ is polygonal.
\end{proposition}

\begin{remark}\label{rem:lattice_quotients_via_polygons}
	In the case where a lattice is polygonal, lattice congruences may be described in terms of merging/contracting edges. Here we provide the superficial details for this idea together with Figure~\ref{fig:polygonality_quotients}. See~\cite{R16} for the explicit construction. Let~$(x,y_1)$ be an edge of a polygonal lattice~$L$ contracted by an equivalence relation~$\equiv$. As~$L$ is polygonal,~$(x,y_1)$ is in a family of polygons. If~$x$ is not the minimal element and~$y_1$ is not the maximal element a polygon, one can easily check that no calculations of meets or joins have been affected. In the case that~$x$ is minimal, our polygon is~$[x,y_1\vee y_2]$, and we have that~$y_1\vee y_2=y_1\vee z\equiv x\vee z=z$ for any~$z\in[y_2,y_1\vee y_2]$. Thus, all elements of~$[y_2,y_1\vee y_2]$ form an equivalence class. Repeating this with the meet operation results on the polygon collapsing onto only two equivalence classes with minimal elements~$x$ and~$y_2$. This happens in each polygon every time a new edge is contracted. We also refer to this phenomenon as the \defn{forcing}\index{lattice!congruence forcing} of the congruence~$\equiv$\footnote{Forcing is a term most known from its use in set theory. In our case of the study of congruences in lattice theory one can find a vague use of this term in~\cite{W03}. For a concrete use of the term in our context we refer the reader to~\cite{R16}. In recent times one can see that the term has become standard through its appearance in works such as~\cite{PS17},~\cite{PPR22}, and~\cite{BM21}.}.

	Informally speaking, a contraction either contracts no other edge of a polygon if it is on a side of the polygon or contracts both sides together with a minimal relation and its maximal opposite. Figure~\ref{fig:polygonality_quotients} shows examples of this where the original contracted edge is in red and the resulting equivalence classes are in blue. The minimal elements of the equivalence classes are in black.
\end{remark}
\vspace{-0.4cm}
\begin{figure}[h!]
	\centering
	\includegraphics[scale=2.2]{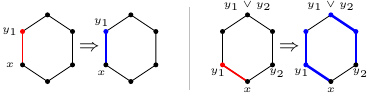}
	\caption[Examples of congruences on polygonal lattices generated by an edge contraction.]{ Examples of congruences (in heavily bolded blue) on polygonal lattices generated by an edge contraction (in bolded red).}
	\label{fig:polygonality_quotients}
\end{figure}

Figure~\ref{fig:quotient_propagation} shows an example of a lattice congruence forcing through a bigger lattice following the cases that can happen in each polygon as in Figure~\ref{fig:polygonality_quotients}. At each step the new contractions are in red and the old ones in blue. Again, the minimal elements are those colored with black. Only the last poset is a lattice.

\begin{figure}[h!]
	\centering
	\includegraphics[scale=0.8]{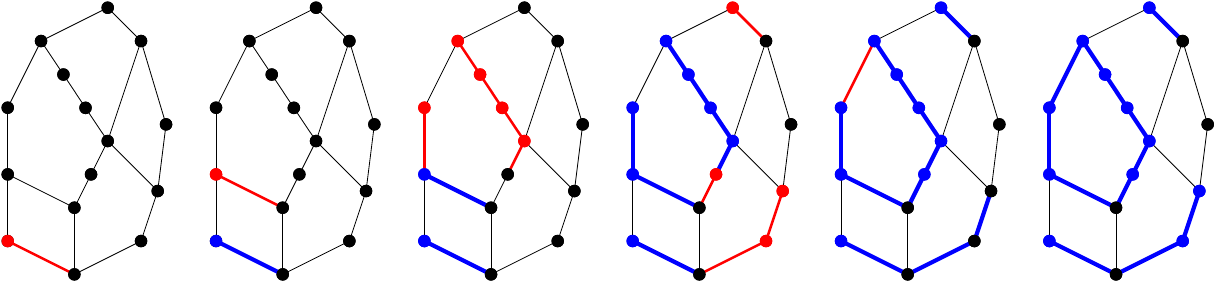}
	\caption[A lattice with a congruence relation forcing through it.]{ A lattice with a congruence relation forcing through it following the polygon cases of Figure~\ref{fig:polygonality_quotients}. The equivalence classes at each step are heavily bolded in blue while the new edges to be contracted are bolded in red.}
	\label{fig:quotient_propagation}
\end{figure}

\section{Convex Geometry}\label{sec:convex_geometry}

In this section we present a selection of basic definitions and results on polytopes based mostly on~\cite{GKZ94},~\cite{Z95}, and~\cite{DRS10}. For a more in depth introduction to polytopes and finer details we refer the reader to~\cite{Z95} and~\cite{DRS10}.

\subsection{Polytopes}\label{subsec:polytopes}

\begin{definition}\label{def:convexity}
	Let~$\cA\subseteq\RR^n$ be a set of~$k\in\ZZ_{\geq0}$ points.

	The \defn{affine hull}\index{hull!affine} of~$\cA$ is \begin{equation*}
		\aff(\cA)=\left\{\sum_{i=1}^k\lambda_i a_i \,:\, a_i\in\cA,\,\lambda_i\in\RR,\,\sum_{i=1}^k\lambda_i=1\right\}.
	\end{equation*} The \defn{cone}\index{cone} or \defn{conic hull}\index{hull!conic} of~$\cA$ denoted~$\cone(\cA)$ is the set of all conic combinations of~$\cA$, that is \begin{equation*}
		\cone(\cA)=\left\{\sum_{i=1}^k\lambda_i a_i \,:\, a_i\in\cA,\,\lambda_i\in\RR_{\geq 0}\right\}.
	\end{equation*}
	Finally, their \defn{convex hull}\index{hull!convex}~$\conv(A)$ is the set of all convex combinations of points in~$\cA$. That is, \begin{equation*}
		\conv(\cA)=\left\{\sum_{i=1}^k\lambda_i a_i \,:\, a_i\in\cA,\,\lambda_i\in\RR_{\geq 0},\,\sum_{i=1}^k\lambda_i=1\right\}.
	\end{equation*}
	Although the definitions are slightly different, the sets change quickly. Figure~\ref{fig:hulls} shows an example of how much these hulls can differ for the same point configuration.

	We say that a set~$\cA$ is \defn{convex}\index{convex set} if~$\conv(\cA)=\cA$. We denote the segment~$\conv(\{\mathbf{x},\mathbf{y}\})$ by \defn{$[{\mathbf{x},\mathbf{y}}]$}.
\end{definition}

\begin{figure}[h!]
	\centering
	\includegraphics[scale=0.9]{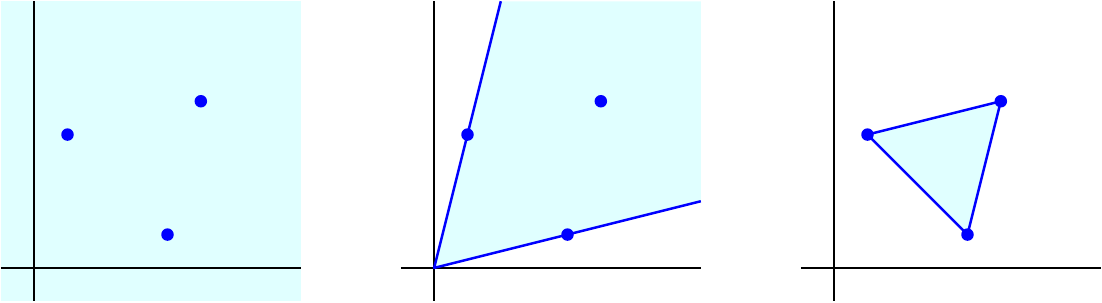}
	\caption[The affine, conic, and convex hulls of a set of points.]{ The affine (left), conic (middle), and convex (right) hulls of a set of points.}
	\label{fig:hulls}
\end{figure}

\begin{definition}\label{def:hyperplane}
	A \defn{hyperplane}\index{hyperplane} in~$\RR^n$ is a set of the form~$\{\mathbf{x}\in\RR^n\, :\,\langle \mathbf{x},\mathbf{v}\rangle=c\}$ obtained by fixing a vector~$\mathbf{v}\in\RR^n\setminus\{\mathbf{0}\}$ and a scalar~$c\in\RR$ where~$\langle \mathbf{x},\mathbf{v}\rangle$ denotes the inner product $x_1v_1+\cdots+x_nv_n$. That is, a subspace of codimension 1 with \defn{normal vector}\index{hyperplane!normal vector}~$v$. Visually, the vector~$\mathbf{v}$ give the orthogonal direction the subspace while~$c$ dictates how shifted it is the origin. In particular,~$\mathbf{0}$ is a point of the subspace if and only if~$c=0$. Every hyperplane separates~$\RR^n$ into two halfspaces. We say that a hyperplane~$H$ \defn{supports} a set~$\cA$ (is a \defn{supporting hyperplane} of~$\cA$) if \begin{itemize}
		\itemsep0em
		\item~$\cA$ is contained completely in one of its halfspaces (i.e.~$\langle \mathbf{x},\mathbf{v}\rangle\geq c$ or~$\langle \mathbf{x},\mathbf{v}\rangle\leq c$ for all~$x\in \cA$),
		\item~$H\cap \cA$ is not empty.
	\end{itemize}
\end{definition}

\begin{definition}\label{def:polytope}{\cite[Thm.1.1]{Z95}}
	A \defn{polytope}\index{polytope}~$P$ is equivalently one of the following:
	\begin{itemize}
		\itemsep0em
		\item \defn{$\cV$-description}\index{polytope!$\cV$-description}: the convex hull of a finite number of points,
		\item \defn{$\cH$-description}\index{polytope!$\cH$-description} the bounded intersection of a finite number of halfspaces.
	\end{itemize}
	Figure~\ref{fig:VHDescriptions} shows an example of a polytope via both descriptions.
	\begin{figure}[h!]
		\centering
		\includegraphics[scale=0.7]{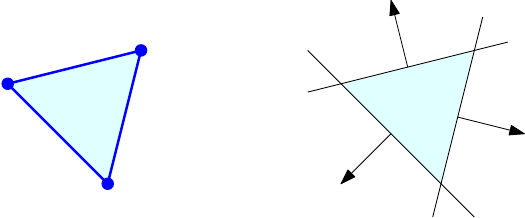}
		\caption[A polytope in~$\RR^2$ with its~$\cV$-description and~$\cH$-description.]{ A polytope in~$\RR^2$ with its~$\cV$-description (left) and~$\cH$-description (right).}
		\label{fig:VHDescriptions}
	\end{figure}

	The \defn{dimension}\index{polytope!dimension} of a polytope~$P$ is the dimension of~$\aff(P)$. The \defn{faces}\index{polytope!face} of~$P$ are the intersections~$P\cap H$ where~$H$ is a supporting hyperplane of~$P$. It follows that the faces of~$P$ are also polytopes themselves. Notice that the empty set~$\emptyset$ and the whole polytope~$P$ are faces of~$P$. Certain faces of a polytope of dimension~$d$ have their own names: \begin{itemize}
		\itemsep0em
		\item A \defn{vertex}\index{polytope!vertex} of~$P$ is a face of dimension 0.
		\item An \defn{edge}\index{polytope!edge} of~$P$ is a face of dimension 1.
		\item A \defn{facet}\index{polytope!facet} of~$P$ is a facet of dimension~$d-1$ (codimension 1 with respect to~$P$).
	\end{itemize}
\end{definition}

The~$\cV$ and~$\cH$ descriptions of polytopes are both useful, so we use them interchangeably. Their equivalence is not trivial by any means~\cite{Z95}, in particular for computational issues.

\begin{definition}\label{def:polyhedron}
	An unbounded intersection of a finite amount of halfspaces is called a~\defn{polyhedron}\index{polyhedron}.
\end{definition}

\begin{remark}\label{rem:relative_dimension}
	Notice that polytopes do not need to have the maximal dimension possible. Whenever this happens we say that the polytope is \defn{full dimensional}\index{polytope!full-dimensional}. Most of our polytopes are not full dimensional but do live in nice subspaces of~$\RR^n$.
\end{remark}

\begin{definition}\label{def:simple_simplicial}
	Let~$P$ be a~$d$-dimensional polytope in~$\RR^n$. We say that~$P$ is \begin{itemize}
		\itemsep0em
		\item \defn{simplicial}\index{polytope!simplicial} if every facet of~$P$ has exactly~$d$ vertices,
		\item \defn{simple}\index{polytope!simple} if every vertex of~$P$ is contained in exactly~$d$ facets of~$P$.
	\end{itemize}
\end{definition}

\begin{example}\label{ex:simplex}
	A polytope~$P$ in~$\RR^n$ that is the convex hull of~$d\leq n+1$ affinely independent points is called a \defn{simplex}\index{polytope!simplex}. Notice that such a simplex is a polytope of dimension~$d-1\leq n$ and that its faces are parametrized by subsets of~$[d]$. The \defn{standard simplex}~$\Delta_{n-1}$ is defined as~$\conv\big(\mathbf{e}_i \,:\, i\in[n]\big)$. Notice that~$\Delta_{n-1}$ is both simple and simplicial and its faces are parametrized via all subsets of~$[n]$. Its~$\cH$-description consists of the inequalities~$x_i\geq 0$ for~$i\in[n]$ together with~$\sum_{i\in [n]}x_i=1$.
\end{example}

\begin{example}\label{ex:cube}
	The \defn{standard cube}~$\PCube_n$ is the~$n$-dimensional polytope~$\conv\big((x_1,\ldots,x_n)\,:\, x_i\in\{0,1\} \big)$. Notice that~$\PCube_n$ is simple but not simplicial. Its~$\cH$-description consists of the inequalities~$0\leq x_i$ and~$x_i\leq 1$ for all~$i\in[n]$. Any other polytope that is combinatorially equivalent (see Definition~\ref{def:realization}) to~$\PCube_n$ is called a \defn{cube}\index{polytope!cube}.
\end{example}

There are diverse ways to obtain new polytopes from old ones. We are particularly interested in the following two.

\begin{definition}\label{def:Minkowski_sum}
	Given polytopes~$P,Q$ in~$\RR^n$ their \defn{Minkowski sum}\index{polytope!Minkowski sum} is \begin{equation*} P+Q=\conv\big(\mathbf{p}+\mathbf{q} \, :\, \mathbf{p}\in P,\, \mathbf{q}\in Q\big).\end{equation*} We denote the sum of~$k$ copies of~$P$ by~$kP$. Figure~\ref{fig:minkowski_sum} shows an example of a Minkowski sum.
\end{definition}
\vspace{-0.4cm}
\begin{figure}[h!]
	\centering
	\includegraphics[scale=1.2]{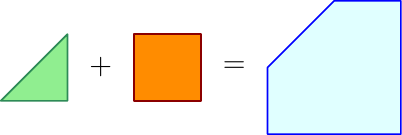}
	\caption[The Minkowski sum of two polytopes.]{ The Minkowski sum of a triangle and a square.}
	\label{fig:minkowski_sum}
\end{figure}
\vspace{-0.4cm}
\begin{example}\label{ex:zonotope}
	A \defn{zonotope}\index{polytope!zonotope} is a polytope of the form~$Z=\mathbf{z}+\sum_{i=1}^k [\mathbf{0},\mathbf{x_i}]$ for a collection of vectors~$\mathbf{z},\mathbf{x_1},\ldots,\mathbf{x_k}\in\RR^n$. Equivalently said, a zonotope is a Minkowski sum of line segments or the affine projection of a cube~\cite{Z95}.
\end{example}

\begin{definition}\label{def:removing_hyperplanes}
	Let~$P$ be a polytope defined by a family of halfspaces~$\cH$. If for a subset~$\cH'\subset \cH$ the intersection of the corresponding halfspaces is still bounded, then the resulting polytope~$Q$ is said to be obtained by \defn{removing facets}\index{polytope!removing facets} of~$P$. An example of this process is shown in Figure~\ref{fig:removahedron}.
\end{definition}
\vspace{-0.25cm}
\begin{figure}[h!]
	\centering
	\includegraphics[scale=0.65]{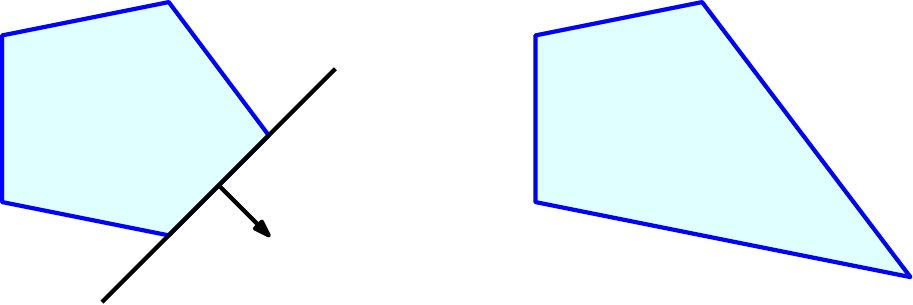}
	\caption{A polytope obtained by removing facets.}
	\label{fig:removahedron}
\end{figure}

\subsection{Fans}\label{ssec:fans}

Cones are useful to us as they share an intimate relation with polytopes. We begin by describing how they share some properties of polytopes.

\begin{definition}\label{def:cones}
	Given a cone~$C$, the \defn{dimension}\index{cone!dimension} of~$C$ is the dimension of its affine hull~$\aff(C)$. The faces of dimension~$1$ of~$C$ are called \defn{rays}\index{cone!ray} and faces of dimension~$d-1$ are called \defn{facets}\index{cone!facet}. A cone of~$\cF$ of codimension~$0$ is called a \defn{chamber}\index{cone!chamber}.
\end{definition}

\begin{definition}\label{def:fans}
	A collection of cones~$\cF$ in~$\RR^n$ is a \defn{fan}\index{fan} if it satisfies: \begin{itemize}
		\itemsep0em
		\item if~$F$ is a face of a cone~$C\in\cF$, then~$F\in\cF$,
		\item if~$C_1,C_2\in\cF$, then the cone~$C_1\cap C_2$ is a common face of~$C_1$ and~$C_2$.
	\end{itemize}
	We say that~$\cF$ is \defn{complete}\index{fan!complete} if~$\bigcup_{C\in\cF}C=\RR^n$, and that~$\cF$ is \defn{pointed}\index{fan!pointed} if~$\{\mathbf{0}\in\cF\}$. The \defn{dimension}\index{fan!dimension} of~$\cF$ is the dimension of the affine hull of the union of its cones.
\end{definition}

In our case all fans are assumed to be complete and pointed. Moreover, we can construct a fan from a polytope in the following way.

\begin{definition}\label{def:normal_fan}
	Let~$P$ be a polytope. A \defn{normal vector}\index{polytope!normal vector} of a facet~$F$ of~$P$ is the normal vector of a hyperplane that supports~$F$. The \defn{normal cone}\index{polytope!normal cone} of a non-empty face~$F$ is the cone formed from all normal vectors of~$F$. The collection of all normal cones of the faces of~$P$ is called the \defn{normal fan}\index{polytope!normal fan} of~$P$ and denoted~$\cN(P)$. Figure~\ref{fig:normalFan} shows an example of a polytope together with its normal fan.
\end{definition}

\begin{figure}[h!]
	\centering
	\includegraphics[scale=0.6]{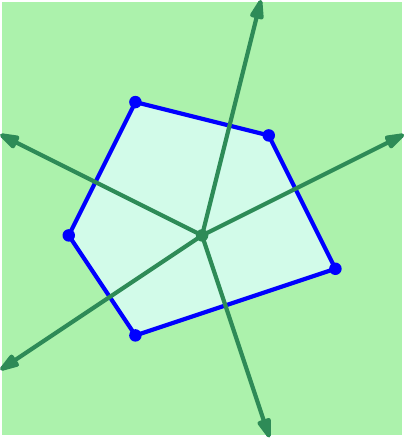}
	\caption[A polytope together with its normal fan.]{ A polytope (blue) together with its normal fan (green).}
	\label{fig:normalFan}
\end{figure}

Since normal cones come from polytopes, their faces have a dual-like description based on their original polytopes.

\begin{remark}\label{rem:bijection_normal_cone_polytope}
	Let~$P$ be a polytope and~$\cN(P)$ its normal cone. The following are in correspondence:
	\begin{itemize}
		\itemsep0em
		\item vertices of~$P$ with chambers of~$\cN(P)$,
		\item edges of~$P$ with facets of~$\cN(P)$,
		\item facets of~$P$ with rays of~$\cN(P)$.
	\end{itemize}
	In general, if~$P$ is full dimensional, the~$k$-dimensional faces of~$P$ correspond to~$(n-k)$-dimensional faces of~$\cN(P)$.
\end{remark}

\begin{definition}\label{def:fan_coarseing_refinement}
	Given two fans~$\cF_1,\cF_2$ of~$\RR^n$, we say that~$\cF_1$ \defn{refines}\index{fan!refinement}~$\cF_2$ if every cone of~$\cF_2$ is a union of cones of~$\cF_1$. Conversely, we say that~$\cF_2$ \defn{coarsens}\index{fan!coarsening}~$\cF_1$.
\end{definition}

\begin{example}\label{ex:fan_examples} The following are some examples of fans:
	\begin{itemize}
		\itemsep0em
		\item For the simplex~$\conv(\{\mathbf{0},\mathbf{e_1},\ldots,\mathbf{e_n}\})$ in~$\RR^n$, its fan is the fan generated by the collection of vectors~$-\mathbf{e_1},\ldots,-\mathbf{e_n}$ and~$\mathbf{1}$.
		\item For a zonotope of the form~$Z=\mathbf{z}+\sum_{i=1}^k [\mathbf{0},\mathbf{x_i}]$ the normal fan is the fan given by the arrangement of the respective hyperplanes normal to~$\mathbf{x_1},\ldots,\mathbf{x_k}$.
		\item For a Minkowski sum~$P+Q$, the normal fan~$\cN(P+Q)$ is the common refinement of~$\cN(P)$ and~$\cN(Q)$.
	\end{itemize}
\end{example}

\subsection{Realizations}\label{sec:realizations}

Let~$P$ be a polytope in~$\RR^n$ with vertices~$\{\mathbf{p_1},\ldots,\mathbf{p_r}\}$. The search of which sets of vertices can play the role of defining a particular polytope is not easy. Moreover, knowing when two sets of vertices give the same polytope is complicated as well. Similarly, if we are given a set of vertices and its convex hull, its is not immediate to know how much one can budge a vertex and still preserve the combinatorics of the original polytope. This begs the question, is there a `good' set of vertices that can be constructed algorithmically/combinatorially to define a particular polytope? This can be called the realization problem over polytopes. Linked to this problematic, another classical problem of polytopes is to understand the combinatorial structure of its~$1$-skeleton.

\begin{definition}\label{def:1-skeleton}
	The \defn{$1$-skeleton} of~$P$ is the graph on the vertices and edges of~$P$. In the case that there is another polytope~$Q$ such that the~$1$-skeleton of~$P$ is isomorphic to a subgraph of the~$1$-skeleton of~$Q$, we say that~$P$ is \defn{embeddable}\index{polytope!embeddable} in~$Q$.
\end{definition}

The~$1$-skeleton of a polytope sometimes corresponds to the Hasse diagram of a poset by orienting the polytope in a generic direction. A non-example of this is the simplex~$\Delta_n$ as it does not have a corresponding poset since a Hasse diagram cannot have triangles. In Chapter~\ref{chap:sorder_realizations} we study this particular problem for a particular polytope. First we need the following notions to get there.

\begin{definition}\label{def:face_lattice}
	The \defn{face lattice}\index{polytope!face lattice}~$Fl(P)$ of a polytope~$P$ is the poset of all faces of~$P$ ordered by inclusion. The meet of two faces is their intersection and the join is the smallest face containing them.
\end{definition}

\begin{definition}\label{def:realization}
	A polytope~$Q={\conv(\{\mathbf{q}_i\}_{i=1}^{r})}$ is \defn{combinatorially equivalent}\label{polytope!equivalence!integral} to~$P$ if the face lattices~$Fl(P)$ and~$Fl(Q)$ are isomorphic. Given a poset~$(F,\leq)$, we say that~$P$ \defn{realizes}\label{poset!realization} the order~$(F,\leq)$ if there is an order-preserving bijection between an orientation of the~$1$-skeleton of~$P$ and the elements of~$F$.
\end{definition}

It is possible to study in more generality the realizations of polytopes via their realization space, but such topic falls outside the scope of this thesis. We refer the reader to~\cite{RG06}.

\begin{example}\label{ex:boolean_realization_cube}
	The boolean lattice on~$[n]$ is realized by the~$\PCube_n$, and it is the face lattice of the simplex~$\Delta_{n-1}$. In this case each vertex of the cube corresponds to the characteristic vectors of a subset of~$[n]$. The case~$n=3$ is shown in Figure~\ref{fig:realizationCube}.
\end{example}

\begin{figure}[h!]
	\centering
	\includegraphics[scale=1]{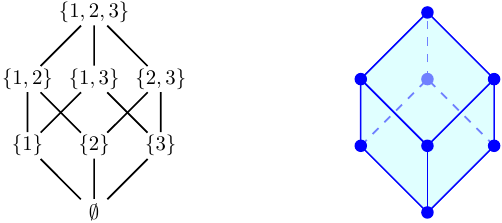}
	\caption{The boolean lattice of~$[3]$ and one of its realizations.}
	\label{fig:realizationCube}
\end{figure}

We now give other notions under which polytopes can be considered equivalent which are exemplified in Figure~\ref{fig:polytopeEquivalences}

\begin{definition}\label{def:combinatorial_equivalence}
	Let~$P\subset \RR^n$ and~$Q\subset\RR^m$ be two polytopes. We say that~$P$ and~$Q$ are \begin{itemize}
		\itemsep0em
		\item \defn{normal equivalent}\index{polytope!equivalence!normal} if we have~$\cN(P) = \cN(Q)$ for their normal fans,
		\item \defn{integrally equivalent}\index{polytope!equivalence!integral} if there exists an affine function~$f:\RR^n\to \RR^m$ whose restriction~$f:P\to Q$ is a bijection and~$f(P\cap\ZZ^n)=Q\cap\ZZ^m$, that is,~$f$ preserves the lattice structure.
	\end{itemize}
\end{definition}

\begin{figure}[h!]
	\centering
	\includegraphics[scale=0.8]{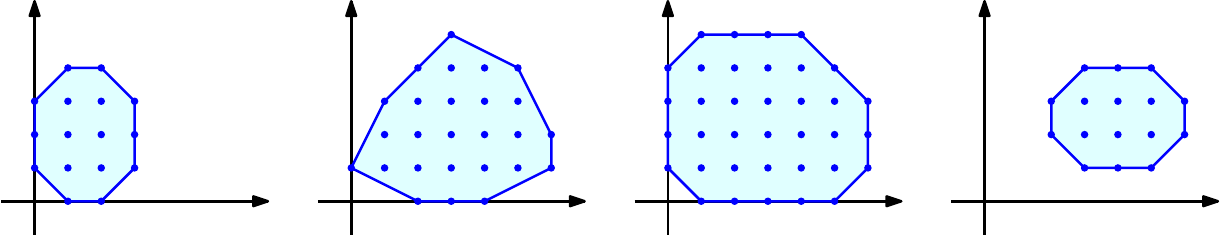}
	\caption[A collection of combinatorially equivalent polytopes.]{ A collection of combinatorially equivalent polytopes. All except the second one are normal equivalent and only the first and last are integrally equivalent.}
	\label{fig:polytopeEquivalences}
\end{figure}
\vspace{-0.2cm}

\subsection{Complexes and Subdivisions}\label{sec:subdivisions}

Apart from studying polytopes by themselves, we can study the structures they make gluing them together into complexes or dividing them into smaller polytopes giving subdivisions.

\begin{definition}\label{def:polytope_complex}
	A \defn{polyhedral complex}\index{complex!polyhedral}~$\cC$ is a collection of polyhedra such that \begin{enumerate}
		\itemsep0em
		\item if~$P\in\cC$ and~$F$ is a face of~$P$, then~$F\in\cC$,
		\item if~$P_1,P_2\in\cC$, then~$P_1\cap P_2$ is a face of both~$P_1$ and~$P_2$.
	\end{enumerate}
	The polyhedra forming a complex~$\cC$ are called the \defn{cells}\index{complex!cells} of the complex. A cell is a \defn{boundary cell} if it has a supporting hyperplane that also supports~$\cC$. Otherwise, it is an \defn{interior cell}.

	The polyhedral subcomplex of cells of dimension~$k$ is the \defn{$k$-dimensional skeleton} of~$\cC$. If all maximal cells under inclusion have the same dimension, then~$\cC$ is called a \defn{pure complex}\index{complex!pure}.

	A polyhedral complex is a \defn{polytopal complex}\index{complex!polytopal} if all its polyhedra are bounded (that is, polytopes). The \defn{face poset}\index{complex!face poset} of a polyhedral complex~$\cC$ is the poset~$(\cC,\subseteq)$ and the dimension of~$\cC$ is the maximal dimension of its polyhedra.
\end{definition}

See Figure~\ref{fig:polytopal_complex} for an example of a polytopal complex and a polyhedral complex.

\begin{figure}
	\centering
	\includegraphics[scale=1.5]{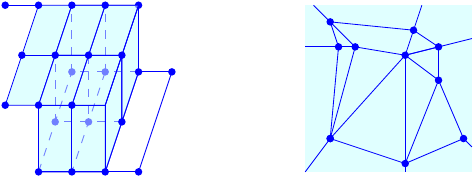}
	\caption[A non-pure polytopal complex and a pure polyhedral complex.]{ A non-pure polytopal complex of dimension~$3$ and a pure polyhedral complex of dimension~$2$.}
	\label{fig:polytopal_complex}
\end{figure}

\begin{example}\label{ex:face_posets}
	Let~$H$ be a polyhedron.
	\begin{itemize}
		\item The complex~$\cC(H)$ is the polytopal complex formed by all of its faces. If~$H$ is a polytope, the face poset is exactly the face lattice~$Fl(H)$.
		\item Let~$H$ be a polyhedron. The boundary complex~$\cC(\partial H)$ is the subcomplex of~$\cC(H)$ given by all proper faces of~$H$.
	\end{itemize}
\end{example}

We are interested in a particular complex of polytopes that is obtained from a polytope by dividing it geometrically.

\begin{definition}\label{def:subdivision}
	Consider a point configuration~$\cA$ with polytope~$P=\conv(\cA)$ and~$\dimension(P)=n$. A \defn{polytopal subdivision}\index{polytope!polytopal subdivision} of~$\cA$ is a collection of polytopes~$\cS=\{S_i:=\conv(\cA_i)$\} such that \begin{enumerate}
		\itemsep0em
		\item~$\cA_i\subset \cA$ and~$\dimension(S_i)=n$ for all~$i$,
		\item any intersection~$S_i\cap S_j$ is a face of~$S_i$ and~$S_j$ and~$\cA_i\cap S_i\cap S_j = \cA_j\cap S_i\cap S_j$  for all pairs~$i,j$,
		\item~$\bigcup_i S_i=P$.
	\end{enumerate}
	When all points in~$\cA$ are vertices of~$P$ we say that~$\cS$ is a \defn{polytopal subdivision} of~$P$.
\end{definition}

\begin{definition}\label{def:subdivision_refinement}
	Let~$\cS=\{S_i\}$,~$\cT=\{T_j\}$ be two subdivisions of a point configuration~$\cA$. We say that~$\cS$ \defn{refines}\label{polytope!subdivision refinement}~$\cT$ if for every~$T_j$, the collection of~$S_i$ such that~$S_i\subseteq T_j$ is a subdivision of~$T_j$. Moreover, this implies that the set of subdivisions of~$\cA$ is a partially ordered set with respect to refinement. The minimal elements of such a poset are the \defn{triangulations}\index{polytope!triangulation} of~$\cA$ and the maximal element is~$P$. Figure~\ref{fig:subdivision_triangulation} shows an example of a general subdivision and a triangulation of a point configuration.
\end{definition}

\begin{figure}[h!]
	\centering
	\includegraphics[scale=1]{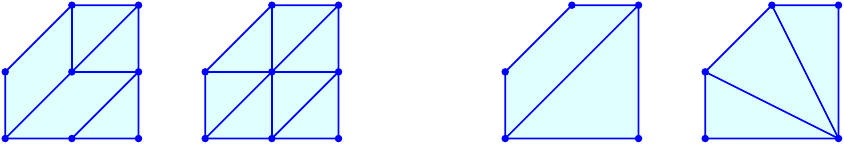}
	\caption[Two subdivisions of a point configuration and two subdivisions of a polytope.]{ Two subdivisions of a point configuration and two subdivisions of a polytope. Only the second and fourth subdivisions are triangulations in their respective cases.}
	\label{fig:subdivision_triangulation}
\end{figure}
\begin{definition}
	Let~$\cA$ be a point configuration in~$\RR^n$. A \defn{height function}\label{height function} is a function~$h:\cA\to \RR_{\geq 0}$.
\end{definition}

We now show a way to use height functions to obtain subdivisions of polytopes.

\begin{definition}\label{def:subdivison_from_height_function}
	Let~$h$ be a height function of a point configuration~$\cA$ in~$\RR^n$ such that~$P=\conv(\cA)$ and consider the~$n+1$-dimensional polyhedron \begin{equation*}
		P_h=\conv\big((\mathbf{x},y)\, :\,y\leq h(\mathbf{x}), \mathbf{x}\in\cA,\, y\in\RR\big).
	\end{equation*}

	The upper faces of~$P_h$ can be seen to be defined by the piecewise-linear function~$f_h:P\to\RR$ with~$f_h(\mathbf{x})=\max(y\,:\,(\mathbf{x},y)\in P_h)$. Since the upper facets are~$n$-dimensional the projection over the last coordinate of them gives polytopes~$S_i$. Then~$\cS=\{S_i\}$ is a subdivision of~$P$. If~$h$ is generic then the subdivision is a triangulation of~$P$.

	A subdivision~$\cS$ obtained in this way using a height function is a \defn{regular subdivision}\index{polytope!regular subdivision}. Such a height function is said to be \defn{admissible}\index{height function!admissible} with respect to~$\cS$.
\end{definition}

\begin{definition}\cite{S05}\label{def:mixed_subdivision}
	Consider~$n$-dimensional polytopes~$P_1,\ldots,P_k$ in~$\RR^n$ with~$P_i=\conv(\cA_i)$. A \defn{Minkowski cell}\index{polytope!Minkowski cell} of the Minkowski sum~$\sum P_i$ is a full dimensional polytope~$B$ of the form~$\sum B_i$ where~$B_i=\conv(\cA'_i)$ and~$\cA'_i\subseteq \cA_i$.

	A \defn{mixed subdivision}\label{polytope!mixed subdivision} of said Minkowski sum is a collection of Minkowski cells~$\cB$ such that \begin{itemize}
		\itemsep0em
		\item~$\bigcup_{B\in \cB} B=\sum P_i$,
		\item for any Minkowski cells~$B=\sum B_i$ and~$B'=\sum B'_i$, then for~$i\in [k]$ the intersection~$B_i\cap B'_i$ is a common face.
	\end{itemize}

	A \defn{fine mixed subdivision} is a minimal mixed subdivision via containment of its summands.
\end{definition}

In Figure~\ref{fig:minkowski_subdivision} we show an example of a Minkowski cell inside a mixed subdivision. Notice that the Minkowski cells are not necessarily always simplices. Still, they are Minkowski sums of simplices if the mixed subdivision is fine.

\begin{figure}[h!]
	\centering
	\includegraphics[scale=1.4]{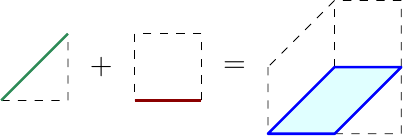}
	\caption{The Minkowski cell in a mixed subdivision of a Minkowski sum.}
	\label{fig:minkowski_subdivision}
\end{figure}

Sometimes, instead of studying a certain characteristic of a polytope, one can study another structure on a related polytope. Fortunately for us this is the case between certain mixed subdivisions and polytopal subdivisions of certain polytopes via a theorem called the \defn{Cayley trick}. The Cayley trick is useful to us as it not only permits us to relate the combinatorics of distinct polytopes but also reduces greatly the ambient dimension where we work. For an extensive view of applications of the Cayley trick we refer the reader to~\cite{S94},~\cite{HRS00},~\cite{S05}, and~\cite{DRS10}.

\begin{definition}\label{def:Cayley_embedding}
	Given polytopes~$P_1,\ldots,P_k$ in~$\RR^n$ we define the \defn{Cayley embedding}\index{polytope!Cayley embedding} as \begin{equation*}
		\cC(P_1,\ldots,P_k):=\conv\big(\{\mathbf{e_1}\}\times P_1, \ldots, \{\mathbf{e_k}\}\times P_k\big)\subset \RR^{k}\times \RR^n. \end{equation*}
\end{definition}

\begin{proposition}[\defn{The Cayley trick }\index{polytope!Cayley trick}{\cite[Thm. 1.4]{S05}}]\label{prop:cayley_trick}
	Let~$P_1,\ldots,P_k$ be polytopes in~$\RR^n$. The regular polytopal subdivisions (resp.\ triangulations) of the Cayley embedding~$\cC(P_1,\ldots,P_k)$ are in bijection with the regular mixed subdivisions (resp.\ fine mixed subdivisions) of~$P_1+\cdots+P_k$.
\end{proposition}

\begin{remark}\label{rem:cayley_trick_proof}
	The bijection consists of taking a polytopal subdivision of~$\cC(P_1, \ldots, P_k)$ and intersecting it with the subspace~${(\sum \frac{1}{k}\cdot\mathbf{e}_i)}_{i=1}^k\times \RR^n$. The result of this operation is a mixed subdivision of~$P_1+\cdots +P_k$ up to a scaling of~$\frac{1}{k}$. Figure~\ref{fig:Cayley_trick} exemplifies this process.

\begin{figure}[h!]
	\centering
	\includegraphics[scale=1.6]{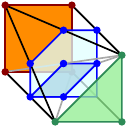}
	\caption[The Cayley trick illustrated.]{ The Cayley trick illustrated following Figure~\ref{fig:minkowski_subdivision}. Figure inspired from~\cite{DRS10}.}
	\label{fig:Cayley_trick}
\end{figure}
\end{remark}

\section{Automata Theory}\label{sec:automata_theory}

We now move on to describe the bases of automata following notation from~\cite{FS09}. For a wider introduction to automata see~\cite{E74}.

\begin{definition}\label{def:automaton}
	An \defn{alphabet}\index{automaton!alphabet} is a set such that its elements are called \defn{letters}\index{automaton!letter}~$\cA$. A \defn{word}\index{automaton!word} on this alphabet is a finite sequence of letters. A collection of words is called a \defn{language}\index{automaton!language}.

	A finite \defn{automaton}\index{automaton}~$\AA$ is a directed multigraph with edges labeled by letters from a finite alphabet~$\cA$. The vertices of~$\AA$ are called \defn{states}\index{automaton!state} and denoted by~$Q$. The edges of~$\AA$ are said to be the \defn{transitions}\index{automaton!transition} of the automaton. Unless stated otherwise, our automata have only 1 starting state denoted by~$q_0$. The set of final states is denoted by~$Q_f\subseteq Q$. Graphically we  represent a final state as a double circle and a non-final state as a simple circle.

	$\AA$ is \defn{deterministic}\index{automaton!deterministic} if for every state~$q\in Q$ and letter~$a\in \cA$ there is at most one transition labeled by~$a$ leaving~$q$. Otherwise,~$\AA$ is \defn{non-deterministic}\index{automaton!non-deterministic}. If there is exactly one transition labeled by~$a$ leaving~$q$ for all~$a\in\cA$ and~$q\in Q$, we say that~$\AA$ is \defn{complete}\index{automaton!complete}.
\end{definition}

Automata are useful as they process words in the alphabet~$\cA$ and determine a language out of said words.

\begin{definition}\label{def:aut_accepting_rejecting}
	Let~$\AA$ be an automaton with alphabet~$\cA$ and~$w=w_i\ldots w_n$ a word of~$\cA$. We say that~$\AA$ \defn{accepts}\index{automaton!accepts/rejects}~$w$ if there is a sequence of edges in~$\AA$ starting from~$q_0$ and ending in a final state~$q_f$ such that the concatenation of the labels of the edges is precisely~$w$. Otherwise, we say that~$\AA$ \defn{rejects}~$w$.

	Informally in a deterministic finite automaton (\defn{DFA}\index{automaton!DFA}), this amounts to reading the word~$w$ from left to right and following the transitions of~$\AA$ accordingly. If the word is accepted, we also say that~$\AA$ \defn{recognizes}\index{automaton!recognize}~$w$. See Figure~\ref{fig:automataFlajoulet} for an example.
\end{definition}

DFA are of particular interest to us since they can be used to characterize patterns in words.

\begin{figure}[h!]
	\centering
	\includegraphics[scale=1]{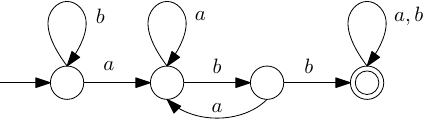}
	\caption[An automaton that recognizes words containing the pattern~$abb$.]{ An automaton that recognizes words containing the pattern~$abb$. Figure based on~\cite{FS09}.}
	\label{fig:automataFlajoulet}
\end{figure}

\vspace{-0.5cm}

\begin{definition}\label{def:language}
	Let~$\AA$ be an automaton with alphabet~$\cA$. The \defn{language}\index{automaton!language} of~$\AA$ is the set of words accepted by~$\AA$. We denote it by~$\cL(\AA)$. In particular, a language of a DFA or a language recognized by a DFA is called a \defn{regular language}\index{automaton!regular language}.
\end{definition}

\begin{definition}\label{def:aut_product}
	Let~$\AA_1$,~$\AA_2$ be automata over an alphabet~$\cA$. The \defn{product automaton}\index{automaton!product}~$\AA_1\times \AA_2$ is the automaton over~$\cA$ with state set~$Q_1\times Q_2$. The initial state is~$(q_{01},q_{02})$ and the set of final states is~$Q_{f1}\times Q_{f2}$. There is a transition between~$(q_1,q_2)$ and~$(q'_1,q'_2)$ labeled~$a\in\cA$ if and only if there exists transitions between~$q_1$ and~$q'_1$ and~$q_2$ and~$q'_2$ both with label~$a$.
\end{definition}

We provide the product of two automata in Figure~\ref{fig:automata_product}.

\begin{figure}[h!]
	\centering
	\includegraphics[scale=0.9]{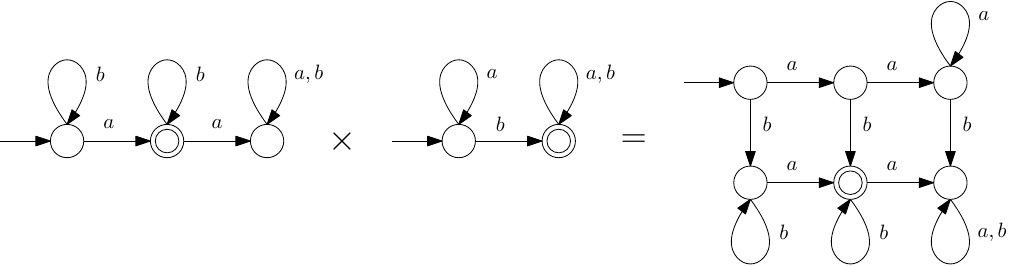}
	\caption{Two automata and their product.}
	\label{fig:automata_product}
\end{figure}

%% file: includes/contenu/chap_preliminaires_weak_order.tex

\chapter{Weak Order, Quotients, and Generalizations}\label{chap:prelim_weak_order}

\addcontentsline{lof}{chapter}{\protect\numberline{\thechapter} {Weak Order, Quotients, and Generalizations }}
\addcontentsline{lot}{chapter}{\protect\numberline{\thechapter} {Weak Order, Quotients, and Generalizations }}

This chapter lays out our main objects of study: the weak order, permutrees, other Coxeter groups, and the~$s$-weak order. Since they are all objects that have as a common point permutations, we bundle them together in what follows.

\section{Permutations}\label{sec:weak_order}

We begin stating some properties of permutations and the weak order that are at the core of our thesis.

\subsection{Combinatorics}\label{ssec:weak_order_permutations}

We present the main results on permutations and the weak order that we need. For a more in depth view of this subject we recommend~\cite{CSW16},~\cite{BB06}, and~\cite{S11}.

\begin{definition}\label{def:permutations_presentations}
	Given~$n\in\NN$, the set of \defn{permutations}\index{permutation} of size~$n$ denoted by \defn{$\fS_n$} is the set of bijections from~$[n]$ to~$[n]$. For us permutations are denoted in \defn{1-line notation}\index{permutation!1-line notation}, that is, as a sequence of numbers~$\pi=\pi_1\pi_2\ldots\pi_n$ where~$\pi_i:=\pi(i)$.

	Another way to represent permutations is via its factorization into cycles. An \defn{l-cycle}\index{permutation!cycle} is a sequence~$(i_1\,\, i_2\,\,\ldots\,\, i_l)$ such that~$\pi_{i_k}=i_{k+1}$ and~$\pi_{i_l}=i_1$. A permutation is then a product of disjoint cycles which is called its \defn{cycle decomposition}\index{permutation!cycle decomposition}.
\end{definition}

The following is a visual way to describe a permutation useful to us.

\begin{definition}\label{def:permutation_table}
	The \defn{table}\index{permutation!table} of a permutation~$\pi$ is the~$n\times n$ grid with points in positions~$(\pi(i),i)$. The positions equivalently are given by~$(i,\pi^{-1}(i))$.
\end{definition}

\begin{example}\label{ex:permutation_cycle_table}
	The permutation~$\pi=41325$ can be represented as~$(1\,4\,2)(3)(5)$ or~$(3)(4\,2\,1)(5)$. Its table is given in Figure~\ref{fig:table41325}.

	\begin{figure}[h!]
		\centering
		\includegraphics[scale=1.25]{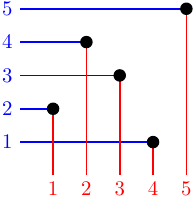}
		\caption{The table of~$\pi=41325$.}\label{fig:table41325}
	\end{figure}
\end{example}

\begin{definition}\label{def:permutation_linear_extension}
	Given a permutation~$\pi\in\fS_n$ in 1-line notation, its corresponding \defn{total order}\index{permutation!total order} is the order given by~$\pi_1<\pi_2<\cdots<\pi_n$. Following this, we say that given a poset~$(P,\leq_P)$, the set of its linear extensions\index{poset!linear extension} is a subset of~$\fS_n$.
\end{definition}

\begin{definition}\label{def:permutations_transpositions}
	Permutations whose cycle decomposition is formed only by~$1$-cycles and a~$2$-cycle~$(i\,\, j)$ with~$i<j$ are called \defn{transpositions}\index{permutation!transposition} and are denoted by \defn{$\tau_{i,j}$}. If~$j=i+1$ we call them \defn{adjacent transpositions}\index{permutation!transposition!adjacent} and denote them by \defn{$s_i$}.

	As we are working with permutations, we also consider their actions on permutations themselves. We say that transposition~$\tau_{i,j}$ acts on a permutation~$ \pi$ \begin{itemize}
		\itemsep0em
		\item on its \defn{right}\index{permutation!right inversion} if it interchanges the values of~$\pi_i$ and~$\pi_j$ and leaves the rest of the elements the same (i.e.\ interchanges positions~$i$ and~$j$),
		\item on its \defn{left}\index{permutation!left inversion} if it interchanges the values of~$i$ and~$j$ and leaves the rest of the elements the same (i.e.\ interchanges positions~$\pi^{-1}_i$ and~$\pi^{-1}_j$).
	\end{itemize} We denote these actions as~$\pi\circ\tau_{i,j}$ and~$\tau_{i,j}\circ\pi$ respectively.
\end{definition}

\begin{example}\label{ex:left_right_transpositions}
	Consider the permutation~$\pi=41325$. Then~$\pi\circ\tau_{2,4}=42315$ via the right action, and~$\tau_{2,4}\circ\pi=21345$ via the left.
\end{example}

\begin{definition}\label{def:pattern_avoidance_containment}
	Given words~$\pi,\sigma$ on~$[n]$ of respective lengths~$r>s$, we say that~$\pi$ \defn{contains} the pattern~$\sigma$ if there is a subsequence of~$\pi$ whose relative order is isomorphic to the relative order of~$\sigma$. If no such subsequence exists, we say that~$\pi$ \defn{avoids}\index{pattern!contains/avoids} the pattern~$\sigma$.
\end{definition}

\begin{example}
	Let~$\pi=1352645$. Then~$\pi$ contains the patterns~$231$ and~$121$ as it respectively contains the subsequences~$352$ and~$565$. It avoids the pattern~$321$ as no subsequence of~$\pi$ has relative order isomorphic to~$321$.
\end{example}

\begin{definition}
	Let~$\pi\in\fS_n$ be a permutation and~$(i,j)\in{[n]}^2$ such that~$i<j$. If~$\pi_i>\pi_{i+1}$ (resp.~$\pi_i<\pi_{i+1}$) then~$i$ is said to be a \defn{descent}\index{permutation!descent} (resp.\ \defn{ascent}\index{permutation!ascent}) of~$\pi$.
	We say that~$(i,j)$ is an \defn{inversion}\index{permutation!inversion} of~$\pi$ if~$\pi^{-1}_i>\pi^{-1}_j$ and denote by \defn{$\inv(\pi)$} the set of inversions of~$\pi$. Otherwise,~$(i,j)$ is a \defn{version}\index{permutation!version} and the set of versions is denoted \defn{$\ver(\pi)$}. Letting~$a_i:=|\{j\in[i+1,n]\,:\,\pi^{-1}_{i}>\pi^{-1}_{j}\}|$ (i.e.\ the amount of elements transposed relative to~$i$), the sequence~$(a_1,\ldots,a_{n-1})$ is called the \defn{Lehmer code}\index{permutation!Lehmer code} or \defn{inversion vector}\index{permutation!inversion vector} of~$\pi$. Note that we omit~$a_{n}$ since it is always 0.
\end{definition}

\begin{remark}\label{rem:inversions_are_coinversions}
	The reader may notice that our definition of inversions can also be found in the literature as that of coinversions. Here we have opted for the definition using the inverse permutation in aims to be consistent with~\cite{PP18} and how to relate tables of permutations with binary trees, permutrees, and decreasing trees. 
\end{remark}

\begin{proposition}[{\cite[Prop.1.3.12]{S11}}]\label{prop:lehmer_code_bijection}
	Let~$\cB_n:=[0,n-1]\times[0,n-2]\times\cdots\times [0,1]$. The map~$\cL:\fS_n\to\cB_n$ sending each permutation to its Lehmer code is a bijection.
\end{proposition}

\begin{example}
	Consider the permutation~$\pi=41325$. Then \begin{equation*}
		\begin{split}
			\inv(\pi)&=\{(1,4),(2,3),(2,4),(3,4)\},\\
			\ver(\pi)&=\{(1,2),(1,3),(1,5),(2,5),(3,5)\},
		\end{split}
	\end{equation*}

	and its Lehmer code is~$(1,2,1,0)$.
\end{example}

\begin{proposition}[{\cite[Thm.2]{GR63}},{\cite[Lem.7-2.4]{CSW16}}]\label{prop:inv_sets_trans_cotrans}
	A subset of~$\NN^2$ of the form~$E=\{(i,j)\,:\, i<j \}$ is the inversion set of a permutation in~$\fS_n$ if and only if~$E$ is transitive and cotransitive.
\end{proposition}

\subsection{Weak Order}\label{ssec:weak_order}

Permutations allow themselves to be partially ordered in several ways. We focus on the orders that come from the containment of inversion sets. Enter the main object of this thesis, the weak order.

\begin{definition}\label{def:weak_order_perms}
	The \defn{(right) weak order}\index{permutation!right weak order} on~$\fS_n$ is the partial order~$\leq$ such that for~$\pi,\sigma\in\fS_n$ we have that~$\pi\leq\sigma \iff \inv(\pi)\subseteq\inv(\sigma).$

	The minimal (resp.\ maximal) element of~$\fS_n$ under this definition is the \defn{identity permutation}\index{permutation!identity}~$\defn{e}=123\ldots n$ (resp.\ \defn{longest element}\index{permutation!longest element}~$w_0=n\ldots 321$).
\end{definition}

\begin{figure}[h!]
	\centering
	\includegraphics[scale=0.9]{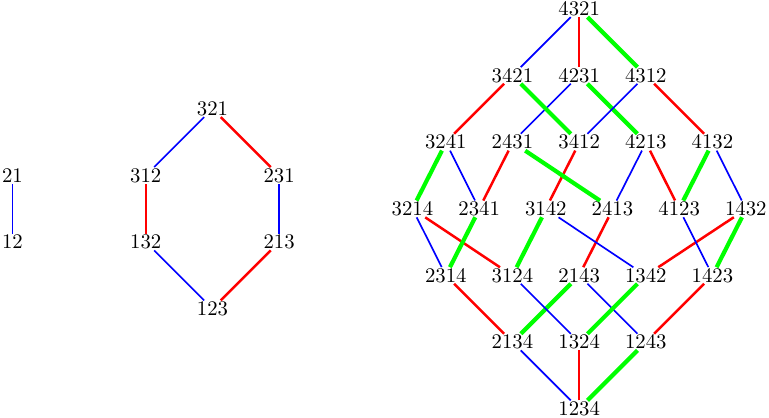}
	\caption[The weak orders~$\fS_2$,~$\fS_3$, and~$\fS_4$ with edges~${\color{blue}s_1=(1\,\,2)}$,~${\color{red}s_1=(2\,\,3)}$, and~${\color{green}s_3=(3\,\,4)}$.]{ The weak orders for~$\fS_2$,~$\fS_3$, and~$\fS_4$ with edges given by the adjacent transpositions~${\color{blue}s_1=(1\,\,2)}$,~${\color{red}s_1=(2\,\,3)}$ (bolded), and~${\color{green}s_3=(3\,\,4)}$ (heavily bolded).}\label{fig:weak_order_3}
\end{figure}

\begin{proposition}[\cite{CSW16}]\label{prop:perm_cover_relations}
	The cover relations of~$(\fS_n,\leq)$ are given by~$\pi\lessdot\sigma$ if and only if~$\sigma=\pi\circ s_{i}$ for some~$i\in[n-1]$. Moreover,~$\inv(\sigma) = \inv(\pi)\cup \{(\pi_i,\pi_{i+1})\}$.
\end{proposition}

This is a powerful proposition as it allows us to think of all edges of the weak order as adjacent transpositions. That is, to find what order of adjacent transpositions gives a permutation, it is enough to follow a route on the poset. Figure~\ref{fig:weak_order_3} shows how this is the case for~$n=2,3,4$. Furthermore, we get the following result.

\begin{corollary}\label{cor:permutaiton:ranked}
	The weak order~$(\fS_n,\leq)$ is ranked by the amount of adjacent transpositions needed to form a permutation starting from the identity permutation. That is,~$\rank(\pi)=|\inv(\pi)|$ with minimal and maximal values~$\rank(e)=0$ and~$\rank(w_0)=\gbinom{n}{2}$.
\end{corollary}

Both inversions and the weak order can be defined for the left action of transpositions on permutations. This gives the \defn{left weak order}\index{permutation!left weak order} which we state here for completeness.

\begin{definition}
	The cover relations of~$(\fS_n,\leq_L)$ are given by~$\pi\lessdot\sigma$ if and only if~$\sigma=s_{i}\circ\pi$ for some~$i\in[n-1]$.
\end{definition}

The left and right weak order are not the same. However, they are isomorphic and thus the results we write for the right weak order are equally true for the left weak order. The following proposition gives the isomorphism and Figure~\ref{fig:leftRightWeakOrder4} shows the left and right weak orders for~$n=4$.

\begin{proposition}\label{prop:left_right_weak_order_isomorphic}
	The map~$\pi\mapsto \pi^{-1}$ is an isomorphism between the left and right weak orders.
\end{proposition}

\begin{figure}[h!]
	\centering
	\includegraphics[scale=0.9]{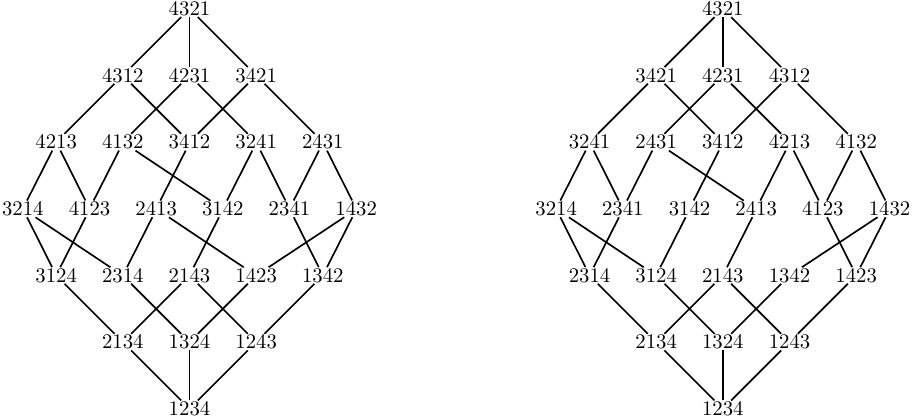}
	\caption[The left and right weak orders on~$\fS_4$.]{ The left (left) and right (right) weak orders on~$\fS_4$.}
	\label{fig:leftRightWeakOrder4}
\end{figure}

\begin{proposition}[\cite{GR63},\cite{HP07}]\label{prop:permutation_lattice}
	The weak order~$(\fS_n,\leq)$ is a lattice. Moreover, the meet and join of~$\pi,\sigma\in\fS_n$ are given by the permutations~$\pi\wedge\sigma$ and~$\pi\vee\sigma$ such that \begin{equation*}
		\begin{split}
			\inv(\pi\wedge\sigma) &= {\{\inv(\pi)\cup\inv(\sigma)\}}^{tc},\\
			\ver(\pi\vee\sigma) &= {\{\ver(\pi)\cup\ver(\sigma)\}}^{tc}.
		\end{split}
	\end{equation*}
\end{proposition}

\begin{proposition}[{\cite[Thm.7-4.3]{CSW16}}]\label{prop:permutation_autoduality}
	The weak order~$(\fS_b,\leq)$ is an autodual lattice via the involutive dual automorphism~$f$ such that~$\inv(f(\pi))=\inv(w_0)\setminus\inv(\pi)$ for~$\pi\in\fS_n$. In particular, it satisfies~$\pi\wedge f(\pi)=e$ and~$\pi\vee f(\pi)=w_0$ for all~$\pi$.
\end{proposition}

\subsection{Permutahedra}\label{ssec:weak_order_permutahedron}

In addition of having a nice combinatorial structure, the weak order enjoys several geometric interpretations which also have rich combinatorial properties.

\begin{definition}\label{def:permutahedron}
	The \defn{permutahedron}\index{polytope!permutahedron} is the polytope~$\PPerm_n$ defined equivalently as:
	\begin{itemize}
		\itemsep0em
		\item the convex hull of the coordinates~$(\sigma_i,\ldots,\sigma_n)$ for~$\sigma\in\fS_n$,
		\item the intersection of the following hyperplane and half-spaces \begin{equation*}
			      \left\{\mathbf{x}\in\RR^n \,:\,\sum_{i\in [n]}x_i=\gbinom{n+1}{2}\right\} \cap \bigcap_{\emptyset\subsetneq I \subsetneq[n]} \left\{\mathbf{x}\in\RR^n \,:\, \sum_{i\in I}x_i\geq\gbinom{|I|+1}{2}\right\},
		      \end{equation*}
		\item the shifted zonotope~$\sum_{1\leq i<j \leq n}[\mathbf{e}_i,\mathbf{e}_j]$.
	\end{itemize}
\end{definition}

See Figure~\ref{fig:perm4} for an example of~$\PPerm_4$.

\begin{proposition}\label{prop:orientation_permutahedron}
	Let~$\mathbf{v}=(\mathbf{w_0})-(\mathbf{e})=(n-1,n-3,\ldots,-n+3,-n+1)={(2i-n-1)}_{i\in[n]}$. The Hasse diagram of the weak order~$(\fS_n,\leq)$ is isomorphic to the~$1$-skeleton of the permutahedron~$\PPerm_n$ oriented with the vector~$\mathbf{v}$.
\end{proposition}

\begin{figure}[h!]
	\centering
	\includegraphics[scale=0.6]{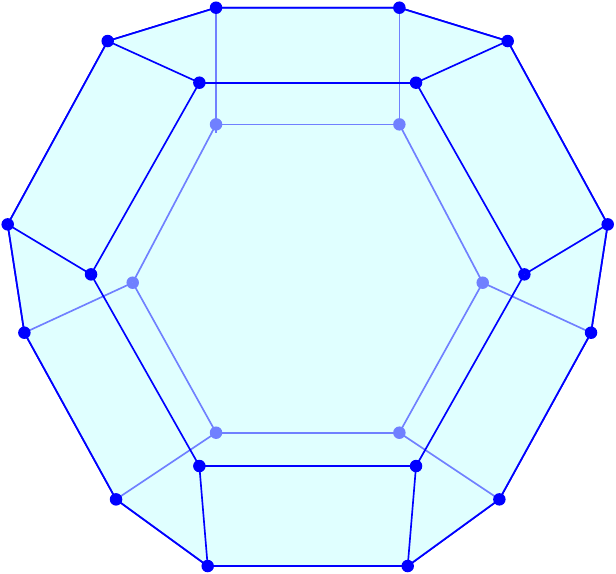}
	\caption{The permutahedron~$\PPerm_4$.}
	\label{fig:perm4}
\end{figure}

There is a detail we must take care of when talking about Proposition~\ref{prop:orientation_permutahedron}. Our weak order~$(\fS_n,\leq)$ is the right weak order whereas the poset obtained from the oriented~$1$-skeleton of~$\PPerm_n$ gives the left weak order. We know from Proposition~\ref{prop:left_right_weak_order_isomorphic} that both orders are isomorphic already. We go one step further and present a way to see both of this orders by describing the faces of the permutahedron combinatorially through different combinatorial objects. See Figure~\ref{fig:permutations_3_packed_words_ordered_paritions}. 

\begin{definition}[{\cite{NT06},\cite{HNT08}}]\label{def:packed_words}
	A finite word over the alphabet~$\NN_{>0}$ is \defn{packed}\index{packed word} if all the letters between~$1$ and its maximum~$m$ appear at least once.
\end{definition}

\begin{definition}\label{def:ordered_partitions}
	An \defn{ordered partition}\index{ordered partition}~$\lambda$ of~$[n]$ into~$k$ parts is a sequence~$\tau=(\tau_1,\ldots,\tau_k)$ such that the disjoint union of the parts gives~$\bigsqcup_{i\in[k]}\tau_i=[n]$.
\end{definition}

The following propositions describing the faces of the permutahedron in terms of packed words and ordered partitions are well known and thus are adapted to our context. We refer the reader to~\cite{NT06},~\cite{HNT08},~\cite{S97}, and~\cite{M03} for further details and connections of this phenomenon.

\begin{proposition}\label{prop:permutahedron_facets_packed_words}
	The faces of the permutahedron~$\PPerm_n$ are in bijection with the packed words of length~$n$. The face corresponding to a packed word~$w$ with maximum~$k$ is the~$(n-k)$-dimensional face given by the convex hull of the~$\sigma\in\fS_n$ such that if~$w_i<w_j$ then~$\sigma^{-1}_i<\sigma^{-1}_{j}$.
\end{proposition}

\begin{proposition}\label{prop:permutahedron_facets_ordered_partitions}
	The faces of the permutahedron~$\PPerm_n$ are in bijection with the ordered partitions of~$[n]$. The face corresponding to the ordered partition~$\tau=(\tau_1,\ldots,\tau_k)$ is the~$(n-k)$-dimensional face given by the convex hull of the~$\sigma\in\fS_n$ such that~$\sigma\prec\tau$ where~$\prec$ is the refinement order for partitions.
\end{proposition}

\begin{figure}[h!]
	\centering
	\includegraphics[scale=1.6]{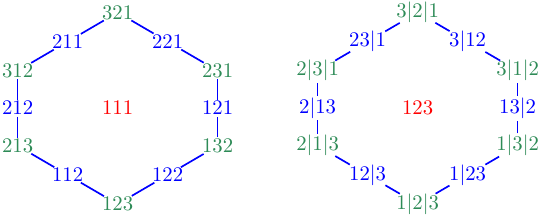}
	\caption[$\PPerm_3$ with faces labeled with packed words and ordered partitions.]{ $\PPerm_3$ with faces labeled with packed words (left) and ordered partitions (right).}\label{fig:permutations_3_packed_words_ordered_paritions}
\end{figure}

Following Figure~\ref{fig:permutations_3_packed_words_ordered_paritions} one might sense that ordered partitions correspond to the right weak order and packed words to the left weak order. That is part of a more general story that lands outside the scope of this thesis. We refer the reader for more details about this respectively to~\cite{S97},~\cite{M03}, and~\cite{P20} for the ordered partitions side and~\cite{NT06},~\cite{T97}, and~\cite{H22} for the packed words side.

Having talked about the permutahedron~$\PPerm_n$, we now move to talk about its normal fan. Via the zonotope description of Definition~\ref{def:permutahedron} we can describe it easily as follows.

\begin{proposition}\label{prop:permutahedron_normal_fan}
	The normal fan~$\cN(\PPerm_n)$ is the fan formed by the collection of hyperplanes of the form~$x_i=x_j$ for all~$1\leq i<j\leq n$. This fan is known as the \defn{braid fan}\index{polytope!permutahedron!braid fan}.
\end{proposition}

\begin{definition}\label{def:generalized_permutahedra}
	A polytope~$P$ with~$n:=\dim(P)$ is a \defn{generalized permutahedron}\index{generalized permuthaedron} if its normal fan~$\cN(P)$ coarsens the braid fan~$\cN(\PPerm_n)$. Generalized permutahedra are also equivalently defined as any polytope~$P\in\RR^{n}$ such that any edge~$(\mathbf{v},\mathbf{u})$ satisfies~$\mathbf{v}-\mathbf{u}=\lambda(\mathbf{e_i}-\mathbf{e_j})$ for some~$\lambda\in\RR$ and~$1\leq i<j\leq n$.  A polytope obtained by removing facets (i.e.\ deleting inequalities in the~$\cH$-description) from~$\PPerm_n$ is a called a \defn{removahedron}\index{polytope!removahedron}.
\end{definition}

\begin{example}\label{ex:cube_gen_permutahedron}
	The zonotope~$\conv((0,0,0),(1,-1,0),(0,1,-1),(1,0,-1))$ is a generalized permutahedron. Notice that this polytope is combinatorially equivalent to~$\PCube_2$. In general all polytopes that are combinatorially equivalent to~$\PCube_n$ are generalized permutahedra in this sense.
\end{example}

\begin{example}
	Generalized permutahedra do not need to strictly coarsen the braid fan. Figure~\ref{fig:generalizedpermutahedra} shows three examples of generalized permutahedra. The first one is the permutahedron~$\PPerm_4$ itself while the second is a polytope combinatorially equivalent to~$\PPerm_4$. The third one is a generalized permutahedron with a coarser fan.
	\begin{figure}[h!]
		\centering
		\includegraphics[scale=0.8]{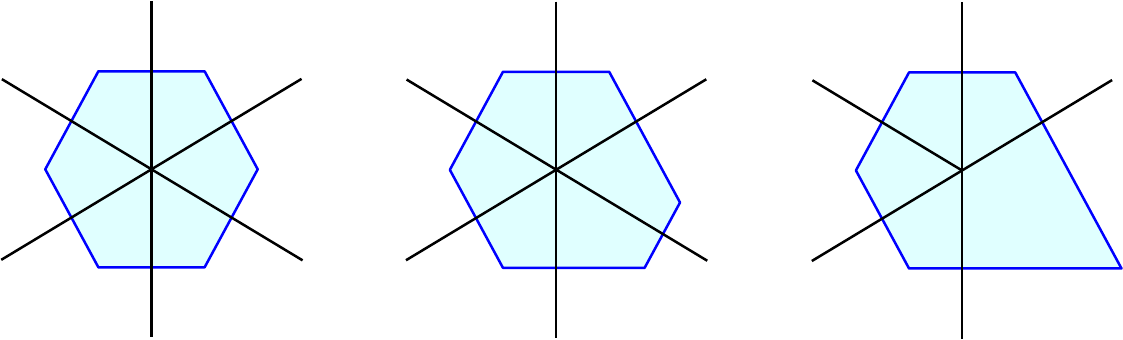}
		\caption{3 examples of Generalized permutahedra together with their normal fans.}\label{fig:generalizedpermutahedra}
	\end{figure}
\end{example}

\subsection{Cubical Embeddability}\label{ssec:cubical permutahedron}

Apart from the standard permutahedron, another possible combinatorial realization of~$\PPerm_n$ is its embedding into a dilatation of the~$n-1$-cube~$\PCube_{n-1}$ (\cite{BF71},~\cite{RR02}). This embedding is shown in Figure~\ref{fig:permutahedron_cubic} for~$n=4$. This structure appears also as a particular case of~$\delta$-cliff posets (see~\cite[Prop.1.2.1.]{CG22})

\begin{proposition}[{\cite[Thm.3.1]{RR02}}]\label{prop:cubical_permutahedron}
	Let~$Q_{n-1}=[0,n-1]\times\cdots\times[0,1]$. The permutahedron~$\PPerm_n$ is embeddable in the cube~$Q_{n-1}$ via the function that sends a permutation to its Lehmer code (inversion vector).
\end{proposition}

\begin{figure}[h!]
	\centering
	\includegraphics[scale=0.6]{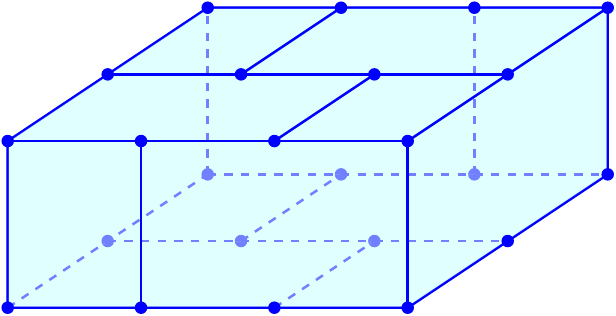}
	\caption{The cubical embedding of~$\PPerm_4$.}\label{fig:permutahedron_cubic}
\end{figure}

\section{Binary Trees}\label{sec:binary_trees}

Permutations share connections with many combinatorial families. One of the most classical ones come from binary trees which we now proceed to describe.

\subsection{Combinatorics}\label{ssec:binary_combinatorics}

\begin{definition}\label{def:binary_tree}
	A \defn{binary tree}\index{binary tree} is a rooted plane tree (i.e. with a concrete embedding) on the nodes~$\{v_1,\ldots,v_n\}$ where each node has exactly two subtrees. In other words, each node has exactly two children and one parent. We denote the collection of binary trees with~$n$ nodes by~$\cBT_n$. The left (resp.\ right) subtree of a node~$v_i$ is denoted by~$L_i$ (resp.~$R_i$). We say that \defn{$i\to j$} if either~$v_i\in L_j$ or~$v_i\in R_j$. Given a binary tree~$T$, its \defn{partial order}\index{binary tree! partial order} on~$[n]$ is defined by~$i<j$ if~$i\to j$.
\end{definition}

Binary trees are an interesting family of combinatorial objects. In particular, they are counted by the \defn{Catalan numbers}\index{Catalan numbers}~$C_n=\frac{1}{n+1}\gbinom{2n}{n}$. Catalan numbers appear in a wide variety of combinatorial problems as they count a plethora of combinatorial families. So many in fact (200+) that we avoid going in detail about them. We refer the dauntless reader to~\cite{S15} for a more in depth view of this subject. Some other objects counted by Catalan numbers include: \begin{itemize}
	\itemsep0em
	\item $312$-avoiding permutations,
	\item triangulations of a convex~$(n+2)$-gon,
	\item Dyck paths of length~$2n$,
	\item non-crossing matchings of~$[2n]$,
	\item non-crossing partitions of~$[n]$,
	\item ballot sequences of length~$2n$.
\end{itemize}

On top of being interesting between themselves, Catalan families~$(C_n)$ also have connections with other combinatorial families by how they are labeled. For example, binary trees with increasing labeling on both left and right subtrees give factorial families~$(n!)$. If the labeling is increasing in one side and free in the other one obtains families counted by parking functions~$((n+1)^{n-1})$. A free labeling in both subtrees gives families counted by~$n!C_n$. See~\cite{CG19} for more information on this idea.

\begin{definition}\label{def:binary_tree_inordering}
	Taking an anticlockwise walk, we label the nodes of a binary tree with the set~$[n]$ whenever we visit a node for the second time. In other words, all labels of vertices in~$L_i$ (resp.~$R_i$) are smaller (resp. larger) than~$i$. This is called the \defn{in-order labeling}\index{binary tree!in-order labeling} of the binary tree. Figure~\ref{fig:binary_inordering} shows an example of a binary tree with its in-order labeling.
\end{definition}

\begin{figure}[h!]
	\centering
	\includegraphics[scale=1.1]{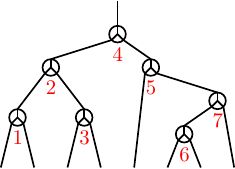}
	\caption{A binary tree with its in-order labeling.}\label{fig:binary_inordering}
\end{figure}

\subsection{Tamari Lattice}\label{ssec:tamari_order}

We now move on to describe the Tamari lattice which shares a key relation with the weak order classically studied through the Stack-Sort algorithm as in~\cite{K73}, and~\cite{K97}. Historically, the Tamari lattice was first described using parenthisations in~\cite{HT72}. Here we proceed instead using binary trees and basing our statements on~\cite{L04} and~\cite{P86}.

\begin{definition}\label{def:edge_rotation}
	Let~$T\in\cBT_n$ be an in-ordered binary tree with an edge~$(v_i,v_j)$ where~$1\leq i<j\leq n$ such that~$i \to j$. An \defn{$ij$-edge rotation}\index{binary tree!edge rotation} is the operation of replacing the right subtree of~$v_i$ by the subtree with root~$v_j$ and the left subtree of~$v_j$ by~$R_i$. Figure~\ref{fig:binary_tree_rotation} shows an example of an edge rotation.

	Given two binary trees~$T_1,T_2$ we say that~$T_1\lessdot T_2$ if and only if~$T_2$ can be obtained from~$T_1$ by a single \defn{edge rotation}\index{binary tree!edge rotation}.

	\begin{figure}[h!]
		\centering
		\includegraphics[scale=1.5]{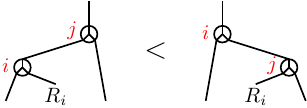}
		\caption{Rotation on binary trees.}\label{fig:binary_tree_rotation}
	\end{figure}
\end{definition}

The poset obtained by edge rotations is called the \defn{Tamari lattice}\index{Tamari lattice}. To prove the lattice property, Huang and Tamari~\cite{HT72} relied on a bijection between parenthisations on~$n$ objects and bracketing functions on~$n$ objects. We forgo this definition and instead use the following equivalent formulation given directly in terms of binary trees.

\begin{definition}[{\cite{K93},~\cite{P86},~\cite{BW96}}]\label{def:bracket_vector}
	Let~$T$ be a binary tree with vertex set~$[n]$ labeled in in-order. Its \defn{bracket set}\index{binary tree!bracket!set} is~$B(T):=\{(i,j)\,:\,j\in R_i\}$ and its \defn{bracket components}\index{binary tree!bracket!components} are~${B(T)}_i=\{j\in [n] \,:\, (i,j)\in B(T)\}$.
	To a bracket set we associate a \defn{bracket vector}\index{binary tree!bracket!vector}~$\vec{b}(T)=(b_1,\ldots,b_{n-1})$ such that~$b_i=|{B(T)}_i|$.
\end{definition}

Notice that we do not consider~${B(T)}_n$ as~$R_n=\emptyset$. Figure~\ref{fig:tamari_3_bracket_sets} presents all bracket sets for~$\cBT_3$. We now characterize which vectors are bracket vectors of binary trees with our terminology.

\begin{figure}[h!]
	\centering
	\includegraphics[scale=1]{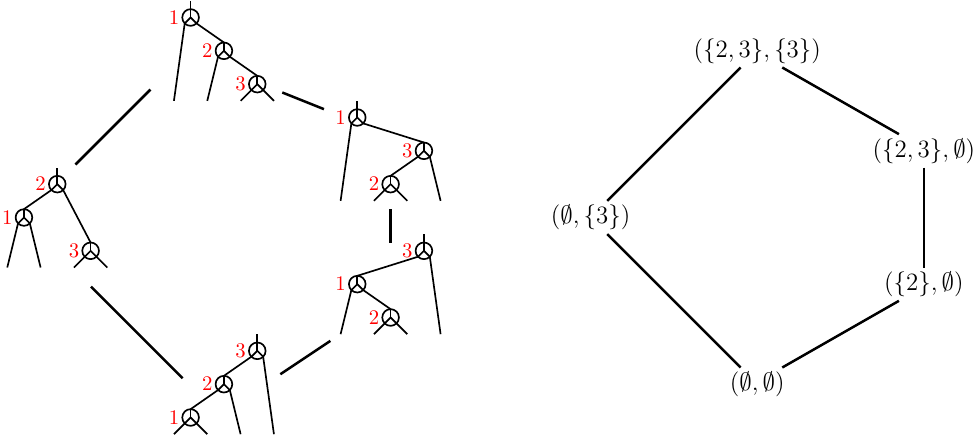}
	\caption[The Tamari lattice for~$n=3$ and its corresponding bracket sets.]{ The Tamari lattice for~$n=3$ and its corresponding bracket sets represented by their components.}\label{fig:tamari_3_bracket_sets}
\end{figure}

\begin{proposition}[{\cite{HT72}}]\label{prop:bracket_vector_characterization}
	The bracket set map is a bijection from binary trees to the sets~$B\in 2^{[n]}$ such that their components satisfy
	\begin{enumerate}
		\itemsep0em
		\item~$B_i=\emptyset$ or~$B_i=\{i+1,i+2,\ldots,i+l\}$ for some~$l>0$,

		\item if~$j\in B_i$, then~$B_j\subseteq B_i$.
	\end{enumerate}
\end{proposition}

The bracket set has similarities to inversions of permutations and the notation is reminiscent of that of Proposition~\ref{prop:cubical_permutahedron}. This is not a coincidence, and we study it in Chapter~\ref{chap:permutree_vectors}.

\begin{proposition}[{\cite{HT72}}]\label{prop:tamari_lattice_meet}
	Given two binary trees~$T,T'$ with~$n$ vertices, there exists the binary tree~$T\wedge T'$ under the binary tree rotation order. Moreover, it satisfies \begin{equation}\label{eq:binary_meet}{B(T\wedge T')}_i={B(T)}_i\cap {B(T')}_i.\end{equation}
\end{proposition}

Figure~\ref{fig:tamari_meet} illustrates the meet operation between two binary trees in~$\cBT_5$.

\begin{figure}[h]
	\centering
	\includegraphics[scale=1]{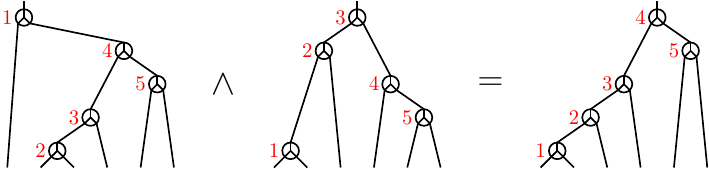}
	\caption[The meet of two binary trees.]{ The meet of two binary trees. Calculating their bracket sets via Equation~\ref{eq:binary_meet} yields~$(\{2,3,4,5\},\emptyset,\emptyset,\{5\})\cap(\emptyset,\emptyset,\{4,5\},\{5\})=(\emptyset,\emptyset,\emptyset,\{5\})$.}\label{fig:tamari_meet}
\end{figure}

\begin{corollary}[{\cite{HT72}}]\label{thm:tamari_lattice_proof}
	The poset of binary trees~$(\cBT_n,\leq)$ is a lattice.
\end{corollary}

\subsection{Sylvester Congruence}\label{ssec:tamari_congruence}

The Tamari lattice can be obtained as a lattice quotient of the weak order via the sylvester congruence in the following way.

\begin{definition}[{\cite{HNT05}}]\label{def:cong_sylvster}
	The \defn{sylvester congruence}\index{lattice!congruence!sylvester} is a lattice congruence over~$\fS_n$ given as the transitive closure on relations of the form \begin{equation*}
		UikVjW \equiv_{sylv} UkiVjW
	\end{equation*} where~$i<j<k$ and~$U,V,W$ are words in the over~$[n]$.
\end{definition}

\begin{proposition}\label{prop:tamari_lattice_congruence}
	The Tamari lattice is isomorphic to~$\fS_n/{\equiv_{sylv}}$.
\end{proposition}

Figure~\ref{fig:sylvester_congruence} shows the sylvester congruence and the resulting Tamari lattice for~$n=4$.

\begin{figure}[h!]
	\centering
	\includegraphics[scale=0.75]{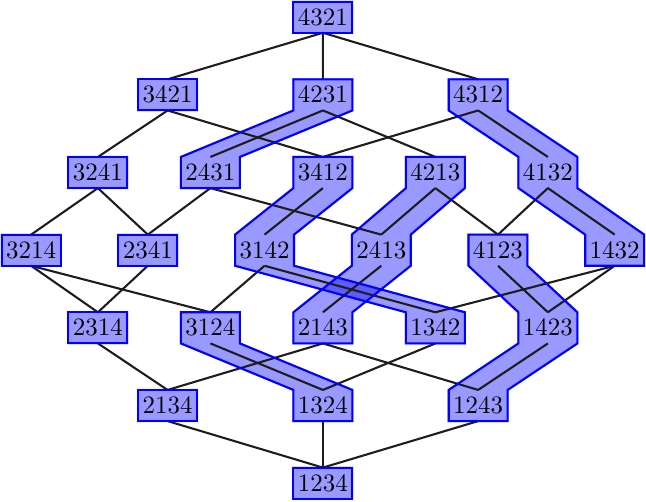}\hspace{2cm}
	\includegraphics[scale=0.6]{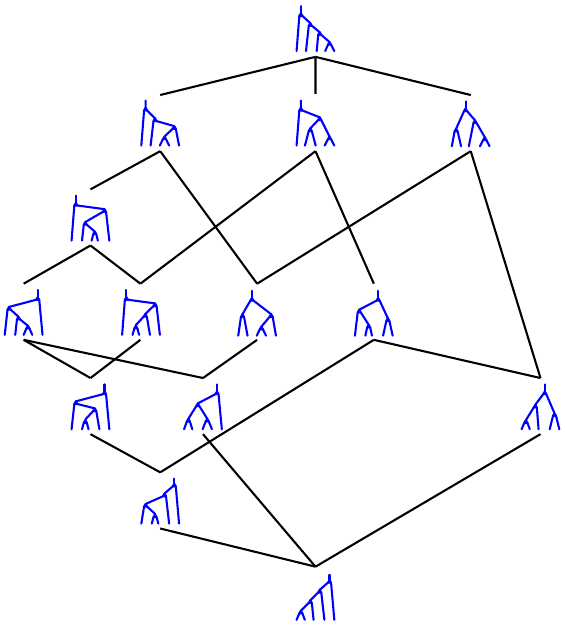}
	\caption[The sylvester congruence and the Tamari lattice for~$n=4$.]{ The sylvester congruence for~$n=4$ (left) and the Tamari lattice (right). Figure from~\cite{PS17}.}\label{fig:sylvester_congruence}
\end{figure}

This congruence can also be characterized as follows.

\begin{proposition}[{\cite{HNT05},~\cite{K73},~\cite{PP18}}]\label{prop:sylvester_congruence_characterization}
	The sylvester congruence~$\equiv_{sylv}$ on~$\fS_n$ can be defined equivalently as the equivalence relation whose classes are:
	\begin{itemize}
		\itemsep0em
		\item the linear extensions of binary trees,
		\item the fibers of the Stack-sorting algorithm.
	\end{itemize}
	Furthermore, the following objects are in bijection: \begin{itemize}
		\itemsep0em
		\item binary trees on~$n$ vertices,
		\item sylvester congruence classes,
		\item permutations that avoid the pattern~$312$.
	\end{itemize} See Figure~\ref{fig:binary_tree_to_permutations} for an example.
\end{proposition}

\begin{figure}[h!]
	\centering
	\includegraphics[scale=1.75]{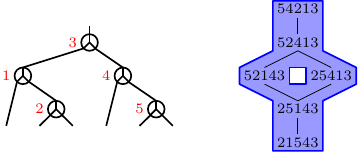}
	\caption[A binary tree on 5 nodes and its corresponding~$3ki$-avoiding permutations as a congruence class.]{ A binary tree on 5 nodes (left) and its corresponding~$3ki$-avoiding permutations as a congruence class (right).}\label{fig:binary_tree_to_permutations}
\end{figure}

\subsection{Associahedra}\label{ssec:tamari_assoc}

The Tamari lattice has several geometric realizations associated to the diverse lenses from which it can be studied (see~\cite{HL07}). Here we concentrate on the one described in~\cite{L04}.

\begin{definition}\label{def:points_associahedron}
	Any polytope whose~$1$-skeleton realizes the Tamari lattice is called the \defn{Associahedron}\index{polytope!associahedron}~$\PAsoc_n$ or \defn{Stasheff polytope}\index{polytope!Stasheff polytope}.
\end{definition}

See Figure~\ref{fig:asoc4} for an example of~$\PAsoc_4$.

\begin{proposition}\label{prop:associahedron}
	The \defn{associahedron} is the polytope~$\PAsoc_n$ defined equivalently as:
	\begin{itemize}
		\itemsep0em
		\item~\cite{L04} the convex hull of the coordinates~$(l_1r_1,\ldots,l_nr_n)$ for~$T\in{\cBT}_n$ where~$l_i$ (resp.~$r_i$) is the number of leaves of~$L_i$ (resp.~$R_i$),
		\item~\cite{SS93} the intersection of the following hyperplane and half-spaces \begin{equation*}
			      \left\{\mathbf{x}\in\RR^n \,:\,\sum_{i\in [n]}x_i=\gbinom{n+1}{2}\right\} \cap \bigcap_{1\leq i \leq j\leq n} \left\{\mathbf{x}\in\RR^n \,:\, \sum_{i\leq \ell\leq j}x_\ell\geq\gbinom{j-i+2}{2}\right\},
		      \end{equation*}
		\item~\cite{P09} the shifted Minkowski sum of the faces~$\Delta_{[i,j]}$ of~$\Delta_{n-1}$ for all~$1\leq i\leq j \leq n$ where $\Delta_X:=\conv(\mathbf{e_x}\, :\, x\in X)$.
	\end{itemize}
\end{proposition}

For more constructions of the associahedron we refer the reader to~\cite{HL07},~\cite{CSZ15}, and~\cite{PSZ23}.

\begin{figure}[h!]
	\centering
	\includegraphics[scale=0.7]{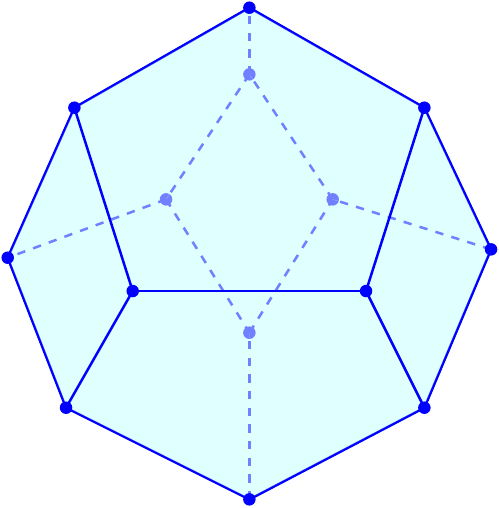}
	\caption{The associahedron~$\PAsoc_4$.}
	\label{fig:asoc4}
\end{figure}

\begin{proposition}\label{prop:orientation_associahedron}
	Let~$\mathbf{v}=(\mathbf{w_0})-(\mathbf{e})=(n-1,n-3,\ldots,-n+3,-n+1)={(2i-n-1)}_{i\in[n]}$. The~$1$-skeleton of the associahedron~$\PAsoc_n$ oriented with the vector~$\mathbf{v}$ is isomorphic to the Hasse diagram of the Tamari lattice.
\end{proposition}

\subsection{Cubical Embeddability}\label{ssec:tamari_cubical}

Like in the case of the permutahedron, the associahedron also possesses a cubical embedding. Although this way of representing the associahedron via coordinates in the cube~$Q_n=[0,n-1]\times\cdots\times[0,1]$ has been known since the 80s (see~\cite{P86} and~\cite{BW96}), the first explicit illustration of this embedding as an actual cube seems to date from a video lecture of Knuth in 1993~\cite{K93}.  We invite (in genuine interest) the archaeological reader to find an older illustration of this embedding into~$Q_{n-1}$. Thanks to Francisco Santos Leal we have learned of a first appearance of this phenomenon due to Stasheff in 1963~\cite{S06}. Still, Stasheff's embedding comes not from polytopes but from a complex~$K_i$ given by a cell decomposition of the boundary of a convex body. Taking logarithms of the defining inequalities of~$K_i$ seems to yield our desired structure. As such, our archaeological invitation still stands to find who studied these complexes visually in a cubical manner before Knuth's video!

The polytopal cubic phenomenon that we are interested on has appeared in recent works related with Tamari intervals~\cite{C22}, parabolic Tamari lattices in Coxeter groups of type~$B$~\cite{FMN21}, Fuss-Catalan posets~\cite{CG22}, and Hochschild lattices~\cite{C21}~\cite{PP23}.

In Figure~\ref{fig:associahedron_cubic} we show the cubical embedding of~$\PAsoc_4$ following~\cite{K93}.

\begin{proposition}[{\cite{K93}}]\label{:cubic_associahedron}
	The associahedron~$\PAsoc_n$ is embeddable in the cube~$Q_{n-1}=[0,n-1]\times\cdots\times[0,1]$ via the function that sends a binary tree to its bracket vector.
\end{proposition}

\begin{figure}[h!]
	\centering
	\includegraphics[scale=0.7]{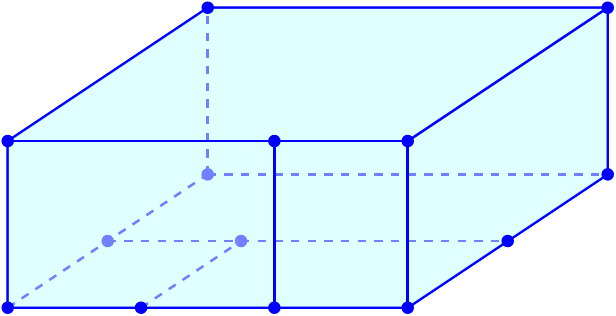}
	\caption[The cubical embedding of~$\PAsoc_4$.]{ The cubical embedding of~$\PAsoc_4$. Figure based on the video~\cite{K93}.}\label{fig:associahedron_cubic}
\end{figure}

\section{Permutrees}

Binary trees are part of a more general family of  combinatorial objects called permutrees. Defined by Pons and Pilaud in~\cite{PP18}, they generalize permutations and binary trees in such a way that they also capture binary sequences and Cambrian trees (see~\cite{LP13} and~\cite{CP17}) that were motivated by the Cambrian lattices of~\cite{R06}.  In this section we give the basic definitions and facts about permutrees from two different perspectives: trees and congruences of the weak order. We follow~\cite{PP18} throughout most of this section.

\subsection{Combinatorics}\label{ssec:permutree_combinatorics}

\begin{definition}\label{def:permutrees}
	A \defn{permutree}\index{permutree} is a directed unrooted tree~$T$ with vertex set~$\{v_1,\ldots,v_n\}$ such that for each vertex~$v_i$:
	\begin{enumerate}
		\itemsep0em
		\item $v_i$ has exactly one or two parents (outward neighbors) and one or two children (inward neighbors). We denote respectively~$LA_i$,~$RA_i$, (resp.~$LD_i$,~$RD_i$) the left and right ancestor (resp.\ descendant) subtree of~$v_i$. In the case that a vertex has only one ancestor (resp.\ descendant) subtree we denote it~$A_i$ (resp.~$D_i$),

		\item if~$v_i$ has two parents (resp.\ children), then all vertices~$v_j\in LA_i$ (resp.~$v_j\in LD_i$) satisfy~$j<i$ and all vertices~$v_k\in RA_i$ (resp.~$v_k\in RD_i$) satisfy~$i<k$.
	\end{enumerate}
	If~$v_j$ is a descendant of~$v_i$ we say that \defn{$j\to i$}. Given a permutree~$T$, its \defn{partial order}\index{permutree!partial order} on~$[n]$ is given by~$j<i$ if and only if~$j\to i$.

	The \defn{decoration}\index{permutree!decoration} of a permutree~$T$ is the vector~$\delta(T)\in{\{\nonee,\downn,\upp,\uppdownn\}}^n$ with entries defined as
	\begin{equation*}
		{\delta(T)}_i=\left\{\begin{array}{ccl}
			\nonee    & \text{ if } & v_i \text{ has one parent and one child,}     \\
			\downn    & \text{ if } & v_i \text{ has one parent and two children,}  \\
			\upp      & \text{ if } & v_i \text{ has two parents and one child,}    \\
			\uppdownn & \text{ if } & v_i \text{ has two parents and two children.}
		\end{array}\right.
	\end{equation*} Letting~$\delta:=\delta(T)$ we say that~$T$ is a~$\delta$-permutree and denote by \defn{$\cPT_n(\delta)$} the collection of all~$\delta$-permutrees on~$n$ vertices.
\end{definition}

\begin{example}\label{ex:permutree_examples}
	Permutrees~$\cPT_n(\delta)$ correspond to: \begin{itemize}
		\itemsep0em
		\item permutations when~$\delta=\nonee^n$ following Definition~\ref{def:decorated_permutation},
		\item binary trees when~$\delta=\downn^n$,
		\item Cambrian trees when~$\delta\in{\{\downn,\upp\}}^n$,
		\item binary sequences of length~$n-1$ when~$\delta=\uppdownn^n$ via the correspondence that the coordinates of the binary sequence are~$s_i=0$ (resp.~$s_i=1$) if the vertex~$v_i$ is a child (resp.\ parent) of~$v_{i+1}$.
	\end{itemize}

	Figure~\ref{fig:permutree-examples} contains several examples of permutrees with distinct decorations.
\end{example}

\begin{figure}[h!]
	\centering
	\includegraphics[scale=1]{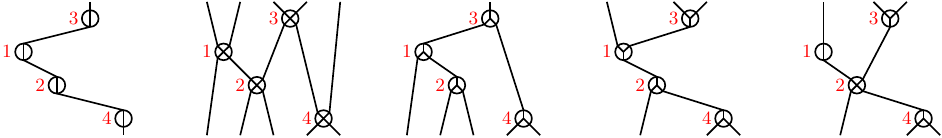}
	\caption[5 examples of~$\delta$-permutrees on 4 vertices.]{ 5 examples of~$\delta$-permutrees on 4 vertices. These permutrees respectively correspond to the permutation~$4213$, the binary sequence~$101$, a rooted binary tree, a Cambrian tree and a generic permutree.}\label{fig:permutree-examples}
\end{figure}

\begin{remark}\label{rem:permutree_leftmost_rightmost_labels}
	Notice that the decorations~$\delta_1$ and~$\delta_n$ do not actually affect the structure of the~$\delta$-permutree since the subtrees~$LA_1$,~$LD_1$,~$RA_n$, and~$RD_n$ are always empty. We never make use of these subtrees, so we always take~$\delta_1=\delta_n=\nonee$ for simplicity.
\end{remark}

All of our drawings of~$\delta$-permutrees have their edges directed upwards and thus, they are presented unoriented. As well, the vertices~$v_i$ appear from left to right in ascending order. We make this more precise now with an explicit algorithm that is our main tool for constructing and drawing permutrees.

\begin{figure}[h!]
	\centering
	\includegraphics[scale=0.8]{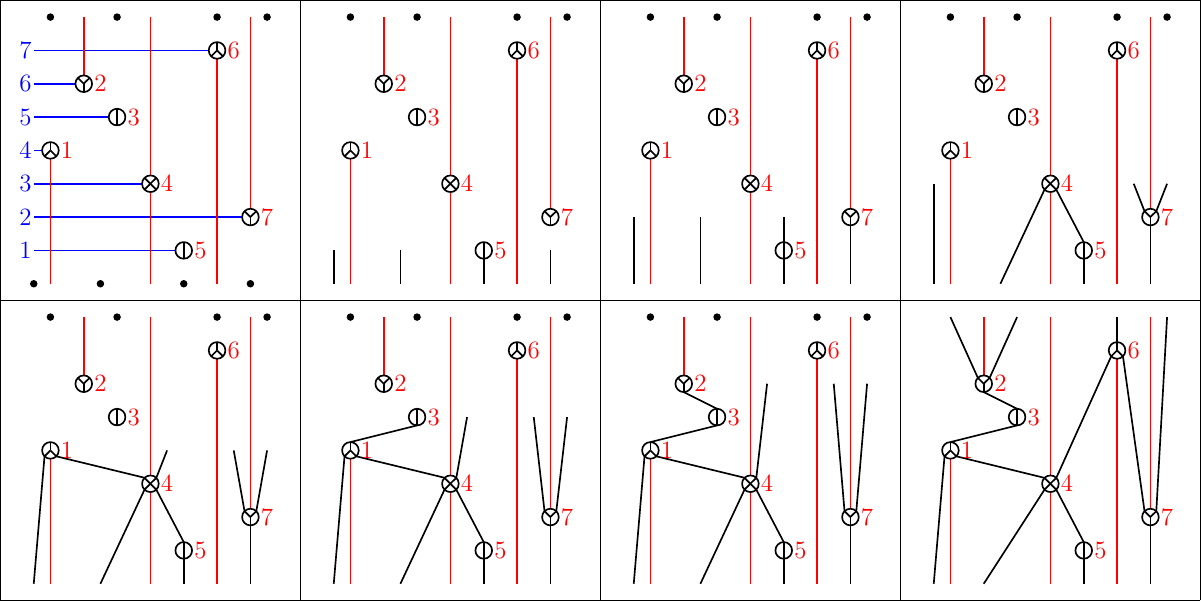}
	\caption[The insertion algorithm of~$\downn\upp\nonee\uppdownn\nonee\downn\downn$-permutrees applied to the table of the permutation~$5741326$.]{ The insertion algorithm of~$\downn\upp\nonee\uppdownn\nonee\downn\downn$-permutrees applied to the table of the permutation~$5741326$. Figure inspired from~\cite{PP18}.}\label{fig:permutree_insertion_example}
\end{figure}

\begin{definition}\label{def:decorated_permutation}
	A \defn{$\delta$-decorated permutation}\index{permutree!decorated permutation} is a permutation table with each point~$(\pi(i),i)$ decorated by~$\delta_i$. The \defn{insertion algorithm}\index{permutree!insertion algorithm} is the following procedure. Let~$\pi$ be a~$\delta$-decorated permutation. We draw a red wall under (resp.\ above) all~$v_i$ such that~$\delta_i\in\{\downn,\uppdownn\}$ (resp.~$\delta_i\in\{\upp,\uppdownn\}$). The red walls partition the table in zones where each vertex can only see his next immediate children (resp.\ parents) looking downwards (resp.\ upwards). Start the algorithm by generating a string from the bottom of these areas. Now extend these strings by units of height 1 along the table from bottom to top as if reading the permutation~$\pi$. At each step, if the height of a string is the same of a vertex~$v_i$ of the table, we say that if \begin{itemize}
		\itemsep0em
		\item~$\delta_i\in\{\upp,\nonee\}$, the vertex catches the only string it can see,
		\item~$\delta_i\in\{\downn,\uppdownn\}$, the vertex catches both strings it can see and merges them,
	\end{itemize}
	and then if \begin{itemize}
		\itemsep0em
		\item~$\delta_i\in\{\downn,\nonee\}$, the vertex releases a single string,
		\item~$\delta_i\in\{\upp,\uppdownn\}$, the vertex releases two strings around the red wall above it.
	\end{itemize}
	The algorithm ends when the strings have lengths~$n+1$. The resulting~$\delta$-permutree is denoted~$\Theta(\pi)$. Figure~\ref{fig:permutree_insertion_example} has a complete example of this algorithm.
\end{definition}

The insertion algorithm gives us all possible~$\delta$-permutrees via the following propositions.

\begin{proposition}\label{prop:permutree_insertion_bijection}
	The insertion algorithm is a surjection between~$\delta$-decorated permutations and~$\delta$-permutrees.
\end{proposition}

\begin{proposition}\label{prop:permutree_insertion_linear_extension}
	Let~$T\in\cPT_n(\delta)$. The permutations~$\pi$ such that~$\Theta(\pi)=T$, are exactly the linear extensions of the poset of the~$\delta$-permutree~$T$.
\end{proposition}

In Figure~\ref{fig:permutree_to_permutations} we show an example of a~$\delta$-permutree and its corresponding linear extensions.

\begin{figure}[h!]
	\centering
	\includegraphics[scale=1.75]{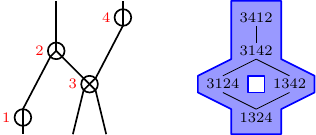}
	\caption[A~$\nonee\downn\uppdownn\nonee$-permutree and its corresponding linear extensions.]{ A~$\nonee\downn\uppdownn\nonee$-permutree (left) and its corresponding linear extensions corresponding to permutations simultaneously avoiding the patterns~$2ki$,~$3ki$, and~$ki3$ as a congruence class (right).}
	\label{fig:permutree_to_permutations}
\end{figure}

As well, the insertion algorithm leads to the following enumeration result.

\begin{lemma}\label{lem:permutree_cardinality_edges}
	The number of edges of a~$\delta$-permutree~$T$ is given by \begin{equation*}
		\Big|E\Big(T\Big)\Big|=n+1+\Big|\Big\{i\in[n]:\delta_i\in\{\downn,\upp\}\Big\}\Big|+2\Big|\Big\{i\in[n]:\delta_i\in\{\uppdownn\}\Big\}\Big|
	\end{equation*}
\end{lemma}

\begin{proof}
	There are~$1+|\{i\in[n]:\delta_i\in\{\downn,\uppdownn\}\}|$ (resp.~$1+|\{i\in[n]:\delta_i\in\{\upp,\uppdownn\}\}|$) corresponding to the bottom and top of the zones created by the red walls at the start of the insertion algorithm. The remaining~$n-1$ come from the vertices being connected upwards through the insertion algorithm.
\end{proof}

We can count the number of~$\delta$-permutrees of size~$n$ by recursion in the following way.

\begin{proposition}[{\cite[Cor.2.26]{PP18}}]\label{prop:permutree_recursion}
	For any permutree decoration~$\delta$, the number of~$\delta$-permutrees follows the recursive formula \begin{equation*}
		\big|\cPT(\delta)\big|=\prod_{k\in[\ell]}\sum_{\substack{i\in[b_{k-1},b_k]\cap\delta^{-1}\big(\downn\big)\\J\subseteq[b_{k-1},b_k]\cap\delta^{-1}\big(\nonee\big)}}\big|\cPT(\delta_{[b_{k-1},i-1]\setminus{J}})\big|\big|\cPT(\delta_{[i+1,b_k]\setminus{J}})\big|\big|J\big|!
	\end{equation*} where $\{b_0<\cdots<b_\ell\}=\{0,n\}\cup\delta^{-1}(\uppdownn)$.
\end{proposition}

\subsection{Permutree Lattices}\label{ssec:permutree_lattice}

Like for binary trees, for any fixed decoration~$\delta$ one can define edge rotations on~$\delta$-permutrees.

\begin{definition}\label{def:permutree_rotations}
	Let~$T\in\cPT_n(\delta)$ be a~$\delta$-permutree with an edge~$i\to j$ where~$1\leq i<j\leq n$. An \defn{ij-edge rotation}\index{permutree!edge rotation} is the operation of replacing the (right) subtree of~$v_i$ by~the (left) subtree of~$v_j$ and the (left) subtree by the tree with root~$v_i$, maintaining rest of~$T$ intact. Figure~\ref{fig:permutree_rotations} shows all possible~$ij$-edge rotations.

	The \defn{edge cut} in~$T$ defined by~$i\to j$ is the ordered partition~$(I\,\|\, [n]\setminus{I})$ of the vertex set of~$T$ where~$I$ are the vertices whose undirected paths to~$v_i$ do not visit~$v_j$.
\end{definition}

\begin{example}\label{ex:permutree_edge_cut}
	Consider the~$\nonee\downn\uppdownn\nonee$-permutree given in Figure~\ref{fig:permutree_to_permutations}. The respective edge cuts of the directed edges~$1\to 2$,~$3\to 2$, and~$3\to 4$, are~$(\{1\}\,\|\, \{2,3,4\})$,~$(\{3,4\}\,\|\, \{1,2\})$, and~$(\{1,2,3\}\,\|\, \{4\})$.
\end{example}

\begin{proposition}\label{rotation_edge_cuts}
	The~$ij$-rotation of a~$\delta$-permutree~$T$ is a~$\delta$-permutree~$T'$ whose edge cuts are precisely those of~$T$ except the edge cut defined by~$i\to j$.
\end{proposition}

The resulting poset is called the \defn{rotation poset of~$\delta$-permutrees}\index{permutree!rotation poset} and its covering relations are characterized by edge rotations. Figure~\ref{fig:permutree_rotations} shows rotations between all possible adjacent vertices and Figure~\ref{fig:permutree_lattice_IXYI} presents an example of such a rotation poset where~$\delta=\nonee\uppdownn\upp\nonee$.

\begin{figure}[h]
	\centering
	\includegraphics[scale=0.905]{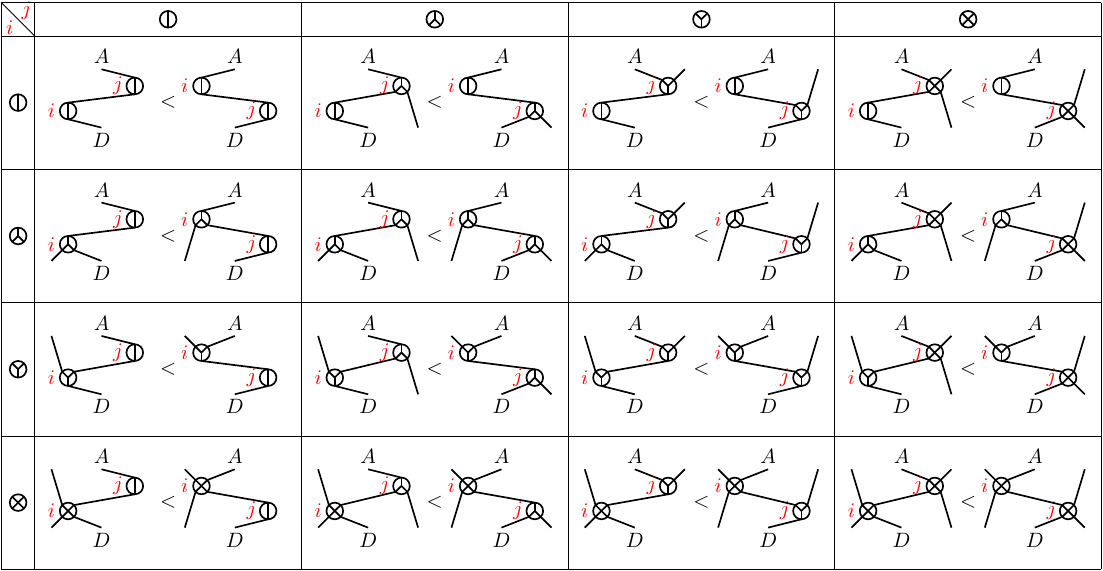}
	\caption[The~$ij$-rotations of~$\delta$-permutrees.]{ All possible~$ij$-rotations of~$\delta$-permutrees. Figure based from~\cite{PP18}.}\label{fig:permutree_rotations}
\end{figure}

\begin{figure}[h]
	\centering
	\includegraphics[scale=0.7]{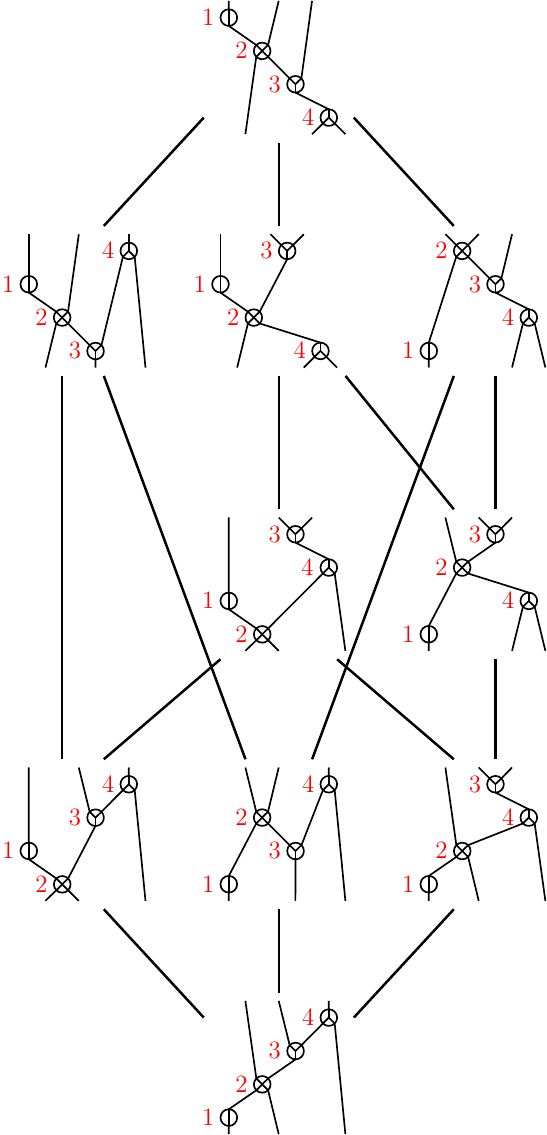}
	\caption{The poset of rotations of~$\nonee\uppdownn\upp\downn$-permutrees.}\label{fig:permutree_lattice_IXYI}
\end{figure}

\begin{remark}\label{rem:permutree_rotations_bounded}
	Notice that~$\delta$-permutree posets are always bounded. The minimal element~$\hat{0}_\delta$ (resp.\ maximal element~$\hat{1}_\delta$) is the~$\delta$-permutree such that~$i\to i+1$ (resp.~$i+1\to i$) for all~$i\in[n-1]$.
\end{remark}

As for binary trees, the rotation poset of permutrees is a lattice.

\begin{proposition}[{\cite[Prop.2.32]{PP18}}]\label{prop:permutree_lattice_property}
	The poset of~$\delta$-permutrees~$(\cPT_n(\delta),\leq)$ is a lattice.

	Moreover, the~$\delta$-permutree lattice is isomorphic to \begin{itemize}
		\itemsep0em
		\item the weak order of~$\fS_n$ if~$\delta=\nonee^n$,
		\item the Tamari lattice if~$\delta=\downn^n$,
		\item the (Type~$A$) Cambrian lattices if~$\delta\in\{\downn,\upp\}^n$,
		\item the boolean lattice if~$\delta=\uppdownn^n$.
	\end{itemize}
\end{proposition}

The proof of~\cite{PP18} of the lattice property uses the theory of lattice quotients. A constructive proof of this fact using similar ideas as Proposition~\ref{prop:tamari_lattice_meet} is possible and is presented in Chapter~\ref{chap:permutree_vectors}.

\subsection{Permutree Congruences}\label{ssec:permutree_congruence}

The~$\delta$-permutree lattice can also be constructed via an equivalence relation on the weak order of~$\fS_n$.

\begin{proposition}[{\cite[Prop.32]{PP18}}]\label{prop:permutree_as_quotient}
	The~$\delta$-permutree lattices are lattice quotients of the weak order.
\end{proposition}

Going further, these lattices can be characterized in the following ways.

\begin{proposition}[{\cite{PP18}}]\label{prop:permutree_quotients}
	Let~$\delta\in\{\nonee,\downn,\upp,\uppdownn\}$. The permutree equivalence relation~$\equiv_\delta$ can be defined equivalently as the equivalence relation whose classes are
	\begin{itemize}
		\itemsep0em
		\item the linear extensions of~$\delta$-permutrees,
		\item the transitive closure on the relations of the form \begin{equation*}
			      \begin{split}
				      UikVjW \equiv_{\delta} UkiVjW & \text{ if } \delta\in\{\downn,\uppdownn\},\\
				      UjVikW \equiv_{\delta} UjVkiW & \text{ if } \delta\in\{\upp,\uppdownn\},
			      \end{split}
		      \end{equation*} where~$i<j<k$ are positive integers and~$U,V,W$ are words in the over~$[n]$,
		\item the fibers of the insertion algorithm (Definition~\ref{def:decorated_permutation}).
	\end{itemize} Furthermore, the following objects are in bijection: \begin{itemize}
		\itemsep0em
		\item permutrees with decoration~$\delta$,
		\item~$\delta$-permutree congruence classes,
		\item permutations that avoid the patterns~$kij$ if~$\delta_j\in\{\downn,\uppdownn\}$ and~$jki$ if~$\delta_j\in\{\upp,\uppdownn\}$ (minima of~$\delta$-permutree congruence classes).
	\end{itemize}
\end{proposition}

\begin{definition}\label{def:decoration_order}
	Consider the order on permutree decorations~$\nonee\prec\{\downn,\upp\}\prec \uppdownn$. Given~$\delta,\delta'\in\{\nonee,\upp,\downn,\uppdownn\}^n$, we say that~$\delta$ \defn{refines}\index{permutree!refinement}~$\delta'$ (resp.~$\delta'$ \defn{coarsens}\index{permutree!coarsening}~$\delta$) denoted~$\delta\preceq\delta'$ if~$\delta_i\preceq\delta'_i$ for all~$i\in[n]$.
\end{definition}

\begin{proposition}\label{prop:permutree_matrioshcka}
	Let~$\delta,\delta'\in\{\nonee,\upp,\downn,\uppdownn\}^n$ two permutree decorations. If~$\delta\preceq\delta'$ then as congruences~$\equiv_\delta$ refines~$\equiv_{\delta'}$.
\end{proposition}

Figure~\ref{fig:fibersPermutreeCongruences} shows all~$\delta$-permutree congruences for~$\delta\in{\{\nonee,\downn,\upp,\uppdownn\}}^4$ ordered by refinement and with initial and final decorations~$\nonee$ following Remark~\ref{rem:permutree_leftmost_rightmost_labels}. See Figure~\ref{fig:permutree_to_permutations} for an example of the bijection in Proposition~\ref{prop:permutree_quotients}.

\begin{figure}
	\centering
	\includegraphics[scale=0.3,angle=90]{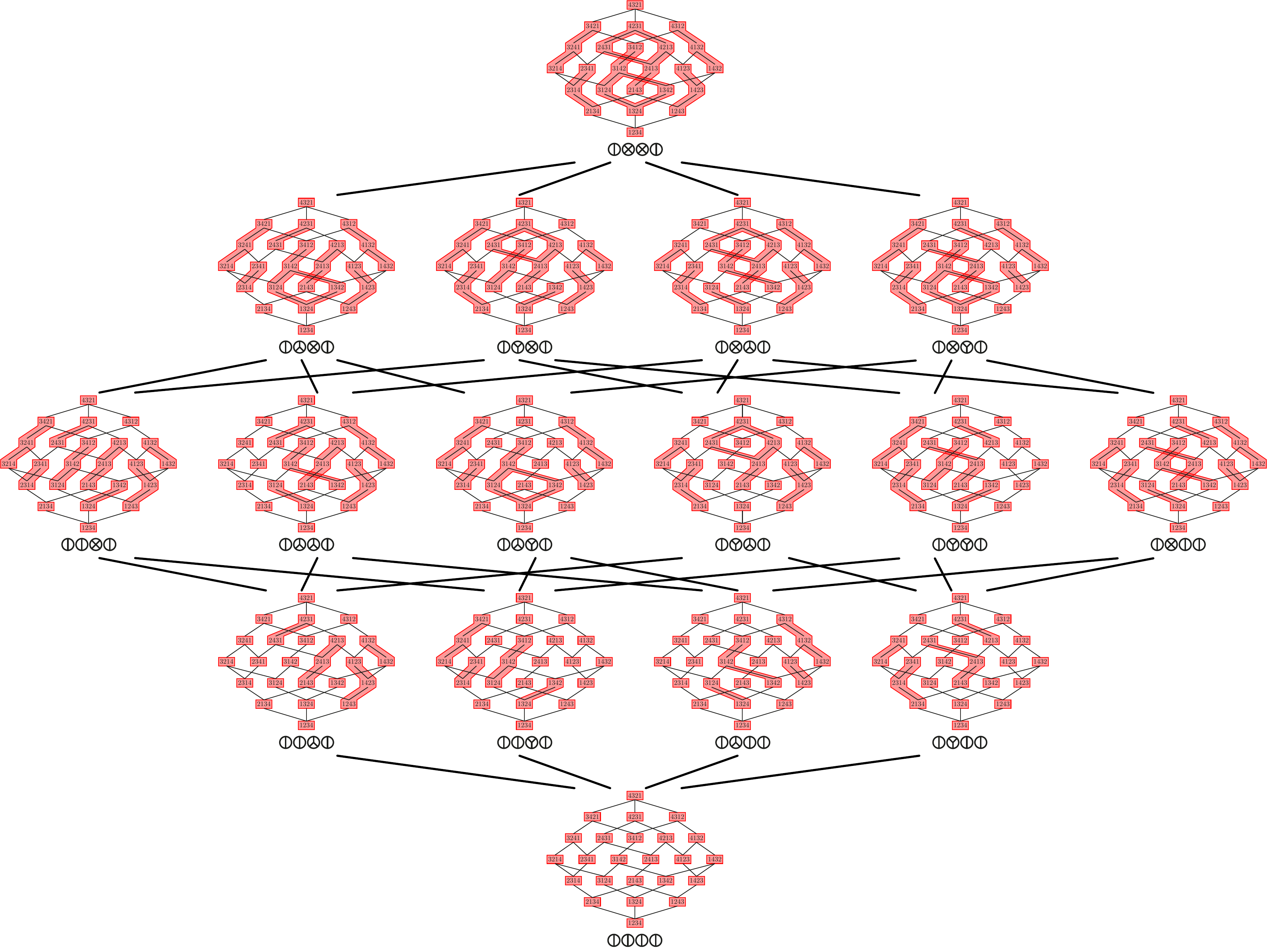}
	\caption[The fibers of all~${\nonee\cdot\{\nonee,\downn,\upp,\uppdownn\}}^2\cdot\nonee$-permutree congruences.]{ The fibers of all~${\nonee\cdot\{\nonee,\downn,\upp,\uppdownn\}}^2\cdot\nonee$-permutree congruences. Figure from~\cite{PP18}.}\label{fig:fibersPermutreeCongruences}
\end{figure}

\begin{remark}
	Following Remark~\ref{rem:lattice_quotients_via_polygons}, the~$\delta$-permutree congruences~$\equiv_\delta$ are also referred to as the rank~$2$ lattice congruences of the weak order. This is due to the fact that they are generated by the forcing of the congruence relations~$12\cdots (j-1)(j+1)j\cdots (n-1)n\equiv 12\cdots (j+1)(j-1)j\cdots (n-1)n$ if~$\delta_j\in\{\downn,\uppdownn\}$ and~$12\cdots j(j+1)(j-1)\cdots (n-1)n\equiv 12\cdots j(j+1)(j-1)\cdots (n-1)n$ if~$\delta_j\in\{\upp,\uppdownn\}$. We explore this idea in a wider context in Subsection~\ref{ssec:Coxeter_sorting}.
\end{remark}

\subsection{Permutreehedra}\label{ssec:permutree_geometry}

\begin{proposition}[{\cite[Thm.3.4]{PP18}}]\label{prop:permutreehedron}
	The~$\delta$-permutree rotation lattice~$(\cPT_n(\delta))$ is realized by the~\defn{$\delta$-permutreehedron}\index{polytope!permutreehedron}~$\PPT(\delta)$ defined equivalently as:
	\begin{itemize}
		\itemsep0em
		\item the convex hull of points of the form \begin{equation*}
			      \mathbf{a}(T)_i =
			      \begin{cases}
				      1+d                           & \text{ if } \delta_i=\nonee,   \\
				      1+d+|LD_i||RD_i|              & \text{ if } \delta_i=\downn,   \\
				      1+d-|LA_i||RA_i|              & \text{ if } \delta_i=\upp,     \\
				      1+d+|LD_i||RD_i|-|LA_i||RA_i| & \text{ if }\delta_i=\uppdownn, \\
			      \end{cases}
		      \end{equation*} where~$d$ is the number of descendants of~$v_i$, and~$T$ is a~$\delta$-permutree,
		\item the intersection of the following hyperplane and half-spaces \begin{equation*}
			      \left\{\mathbf{x}\in\RR^n \,:\,\sum_{i\in [n]}x_i=\gbinom{n+1}{2}\right\} \cap \bigcap_{I\in\mathcal{I}} \left\{\mathbf{x}\in\RR^n \,:\, \sum_{i\in I}x_i\geq\gbinom{|I|+1}{2}\right\},
		      \end{equation*} where~$\mathcal{I}=\{I\subsetneq[n]\,:\,\exists\text{ a~$\delta$-permutree with edge cut }(I\,\|\,[n]\setminus I)\}$.
	\end{itemize}
\end{proposition}

See Figure~\ref{fig:PPT4} for some examples of~$\delta$-permutreehedra.

\begin{proposition}\label{prop:permutreehedron_from_quotients}
	Let~$\mathbf{v}=(\mathbf{w_0})-(\mathbf{e})=(n-1,n-3,\ldots,-n+3,-n+1)={(2i-n-1)}_{i\in[n]}$. The~$1$-skeleton of the~$\PPT(\delta)$ oriented with the vector~$\mathbf{v}$ is isomorphic to the Hasse diagram of the~$\delta$-permutree rotation lattice.
\end{proposition}

\begin{figure}[h!]
	\centering
	\includegraphics[scale=1.2]{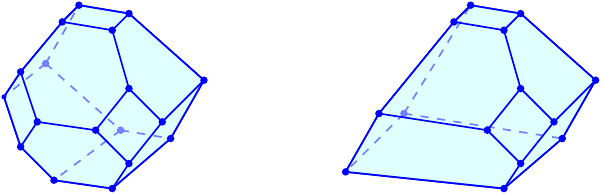}
	\caption[The permutreehedra~$\PPT(\nonee\upp\nonee\nonee)$ and~$\PPT(\nonee\upp\downn\nonee)$.]{ The permutreehedra~$\PPT(\nonee\upp\nonee\nonee)$ (left) and~$\PPT(\nonee\upp\downn\nonee)$ (right).}
	\label{fig:PPT4}
\end{figure}

\begin{remark}\label{rem:permutreehedra_invariance}
	Consider~$\delta,\delta'\in\{\downn,\upp\}$. Although the~$\delta$-permutree lattice might not be isomorphic to the~$\delta'$-permutree lattice, the~$\delta$-permutreehedron is isomorphic to the~$\delta'$-permutreehedron. That is, all these permutreehedrons are associahedrons.
\end{remark}

\section{Finite Coxeter Groups}

In this section we present the basic theory of finite Coxeter groups basing ourselves on~\cite{H90} and~\cite{BB06}.

\begin{definition}\label{def:Coxeter_system_graph}
	Let~$S$ be any set together with integers~$m_{s,t}\in\NN_{>0}$ for all pairs~$(s,t)\in S\times S$ such that \begin{itemize}
		\itemsep0em
		\item~$m_{s,t} = 1$ if and only if~$s=t$,
		\item~$m_{s,t}=m_{t,s}$ for all~$s,t\in S$.
	\end{itemize}
	The \defn{Coxeter graph}\index{Coxeter!graph} is the graph with vertex set~$S$ and edges~$(s,t)$ where~$m_{s,t}\geq 3$. If~$m_{s,t}\geq 4$ the edges are labeled with~$m_{s,t}$, otherwise they are unlabeled. The corresponding \defn{Coxeter group}\index{Coxeter!group} is the group generated as~$W=\langle S \,|\, (st)^{m_{s,t}}=e\rangle$ where~$e$ is the identity element. If~$m_{s,t}\geq 3$ we call~$(st)^{m_{s,t}}=e$ a \defn{braid relation}\index{Coxeter!braid relation}.

	The pair~$(W,S)$ is called \defn{Coxeter system}\index{Coxeter!system},~$S$ is the set of \defn{Coxeter generators}\index{Coxeter!generators} and~$|S|$ is the \defn{rank}\index{Coxeter!rank} of~$W$. The system~$(W,S)$ is said to be \defn{irreducible}\index{Coxeter!irreducible} if its Coxeter graph is connected. Figure~\ref{fig:CoxeterGroups} contains the Coxeter graphs of all irreducible finite Coxeter groups.
\end{definition}

\begin{proposition}[{\cite{C35}}]\label{prop:Coxeter_classification}
	An irreducible Coxeter Group is finite if and only if its Coxeter graph appears in  Figure~\ref{fig:CoxeterGroups}.
	\begin{figure}[h!]
		\centering
		\includegraphics[scale=1.1]{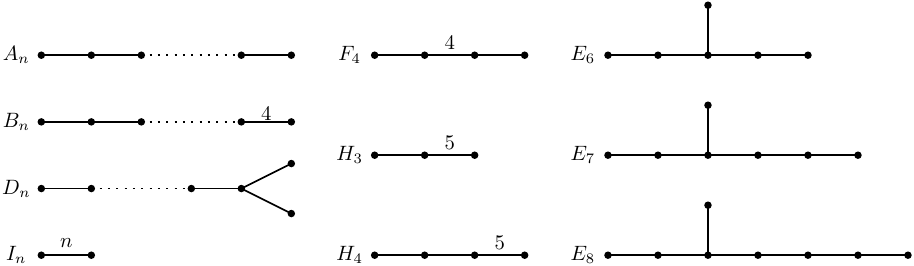}
		\caption[The Coxeter graphs of all irreducible finite Coxeter groups.]{ The Coxeter graphs of all irreducible finite Coxeter groups. The infinite family~$A_n$ is defined for~$n\geq 1$,~$B_n$ for~$n\geq 2$,~$D_n$ for~$n\geq 4$, and~$I_n$ for~$n\geq 5$.}
		\label{fig:CoxeterGroups}
	\end{figure}
\end{proposition}

\begin{remark}\label{rem:braid_relations}
	Since~$m_{s,s}=1$, we have that~$s^2=e$ for all generators. Thus, the relation~$(st)^{m_{s,t}}=e$ implies \begin{equation*}\underbrace{ststst\cdots st}_\text{$m_{s,t}$ times}=\underbrace{tststs\cdots ts}_\text{$m_{s,t}$ times}.\end{equation*}
\end{remark}

\begin{remark}\label{rem:reducibles}
	If the Coxeter graph of a Coxeter system~$(W,S)$ consists of the connected components~$G_1,\ldots,G_m$, then~$W$ is the direct product of the Coxeter groups~$W_1\times\cdots\times W_m$ of each connected component.
\end{remark}

\begin{example}\label{ex:Coxeter_examples}
	Consider a graph~$G$ of~$n$ isolated vertices. This signifies that~$m_{s,t}=2$ for all~$s,t\in S$ and the corresponding Coxeter group~$W=A_1\times\cdots\times A_1$ is the group~$\ZZ_2\times\cdots\times\ZZ_2$.
\end{example}

\begin{example}
	Consider the Coxeter system of type~$(I_n,\{s,t\})$ where~$m_{s,t}=n$. Any word in this group is expressed via these generators and thus only a product of~$s$ and~$t$. Sending~$s$ and~$t$ to the appropriate reflections of an~$n$-gon, one can see that the Coxeter group~$I_n=\langle s,t\,|\,s^2=t^2=(st)^n=e\rangle$ is isomorphic to~$D_n=\langle s,r\,|\,s^2=(sr)^2=(r)^n=e\rangle$.
\end{example}

\begin{definition}\label{def:Coxeter_reflections}
	Given a Coxeter system~$(W,S)$, its \defn{reflections}\index{Coxeter!reflection} are $T=\{wsw^{-1} \,:\, s\in S,\,w\in W\}$. In this way the generators in~$S$ are also called the \defn{simple reflections} of~$W$  and the \defn{right and left inversion sets}\index{Coxeter!inversions sets} of~$w\in W$ are the sets \begin{equation*}
		\begin{split}
			I_R(w)=& \{t\in T\,:\,l(wt)<l(w)\},\\
			I_L(w)=& \{t\in T\,:\,l(tw)<l(w)\}.
		\end{split}
	\end{equation*}
\end{definition}

\begin{definition}\label{def:Coxeter_length}
	Let~$(W,S)$ be a Coxeter system. Considering the elements~$w\in W$ written as products of generators~$w=s_1\cdots s_k$ for~$s_i\in S$, the \defn{length}\index{Coxeter!length} of~$w$ is the minimum~$k$ such that~$w=s_1\cdots s_k$, is denoted as \defn{l(w)}, and~$s_1\cdots s_k$ is said to be a \defn{reduced word} or \defn{reduced expression}\index{Coxeter!reduced word} of~$w$.
\end{definition}

\begin{remark}\label{rem:reduced_word_braid_moves}
	Due to Definition~\ref{def:Coxeter_system_graph} of Coxeter groups, an element~$w\in W$ has a set of reduced words all related by a sequence of braid moves. On one hand, this signifies that the length is well-defined. On the other, this proves that if a transposition~$s_i$ is in a reduced word of~$w$, then~$s_i$ is in every reduced word of~$w$.
\end{remark}

\begin{proposition}
	Some useful properties of the length function include: \begin{itemize}
		\itemsep0em
		\item~$l(ws)=l(w)\pm 1$,
		\item~$l(w)=|I(w)|$,
		\item~$l(w)=l(w^{-1})$,
		\item~$|l(u)-l(w)|\leq l(uw)\leq l(u)+l(w)$.
	\end{itemize}
\end{proposition}

\begin{definition}\label{def:Coxeter_descents}
	Given an element~$w\in W$ of a Coxeter system~$(W,S)$, the \defn{(right) descent set} and \defn{(left) descent set}\index{Coxeter!descents} are \begin{equation*}
		\begin{split}
			D_R(w)=& \{s\in S\,:\,l(ws)<l(w)\},\\
			D_L(w)=& \{s\in S\,:\,l(sw)<l(w)\}.
		\end{split}
	\end{equation*} We call their elements the right (resp. left) \defn{descents} of~$w$.
\end{definition}

\begin{proposition}\label{Coxeter_starting}
	Let~$s\in S$ and~$w\in W$. Then \begin{itemize}
		\itemsep0em
		\item~$s\in D_R(W)$ if and only if there exists a reduced word of~$w$ ending with the letter~$s$.
		\item~$s\in D_L(W)$ if and only if there exists a reduced word of~$w$ starting with the letter~$s$.
	\end{itemize}
\end{proposition}

\begin{definition}\label{def:Coxeter_parabolic}
	Let~$J\subseteq S$ of a Coxeter system~$(W,S)$. The subgroup of~$W$ generated by~$J$ is a \defn{parabolic subgroup}\index{Coxeter!parabolic subgroup} of~$W$ and is denoted by \defn{$W_J$}. The \defn{quotient}\index{Coxeter!parabolic quotient} corresponding to~$J$ is \defn{$W^J$}$=\{w\in W\,:\,l(ws)>l(w) \text{ for all } s\in J\}$ The Coxeter graph of~$W_J$ is the subgraph of the Coxeter graph induced by~$J$. In this context we note by \defn{$W_{\langle s_i\rangle}$} the parabolic group generated by all adjacent transpositions except~$s_i$.
\end{definition}

\begin{proposition}\label{prop:Coxeter_parabolic_quotient}
	For an element~$w\in W$,~$w\in W^{J}$ if and only if no reduced expression of~$w$ ends with a generator in~$J$.
\end{proposition}

\begin{proposition}\label{prop:Coxeter_parabolic_decomposition}
	Given~$w\in W$, there is a unique~$w^J\in W^J$ and~$w_J\in W_J$ such that~$w=w^Jw_J$. Moreover,~$l(w)=l(w^J)+l(w_J)$ and~$w^J$ is the unique element of minimal length in the coset~$wW_J$.
\end{proposition}

\subsection{Weak Order}\label{ssec:Coxeter_weak_order}

For any Coxeter system we can define a myriad of orders. Here we are interested only in the right and left weak orders.

\begin{definition}
	Let~$(W,S)$ be a Coxeter system,~$s\in S$, and~$w\in W$. The \defn{right (resp.\ left) weak order}\index{Coxeter!right weak order}~$\leq$ is the transitive closure of the cover relations~$w\lessdot ws$ (resp.~$w\lessdot sw$) if and only if~$l(w)<l(ws)$ (resp.~$l(w)<l(sw)$).
\end{definition}

\begin{figure}
	\centering
	\includegraphics[scale=1.2]{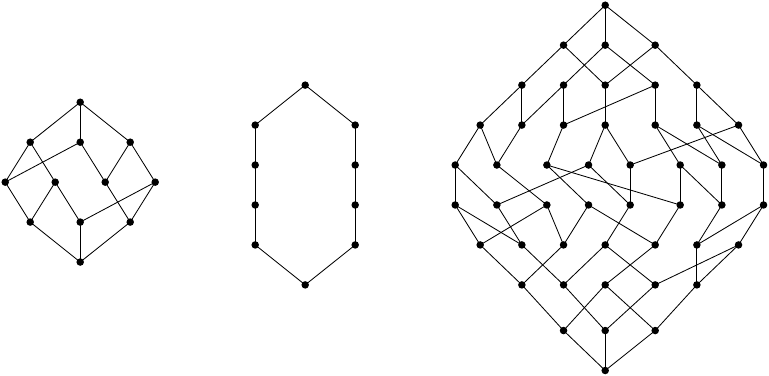}
	\caption{The right weak order of the Coxeter groups~$A_1\times A_2$,~$I_5$, and~$B_3$.}
	\label{fig:Coxeter_weak_orders}
\end{figure}

Figure~\ref{fig:Coxeter_weak_orders} shows some examples of the weak order. In our case of finite Coxeter groups, we obtain the following result.

\begin{proposition}\label{prop:Coxeter_longest_element}
	Let~$W$ be a finite group. There exists an element~$w_0\in W$ such that~$w<w_0$ for all~$w\in W$. We call~$w_0$ the \defn{longest element}\index{Coxeter!longest element} of~$W$.
\end{proposition}

The longest element is crucial for the finite case of Coxeter groups. We list some of its properties.

\begin{proposition}\label{prop:Coxeter_longest_element_properties}
	Let~$w_0$ be the longest element of a Coxeter system~$(W,S)$. Then \begin{itemize}
		\itemsep0em
		\item~$w_0^2=e$,
		\item~$l(w_0)=|T|$.
	\end{itemize}
\end{proposition}

In full generality, the weak order is only a complete semi-lattice. However, in our finite case Proposition~\ref{prop:Coxeter_longest_element} together with Proposition~\ref{prop:semilattice_to_lattice} gives us the lattice property.

\begin{proposition}\label{prop:Coxeter_lattice_proeperty}
	The poset~$(W,\leq)$ is a lattice.
\end{proposition}

Unless stated otherwise,~$\leq$ refers to the right weak order.

\begin{proposition}\label{prop:weak_order_properties} The weak order has the following properties.
	\begin{enumerate}
		\itemsep0em
		\item The reduced words of an element~$w\in W$ are in bijection with maximal chains of the interval~$[e,w]$.
		\item~$u\leq w$ if and only if~$l(u)+l(u^{-1}w)=l(w)$.
		\item~$u\leq w$ if and only if any reduced word of~$u$ is a prefix of a reduced word of~$w$.
		\item If~$s\in D_L(u)\cap D_L(w)$, then~$u\leq w$ if and only if~$su \leq sw$.
		\item~$u\leq_R w$ if and only if~$I_L(u) \subseteq I_L(w)$.
	\end{enumerate}
\end{proposition}

Certain morphisms of Coxeter systems are of particular usefulness to us.

\begin{proposition}\label{prop:Coxeter_automorphisms}
	Consider a Coxeter system~$(W,S)$. \begin{itemize}
		\itemsep0em
		\item The maps~$w\to ww_0$ and~$w\to w_0w$ are antiautomorphisms of~$W$.
		\item The map~$w\to w_0ww_0$ is an automorphism.
		\item The map~$w\to w^{-1}$ is an automorphism sending the right weak order to the left weak order.
	\end{itemize}
\end{proposition}

\subsection{Type~\texorpdfstring{$A$}{}}\label{ssec:type_A}

Coxeter groups of type~$A$ are of particular interest for us since they are precisely the symmetric groups.

Consider~$S=\{s_1,\ldots,s_{n-1}\}$ where identifying generators with adjacent transpositions as~$s_i=(i\,\, i+1)$. Given an adjacent transposition~$s$ and a reduced word~$x$ of a permutation~$\pi\in\fS_n$, the word~$xs$ (resp.~$sx$) corresponds to the permutation~$\pi\cdot s_i$ (resp.~$s_i\cdot\pi$). Whenever we consider there is no room for confusion we use interchangeably a permutation~$\pi$ instead of a reduced word~$x$.

Moreover, a quick calculation shows that the generators~$S$ satisfy the relations~$(s_is_{i+1})^3=e$ and~$(s_is_j)^2=e$ if~$|j-i|>1$. This gives the intuition for the following result.

\begin{proposition}\label{prop:permutations_as_Coxeter}
	$(\fS_n,S)$ is a Coxeter system of type~$A_{n-1}$.
\end{proposition}

This construction leads to a reframing of certain characteristics of permutations in terms of Coxeter groups.

\begin{proposition}
	Let~$\pi\in\fS_n$. Then, \begin{itemize}
		\itemsep0em
		\item~$l(\pi)=|\inv(\pi)|$,
		\item~$D_R(\pi)=\{s_i\in S\,:\,\pi_i>\pi_{i+1}\}$.
	\end{itemize}
\end{proposition}

We finish with a useful lemma given by Proposition~\ref{prop:weak_order_properties} in the case of permutations.

\begin{lemma}\label{lem:perm_starting_with_sj}
	A permutation~$\pi\in\fS_n$ permutes the values~$j$ and~$j+1$ if and only if it has a reduced word starting with~$s_j$.
\end{lemma}

\begin{proof}
	Recall that~$\pi$ permuting~$j$ and~$j+1$ is equivalent to~$\pi$ containing the inversion~$(j\,\, j+1)$. Proposition~\ref{prop:weak_order_properties}\,(5) gives us that this is equivalent to~$s_j\cdot e<\pi$ where~$e$ is the identity permutation. Proposition~\ref{prop:weak_order_properties}\,(3) tells us that this occurs if and only if~$s_j$ is a prefix of a reduced word of~$\pi$ as desired.
\end{proof}

\subsection{Type~\texorpdfstring{$B$}{}}\label{ssec:type_B}

For an easier reading we will use the notation~\defn{$\overline{i}$} for the negative integer~$-i$.

\begin{definition}\label{def:type_B_permutations}
	Let~\defn{$\fS_n^B$} be the group of permutations on~$[\pm n]:=\{\overline{n},\ldots,\overline{1},1,\ldots,n\}$ satisfying~$\pi(\overline{i})=\overline{\pi(i)}$ for all~$i\in[n]$. This group is called the group of all \defn{signed permutations}\index{permutation!signed} on~$[n]$. Given~$w\in\fS_n^B$ we write~$w=w(\overline{n})\cdots w(\overline{1})w(1)\cdots w(n)$ in the usual~$1$-line permutation notation or~$w=[w(1)\cdots w(n)]$ in what is called the \defn{window notation}\index{permutation!signed!window} of~$w$.
\end{definition}

\begin{remark}\label{rem:recovering_A_from_B}
	Notice that~$\fS_n\subset\fS_n^B$ naturally by identifying~$\fS_n$ with the signed permutations~$w$ such that~$w([n])=w([n])$.
\end{remark}

Similar to the definitions for permutations given in Definition~\ref{prop:cubical_permutahedron}, signed permutations have a cycle decomposition as now exemplified.

\begin{example}\label{ex:window_notation_cycle}
	If~$w=5\,\overline{3}\,\overline{2}\,\overline{1}\,4\,\overline{4}\,1\,3\,2\,\overline{5}$, then~$w=[\overline{4}\,1\,3\,2\,\overline{5}]$ and one of its cycle decompositions is~$(1\,\overline{4}\, \overline{2}\, \overline{1}\, 4\, 2)(3)(\overline{3})(5\, \overline{5})$.
\end{example}

Similar to Definition~\ref{def:permutations_transpositions}, cycle decompositions give us adjacent transpositions.

\begin{definition}\label{def:B_simple_transpositions}
	In the context of signed permutations the \defn{adjacent transpositions}\index{permutation!transposition} are~$s_i:=(i\,\, i+1)(\overline{i}\,\, \overline{i+1})$ for~$i\in[n-1]$ and~$s_n:=(\overline{n}\,\, n)$.
\end{definition}

\begin{proposition}\label{prop:B_generators}
	The group of signed permutations~$\fS_n^B$ is generated by the adjacent transpositions~$S^B=\{s_1,\ldots,s_{n-1},s_n\}$.
\end{proposition}

\begin{definition}\label{def:B_inversions}
	We say that~$(i,j)\in[n]^2$ is a \defn{$B$-inversion}\index{permutation!inversion} of~$\pi$ if~$i<j$ and~$\pi^{-1}_i>\pi^{-1}_j$ or~$i\leq j$ and~$\pi^{-1}_{\overline{i}}>\pi^{-1}_j$ and denote by \defn{$\inv_B(\pi)$} the set of~$B$-inversions of~$\pi$. We define the length of~$\pi\in\fS_n^B$ by~$\ell_B(w):=|\inv_B(w)|$.
\end{definition}

\begin{proposition}\label{prop:B_Coxeter}
	$(\fS_n^B,S^B)$ is a Coxeter group of type~$B$.
\end{proposition}

\subsection{Type~\texorpdfstring{$D$}{}}\label{ssec:type_D}

\begin{definition}\label{def:type_D_permutations}
	Let~\defn{$\fS_n^D$} be the subgroup of~$\fS_n^B$ of signed permutations having an even number of negative values in their window notation. This group is called  the group of all \defn{even signed permutations}\index{permutation!signed} of~$[n]$.
\end{definition}

As a subgroup of~$\fS_n^B$, we can describe a generating set of~$\fS_n^D$ in the following way.

\begin{proposition}\label{prop:D_generators}
	The group of signed permutations~$\fS_n^D$ is generated by the adjacent transpositions~$S^D=\{s_1,\ldots,s_{n-1},s_n\}$ where~$s_i:=(i\,\, i+1)(\overline{i}\,\, \overline{i+1})$ for~$i\in[n-1]$ and~$s_n:=(n\,\, \overline{n-1})(n-1\,\, \overline{n})$.
\end{proposition}

\begin{definition}\label{def:D_inversions}
	We say that~$(i,j)\in[n]^2$ is a \defn{$D$-inversion}\index{permutation!inversion} of~$\pi$ if~$i<j$ and~$\pi^{-1}_i>\pi^{-1}_j$ or~$i< j$ and~$\pi^{-1}_{\overline{i}}>\pi^{-1}_j$ and denote by \defn{$\inv_D(\pi)$} the set of~$D$-inversions of~$\pi$. We define the length of~$\pi\in\fS_n^D$ by~$\ell_D(w):=|\inv_D(w)|$.
\end{definition}

\begin{proposition}\label{prop:D_Coxeter}
	$(\fS_n^D,S^D)$ is a Coxeter group of type~$D$.
\end{proposition}

\subsection{Automata}\label{ssec:Coxeter_automata}

As the theory of Coxeter groups is based on words, it is natural that the theory of automata has seen applications on it (see~\cite{BH93} and~\cite{HNW16}). Here we concentrate ourselves on the automaticity of Coxeter groups of~\cite{BH93}.

The automaticity was given in general for all finitely generated Coxeter groups using small roots. Since we do not delve into the geometric aspect of Coxeter groups, we rephrase their results using inversion sets. As a disclaimer, we can only do this since we are working with finite Coxeter groups.

\begin{proposition}\label{Coxeter_automat}
	Let~$(W,S)$ be the Coxeter system of a finite Coxeter group~$W$. The language of reduced words of~$W$ is regular. Moreover, it is recognized by the DFA with states~$Q=\{I(w)\,:\,w\in W\}$ and transitions~$I(w)\to I(ws_i)$ labeled by~$s_i$ if and only if~$s_i\notin D_R(w)$. The starting set is the empty set and all states are final states.
\end{proposition}

\begin{example}\label{ex:Coxeter_language_automata}
	The language of reduced words of~$(\fS_3,\{s_1,s_2\})$ is recognized by the automaton in Figure~\ref{fig:Coxeter_language_automata}.
	\begin{figure}[h!]
		\centering
		\includegraphics[scale=1]{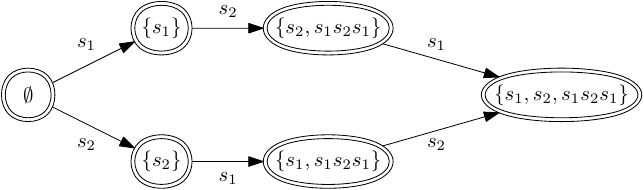}
		\caption{A DFA recognizing the language of reduced words of~$(\fS_3,\{s_1,s_2\})$.}\label{fig:Coxeter_language_automata}
	\end{figure}
\end{example}

\subsection{Coxeter Sorting}\label{ssec:Coxeter_sorting}

In this section we present the theory of~$c$-sorting following~\cite{R07b}. As a starting point we define Cambrian congruences for any Coxeter group.

\begin{definition}\label{def:Coxeter_Cambrian_congruence}
	Let~$W$ be a Coxeter group and~$\overrightarrow{G}$ be a complete orientation of the Coxeter diagram~$G$. The \defn{Cambrian congruence}\index{Cambrian congruence} associated to~$\overrightarrow{G}$ is the smallest lattice congruence that identifies the elements of the interval~$[t,tsts\cdots]$ ($m_{s,t}-1$ letters) for all~$s\to t$.
\end{definition}

Figure~\ref{fig:Coxeter_sorting_orientation} shows an example of a Coxeter graph and the corresponding Cambrian congruence.

\begin{figure}[h!]
	\centering
	\includegraphics[scale=1.5]{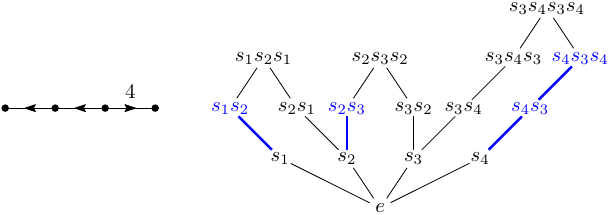}
	\caption[A complete orientation of the Coxeter graph of~$B_4$ and the lattice congruence it specifies.]{ A complete orientation of the Coxeter graph of~$B_4$ and the lattice congruence it specifies. The minimal elements of the lattice congruence are in black and the congruence classes are shown in bolded blue. Notice that we do not draw the squares given by the commuting relations between~$\{s_1,s_3\},\,\{s_1,s_4\}$, or~$\{s_2,s_4\}$ as any congruence relation of rank~$2$ in these cases would force a congruence relation of rank~$1$ following Remark~\ref{rem:lattice_quotients_via_polygons}.}\label{fig:Coxeter_sorting_orientation}
\end{figure}

\begin{remark}\label{rem:cambrian_type_A_are_permutrees}
	Notice that in type~$A$ the Cambrian congruences correspond to the~$\delta$-permutree congruences~$\equiv_\delta$ where~$\delta\in\{\downn,\upp\}^n$ as in Proposition~\ref{prop:permutree_quotients}.
\end{remark}

\begin{definition}\label{def:Coxeter_element}
	Let~$(W,S)$ be a Coxeter system. A \defn{Coxeter element}\index{Coxeter!element} is a word of length~$n=|S|$ where each generator appears exactly once.
\end{definition}

\begin{remark}\label{rem:coxeter_elements_cambrian_congruences}
	Notice that Coxeter elements codify complete orientations of the Coxeter graph and thus Cambrian congruences. Indeed, given a Coxeter element~$c=s_1\cdots s_n$ we orient~$j\to j-1$ (resp.~$j-1\to j$) if~$j$ appears before (resp.\ after)~$j-1$ in~$c$.
\end{remark}

\begin{definition}\label{def:c-sorting}
	Given a Coxeter element~$c$ together with a fixed reduced word~$c=s_1\cdots s_n$, its \defn{infinite word} is~$c^\infty=s_1\cdots s_n|s_1\cdots s_n|s_1\cdots s_n|\cdots$ where ``$|$'' are dividers that mark the separation of each instance of~$c$.

	The \defn{$c$-sorting word} of an element~$w\in W$ is the lexicographically first subword of~$c^\infty$ that is a reduced word of~$w$. We denote it as \defn{$w(c)$}. The~$c$-sorting word~$w(c)$ can be divided into finite factors~$L_1,L_2,\ldots,L_n$ where each factor~$L_i$ is a set of letters of~$w(c)$ appearing between the dividers~$i$ and~$i+1$. We say that an element~$w\in W$ is \defn{$c$-sortable} or \defn{Coxeter-sortable}\index{Coxeter!sortable} if the factors satisfy~$L_1\supseteq L_2\supseteq\ldots\supseteq L_n$.
\end{definition}

\begin{example}\label{ex:Coxeter_sorting}
	Let~$\pi:=3421$ and take the Coxeter element~$c:=s_2\cdot s_1\cdot s_3$. Consider the infinite word~$c^\infty={\color{brown}c}\cdot {\color{blue}c}\cdot {\color{red}c}\cdots={\color{brown}s_2\cdot s_1\cdot s_3}\cdot {\color{blue} s_2\cdot s_1\cdot s_3}\cdot{\color{red} s_2\cdot s_1\cdot s_3}\cdots.$ Then the~$c$-sorting word of~$\pi$ is~$\pi(c)={\color{brown}s_2\cdot s_1\cdot s_3}\cdot{\color{blue} s_2\cdot s_3}$. Since~${\color{brown}L_1}\supset{\color{blue} L_2}\supset{\color{red} L_3}=\emptyset$,~$\pi$ is~$c$-sortable.
\end{example}

\begin{remark}\label{rem:c_infinite_does_not_have_to_be_infinite}
	Notice that~$c^\infty$ does not actually need to be infinite. Indeed, the choosing of a Coxeter element implies the choosing of a reduced word of the largest element~$w_0$.
\end{remark}

\begin{remark}\label{rem:Coxeter_sortability_only_dependent_on_c}
	Although the~$c$-sorting word depends heavily on the choosing of the reduced word of~$c$, the property of being~$c$-sortable does not. Thus, we refer to~$c$-sortable elements without specifying the reduced word of~$c$.
\end{remark}

Coxeter sorting is of interest to us as it recovers the classical notion of stack-sorting~\cite{K73} via another lens.

\begin{proposition}\label{prop:stack_sortable_permutations}
	For the Coxeter word~$c=s_1\cdots s_{n-1}$, the~$c$-sortable permutations are precisely the stack-sortable permutations of~\cite{K73}.
\end{proposition}

The following are some simple lemmas showing the relationship between~$c$-sorting and the weak order.

\begin{lemma}[{\cite[Lems.2.1 \& 2.2]{R07b}}]\label{lem:coxeterElementFacts}
	Consider a Coxeter element of the form~$c = s_\ell \cdot d$ and let~$w\in W$. Then
	\begin{itemize}
		\itemsep0em
		\item if~$w = s_\ell \cdot u$ with~$\ell(w) = \ell(u) + 1$, then~$w(c) = s_\ell \cdot u(d \cdot s_\ell)$,
		\item otherwise,~$w(c) = w(d \cdot s_\ell)$.
	\end{itemize}
\end{lemma}

\begin{lemma}\label{lem:Coxeter_sorting_word_properties}
	Fix a reduced expression of a Coxeter element~$c$. Consider a~$c$-sortable element~$w\in W$ and~$s_i,s_j\in S$ two distinct generators appearing in the~$c$-sorting word~$w(c)$.
	\begin{enumerate}
		\itemsep0em
		\item if~$s_i$ appears before~$s_j$ in~$c$, then~$s_i$ appears before~$s_j$ in~$w(c)$,
		\item if~$s_j$ does not appear in~$w(c)$ between two occurrences of~$s_i$, then it does not appear afterwards.
	\end{enumerate}
\end{lemma}

\begin{proof} We prove both statements separately.
	\begin{enumerate}
		\itemsep0em
		\item As both~$s_i$ and~$s_j$ appear in~$w(c)$, it is immediate from the construction of~$w(c)$ from~$c$.
		\item Notice that each generator can appear at most once in each factor~$L$. Since~$w$ is~$c$-sortable, the factors are weakly decreasing via inclusion. Thus, if~$s_j$ appeared after two occurrences of~$s_i$, it would have to appear between them.
		      \qedhere
	\end{enumerate}
\end{proof}

Coxeter-sorting has a close relationship with lattice congruences of the weak order.

\begin{proposition}\label{prop:Coxeter_sorting_lattice_congruence}
	Let~$c$ be a Coxeter element of a Coxeter group~$W$.
	The~$c$-sortable elements are precisely the bottom elements of the congruence classes of the~$c$-Cambrian congruence~$\equiv_c$ associated to the complete orientation of~$G$ where~$s\to t$ if~$s$ appears before~$t$ in~$c$.
\end{proposition}

\begin{proposition}\label{prop:Coxeter_sorting_congruence_inverse}
	The map~$w\to ww_0$ sends the lattice congruence~$\equiv_c$ to the lattice congruence~$\equiv_{c^{-1}}$.
\end{proposition}

\begin{proposition}[{\cite[Thm.9.1]{R07a}}]\label{prop:Coxeter_sorting_catalan_number}
	The amount of $c$-sortable elements of a Coxeter group $W$ of rank $n$ is the $W$-Catalan number.
\end{proposition}

The $W$-Catalan numbers are presented in Table~\ref{tab:W-W_catalan}.

\begin{table}[h!]
	\centering
	\includegraphics[scale=1.5]{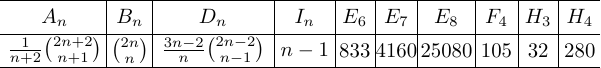}
	\caption{The $W$-Catalan numbers.}
	\label{tab:W-W_catalan}
\end{table}

Following Proposition~\ref{prop:Coxeter_sorting_lattice_congruence} we can define a particular subset of~$c$-sortable elements which help define polytopes whose oriented~$1$-skeleton correspond to the~$c$-Cambrian lattice obtained from the lattice congruence~$\equiv_c$.

\begin{definition}[{\cite{HLT11}}]\label{def:c_singleton}
	Let~$c$ be a Coxeter element. A~$c$-sortable element~$w\in W$ is said to be a~\defn{$c$-singleton} if any of the following equivalent events occurs: \begin{itemize}
		\itemsep0em
		\item $w$ is the only element in its congruence class under~$\equiv_c$,
		\item $w$ has a reduced word that is a prefix of~$w_0(c)$ up to commutations,
		\item $w$ is~$c$-sortable and~$ww_0$ is~$c^{-1}$-sortable.
	\end{itemize}
\end{definition}

\begin{proposition}\label{prop:c_generalized_associahedra}
	Given a Coxeter element~$c$, the intersection of all defining half-spaces of~$\PPerm_n$ whose corresponding facet contains a~$c$-singleton is a polytope whose oriented~$1$-skeleton is the~$c$-Cambrian lattice.
\end{proposition}

\subsection{Permutree Perspectives}

Having all tools in hand, we present the main problematic of this thesis. Following Remark~\ref{rem:cambrian_type_A_are_permutrees}, one can define permutree congruences for any Coxeter type as follows.

\begin{definition}\label{def:permutrees_are_orientations}
	Let~$(W,S)$ be a Coxeter system with its Coxeter diagram~$G_W$. We call a \defn{multi-orientation}~$\overrightarrow{G_W}$ an endowing of each edge of~$G_W$ with either none, one, or both possible orientations. The \defn{permutree congruence} associated to~$\overrightarrow{G_W}$ is the smallest lattice congruence that identifies the elements of the interval~$[t,tsts\cdots]$ ($m_{s,t}-1$ letters) if~$s\to t$.

	As each orientation can have none, one, or two orientations, an identification of each edge with a coordinate gives an encoding of each permutree congruence by a decoration $\delta\in\{\nonee,\downn,\upp,\uppdownn\}^{|E(G_W)|}$. That is, we have that \begin{itemize}
		\itemsep0em
		\item $\delta_{st}=\nonee$ if the edge~$(s,t)$ is not oriented,
		\item $\delta_{st}=\upp$ if the edge~$(s,t)$ is oriented~$s\to t$,
		\item $\delta_{st}=\downn$ if the edge~$(s,t)$ is oriented~$t\to s$,
		\item $\delta_{st}=\uppdownn$ if the edge~$(s,t)$ has both orientations.
	\end{itemize} Figure~\ref{fig:permutree_orientation} shows an example for this in type~$B$. In the cases of types~$A,\, B,\, I,\, F$, and~$H$ these vectors can be simply indexed by~$[n]$ where~$n=|S|$. In types~$D$ and~$E$ one must be careful as there are non-commuting braid relations between non-consecutive indices.
\end{definition}

\begin{figure}[h!]
	\centering
	\includegraphics[scale=1.5]{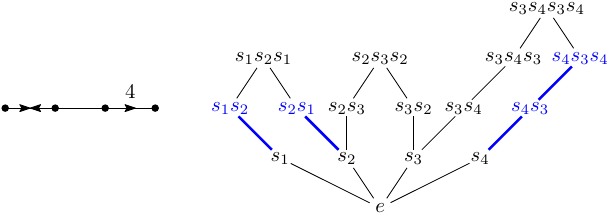}
	\caption[A multi-orientation of the Coxeter graph of~$B_4$ and the permutree congruence it specifies.]{ A multi-orientation of the Coxeter graph of~$B_4$ and the $\uppdownn\nonee\downn$-permutree congruence it specifies. The minimal elements of the lattice congruence are in black and the congruence classes are shown in bolded blue.}\label{fig:permutree_orientation}
\end{figure}

Although Cambrian lattices are well understood in any finite type Coxeter group via their connections with finite type cluster algebras (see~\cite{FZ02},~\cite{FZ03},~\cite{PS13},~\cite{PS15}), permutree lattices are well understood only in type~$A$~\cite{PP18} and partially understood in type~$B$~\cite{PPR22}.

\begin{perspective}\label{pers:Coxeter_permutrees}
	Are there other combinatorial families or methods through which we can study permutree lattices for any finite Coxeter group?
\end{perspective}

\section{\texorpdfstring{$s$}{}-Decreasing Trees}\label{sec:s_weak_order}

In this section we present another generalization of permutations that we study in Part~\ref{part:sorder}. This section is based on~\cite{CP19} and~\cite{CP22}. Let~$s=(s_1,\ldots,s_n)$ be a \defn{weak-composition} (i.e.\ a vector with non-negative integer entries) and \defn{$|s|$}$=\sum_{i=1}^{n}s_i$.

\subsection{\texorpdfstring{$s$}{}-Weak Order}

\begin{definition}\label{def:s_decreasing_trees}
	An \defn{$s$-decreasing tree}\index{$s$-weak order!$s$-decreasing tree} is a rooted plane tree (i.e. with a concrete embedding) on~$n$ internal vertices (called nodes), labeled by~$[n]$, such that the node labeled~$i$ has~$s_i+1$ children and any descendant~$j$ of~$i$ satisfies~$j<i$. We denote by \defn{$T_0^i,\ldots,T_{s_i}^i$} the subtrees of node~$i$ from left to right, and by \defn{$\cT_s$} the set of~$s$-decreasing trees.
\end{definition}

\begin{figure}[h!]
	\centering
	\includegraphics[scale=1.5]{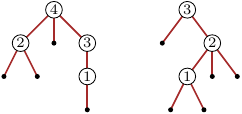}
	\caption[A~$(0,1,0,2)$-decreasing tree and a~$(1,2,1)$-decreasing tree.]{ A~$(0,1,0,2)$-decreasing tree (left) and a~$(1,2,1)$-decreasing tree (right).}
	\label{fig:strees}
\end{figure}

\begin{remark}\label{rem:s_i_zero}
	Since the node~$1$ only has leaves as children, the first entry of~$s$ can be arbitrary and does not influence the combinatorics or geometry of~$s$-decreasing trees. See Figure~\ref{fig:strees} for some examples.
\end{remark}

By its construction the root of an~$s$-decreasing tree is always the number of internal nodes. Moreover, their decreasing aspect allow for a direct calculation of their cardinality. That is, the number of~$s$-trees is given by a generalization of the formula for factorials as follows.

\begin{proposition}\label{prop:num_s_trees}
	Given a weak composition~$s$, the number of~$s$-decreasing trees is \begin{equation*}|\mathcal{T}_s| = \prod_{i=1}^{n-1}\left(1+s_{n-i+1}+s_{n-i+2}+\cdots + s_n\right).
	\end{equation*}
\end{proposition}

\begin{proof}
	There is only one~$(s_n)$-decreasing tree with one node labeled~$n$ and~$s_n+1$ leaves. The next node~$n-1$ together with its~$s_{n-1}+1$ leaves can be placed in any of the~$s_n+1$ leaves from before. This results in~$(1+s_n)$ possible~$(s_{n-1},s_n)$-decreasing trees. Continuing inductively on the length of~$s$, we get that at step~$k$ the number of~$(s_k,\ldots, s_n)$-decreasing trees with labels in~$[k,n]$ is
	$|\mathcal{T}_s| = \prod_{i=1}^{n-k}\bigl(1+\sum_{j=n-i+1}^n s_j\bigr)$.
	We can now place the node labeled~$k-1$ in any of the~$\bigl(1+\sum_{r=k}^n s_r\bigr)$ leaves of any such tree. We finish when we place the node with label~$1$. This gives the desired formula.
\end{proof}

\begin{definition}
	Let~$T$ be an~$s$-decreasing tree and~$1\leq a<c\leq n$. We say that~$a$ is \defn{left} (resp.\ \defn{right})\index{$s$-weak order!$s$-decreasing tree!left and right} of~$c$ if there is a node~$d$ such that~$a<d$,~$c<d$,~$a\in T^d_x$, and~$c\in T^d_y$ where~$x<y$ (resp.~$x>y$). We denote by \defn{$\inv(T)$} the multiset of \defn{inversions}\index{$s$-weak order!$s$-decreasing tree!inversion} of~$T$ formed by pairs~$(c,a)$ with multiplicity (also called cardinality)
	\[\defn{$|(c,a)_T|$} =\left\{\begin{array}{ll}
			\itemsep0em
			0,   & \text{ if } a \text{ is left of } c,  \\
			i,   & \text{ if } a \in T^c_{i},            \\
			s_c, & \text{ if } a \text{ is right of } c.
		\end{array}\right.\]
\end{definition}

Similar to Definition~\ref{def:weak_order_perms} and Proposition~\ref{prop:weak_order_properties}, this allows to define an analogue of the weak order.

\begin{definition}
	Let~$R,T$ be~$s$-decreasing trees. We say that~$R \trianglelefteq T$ if we have that~$\inv(R)\subseteq \inv(T)$. We call~$(\cT_s,\trianglelefteq)$ the \defn{$s$-weak order}\index{$s$-weak order}.
\end{definition}

\begin{figure}[h!]
	\centering
	\includegraphics[scale=0.7]{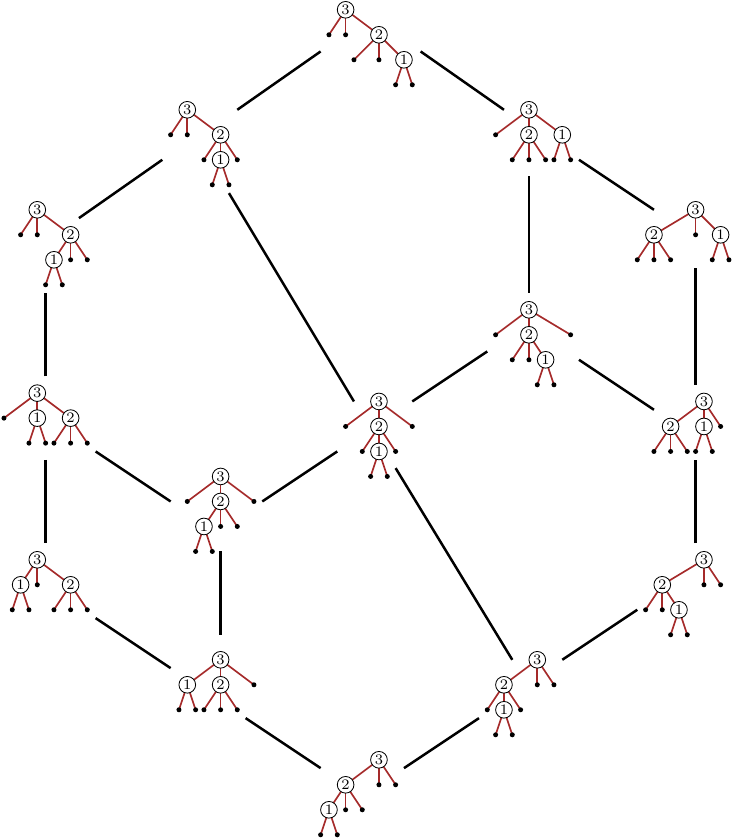}
	\caption[The~$(1,2,2)$-weak order on~$(1,2,2)$-decreasing trees.]{ The~$(1,2,2)$-weak order on~$(1,2,2)$-decreasing trees. Figures based on~\cite{CP22}.}
	\label{fig:strees122}
\end{figure}

As before with Coxeter groups, the~$s$-weak order enjoys several properties coming from the weak order of permutations. A key one for us being the lattice property which can be seen in Figure~\ref{fig:strees122}.

\begin{theorem}[{\cite[Thms.3.2 \& 3.3]{CP19},~\cite[Prop.1.35]{CP22}}]
	The~$s$-weak order on~$s$-decreasing trees is a polygonal lattice. The join of two~$s$-decreasing trees~$R$ and~$T$ is the~$s$-decreasing tree~$R\vee T$ that satisfies~$\inv(R\vee T)=\big(\inv(R)\cup \inv(T)\big)^{tc}$.
\end{theorem}

Notice that the transitive closure requires attention since it is being done over multisets. We postpone the complete definition of the transitive closure for a particular case. Still, the following characterization of multisets of inversions for~$s$-decreasing trees gives a hint about it.

\begin{definition}[{\cite[Def. 1.5, Prop. 1.6]{CP22}}]\label{def:s_tree_inversion_sets}
	The multisets of inversions of~$s$-decreasing trees are exactly the multisets~$I$ that satisfy: \begin{itemize}
		\itemsep0em
		\item \defn{Transitivity}\index{$s$-weak order!$s$-decreasing tree!transitivity}: if~$a<b<c$, then~$|(b,a)_T|=0$ or~$|(c,a)_T|\geq |(c,b)_T|$.
		\item \defn{Planarity}\index{$s$-weak order!$s$-decreasing tree!planarity}: if~$a<b<c$, then~$|(b,a)_T|=s_b$ or~$|(c,b)_T|\geq |(c,a)_T|$.
	\end{itemize}

\end{definition}

\begin{definition}
	Let~$T$ be an~$s$-decreasing tree. An \defn{ascent}\index{$s$-weak order!$s$-decreasing tree!ascent} (resp.\ \defn{descent}\index{$s$-weak order!$s$-decreasing tree!descent}) of~$T$ is a pair~$(a,c)$ such that \begin{enumerate}
		\itemsep0em
		\item~$a\in T^c_{i}$ for some~$0\leq i < s_c$ (resp.~$0< i \leq s_c$),
		\item if~$a < b < c$ and~$a\in T^b_{i}$, then~$i=s_b$ (resp.~$i=0$),
		\item if~$s_a>0$, then~$T^a_{s_a}$ (resp.~$T^a_{0}$) is empty (i.e.\ only a leaf).
	\end{enumerate}
\end{definition}

Notice that in the case where~$s$ has no zeros, the ascents (resp.\ descents) of~$T$ are in bijection with the leaves of~$T$ that are rightmost (resp.\ leftmost) of their parent, excepting the rightmost (resp.\ leftmost) leaf of~$T$. The ascents and descents allow us to define the rotations on~$s$-decreasing trees, which characterize the cover relations of the~$s$-weak order.

\begin{definition}[{\cite{CP22}}]\label{def:s_tree_rotations}
	Let~$T$ be an~$s$-decreasing tree with an ascent~$(a,c)$. The \defn{rotation}\index{$s$-weak order!$s$-decreasing tree!rotation} of~$T$ along~$(a,c)$ is the~$s$-decreasing tree~$T+\{(a,c)\}$ corresponding to the transitive closure multiset of inversions obtained from~$\inv(T)$ after increasing~$|(a,c)|_T$ by 1. If~$A$ is a subset of ascents of~$T$, we denote by \defn{$T+A$} the~$s$-decreasing tree with inversion set~$\big(\inv(T)+A\big)^{tc}$. See Figure~\ref{fig:tree_rotation} for an example of a rotation.
\end{definition}

\begin{proposition}[{\cite[Thm.1.32]{CP22}}]\label{prop:cover_relation_s_weak_order}
	The cover relations of the~$s$-weak order correspond to rotations of~$s$-decreasing trees along ascents.
\end{proposition}

\begin{figure}[h!]
	\centering
	\includegraphics[scale=1.5]{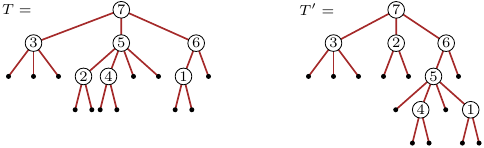}
	\caption[The rotation of a~$(1,1,2,1,3,1,2)$-decreasing tree along an ascent.]{ The rotation of the~$(1,1,2,1,3,1,2)$-decreasing tree~$T$ along the ascent~$(5,7)$ giving~$T'$.}
	\label{fig:tree_rotation}
\end{figure}

With rotations in hand we can define the analogue of the permutahedron for the~$s$-weak order as a combinatorial complex as follows.

\begin{definition}\label{def:s-permutahedron}
	The \defn{$s$-permutahedron}\index{$s$-permutahedron} is the combinatorial complex \defn{$\PSPerm$} with faces~$(T,A)$ where~$T$ is an~$s$-decreasing tree and~$A$ is a subset of ascents of~$T$. The face~$(T,A)$ is contained in~$(T', A')$ if and only if~$[T, T+A]\subseteq [T', A']$ as intervals in the~$s$-weak order.
	In particular, the vertices of~$\PSPerm$ are the~$s$-decreasing trees and the edges correspond to the~$s$-tree rotations.
\end{definition}

Definition~\ref{def:s-permutahedron} portrays~$\PSPerm$ as combinatorial complex. In the original article~\cite{CP19} the authors gave the following conjecture on the geometric structure of~$\PSPerm$.

\begin{conjecture}[{\cite[Conjecture 1]{CP19}}]\label{conj:s-permutahedron}
	Let~$s$ be a weak composition.~$\PSPerm$ can be realized as a polyhedral subdivision of a polytope which is combinatorially isomorphic to the zonotope~$\sum_{1\leq i < j \leq n} s_j (\mathbf{e_i}-\mathbf{e_j})$.
\end{conjecture}

In Chapter~\ref{chap:sorder_realizations} we give a positive answer to this conjecture for when~$s$ is a composition.

\subsection{Stirling \texorpdfstring{$s$}{}-Permutations}

From here onwards~$s=(s_1,\ldots,s_n)$ denotes a \defn{composition} (i.e.~$s_i>0$ for all~$i\in[n]$). This is required for us to study the~$s$-weak order with objects more akin to permutations that we now define.

\begin{definition}\label{def:s-permutation}
	Let~$s$ be a composition. A \defn{Stirling~$s$-permutation}\index{$s$-weak order!$s$-Stirling-permutation} is a permutation of the word~$1^{s_1}2^{s_2}\ldots n^{s_n}$ that avoids the pattern~$121$.
	We denote by~$\mathcal{W}_{s}$ the set of all Stirling~$s$-permutations.
\end{definition}

\begin{proposition}\label{prop:s_tree_to_s_permutation}
	Stirling~$s$-permutations are in bijection with the set of~$s$-decreasing trees by reading the labels of the nodes of an~$s$-decreasing tree in in-order. Moreover, this bijection induces a correspondence between leaves of an~$s$-decreasing tree and prefixes of the corresponding Stirling~$s$-permutation.
\end{proposition}

See Figure~\ref{fig:2112-decreasing-inordered-tree} for an example of the bijection between a Stirling~$(2,1,1,2)$-permutations and a~$(2,1,1,2)$-decreasing tree.

\begin{figure}[h!]
	\centering\includegraphics[scale=1.5]{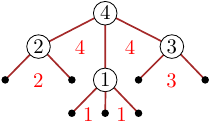}
	\caption[A~$(2,1,1,2)$-decreasing tree with vertices labeled via in-order.]{ A~$(2,1,1,2)$-decreasing tree with vertices labeled via in-order. The corresponding Stirling~$s$-permutation is~$w=241143$.}
	\label{fig:2112-decreasing-inordered-tree}
\end{figure}

\begin{remark}\label{rem:permutrations_from_s}
	Notice that in the case~$s=(k,\ldots,k)$, Stirling~$s$-permutations are exactly the Stirling~$k$-permutations of~\cite{JKP11}. Since decreasing trees can be naturally transformed into increasing trees, the bijection of Proposition~\ref{prop:s_tree_to_s_permutation} recovers the classical bijection of Gessel between~$(k+1)$-ary increasing trees and Stirling~$k$-permutations in this case (see~\cite[$\mathsection$5]{CG19},~\cite{JKP11}, and~\cite{GS78}).
\end{remark}

To describe the covering relations of the~$s$-weak order in terms of Stirling~$s$-permutations we need the following.

\begin{definition}\label{def:a_block}
	Let~$w$ be a Stirling~$s$-permutation.
	For~$a\in [n]$, the \defn{$a$-block~$B_a$}\index{$s$-weak order!$s$-Stirling-permutation!block} of~$w$ is the shortest substring~$u$ of~$w$ containing all~$s_a$ occurrences of~$a$.
\end{definition}

\begin{remark}\label{rem:blocks}
	Some quick facts about~$a$-blocks in a Stirling~$s$-permutation include: \begin{itemize}
		\itemsep0em
		\item An~$a$-block of~$w$ necessarily starts and ends with~$a$ and contains only letters in~$[a]$.
		\item For~$a<b$ we have that either~$B_a\subset B_b$ or~$B_a\cap B_b=\emptyset$. Otherwise, a partial intersection would imply that~$w$ contains the pattern~$121$.
		\item For~$a<c$,~$w$ contains the substring~$ac$ if and only if it is of the form~$w=u_1 B_a cu_2$ where~$u_1$ and~$u_2$ are words on~$[n]$.
	\end{itemize}

	Following Figure~\ref{fig:2112-decreasing-inordered-tree} we have that~$B_1=11$,~$B_2=2$,~$B_3=3$, and~$B_4=4114$.
\end{remark}

\begin{definition}\label{def:ascdesc}
	Let~$w$ be a Stirling~$s$-permutation.
	A pair~$(a,c)$ with~$1\leq a < c \leq n$ is an \defn{ascent}\index{$s$-weak order!$s$-Stirling-permutation!ascent} (resp.\ \defn{descent}\index{$s$-weak order!$s$-Stirling-permutation!descent}) of~$w$ if~$ac$ (resp.~$ca$) is a substring of~$w$.

	If~$w$ is of the form~$w=u_1B_acu_2$ where~$a<c$, the \defn{transposition} of~$w$ along the ascent~$(a,c)$ is the Stirling~$s$-permutation~$u_1cB_au_2$.
	We denote by \defn{$\inv(w)$} the multiset of inversions formed by pairs~$(c,a)$ with multiplicity \defn{$|(c,a)_w|$}$\in [0, s_c]$ the number of occurrences of~$c$ that precede the~$a$-block in~$w$.

	If~$A$ is a subset of ascents of~$w$, we denote by \defn{$w+A$} the Stirling~$s$-permutation with inversion set~$\big(\inv(w)+A\big)^{tc}$.
\end{definition}

\begin{lemma}\label{lem:ascdesc}
	Let~$w$ be a Stirling~$s$-permutation,~$T(w)$ its corresponding~$s$-decreasing tree and~$1\leq a<c\leq n$.
	\begin{enumerate}
		\itemsep0em
		\item The pair~$(a,c)$ is an ascent (resp.\ descent) of~$T(w)$ if and only if it is an ascent (resp.\ descent) of~$w$.
		\item~$|(c,a)_{T(w)}|=|(c,a)_w|$.
	\end{enumerate}
	Moreover, suppose~$(a,c)$ is an ascent of~$T=T(w)$ so that~$w$ is the of the form~$w=u_1B_acu_2$.
	Then~$T'$ is the~$s$-tree rotation of~$T$ along~$(a,c)$ if and only if~$T'=T(w')$ where~$w'=u_1 cB_au_2$.
\end{lemma}
\begin{proof}
	The proofs follow easily from the definitions and Proposition~\ref{prop:s_tree_to_s_permutation}.
\end{proof}

\begin{corollary}\label{prop:cover_relations_multiperm}
	Let~$w$ and~$w'$ be Stirling~$s$-permutations.
	Then~$w'$ covers~$w$ in the~$s$-weak order if and only if~$w'$ is the transposition of~$w$ along an ascent.
\end{corollary}

The analogue of Figure~\ref{fig:strees122} showing the~$s$-weak order on Stirling~$s$-permutations can be found in Figure~\ref{fig:sperms122}.

\begin{figure}[h!]
	\centering
	\includegraphics[scale=0.6]{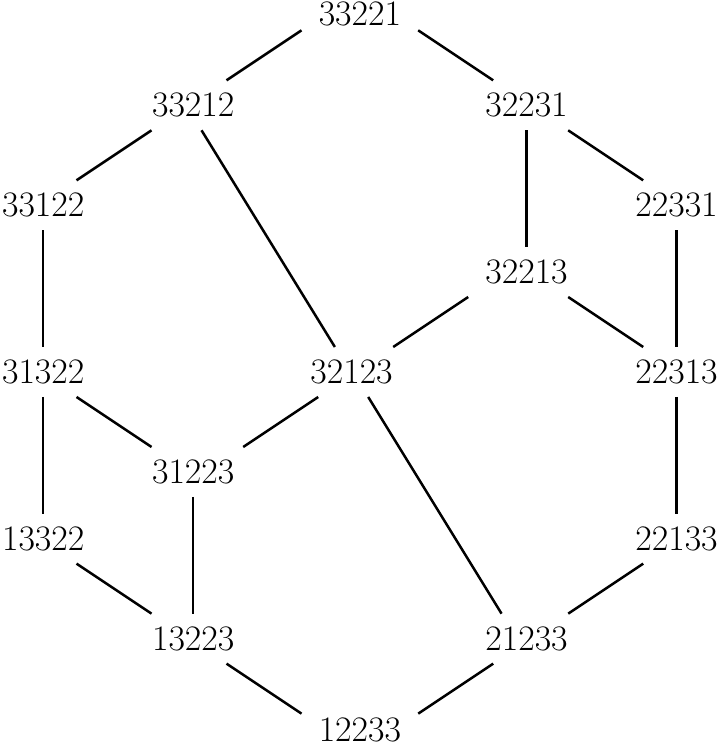}
	\caption{The~$(1,2,2)$-weak order on Stirling~$(1,2,2)$-permutations.}
	\label{fig:sperms122}
\end{figure}

\begin{example}\label{eg:s-perm}
	Let~$s=(1,1,2,1,3,1,2)$ and consider the~$s$-permutation~$w=33725455716$.
	The transposition of~$w$ along the ascent~$(5, 7)$ switches the~$5$-block~$B_5=5455$ of~$w$ with the~$7$ immediately after it and gives~$w'=3372\textcolor{red}{7}\textcolor{blue}{5455}16$. The corresponding rotation in terms of~$s$-decreasing trees is shown in Figure~\ref{fig:tree_rotation}.
\end{example}

This translation of the~$s$-weak order on~$s$-decreasing trees to Stirling~$s$-permutations allows us to consider the~$s$-permutahedron~$\PSPerm$ as the combinatorial complex with faces~$(w,A)$ where~$w$ is a Stirling~$s$-permutation and~$A$ is a subset of ascents of~$w$. It is with this definition that we study Conjecture~\ref{conj:s-permutahedron} in Chapter~\ref{chap:sorder_realizations}.

We finish this section by stating some tools on the transitivity of multisets which are of help to us when working with the faces of the~$s$-permutahedron. The following comes mostly from~\cite{GMPTY23}.

\begin{definition}[{\cite[Def.1.14]{CP22}}]\label{def:multi_transitive_closures}
	Let~$I$ be a multiset of inversions. A \defn{transitivity path} between two values~$c>a$ is a list of values~$a=b_1<\cdots<b_k=c$ such that~$(b_{i+1},b_i)_I>0$ for all~$i\in[k-1]$. In this way the transitive closure~$I^{tc}$ of~$I$ is the multiset of inversions with cardinalities \begin{equation*}
		|(c,a)_{I^{tc}}|:=\max\big(|(c,b_{k-1})_I|\,:\, a=b_1<\cdots<b_k=c \text{ is a transitivity path of } (c,a).\big)
	\end{equation*}
\end{definition}

\begin{definition}[{\cite{GMPTY23}}]\label{def:chain}
	Let~$w$ be a Stirling~$s$-permutation,~$A$ a subset of ascents of~$w$ and~$1\leq a < c\leq n$ such that~$|(c,a)_w|<s_c$. We say that the pair~$(a,c)$ is \defn{$A$-dependent}\index{$s$-weak order!$s$-Stirling-permutation!A-dependent} in~$w$ if there is a sequence~$a\leq b_1 < \cdots < b_k<b_{k+1}=c$ such that:
	\begin{itemize}
		\itemsep0em
		\item $b_1$ is the greatest letter strictly smaller than~$c$ such that~$B_a\subseteq B_{b_1}$,
		\item for all~$i\in[k-1]$, the~$b_{i}$-block~$B_{b_i}$ is directly followed by~$B_{b_{i+1}}$,
		\item $B_{b_k}$ is directly followed by an occurrence of~$c$,
		\item $(b_i,b_{i})\in A$ for all~$i\in[k]$.
	\end{itemize}
\end{definition}

\begin{example}\label{ex:A_dependenacy}
	Let~$w$ be a Stirling~$s$-permutation and~$A$ a subset of ascents of~$w$. \begin{itemize}
		\itemsep0em
		\item If~$(a,c)\in A$, then~$(a,c)$ is~$A$-dependent taking~$b_1=a$ and~$b_2=c$.
		\item Let~$w=33725455716$ and~$A=\{(2,5), (5,7), (1,6)\}$. The pair~$(2,7)$ is~$A$-dependent through the sequence~$b_1=2$,~$b_2=5$ and~$b_3=7$. Meanwhile,~$(2,6)$ is not~$A$-dependent as the only block containing~$B_2$ is~$B_7$ but~$7>6$ and~$B_7$ is followed by~$B_1$ before the next occurrence of~$6$.
	\end{itemize}
\end{example}

\begin{theorem}[{\cite[Prop.3.12]{GMPTY23}}]\label{thm:transitiveclosure}
	Let~$w$ be a Stirling~$s$-permutation and~$A$ a subset of its ascents. Then \begin{equation*}
		\inv(w+A)=\begin{cases}
			|(c,a)_w|+1 & \text{if~$(a,c)$ is~$A$-dependent in }w \\
			|(c,a)_w|   & \text{otherwise.}
		\end{cases}
	\end{equation*}
\end{theorem}

\begin{example}
	Let~$w=33725455716$ and~$A=\{(2,5), (5,7), (1,6)\}$. Augmenting these inversions gives the Stirling~$s$-permutation~$w+A=33775245561$. The pairs whose multiplicity in the multiset of inversion has been increased by~$1$ are~$\{(5,2), (6,1), (7,2), (7,4), (7,5)\}$.
\end{example}

\begin{proof}
	Let be~$I$ the multiset of inversions defined as \begin{equation*}
		|(c,a)_I|:=\begin{cases}
			|(c,a)_w|+1 & \text{if~$(a,c)$ is~$A$-dependent in }w \\
			|(c,a)_w|   & \text{otherwise.}
		\end{cases}
	\end{equation*}

	Notice that for an~$A$-dependent pair~$(a,c)$ and an integer~$d>c$ we have that~$|(d,c)_I|=|(d,a)_I|$. Moreover,~$(a,d)$ is~$A$-dependent if and only if~$(c,d)$ is also~$A$-dependent.	We need to verify that~$I$ is transitive and that it is the smallest transitive multiset of inversions containing~$\inv(w)+A$. We begin with the latter statement.

	First, it is clear that~$\inv(w)+A\subset I$ from the definition of~$I$ as every pair in~$A$ is~$A$-dependent. Now let us show that any transitive multiset of inversions~$I'$ that contains~$\inv(w)+A$ necessarily contains~$I$. Since~$\inv(w)+A\subseteq I'$ is clear that for any pair we have~$|(c,a)_{I'}|\geq |(c,a)_w|$.
	Let~$(a,c)$ be an~$A$ dependent pair in~$w$ through a sequence~$a\leq b_1< \cdots <  b_k< b_{k+1}=c$. We proceed by induction on the length~$k$ of the sequence. For~$k=1$ we have that
	\begin{itemize}
		\item either~$b_1=a$ and~$(a,c)\in A$. Giving us that~$|(c,a)_{I'}|\geq |(c,a)_w|+1$,
		\item or~$a<b_1<c$. In this case~$|(b_1,a)_{I'}|\geq |(b_1,a)_w|>0$ by since Definition~\ref{def:a_block}~$B_a\subset B_{b_1}$. Now we get that \begin{equation}\label{eq:transitivity+A+increasing_pair}
			      |(c,a)_{I'}|\geq|(c,b_1)_{I'}|\geq |(c,b_1)_w|+1=|(c,a)_w|+1
		      \end{equation}
		      where the first inequality comes from transitivity, the second from the previous case as~$(b_1,c)\in A$ and the last equality from~$(a,b_1)$ being~$A$-dependent.
	\end{itemize}
	Suppose that~$k>1$. The induction hypothesis tells us that~$|(b_k,a)_{I'}|\geq |(b_k,a)_w|+1 > 0$. Applying transitivity to~$a<b_k<c$ and using that~$(b_k,c)\in A$ and that~$(a,b_k)$ is~$A$-dependent (such that there is no occurrence of~$c$ between~$a$ and~$b_k$) gives us the inequalities~$|(c,a)_{I'}|\geq |(c,b_k)|_{I'}\geq |(c,b_k)_{w}|+1=|(c,a)_w|+1$. Thus,~$I\subseteq I'$.

	Lastly let us see that~$I$ is transitive in the following cascade of cases. Let~$1\leq a < b < c \leq n$.

	\begin{itemize}
		\itemsep0em
		\item If~$|(b,a)_I|=0$ then there is nothing to prove for this pair.
		\item If~$|(b,a)_I|>0$ we need to see that~$|(c,a)_I|\ge |(c,b)_I|$. \begin{itemize}
			      \itemsep0em
			      \item If~$|(b,a)_w|=0$, then~$(a,b)$ is~$A$-dependant and~$|(c,a)_I|=|(c,b)_I|$ by our observation at the start of the proof.
			      \item Suppose that~$|(b,a)_w|>0$. \begin{itemize}
				            \itemsep0em
				            \item If~$|(c,b)_I|=|(c,b)_w|$, then by the inclusion~$\inv(w)\subset I$ and the transitivity of~$\inv(w)$ we have that~$|(c,a)_I|\geq |(c,a)_w|\geq |(c,b)_w|=|(c,b)_I|$.
				            \item Suppose that~$|(c,b)_I|=|(c,b)_w|+1$ (i.e.~$(b,c)$ is~$A$-dependent). If~$|(c,a)_w|\geq |(c,b)_w|+1$, we have~$|(c,a)_I|\geq |(c,a)_w|\geq |(c,b)_w|+1 =|(c,b)_I|$. Otherwise, we have~$|(c,a)_w|=|(c,b)_w|=:i$. From our assumption that~$|(b,a)_w|>0$ we know that the~$a$-block~$B_a$ appears in~$w$ between the first occurrence of~$b$ and the~$i$-th occurrence of~$c$. Thus,~$(b,c)$ being~$A$-dependant implies that~$(a,c)$ is also~$A$-dependant as the sequence of~$(a,c)$ is included in the sequence of~$(b,c)$ in~$w$. These two~$A$-dependencies together with the transitivity of~$w$ for~$a<b<c$ concludes that~$|(c,a)|_I=|(c,a)|_w+1\geq|(c,b)|_w+1=|(c,b)|_I$ as wished. \qedhere
			            \end{itemize}
		      \end{itemize}
	\end{itemize}
\end{proof}

%% file: includes/contenu/chap_permutrees_vectors.tex

\chapter{Inversion and Cubic Vectors for Permutrees}\label{chap:permutree_vectors}

\addcontentsline{lof}{part}{\protect\numberline{\thepart}Permutrees}
\addcontentsline{lot}{part}{\protect\numberline{\thepart}Permutrees}

\addcontentsline{lof}{chapter}{\protect\numberline{\thechapter}Inversion and Cubic Vectors for Permutrees}

In this chapter we present two generalizations of the bracket vectors of binary trees for permutrees based on~\cite{T23}. The first generalization which we call the inversion vectors helps prove in a constructive manner the lattice property for~$\delta$-permutree rotation posets. The second generalization called the cubic vectors allows for the construction of an embedding of these lattices onto a cube.

\section{Inversion Vectors}

We begin defining inversion vectors for~$\delta$-permutrees with the aim of proving the lattice property of~$\delta$-permutrees rotation posets (Proposition~\ref{prop:permutree_lattice_property}) in a constructive manner.

Recall that~$j\to i$ if~$v_j$ is a descendant of~$v_i$.

\begin{definition}\label{def:inversion_vector_permutrees}
	Consider~$T\in\cPT_n(\delta)$ to be a~$\delta$-permutree. Its \defn{inversion set}\index{permutree!inversion!set} and \defn{inversion components}\index{permutree!inversion!components} are \begin{equation*}
		\begin{split}
			B(T)&:=\{(i,j)\,:\, i<j  \text{ and } j\to i \},\\
			{B(T)}_i&:=\{j\in [n] \,:\, (i,j)\in B(T)\}.
		\end{split}
	\end{equation*} That is, all $j>i$ such that~$v_j$ is a descendant of~$v_i$. An inversion set has an associated \defn{inversion vector}\index{permutree!inversion!vector}~$\vec{b}(T)=(b_1,\ldots,b_{n-1})$ such that~$b_i=|{B(T)}_i|$.
\end{definition}

Since~$RD_{n}=\emptyset$, its component does not alter the combinatorics of inversion sets and thus, we do not consider it. Figure~\ref{fig:permutree-inversion-vector-IXYI} contains the inversion vectors for all~$\nonee\uppdownn\upp\nonee$-permutrees.

\begin{figure}[h!]
	\centering
	\includegraphics[scale=0.6]{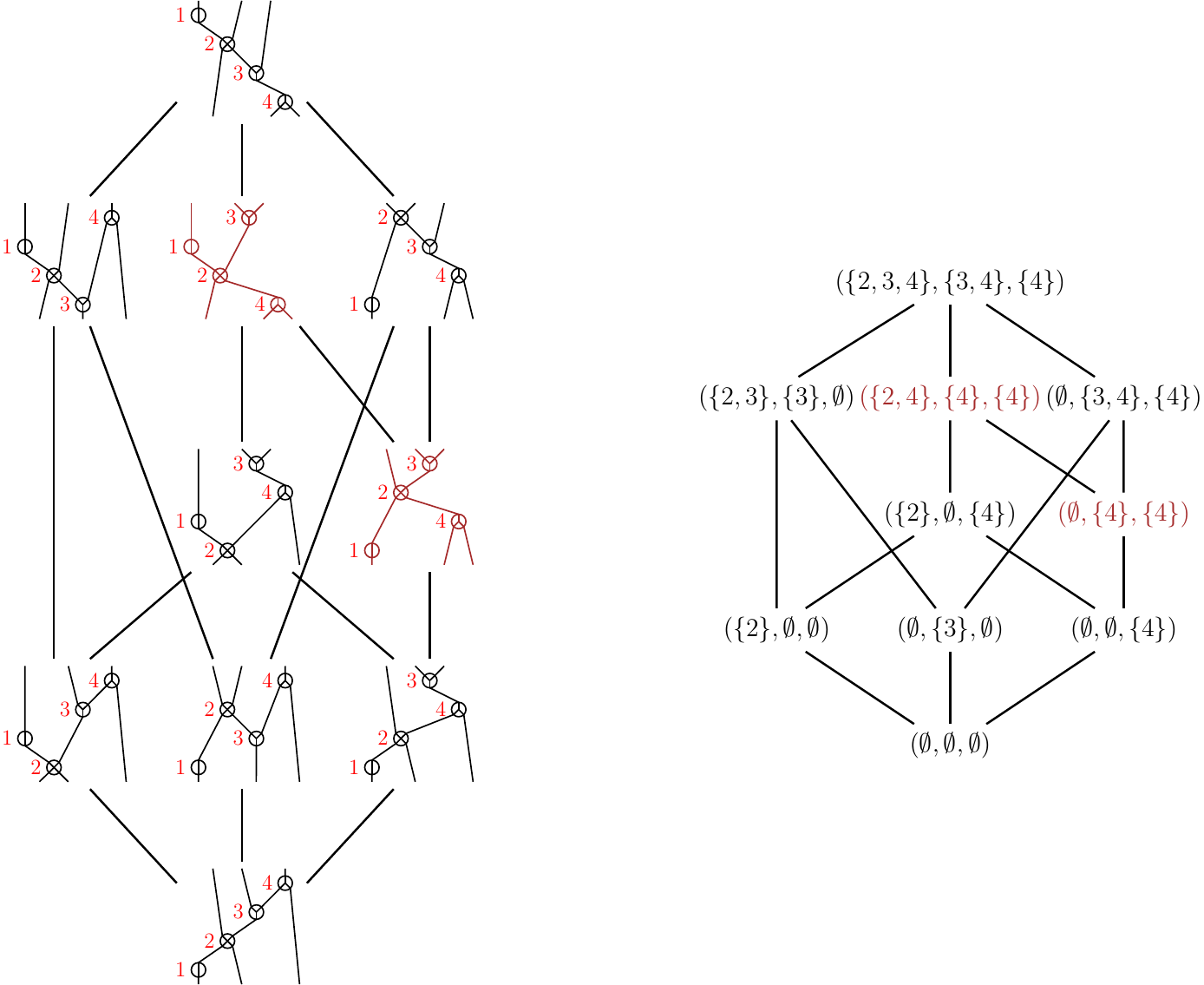}
	\caption[The rotation lattice of~$\nonee\uppdownn\upp\downn$-permutrees with their respective inversion sets.]{ The rotation lattice of~$\nonee\uppdownn\upp\downn$-permutrees (left) with their respective inversion sets represented via their components (right). The elements in brown correspond to the permutrees that are not extremal.}\label{fig:permutree-inversion-vector-IXYI}
\end{figure}

\begin{example}
	Let~$\hat{1},\,\hat{0},\, T_l,$ and~$T_r$ respectively be the top, bottom, middle left, and middle right elements of the lattice of~$\nonee\uppdownn\upp\nonee$-permutrees as in Figure~\ref{fig:permutree-inversion-vector-IXYI}. Then \begin{equation*}
		\begin{split}
			B(\hat{1})&=\{(1,2),(1,3),(1,4),(2,3),(2,4),(3,4)\},\\
			B(\hat{0})&=\emptyset,\\
			B(T_r)&=\{(2,4),(3,4)\},\\
			B(T_l)&=\{(1,2),(3,4)\}.
		\end{split}
	\end{equation*}
\end{example}

\begin{lemma}\label{lem:permutree_inversion_sets}
	Let~$1\leq i<j<k\leq n$. The set of inversion sets~$\{B(T)\,:\, T\in\cPT_n(\delta)\}$ is the set of all subsets~$E\subseteq \{(i,j)\in[n]^2\,:\,1\leq i<j\leq n\}$ such that \begin{enumerate}
		\itemsep0em
		\item $E$ is transitive,
		\item $E$ is cotransitive,
		\item if~$\delta_j\in\{\downn,\uppdownn\}$,~$(i,j)\notin E$, and~$(j,k)\in E$, then~$(i,k)\notin E$,
		\item if~$\delta_j\in\{\upp,\uppdownn\}$,~$(i,j)\in E$, and~$(j,k)\notin E$, then~$(i,k)\notin E$.
	\end{enumerate}
\end{lemma}

\begin{proof}
	Let~$T$ be a~$\delta$-permutree and~$E:=B(T)$. If~$(i,j),(j,k)\in E$, then~$j\rightarrow i,\,k\rightarrow j$, evidently~$k\rightarrow i$. That is,~$(i,k)\in E$ and~$E$ is transitive. The fact that~$E$ is cotransitive follows a similar argument. For property 3, the facts that~$(i,j)\notin E$ and~$(j,k)\in E$ respectively mean that~$v_i\in LD_j$ and~$v_k\in RD_j$. Thus,~$v_k$ is not a child of~$v_i$ and~$(i,k)\notin E$. Property 4 follows a similar argument.

	For the opposite direction we wish to construct a~$\delta$-permutree~$T(E)$ in accordance with the elements in~$E$. Let~$E_i=\{j\in [n]\,:\, (i,j)\in E\}$ be the components of~$E$. Notice that~$E_n=\emptyset$ and that if~$(i,j)\in E_i$, and~$(j,k)\in E_j$ then~$(i,k)\in E_i$ due to~$E$ being transitive. With this in mind, we can construct~$T(E)$ in the following way. Take an~$n\times n$ grid. As step~$0$, place vertex~$v_n$ anywhere in the last column. Now for step~$i$, place the vertex~$v_{n-i}$ in the~$n-i$-th column and at the height such that it is above (resp.\ below) all~$j$ such that~$(n-i,j)\in E$ (resp.~$(n-i,j)\notin E$). If such height was already used by another vertex, move the placed vertices up or down as required maintaining the relative others established in the previous steps. After step~$n-1$ we get a permutation table. Decorate each vertex~$v_i$ with the decoration~$\delta_i$. Following the insertion algorithm (see Definition~\ref{def:decorated_permutation}) we obtain a~$\delta$-permutree~$T(E)$. Notice that the placement of the vertices in the grid ensures transitivity and cotransitvity and that the red walls from the decorations in the insertion algorithm accomplish properties 3 and 4. See Figure~\ref{fig:permutree_inversion_set_insertion_algorithm} for an example.
\end{proof}

\begin{figure}[h!]
	\centering
	\includegraphics[scale=0.7]{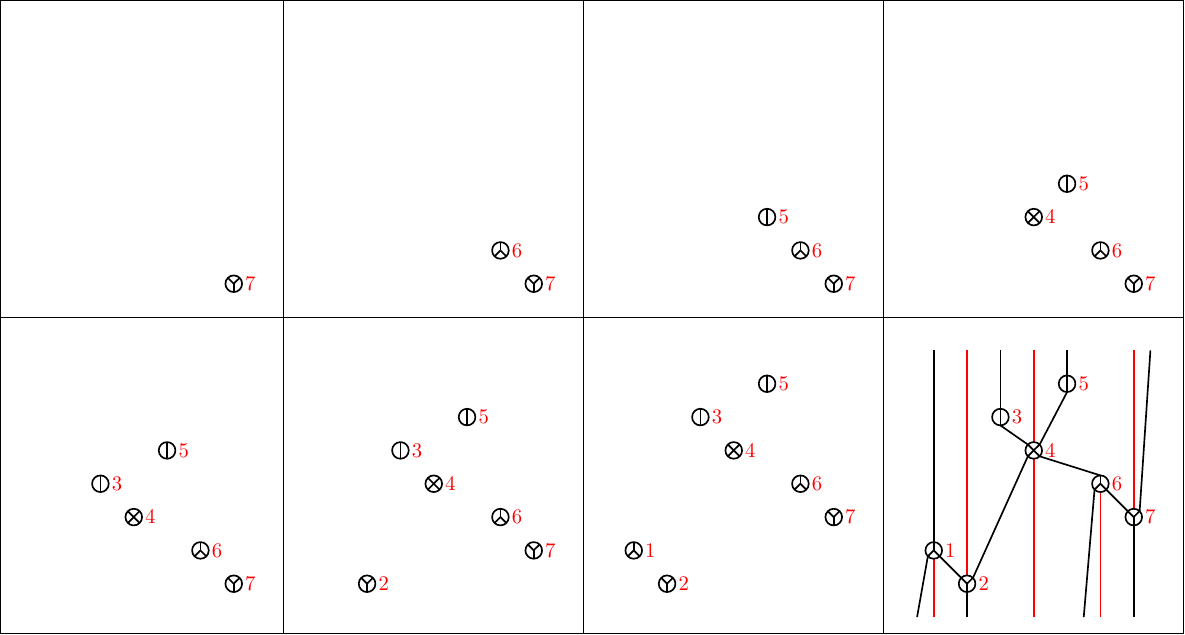}
	\caption[The construction of a~$\downn\upp\nonee\uppdownn\nonee\downn\downn$-permutree from an inversion vector.]{ The construction of the~$\downn\upp\nonee\uppdownn\nonee\downn\upp$-permutree corresponding to the inversion set~$\{(1,2),(3,4),(3,6),(3,7),(4,6),(4,7),(5,6),(5,7),(6,7)\}$.}\label{fig:permutree_inversion_set_insertion_algorithm}
\end{figure}

\begin{remark}\label{rem:inversion_set_inclusion}
	In the case of cover relations~$T\lessdot T'$ between~$\delta$-permutrees, that is, an~$ij$-edge rotation from~$T$ to~$T'$, Definition~\ref{def:permutree_rotations} tells us that such a rotation only affects the edge between~$v_i$ and~$v_j$ while the rest of the tree remains the same. In terms of inversions this means that the rotation only turns into inversions the pairs of the form~$(i,x)$ where~$x\in RD_j(T)$ and if~$\delta_i\in\{\upp,\uppdownn\}$, also all the pairs that depend on these in a transitive manner. That is,~$B(T')=(B(T)\cup \{(i,j)\})^{tc}$ no matter the decoration~$\delta$.
\end{remark}

\begin{remark}\label{rem:recovering_CPP_1}
	The characterization of inversion sets was already given in~\cite[Section 2.3.2]{CPP19} where they are called \emph{IPIP} (\emph{PIP} meaning permutree interval poset). With Lemma~\ref{lem:permutree_inversion_sets}, not only we have characterized inversion sets for permutrees but also described how to recover the permutree given its inversion set which is not done in~\cite{CPP19}.
\end{remark}

To use inversion sets as a tool we need first to show that we can describe the~$\delta$-permutree rotation order via their containment.

\begin{lemma}\label{lem:inversion_set_inclusion}
	Let~$T,T'$ be two~$\delta$-permutrees. Then~$T<T'$ if and only if~$B(T)\subset B(T')$.
\end{lemma}

\begin{proof}
	Suppose that~$T<T'$. Since the rotation order on permutrees is the transitive closure of the covering relations given by the rotations in Figure~\ref{fig:permutree_rotations}, it is enough to prove this in the case that~$T'$ covers~$T$. Remark~\ref{rem:inversion_set_inclusion} tells us that in such a case~$B(T')=(B(T)\cup \{(i,j)\})^{tc}$ and thus~$B(T)\subset B(T')$.

	Before moving to the other direction let~$\hat{0}$ be the minimal~$\delta$-permutree in the rotation lattice. The fact that~$T\lessdot T'$ implies that~$B(T)={(B(T)\cup\{(i,j)\})}^{tc}$ tells us that for any chain~$\hat{0}=T_0\lessdot T_1\lessdot\cdots\lessdot T_{l}\lessdot T_{l+1}=T$ in the interval~$[\hat{0},T]$, we have that a \defn{sequence} of inversions~$(i_x,j_x)$ such that~$B(T_x)={(B(T_{x-1})\cup\{(i_x,j_x)\})}^{tc}$ for all~$x\in[l]$. In this way we say that a sequence~$(i_1,j_1),\ldots,(i_l,j_l)$ generates~$T$. It is easy to see that~$T<T'$ if and only if for every sequence~$(i_1,j_1),\ldots,(i_l,j_l)$ that generates~$T$ there exists a sequence that generates~$T'$ of the form~$(i_1,j_1),\ldots,(i_l,j_l),(i_{l+1},j_{l+1}),\ldots,(i_{l'},j_{l'})$. Take notice that the length of the chains in an interval of permutrees might not always be the same.

	Now suppose that~$B(T)\subset B(T')$ and let~$(i_1,j_1),\ldots,(i_l,j_l)$ be a sequence of~$T$. We claim that for all~$x\in\{0,\ldots,l\}$ the sequence~$(i_1,j_1),\ldots,(i_x,j_x)$ is the start of sequence of~$T'$. Let~$T_x$ correspond to the~$\delta$-permutree corresponding to the claim corresponding to~$x$. Notice that the claim for~$x=l$ amounts to proving~$T<T'$. We proceed by induction on the length of the chain which is given by~$x$. If~$x=0$ then~$T_0=\hat{0}$ and the claim is trivial. Now suppose that the claim holds for~$x-1$ and~$(i_1,j_1),\ldots,(i_{x-1},j_{x-1})$ is the start of a sequence of~$T'$, that is,~$T_{x-1}<T$. Suppose as well that the~$T_{x}\nless T'$. Since~$T_x< T'$, this means that with the~$i_xj_j$-edge rotation,~$T_{x}$ obtained an inversion that~$T'$ does not possess. This is a contradiction since~$B(T_{x})={(B(T_{x-1})\cup\{(i_x,j_x)\})}^{tc}\subset B(T)\subset B(T)$ as all sets~$B$ are transitive. Thus,~$T_x<T$ for all~$x\in\{0,\ldots,l\}$ and~$T<T'$.
\end{proof}

\begin{theorem}\label{thm:permutree_meet}
	Given two~$\delta$-permutrees~$T,T'$ on~$n$ vertices, there exists a~$\delta$-permutree~$T\wedge T'$ under the~$\delta$-permutree rotation order. Moreover, it satisfies \begin{equation}\label{eq:inversion_set_meet}
		B(T\wedge T')=B(T)\cap B(T')\cap \{(i,j)\,:\,\forall\,i<l<j,\, (i,l) \text{ or } (l,j)\in B(T)\cap B(T')\}.
	\end{equation}
\end{theorem}

\begin{example}\label{ex:permutree_meet}
	Figure~\ref{fig:permutree_meet_example} presents the meet operation between~$\delta$-permutrees. Taking the last set of the right-hand side of Equation (\ref{eq:inversion_set_meet}) as~$I$, we have that the corresponding inversion sets in this case are \begin{equation*}
		\begin{split}
			B(T)&=           \{(2,3),(2,4),(2,5),(3,4)\}                         \\
			B(T')&=          \{(1,2),(1,4),(1,5),(2,4),(2,5),(3,4),(3,5),(4,5)\} \\
			I&=              \{(1,2),(1,3),(1,4),(2,3),(2,4),(3,4),(3,5),(4,5)\} \\
			B(T\wedge T')&=  \{(2,4),(3,4)\}
		\end{split}
	\end{equation*} and satisfy Theorem~\ref{thm:permutree_meet}. Translating this into the components of the inversion sets we have that \begin{multicols}{2}
		\noindent
		\begin{equation*}
			\begin{split}
				\emptyset\cap\{2,4,5\}\cap\{2,3\}&=\emptyset\\
				\{3,4,5\}\cap\{4,5\}\cap\{3,4\}&=\{4\}
			\end{split}
		\end{equation*}
		\begin{equation*}
			\begin{split}
				\{4\}\cap\{4,5\}\cap\{4,5\}&=\{4\}\\
				\emptyset\cap\{5\}\cap\{5\}&=\emptyset\\
			\end{split}
		\end{equation*}
	\end{multicols}
\end{example}

\begin{figure}[h!]
	\centering
	\includegraphics[scale=1]{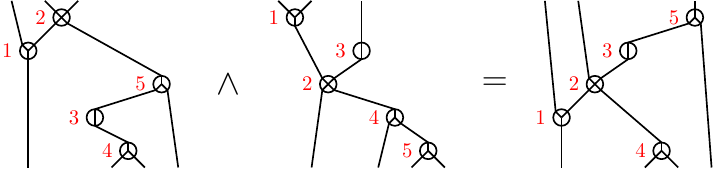}
	\caption{The meet of two~$\upp\uppdownn\nonee\downn\downn$-permutrees.}\label{fig:permutree_meet_example}
\end{figure}

\begin{remark}\label{rem:recovering_binary_meet_from_permutrees}
	Notice that inversion sets of~$\{\downn\}^n$-permutrees are bracket sets of binary trees as in Definition~\ref{def:bracket_vector}. We recover Proposition~\ref{prop:tamari_lattice_meet} whenever~$\delta\in\{\downn\}^n$ as~$B(T)\cap B(T')$ is contained in the last set of Equation (\ref{eq:inversion_set_meet}). To see this, notice that if~$(i,j)\in B(T)$ then~$v_l\in RD_i$ for all~$i<l<j$. If not, it would contradict that~$\delta_j=\downn$. Thus,~$(i,l)\in B(T)$. The same argument applies if~$(i,j)\in B(T')$ and thus in this case~$(i,j)\in B(T)\cap B(T')$ implies~$(i,l)\in B(T)\cap B(T')$ giving us the desired inclusion.
\end{remark}

\begin{remark}\label{rem:recovering_CPP_2}
	Our meet operation is similar to the \emph{tdd} operation defined in~\cite[Section 1.2.2]{CPP19} on integer posets. In this context, Theorem~\ref{thm:permutree_meet} can be seen as~\cite[Corollary 2.39]{CPP19} through the \emph{tdd} and our Lemma~\ref{lem:permutree_inversion_sets}. Nevertheless, we give here a direct proof without relying on the results of~\cite{CPP19}.
\end{remark}

We proceed to prove Theorem~\ref{thm:permutree_meet}.

\begin{proof}
	First let us see that~$B(T\wedge T')$ satisfies the conditions of Lemma~\ref{lem:permutree_inversion_sets} and thus defines~$T\wedge T'$ as a~$\delta$-permutree.

	Assume that~$(i,j),(j,k)\in B(T\wedge T')$. As~$(i,j),(j,k)\in B(T)\cap B(T')$, the transitivity of these sets tells us that~$(i,k)\in B(T)\cap B(T')$. We just need to see that~$(i,k)$ is also in the last set of Equation~\ref{eq:inversion_set_meet}. Let~$l$ such that~$i<l<k$. If~$i<l<j$, then as~$(i,j)\in B(T\wedge T')$ we have that either~$(i,l)\in B(T)\cap B(T')$ or~$(l,j)\in B(T)\cap B(T')$. In the former case we are done. In the latter, as~$(j,k)\in B(T)\cap B(T')$, using transitivity we get that~$(l,k)\in B(T)\cap B(T')$, and we are done. If instead~$l=j$, we immediately finish as by assumption~$(i,j),(j,k)\in B(T)\cap B(T')$. Finally, suppose that~$j<l<k$. In this case since~$(j,k)\in B(T\wedge T')$, either~$(j,l)\in B(T)\cap B(T')$ or~$(l,k)\in B(T)\cap B(T')$. In the latter case we finish. For the former, as~$(i,j)\in B(T)\cap B(T')$, using transitivity we get that~$(i,l)\in B(T)\cap B(T')$. Thus, we conclude that~$B(T\wedge T')$ is transitive.

	To see that~$B(T\wedge T')$ is cotransitive notice that its complement is the transitive closure of~$B(T)$ and~$B(T')$. That is,~$B(T\wedge T')^c=(B(T)^c\cup B(T')^c)^{tc}$. By definition of transitive closure it is immediate that~$B(T\wedge T')$ is cotransitive.

	Now suppose that~$\delta_j\in\{\downn,\uppdownn\}$,~$(i,j)\notin B(T\wedge T')$, and~$(j,k)\in B(T\wedge T')$. The last assumption tells us that~$(j,k)\in B(T)\cap B(T')$ and for all~$l$ such that~$j<l<k$, either~$(j,l)\in B(T)\cap B(T')$ or~$(l,k)\in B(T)\cap B(T')$. On the other hand, that~$(i,j)\notin B(T\wedge T')$ means that either~$(i,j)\notin B(T),\, (i,j)\notin B(T')$, or there exists~$i<l^*<j$ such that~$(i,l^*),(l^*,j)\notin B(T)\cap B(T')$. If either~$(i,j)\notin B(T)$ or~$(i,j)\notin B(T')$, then because of Property 3 of Lemma~\ref{lem:permutree_inversion_sets} and the fact that~$(j,k)\in B(T)\cap B(T')$ we have that~$(i,k)\notin B(T)$ and~$(i,k)\notin B(T')$ respectively. That is,~$(i,k)\notin B(T\wedge T')$ and we are done in this case.

	Consider then that~$(i,j)\in B(T)\cap B(T')$ and there exists~$i<l^*<j$ such that~$(i,l^*),(l^*,j)\notin B(T)\cap B(T')$. For contradiction’s sake suppose that~$(i,j)\in B(T\wedge T')$. By definition of~$B(T\wedge T')$ this means that either~$(i,l^*)\in B(T)\cap B(T')$ or~$(l^*,k)\in B(T)\cap B(T')$. The former case is a contradiction with the condition on which~$l^*$ exists, thus either~$(i,l^*)\in B(T)$ or~$(i,l^*)\in B(T')$ and the latter case happens. Without loss of generality suppose that~$(i,l^*)\in B(T)$. As~$(l^*,k)\in B(T)\cap B(T')\subset B(T)$, Property 3 of Lemma~\ref{lem:permutree_inversion_sets} tells us that~$(i,k)\notin B(T)$ and thus~$(i,k)\notin B(T\wedge T')$ as we wanted. The final Property of Lemma~\ref{lem:permutree_inversion_sets} follows a similar proof, and thus we omit it. We conclude that~$B(T\wedge T')$ indeed corresponds to a permutree~$T\wedge T'$.

	Let us now see that~$T\wedge T'$ is in fact the meet of~$T$ and~$T'$. Since~$B(T\wedge T')\subset B(T)$ and~$B(T\wedge T')\subset B(T')$ Lemma~\ref{lem:inversion_set_inclusion} tells us that~$T\wedge T'<T$ and~$T\wedge T'<T'$. Now suppose that there is a~$\delta$-permutree~$S$ such that~$S<T$ and~$S<T'$. We claim that~$S\leq T\wedge T$. Because of Lemma~\ref{lem:inversion_set_inclusion} we know that~$B(S)\subset B(T)\cap B(T')$. Let~$(i,j)\in B(S)\subset B(T)\cap B(T')$ and~$i<l<j$. Notice that if both elements~$(i,l),(l,j)\notin B(S)$, then~$(i,j)\notin B(S)$ as it is cotransitive, and we would have a contradiction. Without loss of generality suppose~$(i,l)\in B(S)$. As~$B(S)\subset B(T)\cap B(T')$, we have that~$(i,l)\in B(T)\cap B(T')$. Thus, for all~$i<l<j$ either~$(i,l)\in B(T)\cap B(T')$ or~$(l,j)\in B(T)\cap B(T')$. Meaning that,~$(i,j)\in B(T\wedge T')$ and we conclude that~$B(S)\subseteq B(T\wedge T')$. By Lemma~\ref{lem:inversion_set_inclusion} we get that~$S\leq T\wedge T'$ as we wished.
\end{proof}

\begin{corollary}\label{cor:permutrees_are_lattices}
	$\mathcal{PT}(\delta)$ is a lattice for any decoration~$\delta\in\{\nonee,\downn,\upp,\uppdownn\}^n$.
\end{corollary}

\begin{proof}
	The~$\delta$-tree rotation poset has a meet thanks to Theorem~\ref{thm:permutree_meet}. Since it is a bounded poset, Proposition~\ref{prop:semilattice_to_lattice} tell us that it is a lattice.
\end{proof}

\section{Cubic Vectors}

Having inversion vectors in hand, the reader might ask if it is the case that inversion vectors also give a cubic embedding of~$\delta$-permutree lattices. This is not the case as can be seen in Figure~\ref{fig:permutree_inversion_cubic_vectors}.

\begin{figure}[h!]
	\centering
	\includegraphics[scale=0.8]{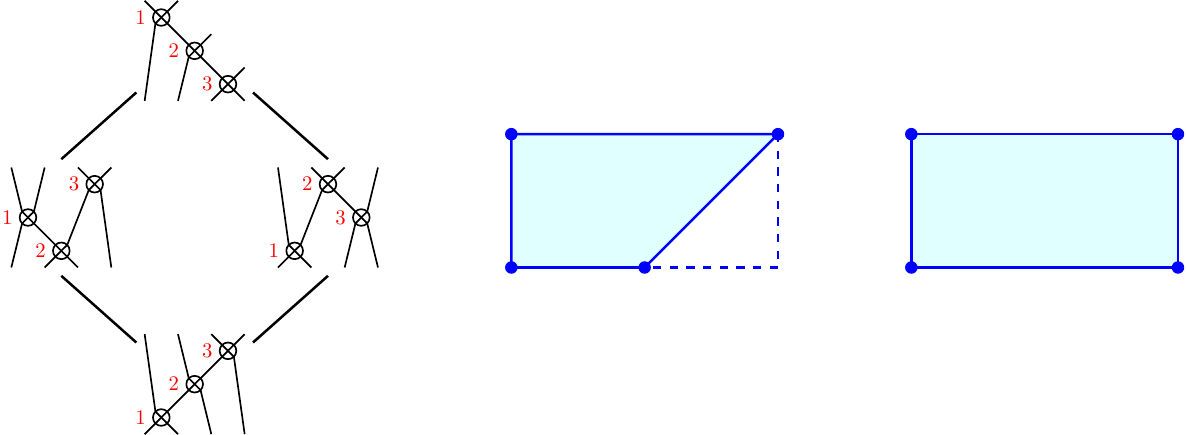}
	\caption[The rotation lattice of~$\uppdownn\uppdownn\uppdownn$-permutrees together with its geometric realizations using inversion vectors and cubic vectors.]{ The rotation lattice of~$\uppdownn\uppdownn\uppdownn$-permutrees (left) together with its geometric realizations using inversion vectors (middle) and cubic vectors (right). Taking~$T$ to be the left~$\uppdownn\uppdownn\uppdownn$-permutree, we have~$\vec{b}(T)=(1,0)$ and~$\vec{c}(T)=(2,0)$.}\label{fig:permutree_inversion_cubic_vectors}
\end{figure}

One can still manage to get such an embedding, it suffices to slightly relax the definition of our sets.

\begin{definition}\label{def:cubic_vector_permutrees}
	Consider~$T\in\cPT_n(\delta)$ to be a~$\delta$-permutree. Its \defn{cubic set}\index{permutree!cubic!set} is \begin{equation*}C(T):=\left\{(i,j)\,:\, \begin{array}{cc}
			i<j  \text{ and } v_j\in D_i & \text{ if } \delta_i\in\{\nonee,\upp\},      \\
			v_j\in RD_i                  & \text{ if } \delta_i\in\{\downn,\uppdownn\},\end{array} \right\}\end{equation*} and its \defn{cubic components}\index{permutree!cubic!components} are~${C(T)}_i=\{j\in [n] \,:\, (i,j)\in C(T)\}$. A cubic set has an associated \defn{cubic vector}\index{permutree!cubic!vector}~$\vec{c}(T)=(c_1,\ldots,c_{n-1})$ such that~$c_i=|{C(T)}_i|$.
\end{definition}

\begin{remark}\label{rem:cubic_set_inclusion}
	Like in Remark~\ref{rem:inversion_set_inclusion}, we have that for a covering relation of~$\delta$-permutrees~$T\lessdot T'$ the respective cubic sets satisfy~$C(T')={(C(T)\cup\{(i,j)\})}^{tc}$. The key difference between the transitive closures of cubic vectors against inversion vectors is that the transitive closure turns into inversions the pairs of the form~$(i,x)$ where~$x\in RD_j(T)$ and nothing else. This is a consequence of the replacement of the condition~$j\to i$ in inversion sets to~$j\in RD_i$ (resp.~$i<j$ and~$j\in D_i$) in cubic sets.
\end{remark}

\begin{definition}\label{def:cubicrealizationpermutrees}
	Let~$T,T'\in\cPT_n(\delta)$. We say that there is an edge between~$\vec{c}(T)$ and~$\vec{c}(T')$ if and only if~$T\lessdot T'$. The convex hull of the cubic vectors together with this collection of edges is called the \defn{cubical realization}~$\cC_\delta$ of~$(\mathcal{PT}(\delta),\leq)$.
\end{definition}

\begin{example}\label{ex:cubic_realization_tamari_others}
	If~$\delta=\downn\downn\downn\downn$ (resp.~$\delta=\nonee\nonee\nonee\nonee$), the cubic vector reduces to the bracket vector of binary trees (resp.\ to the Lehmer code of permutations), and we recover the cubic realization of the Tamari lattice in~\cite{K93} and~\cite{C22} (resp. of the weak order of~\cite{BF71} and~\cite{RR02}). See Figure~\ref{fig:permutreehedron_cubic} for these cubic realizations and other examples.
\end{example}

\begin{figure}[h!]
	\centering
	\includegraphics[scale=1]{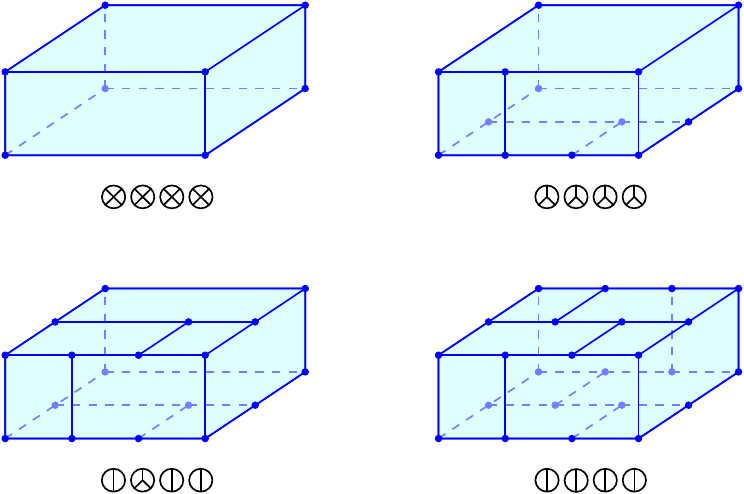}
	\caption{The cubical realization~$\cC_\delta$ of several permutreehedra.}\label{fig:permutreehedron_cubic}
\end{figure}

We now enunciate several properties of cubical realization which culminate in showing that~$\cC_\delta$ is an embedding of~$\PPT_n(\delta)$ into the cube~$Q_n=[0,n-1]\times\cdots\times[0,1]$.

\begin{theorem}\label{thm:cubic_property_edge_direction}
	If~$T,T'\in\cPT_n(\delta)$ are~$\delta$-permutrees such that~$T< T'$ in the~$\delta$-permutree rotation lattice, then~$\vec{c}(T)<_{lex} \vec{c}(T')$ and the edges of~$C_\delta$ have directions~$\mathbf{e_i}$.
\end{theorem}

\begin{proof}
	We prove it for covering relations~$T\lessdot T'$ as the~$\delta$-rotation order is the transitive closure of the relations in Figure~\ref{fig:permutree_rotations}. Following Remark~\ref{rem:cubic_set_inclusion} we have that the inclusion~$C(T')=(C(T)\cup\{(i,j)\})^{tc}$ and the relations between components~$C(T)_i\subset C(T')_i$ and~$C(T)_j= C(T')_j$ if~$j\neq i$. This tells us that~$\vec{c}(T')-\vec{c}(T)=(0,\ldots,0,c(T')_i-c(T)_i,0,\ldots,0)$. Therefore,~$\vec{c}(T)\leq_{lex}\vec{c}(T')$ and the edge~$[\vec{c}(T),\vec{c}(T')]$ has direction~$\mathbf{e_i}$.
\end{proof}

\begin{theorem}\label{thm:cubic_property_convec_cube}
	$\cC_\delta$ is normal equivalent to~$\PCube_{n-1}$ (i.e. has the same normal fan).
\end{theorem}

\begin{proof}
	Let~$Q_{n-1}:=[0,n-1]\times\cdots\times[0,1]$. First note that~$0\leq c(T)_i\leq n-i$ for all~$i\in [n-1]$ meaning that~$\mathcal{C}_{\delta}\subset Q_{n-1}$. To see the reverse inclusion it is enough to prove that all vectors~$\vec{r}=(r_1,\ldots,r_{n-1})$ where~$r_i\in\{0,n-i\}$, have a preimage through the function~$f:\cPT_n(\delta)\to \cC_\delta$ such that~$f(T)=\vec{c}(T)$. We call such preimages \defn{extremal}\index{permutree!extremal}~$\delta$-permutrees.

	Take any such~$\vec{r}$. We now present how to construct a~$\delta$-permutree in the preimage~$f^{-1}(\vec{r})$. Consider an~$n\times n$ grid. At step~$1$ place~$v_1$ at~$(1,1)$ (resp.~$(1,n)$) if~$r_1=0$ (resp.~$r_1=n-1$). At step~$i$ place~$v_i$ at~$(i,d)$ (resp.~$(i,n-u)$) where~$d:=|\{j\in[n]\,\: j<i \text{ and } r_j=0\}|$ (resp.~$u:=|\{j\in[n]\,\: j<i \text{ and } r_j=n-j\}|$). After step~$n-1$ place~$v_n$ in the only coordinate of column~$n$ that shares no vertex horizontally. Thus, we get a permutation table. Decorate each vertex~$v_i$ with the decoration~$\delta_i$. Following the insertion algorithm (see Definition~\ref{def:decorated_permutation}) we obtain a~$\delta$-permutree~$T(\vec{r})$.

	Notice that in~$T(\vec{r})$, for each vertex~$v_i$ we have either~$|RD_i|=0$ (resp.~$|\{j\in[n]\,:\,i<j \text{ and } v_j\in D_i\}|=0$) or~$|RD_i|=n-i$ (resp.~$|\{j\in[n]\,:\,i<j \text{ and } v_j\in D_i\}|=n-i$) That is, the values corresponding to~$\vec{r}$. Therefore,~$\cC_\delta=Q_{n-1}$ and normal equivalent to~$\PCube_{n-1}$.
\end{proof}

\begin{remark}\label{rem:cubic_embedding_interior_points}
	Notice that since the interior of~$Q_{n-1}$ has no integer points, we have that all cubic coordinates are on the surface of~$\mathcal{C}_{\delta}$.
\end{remark}

Figure~\ref{fig:permutree_cubic_insertion_algorithm} shows an example of the construction of extremal permutrees described in the proof of Theorem~\ref{thm:cubic_property_convec_cube}. In Figure~\ref{fig:permutree-inversion-vector-IXYI} the extremal~$\nonee\uppdownn\upp\downn$-permutrees are colored in black while the~$2$ that are extremal are colored in brown. We now show that these preimages are unique as a part of the following bigger result.

\begin{figure}[h!]
	\centering
	\includegraphics[scale=0.8]{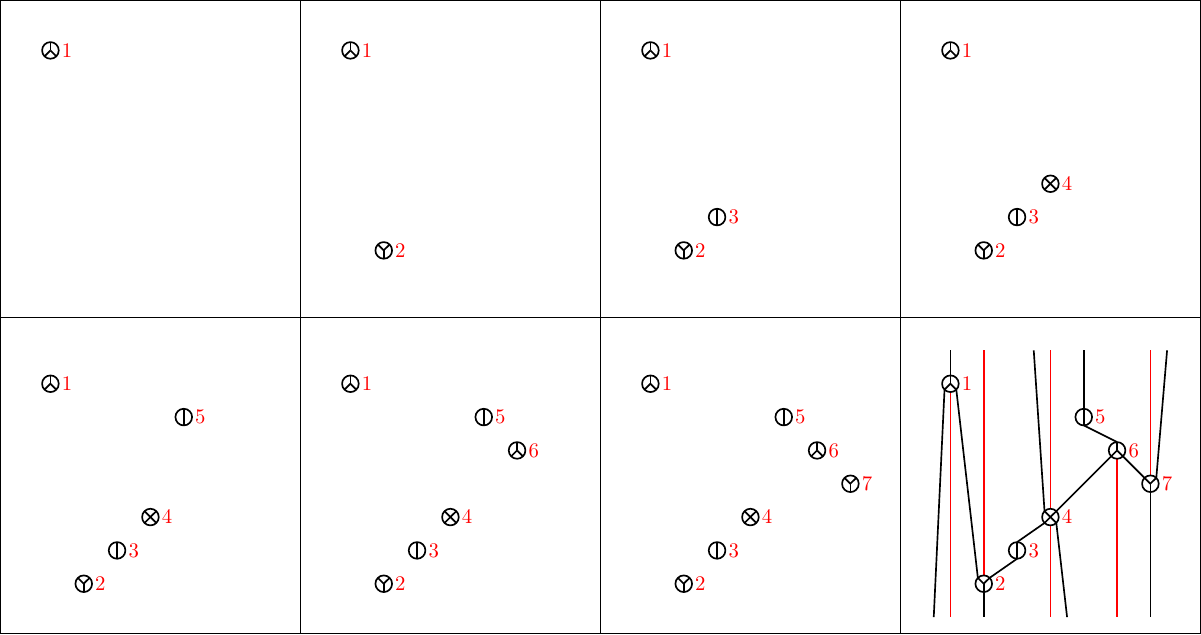}
	\caption[The construction of an extremal~$\downn\upp\nonee\uppdownn\nonee\downn\upp$-permutree.]{ The construction of the extremal~$\downn\upp\nonee\uppdownn\nonee\downn\upp$-permutree corresponding to the corner~$(6,0,0,0,2,1)\in Q_{6}$.}\label{fig:permutree_cubic_insertion_algorithm}
\end{figure}

\begin{theorem}\label{thm:cubic_property_injective}
	The map~$f:\cPT_n(\delta)\to\cC_\delta$ sending a~$\delta$-permutree to its cubic vector is injective.
\end{theorem}

\begin{proof}
	Consider~$T,T'\in \mathcal{PT}_{\delta}$ two different~$\delta$-permutrees. Due to them being different, there is a maximal vertex~$i$ such that~$RD(T)_i\neq RD(T')_i$ (resp.~$\{j\in[n]\,:\, i<j \text{ and } v_j\in D(T)_i\}\neq\{j\in[n]\,:\, i<j \text{ and } v_j\in D(T')_i\}$). If~$c(T)_i=|RD(T)_i|\neq|RD(T')_i|=c(T')_i$ (or the equivalent in the~$D_i$ case) we are done. Otherwise, there exists a maximal vertex~$r\in RD(T)_i\cap RD(T')_i$ such that~$RD(T)_r\neq RD(T')_r$ which contradicts the existence of~$i$. Therefore,~$\vec{c}(T)\neq\vec{c}(T')$.
\end{proof}

\begin{theorem}\label{thm:cubic_property_embedding}
	$\mathcal{C}_\delta$ is an embedding of the~$\delta$-permutreehedron. In particular, maximal cells of~$\mathcal{C}_{\delta}$ are in bijection with facets of the~$\delta$-permutreehedron.
\end{theorem}

\begin{proof}
	Recall from Proposition~\ref{prop:permutreehedron} that the facets of the~$\delta$-permutreehedron are in bijection with the proper subsets~$I\subsetneq[n]$ such that there is a~$\delta$-permutree that admits~$(I\,\|\,{[n]}\setminus {I})$ as an edge cut. Let~$J:=[n]\setminus I$. As~$\delta$-permutrees are connected, edge cuts partition a~$\delta$-permutree~$T$ into a~$\delta_I$-permutree~$T_I$ and a~$\delta_J$-permutree~$T_J$ which as subtrees are connected only via an edge~$(i,j)$ such that~$i\to j$ and~$v_i\in T_i$ and~$v_j\in T_j$. 
	
	Take an edge-cut~$(I\,\|\, J)$. We proceed to construct a cell~$K$ of~$\mathcal{C}_\delta$ containing all cubic vectors~$\vec{c}(T)$ of~$\delta$-permutrees~$T$ that admit said edge-cut. Consider the minimal elements~$\hat{0}_{\delta_I}$ and~$\hat{0}_{\delta_J}$ (resp. maximal elements~$\hat{1}_{\delta_I}$ and~$\hat{1}_{\delta_J}$). Connecting them via the insertion algorithm gives us the~$(I\,\|\, J)$-admitting~$\delta$-permutree~$\underline{T}$ given by~$\underline{T}_I:=\hat{0}_{\delta_I}$ and~$\underline{T}_J:=\hat{0}_{\delta_J}$ (resp.~$\overline{T}$ given by~$\overline{T}_I:=\hat{1}_{\delta_I}$ and~$\overline{T}_J:=\hat{1}_{\delta_J}$). Notice that~$\underline{T}$ (resp.~$\overline{T}$) is the minimal (resp.\ maximal)~$\delta$-permutree that admits~$(I\,\|\, J)$ as an edge cut. This in turn shows that~$\vec{c}(\underline{T})$ (resp.~$\vec{c}(\overline{T})$) is the lexicographical minimal (resp.\ maximal) cubic vector that relate with this edge-cut. Thus, we define our cell as~$K:=\{\vec{c}(T)\in\mathcal{C}_\delta\,:\,\underline{T}\leq T\leq\overline{T}\}$.

	Let us see that~$K$ is maximal by showing it is contained in a hyperplane. Suppose that~$n\in J$. In such case, for any~$\delta$-permutree~$T$ such that~$\underline{T}\leq T\leq\overline{T}$ we have that~$RD(T)_{\max(I)}=\emptyset$ (resp.~$\{j\in[n]\,:\, \max(I)<j \text{ and } v_j\in D(T)_{\max(I)}\}=\emptyset$) and we conclude that get that~$K$ is in the hyperplane~$x_{\max(I)}=0$. If instead~$n\in I$, then we obtain that~$K$ is in the hyperplane~$x_{\max(J)}=n-\max(J)$ following a similar argument. Finally, note that all other~$n-2$ entries of the cubic vectors change between~$\underline{T}$ and~$\overline{T}$ through rotations between the vertices~$I$ or~$J$. This together with Theorem~\ref{thm:cubic_property_edge_direction} gives us that~$K$ is a maximal cell of~$\mathcal{C}_\delta$.

	The conjunction of Theorems~\ref{thm:cubic_property_edge_direction},~\ref{thm:cubic_property_convec_cube}, and~\ref{thm:cubic_property_injective} and our bijection between facets and cells gives us that~$\mathcal{C}_\delta$ is an embedding of the~$\delta$-permutreehedron.
\end{proof}

%% file: includes/contenu/chap_permutrees_sorting.tex

\chapter{Permutree Sorting}\label{chap:permutree_sorting}

\addcontentsline{lof}{chapter}{\protect\numberline{\thechapter}Permutree Sorting}
\addcontentsline{lot}{chapter}{\protect\numberline{\thechapter}Permutree Sorting}

In this chapter we present a way to study permutree congruences via the theory of automata. This chapter is based directly on the article~\cite{PPT23}.

Recall from Subsection~\ref{ssec:type_A} that the symmetric group~$\fS_n$ together with the set~$S=\{s_1,\ldots,s_{n-1}\}$ of simple reflections~$s_i=(i\;\; i+1)$ is a Coxeter system of Type~$A$. In this way each permutation can be represented by a set of reduced words in the generators~$S$. We begin by rephrasing~$\delta$-permutrees in a notation that is more convenient for our purposes.

\begin{definition}
	Let~$\delta\in\{\nonee,\upp,\downn,\uppdownn\}^n$. We denote~$U:=\{j\in[2,n-1]\,:\,\delta_j\in\{\upp,\uppdownn\}\}$ and~$D:=\{j\in[2,n-1]\,:\,\delta_j\in\{\downn,\uppdownn\}\}$. In this context and following Proposition~\ref{prop:permutree_quotients}, we call \defn{$(U,D)$-permutree minimal} the minimal permutations of~$\equiv_\delta$.
\end{definition}

As a rephrasing, this chapter is dedicated to study how to discern the reduced words of~$(U,D)$-permutree minimal permutations.

\begin{remark}\label{rem:U_D_as_orientations}
	Following Definition~\ref{def:permutrees_are_orientations}, the sets~$U$ and~$D$ can also be defined as~$j\in U$ (resp.~$j\in D$) if~$j\to j-1$ (resp.~$j-1\to j$).
\end{remark}

\section{Single Automata}\label{sec:single_automata}

We begin by studying the case where~$U=\emptyset$ and~$D=\{j\}$ (resp.~$U=\{j\}$ and~$D=\emptyset$) for some~$j\in[2,n-1]$.

\begin{definition}\label{def:permutree_automata_single}
	Consider~$U=\{j\}$ (resp.~$D=\{j\}$) for some~$j\in[2,n-1]$ and the set of generators~$S$ as an alphabet. We define the automaton \defn{$\UU(j)$}\index{permutree!automaton~$\UU$} (resp.\ \defn{$\DD(j)$}\index{permutree!automaton~$\DD$}) recursively following Figure~\ref{fig:permutree_automata_single_recursive} with automata~$\UU(n)$ (resp.~$\DD(0)$) defined for consistency. As our automata are complete with each node having the~$n-1$ transitions labeled by adjacent transpositions~$s_1,\ldots,s_{n-1}$, all missing transitions in our figures are meant to be loops. Figure~\ref{fig:permutree_automata_single_full} shows the complete automata~$\UU(j)$ and~$\DD(j)$.
\end{definition}

\begin{figure}[h!]
	\centering
	\includegraphics[scale=1]{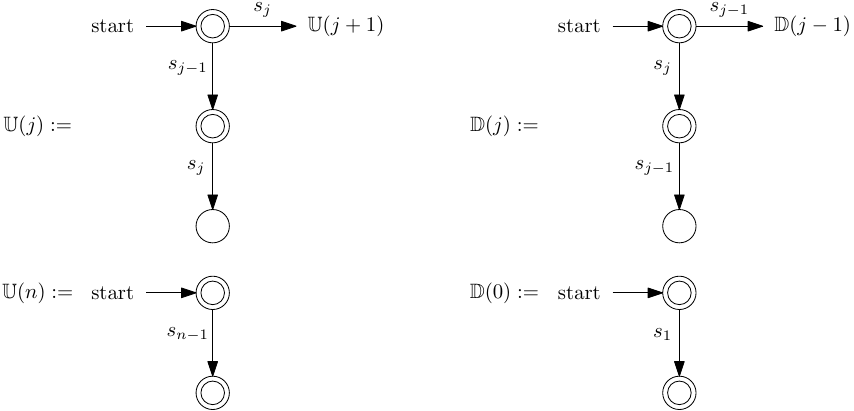}
	\caption[The automata~$\UU(j)$ and~$\DD(j)$ defined recursively.]{ The automata~$\UU(j)$ (left) and~$\DD(j)$ (right) defined recursively.}\label{fig:permutree_automata_single_recursive}
\end{figure}

\begin{figure}[h!]
	\centering
	\includegraphics[scale=0.9]{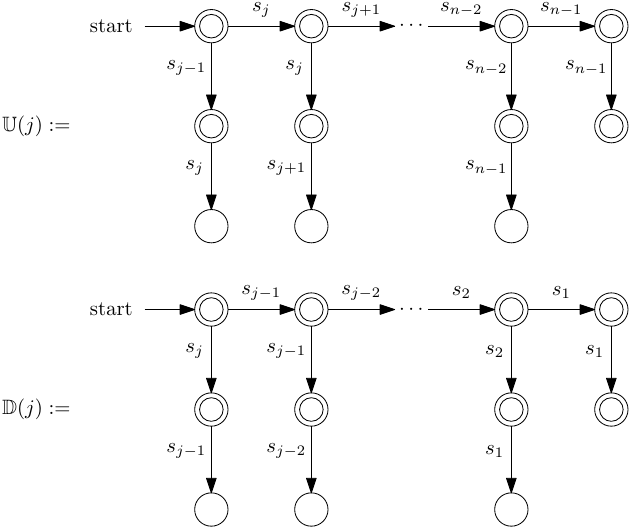}
	\caption{The complete automata~$\UU(j)$ and~$\DD(j)$.}\label{fig:permutree_automata_single_full}
\end{figure}

\begin{remark}\label{rem:automata_read_red_words}
	As our automata are complete, it is immediate from their definition that the languages~$\cL(\UU)$ and~$\cL(\DD)$ are infinite. Thus, we narrow our scope and instead of asking what is the language recognized by them, we ask which reduced words are accepted (read from left to right) by them and in which states they are accepted. This leads us to the following definition.
\end{remark}

\begin{definition}\label{def:healthy_ill_dead_states}
	We say that a state of~$\UU(j)$ (resp.~$\DD(j)$) is \defn{healthy}\index{permutree!automata!healthy/ill/dead}, \defn{ill}, or \defn{dead}, if it is respectively in the top, middle, bottom row of the automaton. The sequence of healthy states of~$\UU(j)$ (resp.~$\DD(j)$) is called the \defn{spine}\index{permutree!automata!spine} the automaton.
\end{definition}

\begin{remark}\label{rem:automata_and_walks_cox_graph}
	One didactic way of seeing our automata is the following. Given an orientation~$j\to j-1$ (resp.~$j-1\to j$), the spine of~$\UU(j)$ (resp.~$\DD(j)$) can be built taking a walk on the Coxeter graph starting from the start of the orientation (i.e.~$j$ (resp.~$j-1$)) to the end opposite to the orientation. Afterwards the other states can be built recursively following the pattern given in the Figure~\ref{fig:permutree_automata_single_recursive}.
\end{remark}

With this definition our final/accepting states are partitioned into healthy or ill states while rejecting/non-final states are dead states. To explain why these states receive these names right now is a spoiler. First, we need to redefine slightly what it means to avoid a pattern in this chapter.

\begin{definition}\label{def:permutree_pattern_avoidance}
	Fix a~$j\in[2,n-1]$. We say that a permutation~$\pi\in\fS_n$ \defn{avoids}\index{permutation!avoidance} the pattern~$jki$ (resp.~$kij$) if for all~$1\leq i<j$ and~$j<k\leq n$, the word~$jki$ (resp.~$kij$) is not a subword of~$\pi$.
\end{definition}

\begin{remark}\label{rem:comparing_pattern_avoidances}
	We bring special attention to the reader to notice that since~$j$ is fixed in Definition~\ref{def:permutree_pattern_avoidance}, this pattern avoidance is different from the usual as in Definition~\ref{def:pattern_avoidance_containment}.
\end{remark}

\begin{example}
	The permutation~$42135$ avoids~$2ki$,~$3ki$, and~$4ki$ (and thus the pattern~$231$), but contains~$ki3$ (and thus the pattern~$312$) since it contains the subsequence~$423$.
\end{example}

\subsection{Properties of \texorpdfstring{$\UU(j)$}{} and \texorpdfstring{$\DD(j)$}{}}\label{ssec:single_automata_properties}

Since our automata read reduced words from left to right, it is convenient to us to study how the event of~$\UU(j)$ (resp.~$\DD(j)$) accepting or rejecting a reduced word changes when we multiply on the left said word by the transposition~$s_j$,~$s_{j-1}$, or~$s_i$ with~$i\notin\{i,j-1\}$. The following Lemmas and Examples do this while studying in parallel how our fixed pattern avoidance changes under the same circumstances. The following proofs are written for the case of~$\UU(j)$ and the pattern~$jki$ to ease reading. The proofs for the cases of~$\DD(j)$ and the pattern~$kij$ are similar.

\begin{lemma}\label{lem:automata_multplying_si}
	Let~$\sigma\in\fS_n$ such that~$\sigma([j-1]) = [j-1]$,~$\sigma(j) = j$, and~$\sigma([n]\setminus{[j]}) = [n]\setminus{[j]}$ and~$\tau\in \fS_n$ such that~$\ell(\sigma \cdot \tau) = \ell(\sigma) + \ell(\tau)$, then:
	\begin{enumerate}
		\itemsep0em
		\item~$\tau$ possesses a reduced word accepted by~$\UU(j)$ (resp.~$\DD(j)$) if and only if~$\sigma \cdot \tau$ posses a reduced word accepted by~$\UU(j)$ (resp.~$\DD(j)$),
		\item~$\tau$ avoids~$jki$ (resp.~$kij$) if and only if~$\sigma \cdot \tau$ avoids~$jki$ (resp.~$kij$).
	\end{enumerate}
\end{lemma}

\begin{proof}
	We prove both events separately as follows.
	\begin{enumerate}
		\itemsep0em
		\item The constraints on~$\sigma$ imply that the transpositions~$s_j$ and~$s_{j-1}$ do not appear in any of its reduced words (see Remark~\ref{rem:reduced_word_braid_moves}). Therefore, any reduced word of~$\sigma$ ends in the initial state of~$\UU(j)$, and it is trivially accepted. Thus, the accepting or rejecting of~$\sigma\cdot\tau$ depends uniquely on~$\tau$. The result immediately follows.
		\item The constraints on~$\sigma$ show that the assumption~$1\leq i<j<k \leq n$ implies that~$1\leq \sigma(i)<\sigma(j)=j<\sigma(k) \leq n$. Thus,~$\tau$ contains the pattern~$jki$ if and only if~$\sigma\cdot\tau$ contains the pattern~$j\sigma(k)\sigma(i)$.
		      \qedhere
	\end{enumerate}
\end{proof}

\begin{example}\label{ex:automata_multplying_si}
	Consider~$j := 4$ and the permutations~$\sigma := 312465 = s_2 \cdot s_1 \cdot s_5$,~$\tau_1 := 143256 = s_3 \cdot s_2 \cdot s_3$, and~$\tau_2 := 124536 = s_3 \cdot s_4$.
	Multiplying we obtain~$\sigma\cdot\tau_1=342165$ and~$\sigma\cdot\tau_2 = 314625$. Following the order of Lemma~\ref{lem:automata_multplying_si} we have that:
	\begin{enumerate}
		\itemsep0em
		\item~$\UU(4)$ accepts all reduced words of both~$\tau_1$ and~$\sigma\cdot \tau_1$ on its first ill state, and rejects all reduced words of both~$\tau_2$ and~$\sigma\cdot \tau_2$ on its first dead state,
		\item both~$\tau_1$ and~$\sigma\cdot\tau_1$ avoid~$4ki$, while both~$\tau_2$ and~$\sigma\cdot\tau_2$ contain~$4ki$.
	\end{enumerate}
\end{example}

\begin{lemma}\label{lem:automata_multplying_sj-1}
	Let~$\tau \in \fS_n$ be a permutation with a reduced word starting with~$s_{j-1}$ (resp.~$s_j$) that is accepted by~$\UU(j)$ (resp.~$\DD(j)$), then
	\begin{enumerate}
		\itemsep0em
		\item~$\tau$ does not permute~$j$ and~$j+1$ (resp.~$j-1$ and~$j$),
		\item~$\tau$ avoids~$jki$ (resp.~$kij$).
	\end{enumerate}
\end{lemma}

\begin{proof}
	Consider a reduced word~$w$ starting with~$s_{j-1}$ and accepted by~$\UU(j)$. Since~$w$ starts with~$s_{j-1}$ and is accepted by~$\UU(j)$, it is accepted in the first ill state of~$\UU(j)$. Such state contains only loops and the transitions~$s_j$ to a dead state. Since~$w$ does not arrive to said dead state, we know that~$w$ does not contain the transposition~$s_j$. Following Remark~\ref{rem:reduced_word_braid_moves} we know that no word of~$w$ contains~$s_j$. We conclude our claims as follows:
	\begin{enumerate}
		\itemsep0em
		\item Since~$s_j=(j\;\; j+1)$, not having~$s_j$ in~$w$ implies that~$\tau$ does not permute~$j$ and~$j+1$.
		\item As~$\tau$ does not contain~$s_j$ in any reduced word,~$\tau([j]) = [j]$ and~$\tau([n+1] \setminus{[j]}) = [n+1] \setminus{[j]}$. Thus,~$\tau$ avoids the subword~$ki$ with~$i < j < k$ and by consequence avoids as well~$jki$.
		      \qedhere
	\end{enumerate}
\end{proof}

\begin{example}\label{ex:automata_multplying_sj-1}
	Consider~$j:=4$ and~$\tau:=413265$, with reduced word~$s_3 \cdot s_5 \cdot s_2 \cdot s_1 \cdot s_3$ accepted by~$\UU(4)$. We get that
	\begin{enumerate}
		\itemsep0em
		\item~$\tau$ does not permute~$4$ and~$5$,
		\item~$\tau$ avoids~$4ki$.
	\end{enumerate}
\end{example}

\begin{lemma}\label{lem:automata_multplying_sj}
	Let~$\tau \in \fS_n$ be a permutation that does not permute~$j$ and~$j+1$ (resp.~$j-1$ and~$j$), then
	\begin{enumerate}
		\itemsep0em
		\item~$s_j \cdot \tau$ (resp.~$s_{j-1} \cdot \tau$) possesses a reduced word accepted by~$\UU(j)$ (resp.~$\DD(j)$) if and only if~$\tau$ possesses a reduced word accepted by~$\UU(j+1)$ (resp.~$\DD(j-1)$),
		\item~$s_j \cdot \tau$ (resp.~$s_{j-1} \cdot \tau$) avoids~$jki$ (resp.~$kij$) if and only if~$\tau$ avoids~$(j+1)ki$ (resp.~$ki(j-1)$).
	\end{enumerate}
\end{lemma}

\begin{proof}
	We deal with the two statements separately again:
	\begin{enumerate}
		\itemsep0em
		\item Consider~$w$ a reduced word for~$\tau$ accepted by~$\UU(j+1)$. Since~$\tau$ does not permute~$j$ and~$j+1$, we know that~$s_j$ does not appear in~$w$, and~$s_j \cdot w$ is a reduced word for~$s_j \cdot \tau$. By construction~$s_j \cdot w$ is accepted by~$\UU(j)$.

		      Now suppose that~$s_j \cdot \tau$ possesses a reduced word~$w$ accepted by~$\UU(j)$. Notice that~$w$ cannot start by~$s_{j-1}$ due to Lemma~\ref{lem:automata_multplying_sj-1} but contains~$s_j$ since~$s_j\cdot\tau$ permutes~$j$ and~$j+1$. Using Lemma~\ref{lem:automata_multplying_si} we can also assume that~$w$ begins with~$s_j$ as any other possibilities yield only loops. Thus, after reading~$s_j$, the corresponding accepted suffix is a reduced word for~$\tau$ that is accepted by~$\UU(j+1)$.

		\item Notice that since~$(j\;\; j+1)$ is an inversion of~$s_j \cdot \tau$ but not of~$\tau$, the value~$j+1$ cannot be used as~$k$ to form the pattern~$jki$ in~$s_j \cdot \tau$ and the value~$j$ cannot be used as~$i$ to form the pattern~$(j+1)ki$ in~$\tau$. Since multiplying on the left by~$s_j$ only exchanges the values~$j$ and~$j+1$ the result follows.
		      \qedhere
	\end{enumerate}
\end{proof}

\begin{example}\label{exm:automata_multplying_sj}
	Consider~$j := 4$ and the permutations~$\tau_1 := 142536$ and~$\tau_2 := 142563$ that do not permute~$4$ and~$5$.
	Multiplying we obtain~$s_4 \cdot \tau_1 = 152436$ and~$s_4 \cdot \tau_2 = 152463$. In this scenario we have that
	\begin{enumerate}
		\itemsep0em
		\item the reduced word~$s_4 \cdot s_3 \cdot s_4 \cdot s_2$ of~$s_4 \cdot \tau_1$ is accepted by~$\UU(4)$ and the reduced word~$s_3 \cdot s_4 \cdot s_2$ of~$\tau_1$ is accepted by~$\UU(5)$, while all reduced words of~$s_4 \cdot \tau_2$ are rejected by~$\UU(4)$ and all reduced words of~$\tau_2$ are rejected by~$\UU(5)$,
		\item~$s_4 \cdot \tau_1$ avoids~$4ki$ and~$\tau_1$ avoids~$5ki$, while~$s_4 \cdot \tau_2$ contains~$463$ and~$\tau_2$ contains~$563$.
	\end{enumerate}
\end{example}

We now have in our hands the tools required to prove the hinted relationship between our automata and pattern avoidance.

\begin{theorem}\label{thm:pattern_avoidance_single}
	Fix~$j \in [2,n-1]$ and~$\pi \in \fS_n$. The following statements are equivalent:
	\begin{itemize}
		\itemsep0em
		\item~$\pi$ possesses a reduced word accepted by the automaton~$\UU(j)$ (resp.~$\DD(j)$),
		\item~$\pi$ avoids the pattern~$jki$ (resp.~$kij$) with~$i < j < k$.
	\end{itemize}
\end{theorem}

\begin{proof}
	We proceed by an induction on the length of the permutations. The base case is trivial due to the construction of our automata and the only permutation of length~$0$ being the identity~$12\cdots n$.

	Let~$\pi\in\fS_n$ such that it has a reduced word~$w$ accepted by~$\UU(j)$. Supposing that~$w$ begins with~$s_i$, let~$\tau\in \fS_n$ such that~$\pi=s_i\cdot\tau$.
	\begin{itemize}
		\itemsep0em
		\item If~$i\notin\{j-1,j\}$,~$\tau$ possesses a reduced word accepted by~$\UU(j)$ by Lemma~\ref{lem:automata_multplying_si}\,(1), so that~$\tau$ avoids~$jki$ by induction. Thus,~$\pi = s_i \cdot \tau$ avoids~$jki$ by Lemma~\ref{lem:automata_multplying_si}\,(2).
		\item If~$i = j-1$, then~$\pi$ avoids~$jki$ by Lemma~\ref{lem:automata_multplying_sj-1}\,(2).
		\item If~$i = j$, then~$\tau$ possesses a reduced word accepted by~$\UU(j+1)$ by Lemma~\ref{lem:automata_multplying_sj}\,(1). We obtain by induction that~$\tau$ avoids~$(j+1)ki$. Thus,~$\pi = s_j \cdot \tau$ avoids~$jki$ by Lemma~\ref{lem:automata_multplying_sj}\,(2).
	\end{itemize}

	No matter the case, we see that~$\pi$ avoids~$jki$.

	Now for the reverse direction let~$\pi\in\fS_n$ be a permutation avoiding the pattern~$jki$. We get the following two cases.

	\begin{itemize}
		\itemsep0em
		\item There exists~$m$ minimal such that~$j<m$ and~$\pi^{-1}(j)>\pi^{-1}(m)$. In such a case,~$\pi^{-1}(l)>\pi^{-1}(m)$ for all~$l\in [j,m-1]$. Following Lemma~\ref{lem:perm_starting_with_sj},~$\pi$ possesses a reduced word starting with~$s_{m-1}s_{m-2}\cdots s_{j+1}s_j$. Let~$\sigma=s_{m-1}s_{m-2}\cdots s_{j+1}$ giving us the factorization~$\pi=\sigma\cdot s_j\cdot \tau$ for some~$\tau\in\fS_n$. Using Lemma~\ref{lem:automata_multplying_si}\,(2) and then Lemma~\ref{lem:automata_multplying_sj}\,(2) we obtain that~$\tau$ avoids~$(j+1)ki$. By induction, we obtain that it possesses a reduced word accepted by~$\UU(j+1)$. By Lemmas~\ref{lem:automata_multplying_si}\,(1) and~\ref{lem:automata_multplying_sj}\,(1), we conclude that~$\pi$ possesses a reduced word accepted by~$\UU(j)$.
		\item For all~$m\in[j+1,n]$ we have that and~$\pi(j)<\pi(m)$. Let~$w$ be a reduced word of~$\pi$ which by the previous sentence Lemma~\ref{lem:perm_starting_with_sj} cannot begin by~$s_m$ for~$m\in[j+1,n]$. If it is accepted by~$\UU(j)$, we are done. Otherwise,~$w$ is rejected and up to commutations with transpositions~$s_c$ with~$c\in[1,j-2]$ we have that~$w$ starts with~$s_{j-1}$ and then followed by~$s_j$. Call~$i$ and~$k$ the two elements that are exchanged when the reduced word first uses~$s_j$. We have that~$i < j < k$ and~$\pi$ contains the pattern~$jki$ (because~$j$ and~$k$ are not exchanged in~$\pi$, and~$i$ and~$k$ are already exchanged, so they remain exchanged in~$\pi$). This is a contradiction with our assumption that~$\pi$ avoids~$jki$.
	\end{itemize}
	Thus, in any case we get that~$\pi$ possesses a reduced word accepted by~$\UU(j)$.
\end{proof}

\subsection{The Structure of Accepted Reduced Words of \texorpdfstring{$\UU(j)$}{} and \texorpdfstring{$\DD(j)$}{}}\label{ssec:single_automata_words}

Via Proposition~\ref{prop:permutree_quotients}, Theorem~\ref{thm:pattern_avoidance_single} allows us to distinguish when a permutation is minimal in its permutree class for a congruence given through a single orientation of the Coxeter graph. Still, as the number of reduced word of permutations grows very quickly (for the longest word~$w_0$ there are~$\frac{\binom{n}{2}!}{\prod_{k=1}^{n-1} (2k-1)^{n-k}}$ reduced expressions~\cite[A005118]{OEIS}), trying all reduced words on our automata in hopes of finding one that is accepted is not practical. Therefore, we now move to present certain properties of the set of accepted reduced words of our automata with the aim of describing how to efficiently propose a single candidate reduced word in the context of Theorem~\ref{thm:pattern_avoidance_single}.

\begin{remark}\label{rem:reduced_words_not_closed}
	Given the structure of our automata, a permutation may possess reduced words~$w,w'$ such that~$w$ is accepted by~$\UU(j)$ while~$w'$ is rejected by~$\UU(j)$. For example consider~$\pi=321$ with reduced words~$w=s_2\cdot s_1\cdot s_2$ and~$w'=s_1\cdot s_2\cdot s_1$. Then~$w$ is accepted by~$\UU(2)$ while~$w'$ is rejected by~$\UU(2)$.
\end{remark}

However, we do have the following nice properties on the set of reduced words.

\begin{theorem}\label{thm:acc_red_word_prefix}
	The set of reduced words accepted by~$\UU(j)$ (resp.~$\DD(j)$) is closed by taking prefixes.
\end{theorem}

\begin{proof}
	This follows from the fact that any path to an accepting state in~$\UU(j)$ begins in the initial state and contains only accepting states. Thus, any prefix of a reduced word accepted by~$\UU(j)$ is also accepted by~$\UU(j)$.  Since all prefixes of a reduced word are also reduced words, we obtain the result.
\end{proof}

\begin{theorem}\label{thm:acc_red_word_algorithm}
	Let~$\ell \in [n-1]$ be distinct from~$j-1$ (resp.~$j$).
	A permutation~$\pi \in \fS_n$ that avoids~$jki$ (resp.~$kij$) and reverses~$\ell$ and~$\ell+1$ possesses a reduced word starting with~$s_\ell$ and accepted by~$\UU(j)$ (resp.~$\DD(j)$).
\end{theorem}

\begin{proof}
	Since~$\pi$ reverses~$\ell$ and~$\ell+1$, it possesses a reduced word of the form~$\pi = s_\ell \cdot \tau$ (see Lemma~\ref{lem:perm_starting_with_sj}). To finish proving the proposition we have the following cases:
	\begin{itemize}
		\itemsep0em
		\item if~$\ell = j$, then~$\tau$ does not reverse~$j$ and~$j+1$ and thus avoids~$(j+1)ki$ by Lemma~\ref{lem:automata_multplying_sj}\,(2). Hence,~$\tau$ has a reduced word accepted by~$\UU(j+1)$ by Theorem~\ref{thm:pattern_avoidance_single}, and we conclude by Lemma~\ref{lem:automata_multplying_sj}\,(1).
		\item if~$\ell\notin\{j-1,j\}$. In this case~$\tau$ avoids~$jki$ by Lemma~\ref{lem:automata_multplying_si}\,(2) and thus,~$\tau$ has a reduced word accepted by~$\UU(j)$ by Theorem~\ref{thm:pattern_avoidance_single}. We conclude by Lemma~\ref{lem:automata_multplying_si}\,(1).
		      \qedhere
	\end{itemize}
\end{proof}

\begin{theorem}\label{thm:acc_red_word_same_state}
	Given a permutation~$\pi \in \fS_n$, all the reduced words for~$\pi$ accepted by~$\UU(j)$ (resp.~$\DD(j)$) are accepted at the same state.
\end{theorem}

It is possible to give a direct proof of this theorem by verifying that commutations and braid moves between accepted reduced words do not change the state on which the reduced words are accepted. One can also classify which expressions of reduced words are accepted at each state. We prefer going for the following stronger result which characterizes where the reduced words of a permutation land depending on its inversions. For simplicity, we state it in terms of~$\UU(j)$, although a similar statement holds for the automaton~$\DD(j)$.

\begin{theorem}\label{thm:acc_red_word_same_state_inv}
	Given a permutation~$\pi\in\fS_n$, partition the inversions of~$\pi$ as into the following two sets: \begin{equation*}
		\begin{split}
			\inv^j(\pi) &= \{(i,j)\,:\, i < j \text{ and } \pi^{-1}(i) > \pi^{-1}(j)\}\\
			\inv_j(\pi) &= \{(j,k) \,:\, j < k \text{ and } \pi^{-1}(j) > \pi^{-1}(k)\}
		\end{split}
	\end{equation*}
	Then we have the following properties:
	\begin{enumerate}
		\itemsep0em
		\item if~$|\inv^j(\pi)| = 0$, then all reduced words for~$\pi$ end at the same healthy state of~$\UU(j)$,
		\item if~$|\inv_j(\pi)| = 0$, then all reduced words for~$\pi$ end at the same state of~$\UU(j)$. The state is healthy if~$\pi$ avoids~$ji$, ill if~$\pi$ contains~$ji$ but avoids~$jki$, or dead if~$\pi$ contains~$jki$,
		\item if~$|\inv^j(\pi)| \ne 0 \ne |\inv_j(\pi)|$, all accepted reduced words for~$\pi$ end at the same ill state of~$\UU(j)$ while the rejected reduced words may end at distinct dead states of~$\UU(j)$.
	\end{enumerate}
\end{theorem}

\begin{proof}
	We proceed by induction on the length of a permutation. Notice that the base case is trivial as the identity has no reduced words and corresponds to our initial state which is always accepting. Consider an arbitrary reduced word~$w$ for~$\pi$. Letting it begin by~$s_l$ let~$w=s_l\cdot w'$ and~$\pi=s_l\cdot \tau$ where~$w'$ is a reduced word of~$\tau\in\fS_n$. Depending on the value of~$l$ we have the following cases:
	\begin{itemize}
		\itemsep0em
		\item if~$\ell \notin \{j-1, j\}$, then~$s_\ell$ is a loop in~$\UU(j)$ giving us the equalities~$|\inv^j(\pi)| = |\inv^j(\tau)|$ and~$|\inv_j(\pi)| = |\inv_j(\tau)|$,
		\item if~$\ell = j$, then~$s_j$ goes to the initial state of~$\UU(j+1)$, and we have that~$|\inv^j(\pi)| = |\inv^{j+1}(\tau)|$ and~${|\inv_j(\pi)| = |\inv_{j+1}(\tau)| + 1}$,
		\item if~$\ell = j-1$, then~$s_{j-1}$ goes to the first ill state of~$\UU(j)$ and~$|\inv^j(\pi)| = |\inv^{j+1}(\tau)| + 1$ and~$|\inv_j(\pi)| = |\inv_{j+1}(\tau)|$.
	\end{itemize} By induction, we obtain that the reduced word~$w'$ for~$\tau$ ends as predicted in the statement.
	The previous observations ensure that the reduced word~$w$ for~$\pi$ also does.
\end{proof}

\begin{example}\label{exm:sameStateAcceptedReducedExpressionsRefined}
	We provide examples for each of the three cases of Theorem~\ref{thm:acc_red_word_same_state_inv}.
	\begin{enumerate}
		\itemsep0em
		\item For~$\pi := 4312$, we have that~$|\inv^2(\pi)|=0$ and all of its~$5$ reduced words end at the third healthy state of~$\UU(2)$.
		\item For~$\pi := 32145$ (resp.~$\pi := 43215$, resp.~$\pi := 43251$), we have~$|\inv_4(\pi)| = 0$ and all its~$2$ (resp.~$16$, resp.~$35$) reduced words end at the first healthy (resp.~ill, resp.~dead) state~of~$\UU(4)$.
		\item For~$\pi := 4321$, we have~$|\inv^2(\pi)| = |\{(1,2)\}| = 1$ and~$|\inv_2(\pi)| = |\{(2,3),(2,4)\}| = 2$. Among the~$16$ reduced words of~$\pi$, the automaton~$\UU(2)$ accepts~$7$ at its third ill state, rejects~$7$ at its first dead state, and rejects the other~$2$ at its second dead state.
	\end{enumerate}
\end{example}

Notice that Theorem~\ref{thm:acc_red_word_algorithm} allows us to algorithmically construct a candidate reduced word to verify if a permutation~$\pi$ is~$(\{j\},\emptyset)$-permutree minimal. This would be done by accumulating transpositions~$s_\ell$ depending on if the values~$\ell$ and~$\ell+1$ are permuted and verifying that the resulting reduced word is accepted by~$\UU(j)$.

Here we go for a sorting approach as in~\cite{K73}, meaning that we construct a reduced word accepted by~$\UU(j)$ which is a reduced word of~$\pi$ if and only if~$\pi$ avoids the pattern~$jki$. We call the following algorithm the \defn{$(\{j\},\emptyset)$-permutree sorting}.

\bigskip
\IncMargin{1em}
\SetKwInOut{Input}{Input}\SetKwInOut{Output}{Output}
\SetKwFor{Repeat}{repeat}{}{}
\SetKwIF{If}{ElseIf}{Else}{if}{then}{else if}{else}{}
\DontPrintSemicolon{}
\begin{algorithm}[H]
	\renewcommand{\algorithmcfname}{Algorithm}%
	\Input{a permutation~$\pi \in \fS_n$ and an integer~$j \in [n]$}
	\Output{a reduced word accepted by~$\UU(j)$, candidate reduced word for~$\pi$}
	$w :=\varepsilon$\;
	\Repeat{}{
		\If{$\exists \; \ell \ne j-1$ such that~$\ell$ and~$\ell+1$ are reversed in~$\pi$}{
			$\pi := s_\ell \cdot \pi$, \quad~$w := w \cdot s_\ell$ \;
			\lIf{$\ell = j$}{
				$j := j+1$
			}
		}
	}
	\If{$j-1$ and~$j$ are reversed in~$\pi$}{
		$\pi := s_{j-1} \cdot \pi$, \quad~$w := w \cdot s_{j-1}$ \;
		$w := w \cdot w' \cdot w''$ where~$w'$ sorts~$\pi_{[j]}$ and~$w''$ sorts~$\pi_{[n] \setminus{[j]}}$ \;
	}
	\Return{}~$w$
	\caption{$(\{j\}, \emptyset)$-permutree sorting}
	\label{algo:permutree_sorting_simple}
\end{algorithm}
\bigskip

\begin{example}\label{ex:permutree_sorting_simple}
	Table~\ref{tab:permutree_sorting_example_simple} shows the~$(\{2\},\emptyset)$-permutree sorting algorithm in action for the permutations~$\pi_1 := 3421$ and~$\pi_2 := 4231$. Each row contains the permutation~$\pi_i$, the reduced word~$w_i$, and the values of~$j_i$ and~$\ell_i$ in use at each step of the algorithm.

	Notice that for~$\pi_1$ the algorithm ends with the identity and is thus~$(\{2\},\emptyset)$-sortable, which coincides with the fact that it avoids~$2ki$. In contrast, for~$\pi_2$ the algorithm ends with the permutation~$1243$, meaning that~$\pi_2$ is not~$(\{2\},\emptyset)$-sortable, which coincides with the fact that~$\pi_2$ contains~$2ki$.

	\begin{table}[h!]
		\centering
		\includegraphics[scale=1]{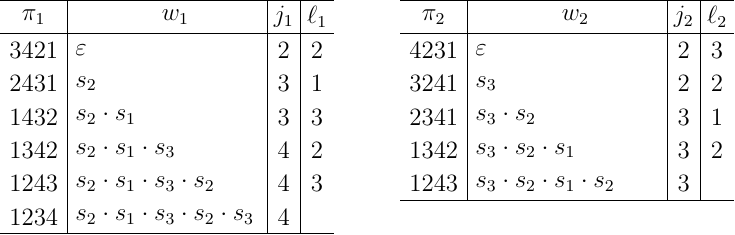}
		\caption[The~$(\{2\},\emptyset)$-permutree sorting of~$3421$ and~$4231$.]{ The~$(\{2\},\emptyset)$-permutree sorting of~$3421$ (left) and~$4231$ (right).}\label{tab:permutree_sorting_example_simple}
	\end{table}
\end{example}

\begin{corollary}\label{coro:algorithm_single}
	For any permutation~$\pi$ and~$j \in [2,n-1]$, Algorithm~\ref{algo:permutree_sorting_simple} returns a reduced~word~$w$ accepted by~$\UU(j)$ with the property that~$w$ is a reduced word for~$\pi$ if and only if~$\pi$ avoids~$jki$.
\end{corollary}

\begin{proof}
	First notice that Algorithm~\ref{algo:permutree_sorting_simple} creates a reduced word following~$\UU(j)$. It begins prioritizing healthy states in lines 2 to 5 by considering transpositions that are loops following Lemma~\ref{lem:automata_multplying_si} and then changing~$j$ to~$j+1$ following Lemma~\ref{lem:automata_multplying_sj}. This repeats for as many transitions as possible until we have to consider~$s_{j-1}$ and go to an ill state in line 6. At this point we cannot use~$s_j$ as this would lead to an ill state. After this point, we can use as many transitions in~$\{s_1,\ldots,s_{j-1}\}$ (resp.~$\{s_{j+1},\ldots,s_{n-1}\}$) to sort~$[j]$ (resp.~$[n]\setminus{[j]}$).

	The resulting reduced word is accepted by~$\UU(j)$ as we never use~$s_j$ after an ill state. If it happens that~$w$ is a reduced word of~$\pi$, then by Theorem~\ref{thm:pattern_avoidance_single}~$\pi$ avoids~$jki$. Conversely, if~$\pi$ avoids~$jki$, by Theorem~\ref{thm:acc_red_word_algorithm} we have that~$w$ started being constructed with the respective transpositions~$s_l$ and lines 2 to 5 of the algorithm, and then forced to end in lines 6 to 9 in an expression unique up to commutations. Thus, it is a reduced word of~$\pi$.
\end{proof}

\begin{remark}
	Notice it is line 8 of Algorithm~\ref{algo:permutree_sorting_simple} which makes it a sorting algorithm as in~\cite{K73}. Indeed,~$w$ becomes a reduced word for~$\pi$ at the end of the algorithm if and only if~$\pi$ becomes the identity permutation. We say that~$\pi$ is then~$(\{j\},\emptyset)$-permutree sortable. If one does not want to go so far one can just verify at line 8 if~$\pi([j]) = [j]$ and~$\pi([n] \setminus{[j]}) = [n] \setminus{[j]}$.
\end{remark}

\subsection{Generating Trees}\label{ssec:single_automata_trees}

Similar to how Theorem~\ref{thm:acc_red_word_algorithm} has an algorithmic consequence, Theorem~\ref{thm:acc_red_word_prefix} and Theorem~\ref{thm:acc_red_word_same_state} have a combinatorial consequence. Namely, they define generating trees for~$(\{j\},\emptyset)$-permutree minimal permutations as subgraphs of the Hasse diagram of the weak order on~$\fS_n$. To construct them without ambiguity let~$\prec$ be a total order on~$\{s_1,\ldots,s_{n-1}\}$ which we call a \defn{priority order}\index{priority order} on adjacent transpositions. Given a~$(\{j\},\emptyset)$-permutree minimal permutation~$\pi\in\fS_n$, let \defn{$\pi(\{j\}, \emptyset, \prec)$} be the~$\prec$-lexicographic minimal reduced word for~$\pi$ that is accepted by~$\UU(j)$ and \begin{equation*}
	\cR(n, \{j\}, \emptyset, \prec):=\big\{ \pi(\{j\}, \emptyset, \prec)\,:\, \pi \in \fS_n \text{ is } (\{j\}, \emptyset)\text{-permutree minimal} \big\}.
\end{equation*}

\begin{theorem}\label{thm:automata_generating_trees_simple}
	The set~$\cR(n, \{j\}, \emptyset, \prec)$ is closed by taking prefixes.
\end{theorem}

\begin{proof}
	Consider a reduced word~$w = u \cdot v$ where~$u\notin\cR(n, \{j\}, \emptyset, \prec)$.
	If~$u$ is rejected by~$\UU(j)$, then~$w$ is rejected as well due to Theorem~\ref{thm:acc_red_word_prefix}. Otherwise, there exists a reduced word~$u'$ representing the same permutation as~$u$, accepted by~$\UU(j)$ such that it is~$\prec$-lexicographic smaller than~$u$.
	Due to Theorem~\ref{thm:acc_red_word_same_state},~$u$ and~$u'$ end at the same state of~$\UU(j)$ and thus if~$w = u \cdot v$ is accepted by~$\UU(j)$, so is~$u' \cdot v$. Since~$u' \cdot v$ is~$\prec$-lexicographically smaller than~$u \cdot v$ and represents the same permutation, we have that~$w$ is not in~$\cR(n, \{j\}, \emptyset, \prec)$.
\end{proof}

Theorem~\ref{thm:automata_generating_trees_simple} gives us a generating tree where the root is the empty reduced word and the parent of a reduced word is obtained by deleting the last letter. Taking this tree as the sequence of associated transpositions in the weak order, we obtain a generating tree on~$(\{j\},\emptyset)$-permutree minimal permutations as a sublattice of the weak order on~$\fS_n$. Figure~\ref{fig:permutree_automata_single_trees} shows all possible trees for~$\fS_4$ and priority order~$s_1 \prec s_2 \prec s_3$.

\begin{figure}
	\centering
	\includegraphics[scale=0.75,angle=90]{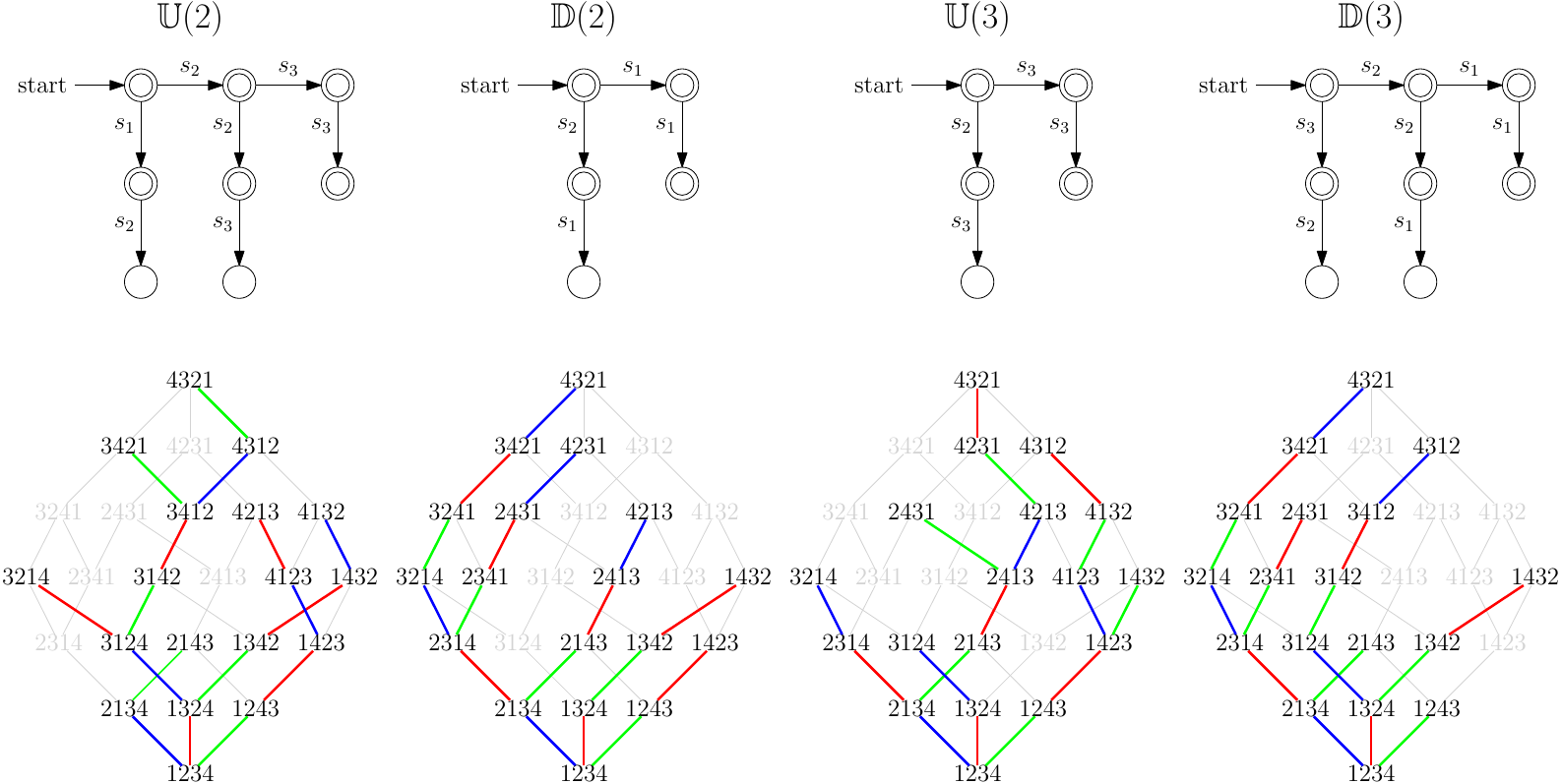}
	\caption{Generating trees for the~$(\{j\}, \emptyset)$ and~$(\emptyset, \{j\})$-permutree minimal permutations of~$\fS_4$, with priority order~${\color{blue}s_1} \prec {\color{red}s_2} \prec {\color{green}s_3}$.}\label{fig:permutree_automata_single_trees}
\end{figure}

\section{Multiple Automata}\label{sec:multiple_automata}

We now move on to the case where~$U$ and~$D$ are arbitrary subsets of~$[2,n-1]$. In this context Theorem~\ref{thm:pattern_avoidance_single} tells us that the following are equivalent statements for a permutation~$\pi\in\fS_n$: \begin{enumerate}
	\itemsep0em
	\item~$\pi$ is a~$(U,D)$-permutree minimal permutation,
	\item letting~$i<j<k$,~$\pi$ avoids the pattern~$jki$ (resp.~$kij$) if~$j\in U$ (resp.~$j\in D$),
	\item for each~$j\in U$ (resp.~$j\in D$)~$\pi$ possesses a reduced word accepted by~$\UU(j)$ (resp.~$\DD(j)$).
\end{enumerate}

Statements (1) and (2) happen simultaneously while statement (3) gives a reduced word for each element of~$U$ and~$D$. Unfortunately the reduced word that is accepted by an automaton might be rejected by another. Such is the case of the following example.

\begin{example}\label{ex:no_common_red_wrd_when_j_U_and_D}
	Let us revisit Remark~\ref{rem:reduced_words_not_closed} in more generality. Let~$j\in[2,n-1]$ and~$U=D=\{j\}$. Consider the permutation~$\pi\in\fS_n$ with~$s_j\cdot s_{j-1}\cdot s_j$ and~$s_{j-1}\cdot s_j\cdot s_{j-1}$ as its only reduced words. Notice that~$s_j\cdot s_{j-1}\cdot s_j$ is accepted by~$\UU(j)$ and rejected by~$\DD(j)$ while~$s_{j-1}\cdot s_j\cdot s_{j-1}$ is accepted by~$\DD(j)$ and rejected by~$\UU(j)$. Thus,~$\pi$ has no reduced word accepted by both automata.
\end{example}

This gives us the following remark.

\begin{remark}\label{rem:no_common_red_wrd_when_U_and_D_not_disjoint}
	If~$U$ and~$D$ are not disjoint, there exists permutations who do not possess a reduced word accepted by all the respective automata~$\UU(j)$ and~$\DD(j)$.
\end{remark}

Hope is not lost though, as this is not the case when~$U$ and~$D$ are not disjoint as the following theorem shows.

\begin{theorem}\label{thm:pattern_avoidance_multiple}
	Let~$U$ and~$D$ be disjoint subsets of~$[2,n-1]$ and~$\pi \in \fS_n$. The following statements are equivalent:
	\begin{itemize}
		\itemsep0em
		\item~$\pi$ possesses a reduced word accepted by~$\UU(j)$ for~$j\in U$ and~$\DD(j)$ for~$j\in D$.
		\item~$\pi$ avoids the patterns~$jki$ for~$j\in U$ and~$kij$ for~$j\in D$ with~$i < j < k$.
	\end{itemize}
\end{theorem}

\begin{proof}
	If~$\pi$ possess a reduced word accepted at the same time by all~$\UU(j)$ for~$j\in U$ and~$\DD(j)$ for~$j\in D$, Theorem~\ref{thm:pattern_avoidance_single} gives us our desired pattern avoidance. For the opposite direction consider~$\pi\in\fS_n$ such that~$\pi$ avoids the patterns~$jki$ for~$j\in U$ and~$kij$ and~$j\in D$ with~$i<j<k$. Let~$U'=\{j\in U \,:\,\inv_j(\pi)\neq\emptyset\}$ and~$D'=\{j\in D \,:\,\inv^j(\pi)\neq\emptyset\}$. Using Theorem~\ref{thm:acc_red_word_same_state_inv} we have that any reduced word of~$\pi$ is accepted by~$\UU(j)$ for~$j\in U\setminus{U'}$ and~$\DD(j)$ for~$j\in D\setminus{D'}$. Since these cases are trivial we can assume without loss of generality that~$U=U'$,~$D=D'$ and one of the two is not empty. Say~$U=U'\neq\emptyset$.

	The following is an adaptation of the proof of Theorem~\ref{thm:pattern_avoidance_single} by induction. Let~$j_\circ:=\max(U)$ and~$m$ minimal such that~$j_\circ<m$ and~$\pi^{-1}(j_\circ)>\pi^{-1}(m)$. By said minimality we have that~$\pi^{-1}(l)>\pi^{-1}(m)$ for all~$l\in[j,m-1]$. Therefore,~$\pi$ contains the corresponding subwords~$mj_\circ l$ and~$l\notin D$ for~$l\in[j+1,m-1]$. Notice as well that~$l\notin U$ by maximality of~$j_0$. Now, due to Lemma~\ref{lem:perm_starting_with_sj} the minimality of~$m$ tells us that~$\pi$ can be written as~$\pi=s_{m-1}s_{m-2}\cdots s_{j_\circ}\cdot\tau$.

	Lemma~\ref{lem:automata_multplying_si}\,(2) and Lemma~\ref{lem:automata_multplying_sj}\,(2) give us that \begin{itemize}
		\itemsep0em
		\item~$\tau$ avoids~$jki$ for all~$j\in U\setminus{\{j_\circ\}}$ and~$kij$ for all~$j\in D\setminus{\{m\}}$,
		\item~$\tau$ avoids~$(j_\circ+1)ki$,
		\item~$\tau$ avoids~$kij_\circ$ if~$m\in D$.
	\end{itemize}
	By our induction step, we have that~$\tau$ possesses a reduced word~$w$ that is accepted by all~$\UU(j)$ (resp.~$\DD(j)$) for~$j\in U\setminus{\{j_\circ\}}$ and~$j=j_\circ+1$ (resp.~$j\in D\setminus{\{m\}}$ and~$j=j_\circ$ if~$m\in D$). Lemma~\ref{lem:automata_multplying_si}\,(1) and Lemma~\ref{lem:automata_multplying_sj}\,(1) give us that~$s_{m-1}\cdots s_{j_\circ}\cdot w$ is a reduced word for~$\pi$ accepted by all~$\UU(j)$ for~$j\in U$ and~$\DD(j)$ for~$j\in D$ as we wished.
\end{proof}

\subsection{Multiplying Automata}\label{ssec:multiple_automata_properties}

The current statement of Theorem~\ref{thm:pattern_avoidance_multiple} uses an automaton per element of~$U$ and~$D$. We can however, use the product of all of them to rephrase the theorem in a more compact way. See Definition~\ref{def:aut_product}\index{permutree!automaton~$\PP$} for the details on the product of automata. We denote by \defn{$\PP(U,D)$} the automaton resulting from the product of all automata~$\UU(j)$ if~$j\in U$ and~$\DD(j)$ if~$j\in D$.

\begin{corollary}\label{coro:pattern_avoidance_product}
	Let~$U$ and~$D$ be disjoint subsets of~$[2,n-1]$. The following statements are equivalent:\begin{itemize}
		\itemsep0em
		\item~$\pi$ possesses a reduced word accepted by~$\PP(U,D)$,
		\item~$\pi$ avoids the patterns~$jki$ for~$j\in U$ and~$kij$~$j\in D$ with~$i < j < k$.
	\end{itemize}
\end{corollary}

\begin{definition}
	Following Definition~\ref{def:healthy_ill_dead_states} we say that a state of~$\PP(U,D)$ is \defn{healthy} if its corresponding states in all~$\UU(j)$ for~$j\in U$ and all~$\DD(j)$ for~$j\in D$ are healthy. 	We say that a state of~$\PP(U,D)$ is \defn{ill} if its corresponding sets have at least one ill (resp.\ dead) state and no dead states. A state of~$\PP(U,D)$ is said to be \defn{dead} if it contains at least one dead state.
\end{definition}

Figure~\ref{fig:permutree_automata_product} illustrates the automata~$\PP(\{4\},\{2\})$ when~$n = 5$ (left),~$\PP(\{3\},\{2\})$ for~$n=4$ (middle), and~$\PP(\{2\},\{4\})$ for~$n=5$ (right).
For the first two automata, we draw the complete automata on top, and their skeletons on the bottom.
Here, we call \defn{skeleton} a simplification of the automaton that recognizes the same reduced words.
It is obtained using the fact that the word is rejected as soon as we reach a dead state, and that the automata~$\UU(n)$ and~$\DD(1)$ accept all reduced words.
For the last automaton, the complete product is too big, so we only draw the reachable healthy states.
We color the transitions in red, blue, or purple depending on whether only~$\UU$, only~$\DD$, or both~$\UU$ and~$\DD$ change state.

\begin{figure}
	\centering
	\includegraphics[scale=0.74,angle=90]{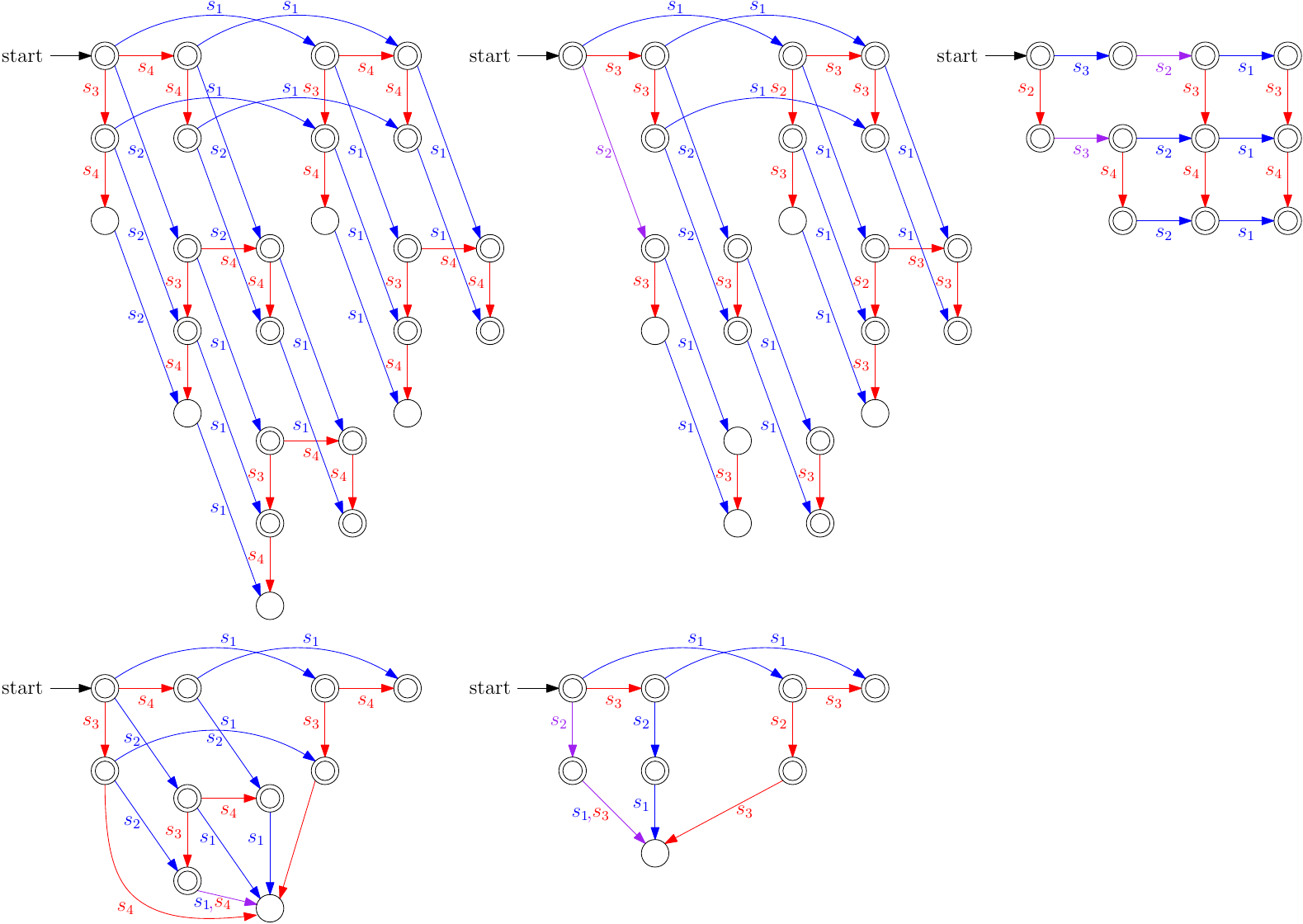}
	\caption[The automaton~$\PP(\{4\},\{2\})$ and its skeleton, the automaton~$\PP(\{3\},\{2\})$ and its skeleton, and the healthy states of the automaton~$\PP(\{2\},\{4\})$.]{ The automaton~$\PP(\{4\},\{2\})$ for~$n = 5$ and its skeleton (left), the automaton~$\PP(\{3\},\{2\})$ for~$n=4$ and its skeleton (middle), and the healthy states of the automaton~$\PP(\{2\},\{4\})$ for~$n=5$ (right).}\label{fig:permutree_automata_product}
\end{figure}

\subsection{The Structure of Accepted Reduced Words of \texorpdfstring{$\PP(j)$}{}}\label{ssec:multiple_automata_words}

Following Section~\ref{ssec:single_automata_words}, we can extend Theorems~\ref{thm:acc_red_word_prefix},~\ref{thm:acc_red_word_algorithm}, and~\ref{thm:acc_red_word_same_state} to~$\PP(U,D)$ as follows.

\begin{theorem}\label{thm:acc_red_word_prefix_multp}
	The set of reduced words accepted by~$\PP(U,D)$ is closed by taking prefixes.
\end{theorem}

\begin{theorem}\label{thm:acc_red_word_algorithm_multp}
	A permutation~$\pi \in \fS_n$ that avoids~$jki$ for~$j\in U$ and~$kij$ for~$j\in D$ and possesses a reduced word starting with~$s_\ell$ such that the transposition~$s_l$ leads to a healthy state of~$\PP(U,D)$, possesses a reduced word starting with~$s_\ell$ and accepted by~$\PP(U,D)$.
\end{theorem}

\begin{theorem}\label{thm:acc_red_word_same_state_multp}
	Given a permutation~$\pi \in \fS_n$, all the reduced words for~$\pi$ accepted by~$\PP(U,D)$ are accepted at the same state.
\end{theorem}

The proofs of these theorems are left out as they are obtained directly from their corresponding single automata versions and the definition of the product of automata.

\subsection{Permutree Sorting}\label{ssec:permutree_sorting}

Corollary~\ref{coro:pattern_avoidance_product} tells us that the set of~$(U,D)$-permutree minimal permutations having reduced words accepted by~$\PP(U,D)$ is non-empty in the case of~$U,D$ being disjoint subsets of~$[2,n-1]$. With this in hand, following the intuition developed in Subsection~\ref{ssec:single_automata_words} and taking into account that the accepted reduced words are closed by prefix as in Theorem~\ref{thm:acc_red_word_prefix_multp}, Theorem~\ref{thm:acc_red_word_algorithm_multp} allows us to think of sorting algorithms following only~$(U,D)$-permutree minimal permutations.

\begin{definition}\label{def:permutree_sorting}
	A \defn{$(U,D)$-permutree sorting algorithm} is a sorting algorithm such that
	\begin{itemize}
		\itemsep0em
		\item applied to a~$(U,D)$-permutree minimal permutation~$\pi$, it only passes through~$(U,D)$-permutree minimal permutations and ends with the identity permutation,
		\item it fails to sort a non~$(U,D)$-permutree minimal permutation~$\pi$.
	\end{itemize}
\end{definition}

\begin{example}\label{ex:stack_sorting_as_permutree_sorting}
	The stack sorting algorithm~\cite{K73} is a~$(\{2, \dots, n-1\}, \emptyset)$-permutree sorting algorithm.
\end{example}

Clearly, we recover Algorithm~\ref{algo:permutree_sorting_simple} as a~$(\{j\},\emptyset)$-permutree sorting algorithm for any~$j \in [2,n-1]$. We now generalize it to a~$(U,D)$-permutree sorting algorithm in a recursive way. As before, the algorithm follows the automaton~$\PP(U,D)$ without explicitly constructing it. Since~$\PP(U,D)$ uses transitions composed by tuples of transpositions, we use the following notation \begin{equation*}
	\moveU(U,\ell) = 
\begin{cases}
U & \text{ if } \ell \notin U, \\
(U \setminus{\{ \ell \}}) \cup \{ \ell + 1 \} & \text{ if } \ell \in U,
\end{cases}
\end{equation*}
\begin{equation*}
	\hspace{0.74cm}\moveD(D,\ell) = 
\begin{cases}
D & \text{ if } \ell + 1 \notin D, \\
(D \setminus{\{ \ell + 1 \}}) \cup \{ \ell \} & \text{ if } \ell + 1 \in D.
\end{cases}
\end{equation*}

\bigskip
\IncMargin{1em}
\SetKwInOut{Input}{Input}\SetKwInOut{Output}{Output}
\SetKwFor{Repeat}{repeat}{}{}
\SetKwIF{If}{ElseIf}{Else}{if}{then}{else if}{else}{}
\SetKwProg{Fn}{Function}{}{}
\DontPrintSemicolon{}
\begin{algorithm}[H]
	\renewcommand{\algorithmcfname}{Algorithm}%
	\Fn{{\rm permutreeSort$(\pi, U, D)$}}{
	\Input{a permutation~$\pi \in \fS_n$ and two disjoint subsets~$U$ and~$D$ of~$[n]$}
	\Output{a reduced word accepted by~$\PP(U,D)$, candidate reduced word for~$\pi$}
	\If{$\exists \; \ell \in [n-1]$ such that~$\ell$ and~$\ell+1$ are reversed in~$\pi$, and~$\ell+1 \notin U$ and~$\ell \notin D$}{
		\Return{}~$s_\ell \cdot \text{permutreeSort}(s_\ell \cdot \pi, \; \moveU(U, \ell), \; \moveD(D,\ell))$ \;
	}
	\If{$\exists \; \ell \in [n-1]$ such that~$\ell$ and~$\ell+1$ are reversed in~$\pi$, and
	\\ ($\ell + 1 \in U$ and~$\pi([\ell+1]) = [\ell +1]$) or ($\ell \in D$ and~$\pi([\ell -1]) = [\ell - 1]$)
	}{
		\Return{}~$s_\ell \cdot \text{permutreeSort}(s_\ell \cdot \pi, \; \moveU(U \setminus{\{ \ell + 1 \}}, \ell), \; \moveD(D \setminus{ \{ \ell \}},\ell))$ \;
	}
	\Return{}~$\varepsilon$ \;
	}
	\caption{$(U,D)$-permutree sorting}
	\label{algo:permutree_sorting_multiple}
\end{algorithm}
\bigskip

\begin{example}\label{ex:permutree_sorting_multple}
	
Table~\ref{tab:permutree_sorting_example_multiple}, has the~$(\{3\},\{2 \})$-permutree sorting of~$3214$,~$1324$ and~$1342$. In Table~\ref{tab:permutree_sorting_example_multiple_2}, we show the~$(\{2\},\{4 \})$-permutree sorting of~$54213$ and~$15342$. Each table contains the permutation~$\pi_i$, the reduced word~$w_i$, the sets~$U_i$ and~$D_i$ and the values of~$\ell_i$ in use at each step of the algorithm together with the value~$k_i$ for which we need to check that~$\pi([k_i])=[k_i]$. A red cross signifies that this last condition has failed. Figure~\ref{fig:permutree_automata_product} contains the corresponding automata~$\PP(\{3\},\{2\})$ and~$\PP(\{2\},\{4\})$.

\begin{table}[h!]
	\centering
	\includegraphics[scale=1]{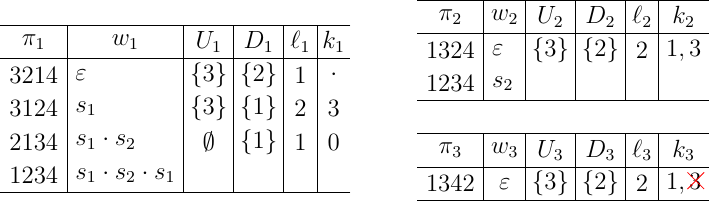}
	\caption[The~$(\{3\},\{2 \})$-permutree sorting of~$3214$,~$1324$, and~$1342$.]{ The~$(\{3\},\{2 \})$-permutree sorting of~$3214$ (left),~$1324$ (top right), and~$1342$ (down right).}\label{tab:permutree_sorting_example_multiple}
\end{table}

\begin{table}[h!]
	\centering
	\includegraphics[scale=1]{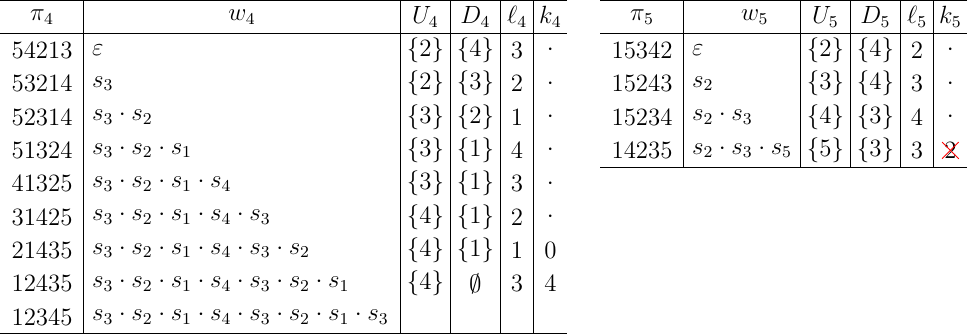}
	\caption[The~$(\{2\},\{4 \})$-permutree sorting of~$54213$ and~$15342$.]{ The~$(\{2\},\{4 \})$-permutree sorting of~$54213$ (left) and~$15342$ (right).}\label{tab:permutree_sorting_example_multiple_2}
\end{table}
\end{example}

\begin{corollary}\label{coro:algorithm_multiple}
	For any permutation~$\pi$ and disjoint subsets~$U,D$ of~$[2,n-1]$, Algorithm~\ref{algo:permutree_sorting_multiple} returns a reduced~word~$w$ accepted by~$\PP(U,D)$ with the property that~$w$ is a reduced word for~$\pi$ if and only if~$\pi$ avoids~$jki$ for~$j\in U$ and~$kij$ for~$j\in D$.
\end{corollary}

\begin{proof}
	First notice that Algorithm~\ref{algo:permutree_sorting_multiple} creates a reduced word following~$\PP(U,D)$. It begins prioritizing healthy states in lines 2 by considering transpositions that transition to other healthy states in all the~$\UU(j)$ or~$\DD(j)$ following Lemma~\ref{lem:automata_multplying_sj}. As we are in the product of said automata, this amounts to updating~$\ell$ to~$\ell+1$ when~$\ell \in U$ and~$\ell+1$ to~$\ell$ when~$\ell+1 \in D$ as in line~3. This repeats for as many transitions as possible until we have to go to an ill state of at least one of the automata. 

	If there is~$\ell+1 \in U$ (resp.~$\ell \in D$) such that~$s_\ell$ is a descent of~$\pi$ and~$\pi([\ell+1]) = [\ell+1]$ (resp.~$\pi([\ell-1]) = [\ell-1]$), then by Theorem~\ref{thm:acc_red_word_same_state_inv}\,(i) any reduced word for~$\pi$ is accepted by the automaton~$\UU(\ell)$ (resp.~$\DD(\ell)$). We can thus start with~$s_\ell$ and forget about the automaton~$\UU(\ell)$ (resp.~$\DD(\ell)$) which is represented in lines 4, 5 and 6. Finally, if none of these options are possible, any reduced word for~$\pi$ leads to a dead state in at least one of the automata, so that~$\pi$ is not~$(U,D)$-sortable. We return the empty reduced word in line 7 in such a case.
\end{proof}

\subsection{Generating Trees}\label{ssec:multiple_automata_trees}

As in Subsection~\ref{ssec:single_automata_trees}, we can also obtain generating trees for the corresponding permutree minimal permutations. Let~$\prec$ be a priority order on~$\{s_1,\ldots,s_{n-1}\}$. Given a~$(U,D)$-permutree minimal permutation~$\pi\in\fS_n$, let \defn{$\pi(U,J,\prec)$} be the~$\prec$-lexicographic minimal reduced word for~$\pi$ that is accepted by~$\PP(j)$ and \begin{equation*}
	\cR(n, U, D, \prec):=\big\{ \pi(U, D, \prec)\,:\, \pi \in \fS_n \text{ is } (U, D)\text{-permutree minimal} \big\}.
\end{equation*}

The proof given in Theorem~\ref{thm:automata_generating_trees_simple} works as well for proving the following theorem, so we skip its proof.

\begin{theorem}\label{thm:automata_generating_trees_multiple}
	The set~$\cR(n, U,D, \prec)$ is closed by taking prefixes.
\end{theorem}

As before, Theorem~\ref{thm:automata_generating_trees_multiple} gives us a generating tree where the root is the empty reduced word and the parent of a reduced word is obtained by deleting the last letter. Taking this tree as the sequence of associated transpositions in the weak order, we obtain a generating tree on~$(U,D)$-permutree minimal permutations as a subgraph of the Hasse diagram of the weak order on~$\fS_n$. Figure~\ref{fig:permutree_automata_multiple_trees} shows all possible trees for~$\fS_4$ and priority order~$s_1 \prec s_2 \prec s_3$.

\begin{figure}
	\centering
	\includegraphics[scale=0.75,angle=90]{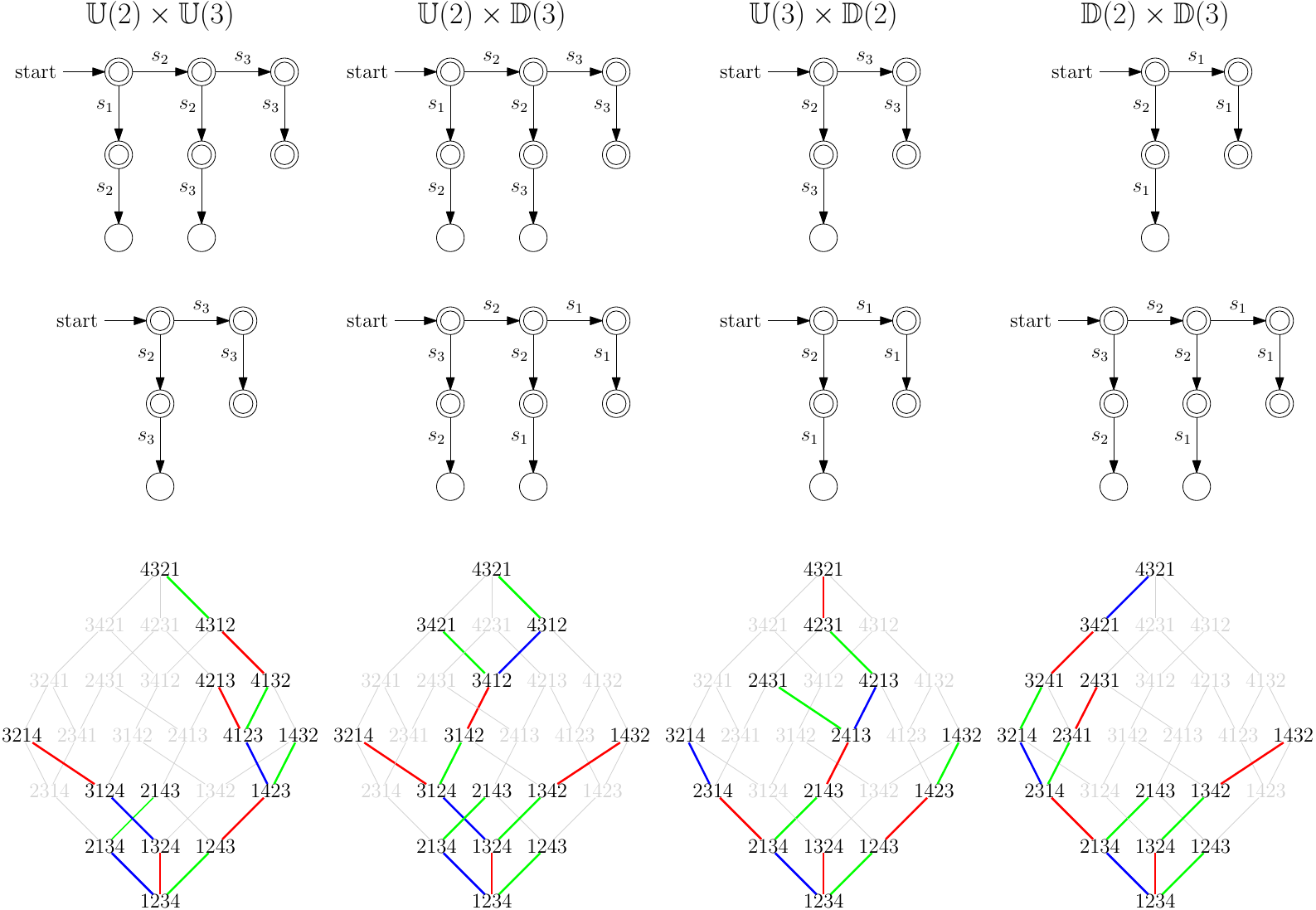}
	\caption{Generating trees for the~$(U,D)$-permutree minimal permutations of~$\fS_4$, with priority order~${\color{blue}s_1} \prec {\color{red}s_2} \prec {\color{green}s_3}$.}\label{fig:permutree_automata_multiple_trees}
\end{figure}

\section{Coxeter Sorting via Automata}\label{sec:coxeter_sorting_automata}

Notice that in Section~\ref{sec:multiple_automata}, the biggest our sets~$U$ and~$D$ could go without a permutation failing to have a common reduced word accepted by all the automata was whenever~$U$ and~$D$ formed a partition of~$[2,n-1]$. Recalling Remark~\ref{rem:U_D_as_orientations} and Definition~\ref{def:Coxeter_Cambrian_congruence}, we see that this case is precisely when~$(U,D)$-permutree congruences are Cambrian congruences in type~$A$ and the number of~$(U,D)$-permutree minimal elements is the Catalan number~$C_n$. This section is devoted to explore this connection with the definitions and technology developed in Subsection~\ref{ssec:Coxeter_sorting}.

\begin{definition}\label{def:U_D_coxeter_element}
	Let~$c$ be a Coxeter element of~$(\fS_n,S)$. \defn{$U_c$} and \defn{$D_c$} are the partition of~$[2,n-1]$ where~$j\in U_c$ (resp.~$j\in D_c$) if~$s_j$ appears before (resp.\ after)~$s_{j-1}$ in~$c$.
\end{definition}

\begin{remark}\label{rem:U_D_automata_coxeter_element}
	Another equivalent way of seeing the construction of~$U_c$ (resp.~$D_c$) is that~$j\in U_c$ (resp.~$j\in D_c$) if~$c$ is accepted by~$\UU(j)$ but not by~$\DD(j)$ (resp.\ is accepted by~$\DD(j)$ but not by~$\UU(j)$).
\end{remark}

Coxeter sorting and~$(U_c,D_c)$-permutree sorting are related by the following theorem.

\begin{theorem}\label{thm:coxter_sorting_permutrees}
	Let~$c$ be a Coxeter element and~$\pi\in\fS_n$. The following statements are equivalent:
	\begin{enumerate}[(i)]
		\itemsep0em
		\item $\pi$ is~$c$-sortable,
		\item the~$c$-sorting word~$\pi(c)$ is accepted by the automaton~$\PP(U_c, D_c)$,
		\item there exists a reduced word for~$\pi$ accepted by the automaton~$\PP(U_c, D_c)$,
		\item for each~$j \in \{2, \dots, n-1\}$, there exists a reduced word for~$\pi$ that is accepted by the automaton~$\UU(j)$ if~$j \in U_c$ and~$\DD(j)$ if~$j \in D_c$,
		\item~$\pi$ avoids~$jki$ for~$j \in U_c$ and~$kij$ for~$j \in D_c$.
	\end{enumerate}
\end{theorem}

Theorem~\ref{thm:pattern_avoidance_multiple} and Corollary~\ref{coro:pattern_avoidance_product} give equivalences~$(iii)\iff (iv)\iff (v)$. In the following lemmas we show the equivalences~$(i)\iff (ii)\iff (iii)$. As a general assumption for the rest of this section let~$c$ be a Coxeter element and~$\pi\in\fS_n$. We begin by showing~$(i)\iff(ii)$.

\begin{lemma}\label{lem:c_word_is_accepted_by_product}
	If~$\pi$ is~$c$-sortable, its~$c$-sorting word~$\pi(c)$ is accepted by~$\PP(U_c,D_c)$.
\end{lemma}

\begin{proof}
	Let~$\pi(c)$ be the~$c$-sorting word. It is enough to see that~$\pi(c)$ is accepted by~$\UU(j)$ for~$j\in U_c$ and~$\DD(j)$ for~$j\in D_c$. We proceed to show it for~$\UU(j)$ since the case for~$\DD(j)$ is similar. We have two possible cases:\begin{itemize}
		\item if~$\pi(c)$ does not contain~$s_j$, then~$\pi(c)$ either remains in the first healthy state of~$\UU(j)$ or ends in the first ill state of~$\UU(j)$.
		\item if~$\pi(c)$ does contain~$s_j$, by Lemma~\ref{lem:Coxeter_sorting_word_properties}\,(1) we have that~$s_j$ appears before~$s_{j-1}$ in~$\pi(c)$, and~$\pi(c)$ moves onto the second healthy state of~$\UU(j)$. Notice that because of the recursive construction of~$\UU(j)$,~$\pi(c)$ can now only end in a dead state if it contains as a reduced subword~$s_{l}V^ls_{l}s_{l+1}$ where~$V^l$ is a reduced word with transpositions in~$S\setminus{\{s_{l},s_{l+1}\}}$ such that~$s_{l-1}\in V^l$. Such containment is impossible due to Lemma~\ref{lem:Coxeter_sorting_word_properties}\,(2).
	\end{itemize}
	Thus, in both cases~$\pi(c)$ is accepted by~$\UU(j)$.
\end{proof}

\begin{lemma}\label{lem:c_word_accepted_product_implies_sortable}
	If the~$c$-sorting word~$\pi(c)$ is accepted by~$\PP(U_c,D_c)$, then~$\pi$ is~$c$-sortable.
\end{lemma}

\begin{proof}
	Assume that~$\pi$ is not~$c$-sortable and let us find an automaton that rejects~$\pi(c)$ between the~$\UU(j)$ and~$\DD(j)$ that define~$\PP(U_c,D_c)$. To do this we follow an induction on the length of~$\pi$ and the size of~$c$.

	Consider~$s_\ell$ to be the first letter of~$c$ and write~$c=s_\ell\cdot d$. As~$s_\ell$ appears before both~$s_{\ell-1}$ and~$s_{\ell+1}$ we have that~$\ell\in U_c$ and~$\ell+1\in D_c$. Notice that~$s_\ell$ transitions from the initial state to the second healthy state of both~$\UU(\ell)$ and~$\DD(\ell+1)$ and stays in the initial state of all~$\UU(j)$ for~$j\in U_c\setminus{\{\ell\}}$ and~$\DD(j)$ for~$j\in D_c\setminus{\{\ell+1\}}$. We now have two possible cases: \begin{itemize}
		\itemsep0em
		\item if~$\pi$ reverses~$\ell$ and~$\ell+1$ and can be written as~$\pi=s_\ell\cdot\tau$. By Lemma~\ref{lem:coxeterElementFacts} we have that the~$c$ sorting word factorizes as~$\pi(c)=s_\ell\cdot \tau(d\cdot s_\ell)$ and~$\tau$ is not~$d\cdot s_\ell$-sortable. By our induction we have that~$\tau(d\cdot s_\ell)$ is a reduced word rejected by an automaton between~$\UU(j)$ for~$j\in U_{d\cdot s_\ell}$ and~$\DD(j)$ for~$j\in D_{d\cdot s_\ell}$. Notice that~$U_{d\cdot s_\ell}=U_c\Delta\{\ell,\ell+1\}$ and~$D_{d\cdot s_\ell}=D_c\Delta\{\ell,\ell+1\}$. Because of this, depending on the value for~$j$ for which the automaton rejects~$\tau(d\cdot s_\ell)$ either Lemma~\ref{lem:automata_multplying_si} or Lemma~\ref{lem:automata_multplying_sj} ensures that~$\pi(c)=s_\ell\cdot\tau(d\cdot s_\ell)$ is also rejected by said automaton.
		\item if~$\pi$ does not reverse~$\ell$ and~$\ell+1$ then~$\pi$ not being~$c$ sortable is equivalent to~$\pi$ not being~$d$-sortable in~$W_{\langle s_\ell\rangle}$. By induction on the size of~$c$,~$\pi(d)$ is not accepted by~$\PP(U_d,D_d)$. Since~$\pi(c)=\pi(d)$ we have that~$\pi(c)$ is rejected by~$\PP(U_c,D_c)$. We finish as~$U_d\subseteq U_c$ and~$D_d\subseteq D_c$. \qedhere
	\end{itemize}
\end{proof}

Moving to~$(ii)\iff(iii)$ notice that~$(ii)\Rightarrow(iii)$ follows immediately from the fact that~$\pi(c)$ is always a reduced word for~$\pi$. The opposite direction is a bit more involved.

\begin{lemma}\label{lem:c_word_is_accepted_product}
	If~$\pi$ possesses a reduced word accepted by~$\PP(U_c,D_c)$, then its~$c$-sorting word~$\pi(c)$ is accepted by~$\PP(U_c,D_c)$.
\end{lemma}

\begin{proof}
	We proceed by induction on the length of~$\pi$ and the size of~$c$. As before, consider~$s_\ell$ to be the first letter of~$c$ and write~$c=s_\ell\cdot d$. As~$s_\ell$ appears before both~$s_{\ell-1}$ and~$s_{\ell+1}$ we have that~$\ell\in U_c$ and~$\ell+1\in D_c$. Notice that~$s_\ell$ transitions from the initial state to the second healthy state of both~$\UU(\ell)$ and~$\DD(\ell+1)$ and stays in the initial state of all~$\UU(j)$ for~$j\in U_c\setminus{\{\ell\}}$ and~$\DD(j)$ for~$j\in D_c\setminus{\{\ell+1\}}$. We now have two possible cases: \begin{itemize}
		\itemsep0em
		\item if~$\pi$ reverses~$\ell$ and~$\ell+1$ and can be written as~$\pi=s_\ell\cdot\tau$. Notice that the reduced word of~$\pi$ that is accepted by~$\PP(U_c,D_c)$ might not be related with~$\tau$. Still, since~$\pi$ has an accepted reduced word by~$\PP(U_c,D_c)$, we have that~$\pi$ avoids the patterns corresponding to the elements of~$U$ and~$D$. With this in hand, Theorem~\ref{thm:acc_red_word_algorithm} tells us that~$\pi$ possesses a reduced word of the form~$w=s_\ell\cdot v$ such that~$w$ is accepted by~$\UU(j)$ for~$j\in U_c$ and~$\DD(j)$ for~$j\in D_c$. Notice that in our context the use of the Theorem~\ref{thm:acc_red_word_algorithm} fails for~$\UU(\ell+1)$ and~$\DD(\ell)$ but as~$\ell\in U_c$ and~$\ell+1\in D_c$ these automata do not affect us.

		      At this point Lemma~\ref{lem:automata_multplying_si} (resp.\ Lemma~\ref{lem:automata_multplying_sj}) gives us that~$\tau$ possesses a reduced word accepted by~$\UU(j)$ for~$j\notin\{\ell,\ell+1\}$ (resp.~$\UU(j+1)$). This coincides with the fact that~$s_{j-1}$ and~$s_{j}$ have not changed order between (resp.~$s_\ell$ appears now after~$s_{l-1}$ and~$s_{l+11}$ after)~$c=s_\ell \cdot d$ to~$d\cdot\ell$. Thus, by induction we have that~$\tau(d\cdot s_\ell)$ is accepted by~$\PP(U_{d\cdot s_\ell},D_{d\cdot s_\ell})$. Using Lemma~\ref{lem:coxeterElementFacts} gives us that~$\pi(c)=s_\ell \cdot \tau(d\cdot s_\ell)$ is accepted by~$\PP(U_c,D_c)$.

		\item if~$\pi$ does not reverse~$\ell$ and~$\ell+1$. We claim that in this case~$\pi$ does not have~$s_\ell$ its reduced words. Notice that this is equivalent to showing that~$\pi([\ell])=[\ell]$ and~$\pi([n]\setminus{[\ell]})=[n]\setminus{[\ell]}$. Suppose that the accepted reduced word~$w$ by~$\PP(U_c,D_c)$ contains~$s_\ell$. As~$\ell\in U_c$ (resp.~$\ell+1\in D_c$) and~$\pi$ does not reverse~$\ell$ and~$\ell+1$, the fact that~$w$ is accepted by~$\UU(\ell)$ (resp.~$\DD(\ell+1)$) means that~$w$ has~$s_{\ell+1}$ followed eventually by~$s_\ell$ before any~$s_{\ell-1}$ (resp.~$s_{\ell-1}$ followed eventually by~$s_\ell$ before any~$s_{\ell+1}$). This implies that~$\pi$ contains the pattern~$k\ell$ for some~$\ell<k$ (resp.~$(\ell+1) i$ for some~$i<\ell+1$). As~$\ell$ and~$\ell+1$ are not reversed,~$\pi$ contains~$k\ell(\ell+1)i$ contradicting twice Theorem~\ref{thm:pattern_avoidance_single}, and we can work as in~$W_{\langle s_\ell\rangle}$ via induction.

		      Notice that the automata corresponding to~$d$ are all the automata of~$c$ without the transitions having~$s_\ell$ and thus any reduced word accepted or rejected for~$j\in U_d$ (resp.~$j\in D_d$) is also accepted or rejected for~$j\in U_c$ (resp.~$j\in U_c$). Notice as well that this also aligns with~$\ell,\ell+1\notin U_d$ (resp.~$\ell,\ell+1\notin D_d$). Since~$\pi$ does not use~$s_\ell$ in any reduced word, the reduced word accepted by~$\PP(U_c,D_c)$ is also accepted by~$\PP(U_d,D_d)$. By induction, we have that~$\pi(d)=\pi(c)$ is accepted by~$\PP(U_d,D_d)$ and thus also by all~$\UU(j)$ for~$j\in U_c\setminus{\{\ell\}}$ and~$\DD(j)$ for~$j\in D_c\setminus{\{\ell+1\}}$. As~$\pi(c)$ also does not use any~$s_\ell$, we have that~$\pi(c)$ is accepted in either the first healthy or ill state of both~$\UU(\ell)$ and~$\DD(\ell+1)$. We conclude that~$\pi(c)$ is accepted by~$\PP(U_c,D_c)$.  \qedhere
	\end{itemize}
\end{proof}

We move to give some remarks about nuances of~$c$-sortability and~$(U_c,D_c)$-permutree sorting.

\begin{remark}
	The fact that a permutation~$\pi$ avoids~$jki$ (resp.~$kij$) for a given~$j$, does not assure that there exists a Coxeter element~$c$ for which~$\pi$ is~$c$-sortable and~$j \in U_c$ (resp.~$j \in D_c$). For example, let~$\pi:=41352$ that avoids~$2ki$ and~$ki4$. Notice that~$\pi$ is not sortable since it contains both~$3ki$ and~$ki3$ via the respective subwords~$352$ and~$413$.
\end{remark}

\begin{remark}
	Lemma~\ref{lem:c_word_is_accepted_by_product} fails when a permutation~$\pi$ is not~$c$-sortable as there might exist~$j \in U_c$ (resp.~$j \in D_c$) for which the~$c$-sorting word~$\pi(c)$ is not accepted by~$\UU(j)$ (resp.~$\DD(j)$) even if~$\pi$ avoids~$jki$ (resp.~$kij$).
	For example, take~$c = s_2 \cdot s_1 \cdot s_3$ and~$\pi = 4213 = s_3 \cdot s_1 \cdot s_2 \cdot s_1 =  s_3 \cdot s_2 \cdot s_1 \cdot s_2 =  s_1 \cdot s_3 \cdot s_2 \cdot s_1$.
	Then~$2 \in U_c$, and the~$c$-sorting word~$\pi(c) = s_1 \cdot s_3 \cdot s_2 \cdot s_1$ is rejected by~$\UU(2)$ while~$s_3 \cdot s_2 \cdot s_1 \cdot s_2$ is accepted by~$\UU(2)$.
\end{remark}

\begin{remark}
	Notice that in the definition of~$c$-sorting, the infinite word~$c^\infty$ is a sorting network as in~\cite{K73}. That is, Our algorithm rely on a series of transpositions to be applied at the appropriate time. With this in mind and taking into account Theorem~\ref{thm:coxter_sorting_permutrees}, it is important to notice that Algorithm~\ref{algo:permutree_sorting_multiple} is not a sorting network. The order on which the values of~$\ell$ are chosen depending on the state of the automata that is being visited and the permutation itself. This demands the following perspective.
\end{remark}

\begin{perspective}\label{pers:sorting_network}
	Let~$U$ and~$D$ be disjoint subsets that do not cover~$[2,n-1]$. Is it possible to find a word~$\tilde{c}$ replacing~$c^\infty$ such that checking for~$\pi(\tilde{c})$ is enough to verify if~$\pi$ is accepted by~$\PP(U,D)$?
\end{perspective}

We finish with a remark on both positive and negative cases for Perspective~\ref{pers:sorting_network}.

\begin{remark}\label{rem:sorting_network}
	Consider the case~$n = 5$,~$U = \{ 2\}$, and~$D=\{ 4 \}$. In this case one can check via computer exploration that no reduced word~$\tilde{c}$ of the maximal permutation~$54321$ can be used as a sorting network. This implies that for all general choices of~$\tilde{c}$, there exists a permutation~$\pi$ which is accepted by~$\PP(U,D)$ although its reduced word~$\pi(\tilde{c})$ is rejected. The healthy states of~$\PP(\{ 2\},\{ 4 \})$ are shown in Figure~\ref{fig:permutree_automata_product}. In this case the problem lies in that certain accepted reduced words can only start with either~$s_2$ or~$s_3$. For example, for~$54213$ as shown in Example~\ref{ex:permutree_sorting_multple} all accepted reduced words start with~$s_3$ whereas for some other permutations such as~$35421$, all accepted reduced words start with~$s_2$. This gives us no possible choice for~$\tilde{c}$.

	Still the answer is positive in the Cambrian case with the~$c$-sorting word when~$U$ and~$D$ form a partition of~$\{2, \dots, n-1\}$ or when~$|U| + |D| = 1$ and we have a single automaton. In this later case,~$\tilde{c}$ is constructed by reading the healthy states of the automaton from left to right, adding at each state the word~$(s_{i_1} \cdots s_{i_k})^k s_j$ where~$s_{i_1}, \dots, s_{i_k}$ are the loops and~$s_j$ is the unique transition going to the next healthy state. This process gives a prefix that can be extended in any way to obtain a proper sorting word~$\tilde{c}$ (that is, a reduced word of~$w_0$). With this construction, if~$U = \{2\}$, we obtain the prefix~$s_3 \cdot s_2 \cdot s_1 \cdot s_3$ and indeed~$s_3 \cdot s_2 \cdot s_1 \cdot s_3 \cdot s_2 \cdot s_1$ acts as a sorting network equivalent to~$(\{2\},\emptyset)$-permutree sorting. This seems to extend to all cases where, at each healthy state of the~$\PP(U,D)$, the choices for the healthy transitions commute like for~$n=5$,~$U = \{ 4 \}$ and~$D = \{ 2 \}$ (see Figure~\ref{fig:permutree_automata_product}). In this case the word~$s_1 \cdot s_2 \cdot s_4 \cdot s_3 \cdot s_2 \cdot s_1 \cdot s_4 \cdot s_3 \cdot s_2$ gives a proper sorting network for~$(\{4\},\{2\})$-permutree sorting.
\end{remark}

\section{Permutree Sorting other Infinite Families}\label{sec:other_infinite_families}

Based on Chapter~\ref{sec:coxeter_sorting_automata} and~\cite{HLT11} and~\cite{D22} we strongly believe that the results of permutree sorting can be extended from type~$A$ to the more general case of Coxeter Groups. In this section we present some observations on recognizing minimal elements of permutree congruences through automata in other Coxeter group based on computer exploration using SageMath~\cite{SAGE}. We first propose some automata for Types~$B$ and~$D$ and then present some definitions and directions that seem promising.

\subsection{Type~\texorpdfstring{$B$}{}}\label{ssec:type_B_sorting}

For this subsection consider the Coxeter system of type~$B$ given by the group of signed permutations~$(\fS_n^B,S^B)$. For this case we have the general definition of automata as follows.

\begin{definition}\label{def:permutree_automata_single_B}
	Consider~$U=\{j\}$ (resp.~$D=\{j\}$) for some~$j\in[2,n-1]$ and the set of generators~$S^B$ as an alphabet. We define the automaton \defn{$\UU(j)$}\index{permutree!automaton~$\UU$} (resp.\ \defn{$\DD(j)$}\index{permutree!automaton~$\DD$}) recursively following Figure~\ref{fig:permutree_automata_single_recursive_B} with automata~$\DD(0)$ defined for consistency. Our automata are complete with all missing transitions being loops. Figure~\ref{fig:permutree_automata_single_full_B} shows the complete automata~$\UU(j)$ and~$\DD(j)$.
\end{definition}

\begin{figure}[h!]
	\centering
	\includegraphics[scale=0.84]{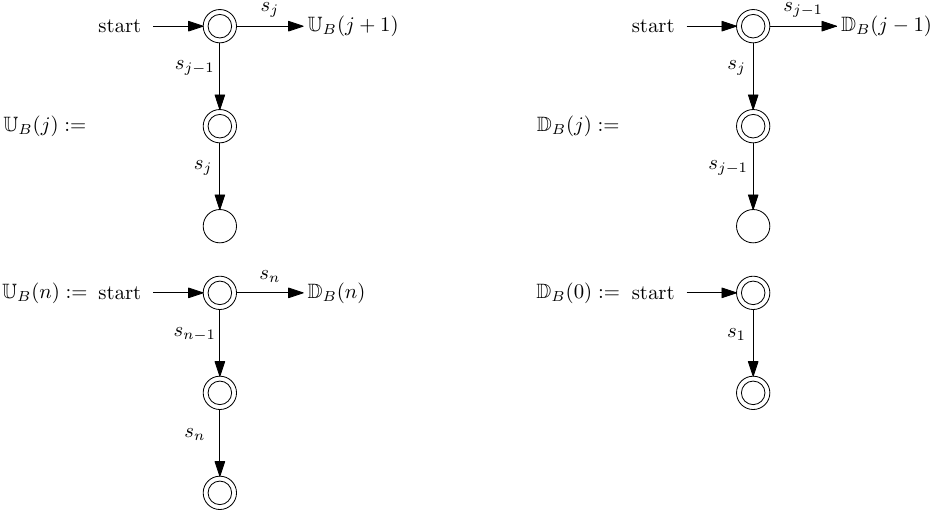}
	\caption[The automata~$\UU_B(j)$ and~$\DD_B(j)$ defined recursively.]{ The automata~$\UU_B(j)$ (left) and~$\DD_B(j)$ (right) defined recursively.}
	\label{fig:permutree_automata_single_recursive_B}
\end{figure}

\begin{figure}[h!]
	\centering
	\includegraphics[scale=0.9]{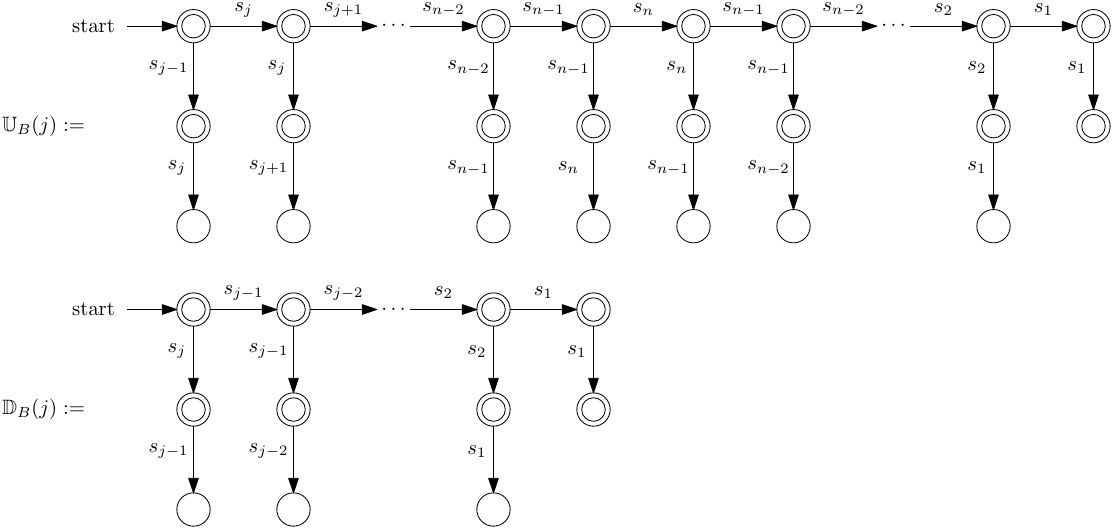}
	\caption{The complete automata~$\UU_B(j)$ and~$\DD_B(j)$.}\label{fig:permutree_automata_single_full_B}
\end{figure}

\begin{remark}\label{rem:type_B_automata_folding}
	Like in Remark~\ref{rem:automata_and_walks_cox_graph}, the automata~$\UU_B(j)$ and~$\DD_B(j)$ can be built via a walk in the Coxeter graph in the opposite direction of the defining orientation with a slight twist. The walk corresponding to~$\DD_B(j)$ ends at~$s_1$ similarly as in type~$A$ while the walk corresponding to~$\UU_B(j)$ ends at~$s_n$ which has an edge with label~$4$ meaning it must go back and travel the graph to the other end making~$\DD_B(j)$ appear. As in other apparitions of Type~$B$, this reflects idea that the Coxeter group~$B_n$ can be thought as folding of the Coxeter groups~$A_{2n}$ along the central symmetry of its Coxeter diagram (see~\cite{HL07}).
\end{remark}

Following in Sections~\ref{sec:single_automata} and~\ref{sec:multiple_automata} we have verified computationally up to~$n=6$ that these automata correctly sort~$(U,D)_B$-permutree minimal elements for all possible combinations of~$U,D\subseteq[2,n-1]$ such that~$U\cap D=\emptyset$. Moreover, the sets of reduced words of accepted elements by these automata and their products also form trees inside the weak order of~$\fS_n^B$. In the case that~$U$ and~$D$ form a partition of~$[2,n-1]$ we correctly recover the fact that there are~$\binom{2n}{n}$ elements that are~$(U,D)_B$-minimal permutree, that is, the~$B$-Catalan number.

\subsection{Type~\texorpdfstring{$D$}{}}\label{ssec:type_D_sorting}

For type~$D$ we do not possess currently even a candidate definition for the general construction of the automata~$\UU_D(j)$ and~$\DD_D(j)$. In certain cases of a single orientation in the Coxeter diagram of~$D_n$ we have found an automaton that recognizes~$(U,D)_D$-permutree minimal elements. To construct it seems to be useful to consider the guiding ideas of Remarks~\ref{rem:automata_and_walks_cox_graph} and~\ref{rem:type_B_automata_folding} by considering a walk on the Coxeter graph with the twist that a fork of the walk at~$s_2$ implies a forking in the spine of the automata and that walks must bounce at least once. This forking of the spine is equivalent to considering several automata and taking their product. We choose this representation for our figures of these automata.

For the case of~$D_4$ we have a family of automata that have been computationally checked to recognize permutree minimal permutations in the corresponding congruences. Figure~\ref{fig:permutree_automata_D4_single_full} shows the automata for the orientations~$2\to 0$ and~$0\to 2$ for~$D_4$. As~$D_4$ is symmetric via rotations, the automata for the other single orientations are obtained by rotating the labels of these diagrams according to the desired orientation around the center vertex corresponding to~$s_2$. For~$n=5$ certain orientations still elude us and for~$n\geq 6$ we have no further indications.

\begin{figure}[h!]
	\centering
	\includegraphics[scale=0.7]{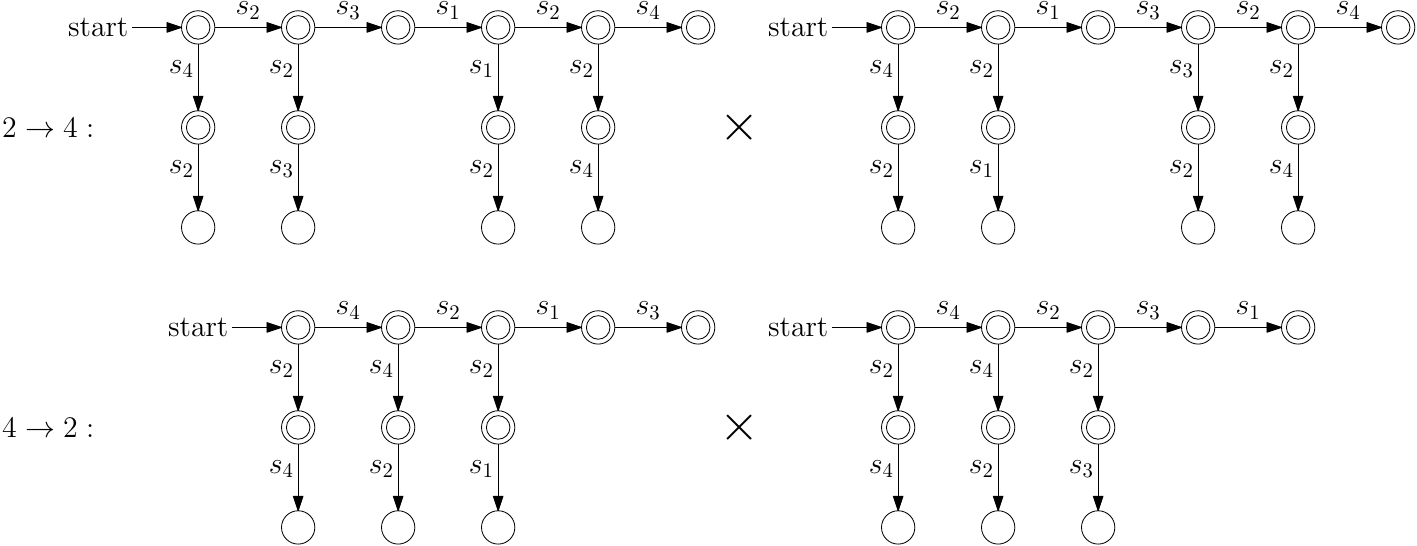}
	\caption[The automata~$\UU_D(0)$ and~$\DD_D(0)$ in~$D_4$.]{ The automata corresponding to the orientations~$0\to 2$ and~$2\to 0$ in the Coxeter diagram of~$D_4$.}
	\label{fig:permutree_automata_D4_single_full}
\end{figure}

We finish this chapter by providing certain definitions that seem to have a hand in determining automata for general Coxeter groups.

\subsection{Ideas and Perspectives}\label{ssec:ideas_perspectives}

Taking inspiration from~\cite{HLT11} and their work on~$c$-singletons, we propose the following definitions.

\begin{definition}\label{def:maccr_smaccr}
	Let~$(W,S)$ be a Coxeter group of rank~$n$ and~$\equiv_\delta$ with~$\delta\in\{\nonee,\downn,\upp,\uppdownn\}^n-1$ the permutree equivalence class corresponding to the multi-orientation of the Coxeter graph. Consider an element~$w\in W$. We say that~$w$ is \begin{itemize}
		\itemsep0em
		\item a \defn{$\delta$-singleton}\index{permutree!singleton} if~$w$ is the only element in its permutree equivalence class,
		\item a \defn{$\delta$-accepter}\index{permutree!accepter} if for all~$w\leq u$, we have that~$u$ is a~$\delta$-singleton,
		\item a \defn{$\delta$-minimal accepter}\index{permutree!minimal accepter} (abbreviated \defn{$\delta$-maccr}) if~$w$ is an accepter and there is no~$s\in S$ such that~$w\lessdot w\cdot s$ and~$w\cdot s$ is an accepter,
		\item a \defn{$\delta$-smaccr} if~$w$ is a maccr and has the shortest length across all maccrs.
	\end{itemize}
\end{definition}

The use of~$\delta$ in these definitions serves only to recall that they depend on the congruence~$\equiv_\delta$ similar to how~$c$-singletons depend on the congruence~$\equiv_c$. We do not propose by this notation that~$\delta$ describes a reduced word of~$w_0$ from which the permutree sortable elements can be derived, as it is the case for the Coxeter element~$c$. We have seen in Remark~\ref{rem:sorting_network} that such reduced word might not always exist for the~$\delta$ case. 

Based on computational evidence we propose the following conjectures. Let~$\equiv_\delta$ be any permutree congruence of a Coxeter group~$W$.

\begin{conjecture}\label{conj:unique_smaccr}
	There exists a unique~$\delta$-smaccr in the left weak order of~$W$. Moreover, the~$\delta$-smaccr is the join of all the smaccrs coming from single orientations.
\end{conjecture}

\begin{conjecture}\label{conj:smaccr_automata}
	For each single orientation associated to~$\delta$, each reduced word of the~$\delta$-smaccr defines the start of the spine of an automaton such that the product of these automata recognizes the permutree minimal elements of this orientation. 

	Furthermore, the product of these automata for each single orientation associated to~$\delta$ is an automaton that recognizes~$\delta$-permutree minimal elements.
\end{conjecture}

\begin{conjecture}\label{conj:maccrs_from_smaccr}
	Let~$x\in W$ be a~$\delta$-smaccr. For each~$\delta$-maccr~$y\in W$ there exists reduced words~$u$ and~$v$ of~$x$ and~$y$ such that~$u$ is a suffix of~$v$. That is, the~$\delta$-smaccr is the smallest~$\delta$-maccr in the left weak order of~$W$.
\end{conjecture}

Given a congruence~$\equiv_\delta$, the reduced words of the~$\delta$-smaccr seem to be good candidates from which to start defining automata whose product recognizes the corresponding~$\delta$-permutree minimal elements. Example~\ref{ex:H3_singleton_automata} shows this for an orientation in~$H_3$.

\begin{example}\label{ex:D4_singleton_automata}
	In~$D_4$ the smaccr of the orientation for~$2\to 4$ (resp.~$4\to 2$) has as set of reduced words~$\{s_2s_3s_1s_2,s_2s_1s_3s_2\}$ (resp.~$\{s_4s_2s_1s_3,s_4s_2s_3s_1\}$). The spines of the automata shown in Figure~\ref{fig:permutree_automata_D4_single_full} start with these reduced words. 
\end{example}

\begin{example}\label{ex:H3_singleton_automata}
	In~$H_3$ the smaccr of the orientation~$2\to 3$ has as set of reduced words the pair~$\{s_3s_2s_1s_3s_2s_3s_2,s_3s_2s_3s_1s_2s_3s_2\}$. The spines of the automata shown in Figure~\ref{fig:permutree_automata_H3_single_full} are formed with these reduced words plus a transition~$s_1$.
	\begin{figure}[h!]
		\centering
		\includegraphics[scale=0.825]{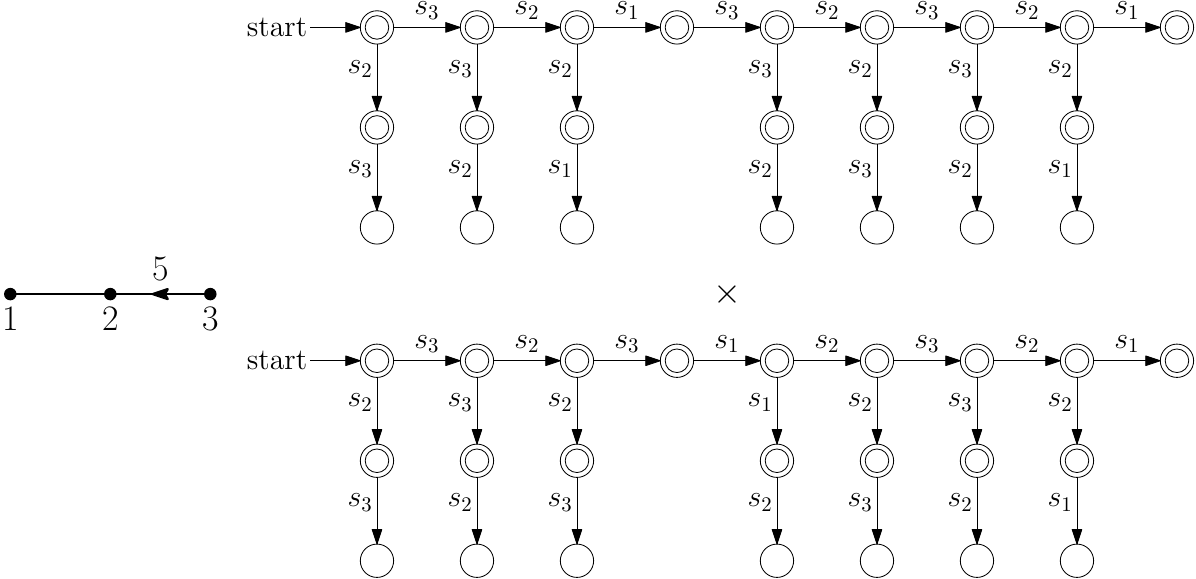}
		\caption{ A single orientation of~$H_3$ and a product of automata that recognizes the minimal elements.}
		\label{fig:permutree_automata_H3_single_full}
	\end{figure}
\end{example}

\begin{remark}
	Let~$\delta$ denote a permutree congruence of a Coxeter group~$W$ corresponding to a single orientation in the Coxeter graph. We have verified the existence of a unique~$\delta$-smaccr for the following Coxeter groups. We list them together with the amount of reduced words they have in each case.
	\begin{itemize}
		\itemsep0em
		\item In~$A_n$ ($n\leq 5$) it has~$1$ reduced word.
		\item In~$B_n$ ($n\leq 4$) it has~$1$ reduced word.
		\item In~$D_4$ it has~$2$ reduced words.
		\item In~$D_5$ it can have anywhere between~$2,5,46,100$ reduced words.
		\item In~$H_3$ it has~$1$ reduced word except for the orientation~$1\to2$ where it has 2.
		\item In~$F_4$ it has 2 reduced words except for any orientation of the middle edge, in which case it has 12. 
	\end{itemize}
\end{remark}

%% file: includes/contenu/chap_sorder_flows.tex

\chapter{Flow Polytopes and Tropical Geometry}\label{chap:sorder_Flows}

\addcontentsline{lof}{part}{\protect\numberline{\thepart}Flow Polytopes}

\addcontentsline{lof}{chapter}{\protect\numberline{\thechapter}Flow Polytopes and Tropical Geometry}

In this chapter we present the preliminaries of flows on graphs and tropical geometry needed for our work on Conjecture~\ref{conj:s-permutahedron}. Each section consists of the main tools known in the corresponding literature together with certain refinements appearing in~\cite{GMPTY23}.

\section{Flows on Graphs}\label{sec:flows_on_graphs}

This section consists of the basic notions of the combinatorics and geometry of flows on graphs. We base our presentation of these topics mainly on~\cite{MM19} and~\cite{MMS19}. In what follows, all graphs are connected loopless digraphs.

\subsection{Graphs and Flows}\label{ssec:graphs_and_flows}

\begin{definition}\label{def:flows}
	Let~$G=(V,E)$ be a graph with vertex set~$V=\{v_0,\ldots,v_n\}$, and a multiset of edges~$E$ where each edge~$(v_i,v_j)\in E$ is oriented~$v_i\to v_j$ if~$i<j$. We respectively denote by \defn{$\cI_i$} and \defn{$\cO_i$} the set of incoming edges and outgoing edges of the vertex~$v_i$.

	A \defn{netflow}\index{graph!netflow} is a vector~$\mathbf{x}=(x_0,\ldots,x_{n})\in\ZZ^{n+1}$ such that~$\sum_{i=0}^{n}x_i=0$. Given a netflow~$\bfa = (a_0,\ldots,a_{n-1},-\sum_{i=0}^{n-1} a_i)$ with~$a_i\in\ZZ_{\geq0}$, a \defn{$\bfa$-flow}\index{graph!flow} of~$G$ is a function~$f: E\to\RR_{\geq 0}$ such that \begin{equation}
		\sum_{e\in \cI_i} f(e) + a_i = \sum_{e\in \cO_i} f(e)
	\end{equation}\label{eq:conservation_of_flow} for all~$i\in [1,n-1]$. In particular, if~$f(e)\in\ZZ_{\geq0}$ for all~$e\in G$, then~$f$ is called an \defn{integer flow}\index{graph!flow!integer} of~$G$. We also call a function~$f:E\to \RR_{\geq 0}$ that satisfies Equation~\ref{eq:conservation_of_flow} an \defn{admissible flow}\index{graph!flow!admissible}. The \defn{support}\index{graph!flow!support} of a flow~$f$ is~$\supp(f)=\{e\in E\,:\,f(e)\neq 0\}$.
\end{definition}

See Figure~\ref{fig:graph_flows} for an example of flows on a graph.

\begin{remark}
	Due to our definition of flows, the netflows that we consider always have a unique sink being vertex~$v_n$. As well, whenever we think of flows and netflows, our graph~$G$ is assumed to have a non-zero amount of incoming and outgoing edges for the vertices~$\{v_1,\ldots,v_{n-1}\}$. This assumption is out of simplicity and not a requirement. If we had such a vertex with no incoming (resp.{} outgoing) edges, any flow corresponding to an edge in~$\cI_i$ (resp.{}~$\cO_i$) would be zero. Therefore, we can always assume we are working modulo such vertices.
\end{remark}

\begin{figure}[h!]
	\centering
	\includegraphics[scale=1.455]{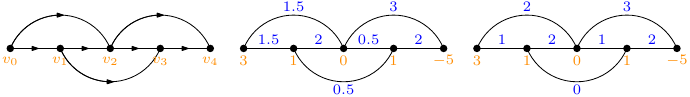}
	\caption[The zigzag graph~$Zig_4$ according to our notation and a pair of~$(3,1,0,1,-5)$-flows on it.]{ The zigzag graph~$Zig_4$ according to our notation and a pair of~$(3,1,0,1,-5)$-flows on it. Only the one in the right is an integer flow.}
	\label{fig:graph_flows}
\end{figure}

\begin{remark}\label{rem:multi_edges}
	For the reader who is not fond of working with graphs having multisets of edges, we remark that for our graphs, studying the multiedge case is equivalent to the single edges. Indeed, for a~$\bfa$-flow~$f$ and a multiedge~$(v_i,v_j)$ of multiplicity~$k$ one can remove~$k-1$ multiplicities and replace them by the family of edges~$(v_{i-l},v_{i-l+1})$ and~$(v_{i-l},v_{j})$ for~$l\in[1,k-1]$ while shifting the indices of the vertices~$\{v_z\}_{z\in[i-1]}$ accordingly. This transforms~$\bfa$ to a netflow~$\bfa^*$ by passing the netflow of~$v_i$ to~$v_{i-(k-1)}$, assigning~$0$ netflow to all vertices~$\{v_{i-l}\}_{l\in{0,k}}$ and all others remain the same. In this sense~$f$ is transformed to an~$\bfa^*$-flow~$f^*$ where~$f^*((v_{i-l},v_{j}))=f((v_i,v_j)_l)$ and~$f^*((v_{i-l},v_{i-(l-1)}))=f^*((v_{i-(l-1)},v_{i-(l-2)}))+f^*((v_{i-(l-1)},v_{j}))$ and all other edges~$e$ remain their original flow.  See Figure~\ref{fig:flow_multi_edges} for an example of this. This is a particular case of~\cite[Cor.2.13]{GHMY21}.
\end{remark}

\begin{figure}[h!]
	\centering
	\includegraphics[scale=1.5]{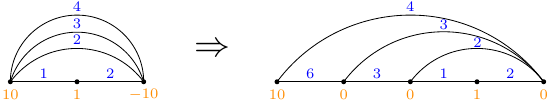}
	\caption{A multiedged graph with a flow and its equivalent flow on a simple graph.}
	\label{fig:flow_multi_edges}
\end{figure}

\subsection{Flow Polytopes}\label{ssec:flow_polytopes}

Notice that in Definition~\ref{def:flows} the admissible flows can be considered as points in~$\RR^{|E|}$. This begs for the following geometric definition.

\begin{definition}\label{def:flow_polytope}
	Let~$G$ be a directed multigraph on~$n$ vertices with~$m$ edges and consider a netflow~$\bfa$. Denote by~$X_G$ the multiset of vectors~$\mathbf{e_i}-\mathbf{e_j}$ for each edge~$(v_i,v_j)$ in~$G$ where~$i<j$ and by~$M$ the matrix with columns the vectors in~$X_G$. The \defn{flow polytope}\index{flow polytope} of~$G$ with netflow~$\bfa$ is \begin{equation*}
		\fpol(\bfa):=\{\mathbf{x}\in\RR^{m}\,:\, M\mathbf{x}=\bfa \text{ and } \mathbf{x}_i\geq 0 \text{ for all } i\in[m]\}.
	\end{equation*}
	Since the points in this polytope correspond to nonnegative real flows on the edges of~$G$, we usually use the following notation for describing the flow polytope
	\begin{equation*}
		\fpol(\bfa) = \Big\{ {\big(f(e)\big)}_{e\in E(G)} \,:\, f \text{ is a } \bfa \text{-flow of } G \Big\} \subset \RR^m.
	\end{equation*}
	We denote by \defn{$\fpol^{\ZZ}(\bfa)$} the set of integer~$\bfa$-flows of~$G$. Between all possible flows, the flows~$\bfd:=(0,d_1,\ldots,d_{n-1},-\sum_{i=1}^{n-1} d_i)$ where~$d_i:=\indeg_i(G)-1$ and~$\bfi:=\mathbf{e_0}-\mathbf{e_n}=(1,0,\ldots,0,-1)$ are of particular interest to us. Since any~$\bfi$-flow of~$G$ corresponds to a route~$R$ of~$G$, we denote such flow by the indicator function~\defn{$\mathds{1}_R$} where~$\mathds{1}_R(e)=\begin{cases}
		1 \text{ if } e\in R,\\
		0 \text{ otherwise }.
	\end{cases}$
\end{definition}

\begin{example}\label{ex:famous_flow_polytope}
	Some examples of polytopes that are integrally equivalent to~$\fpol(\bfi)$ include:
	\begin{itemize}
		\item the simplex~$\Delta_{n-1}$ when~$G$ is the graph with vertex set~$\{v_0,v_1\}$ and the multiedge~$(v_0,v_1)$ with multiplicity~$n$,
		\item the cube~$\PCube_{n-1}$ when~$G$ is the graph with vertex set~$\{v_0,\ldots,v_{n}\}$ and each multiedge~$(v_i,v_{i+1})$ for~$i\in[0,n-1]$ has multiplicity~$2$,
		\item the type~$A$ Chan-Robbins-Yuen polytope ($CRY(n)$~\cite{CRY98},~\cite{CR99}) when~$G$ is the complete graph~$K_{n+1}$,
		\item all order polytopes of strongly planar posets. We refer the reader to~\cite[\S 3.3]{MMS19} for the construction of the corresponding graph.
	\end{itemize}

	Figure~\ref{fig:flows_polytope_example} contains instances of each of these graphs.
\end{example}

\begin{figure}[h!]
	\centering
	\includegraphics[scale=1.3]{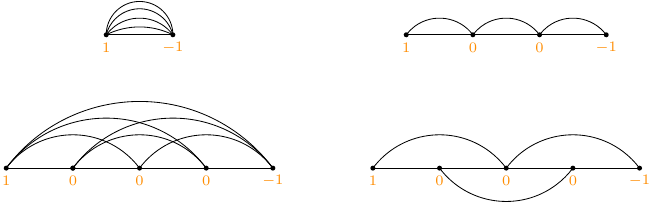}
	\caption[The graphs~$G$ whose respective flow polytopes~$\fpol(\bfi)$ are respectively integrally equivalent to~$\Delta_4$,~$\PCube_2$,~$CRY(4)$, and~$\cO(P)$.]{ The graphs~$G$ whose respective flow polytopes~$\fpol(\bfi)$ are respectively integrally equivalent to~$\Delta_4$ (top left),~$\PCube_2$ (top right),~$CRY(4)$ (bottom left), and~$\cO(P)$ where~$P$ is the chain on~$3$ elements (bottom right).}\label{fig:flows_polytope_example}
\end{figure}

Flow polytopes have plenty of interesting properties. Here we present just a few that are pertinent for our work. For a deeper dive we recommend~\cite{CG78},~\cite{RH70},~\cite{FRD71}, and~\cite{H03}. Before that, we introduce some notation on paths within our graphs.

\begin{definition}\label{def:routes}
	A \defn{path}\index{graph!path} of a graph~$G$ is a connected sequence of edges. That is, a sequence of the form~$((v_{k_0}, v_{k_1}), (v_{k_1}, v_{k_2}), \ldots, (v_{k_l}, v_{K_{k+1}}))$, with~$k_0<k_1<k_2< \cdots <k_l<k_{l+1}$. A maximal path of~$G$ is said to be a \defn{route}\index{graph!route} of~$G$.
\end{definition}

\begin{remark}\label{rem:max_paths}
	For the context of this thesis and due to the specifications of Definition~\ref{def:flows} all routes of~$G$ start at the source~$v_0$ and end at the unique sink~$v_n$.
\end{remark}

\begin{proposition}[{\cite[Lem.2.1]{H03}}]\label{prop:flow_polytope_vertices_alejandro} The flow polytope~$\fpol(\bfa)$ is the convex polytope \begin{equation*}
		\fpol(\bfa) = \conv\Big((f(e)){}_{e\in E(G)} \,:\, \supp(f) \text{ contains no (undirected) cycles}\Big).
	\end{equation*}
\end{proposition}

In the particular case of the netflow~$\bfi$ one obtains that~$\fpol(\bfi)$ is a $0\slash{1}$-polytope in the following way.

\begin{proposition}[{\cite[Cor.3.1]{CG78}}]\label{prop:flow_polytope_vertices_bfi}
	\begin{equation*}
		\fpol(\bfi) = \conv\Big((\mathds{1}_R(e)){}_{e\in E(G)} \,:\, R \text{ is a route of } G\Big).
	\end{equation*}
\end{proposition}

See Figure~\ref{fig:flow_polytope_example_1} for an example.

\begin{figure}[h!]
	\centering
	\includegraphics[scale=1]{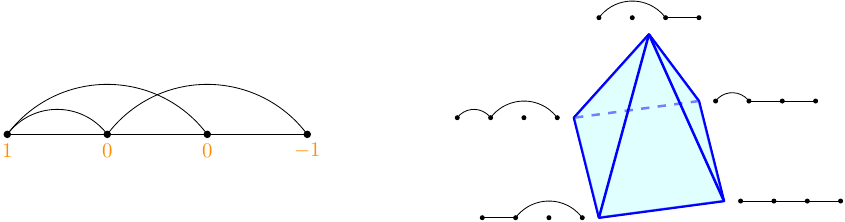}
	\caption[A graph and its flow polytope with flow~$\bfi$.]{ A graph~$G$ and its flow polytope~$\fpol(\bfi)$. Each vertex of the flow polytope is labeled by the route that it corresponds to. Figure adapted from~\cite{MMS19}.}
	\label{fig:flow_polytope_example_1}
\end{figure}

\begin{proposition}[{\cite[\S.1.1]{BV08}}]\label{def:flow_polytope_dimension}
	The flow polytope~$\fpol(\bfa)$ has dimension~$|E|-|V|+1$.
\end{proposition}

Our flow polytopes have a particular connection with the root systems of Coxeter groups when describing their volumes and Ehrhart functions. We present it here making use of the following definition.

\begin{definition}[{\cite[Eq.2.3]{MM19}}]\label{def:Kostant_partition_function}
	The \defn{Kostant partition function}~$K_G:\ZZ^{n+1}\to \ZZ$ is the function \begin{equation*}
		K_G(\bfa):=\left|\left\{(b_i)_{i\in[m]} \,:\,\sum_{i\in[m]}b_i\mathbf{v}_i=\bfa \text{ and }b_i\in\ZZ_{\geq 0}\right\}\right|
	\end{equation*} where~$m:=|E|$ and~$\{\mathbf{v}_1,\ldots,\mathbf{v}_{m}\}$ is the multiset of vectors corresponding to the multiset of edges of~$G$ under the map sending an edge~$(v_i,v_j)$ to the vector~$\mathbf{e_i}-\mathbf{e_j}$.
	Equivalently, the Kostant partition function counts the number of ways to express~$\bfa$ as a linear combination of positive type~$A$ roots with coefficients in~$\NN$.
\end{definition}

\begin{proposition}[{\cite{P1014},\cite{S00},\cite{BV08}}]\label{prop:flow_polytope_volume_full}
	Let~$G$ be a graph on vertices~$\{v_0,\ldots,v_n\}$ and denote~$d_i=\indeg_i(G)-1$. The normalized volume of the flow polytope~$\fpol(\bfi)$ can be expressed as \begin{equation*}
		\vol\big(\fpol(\bfi)\big)=K_G\left(0,d_1,\ldots,d_{n-1},-\sum d_i\right).
	\end{equation*}
\end{proposition}

\begin{example}
	Taking~$G$ as the graph in Figure~\ref{fig:flow_polytope_example_1}, we have that $\vol(\fpol(\bfi))=2$. This coincides with the fact that~$K_G(0,1,1,-2)=2$ as~$(0,1,1-2)$ can be expressed either as~$(\mathbf{e_2}-\mathbf{e_4})+(\mathbf{e_3}-\mathbf{e_4})$ or~$(\mathbf{e_2}-\mathbf{e_3})+2(\mathbf{e_3}-\mathbf{e_4})$.
\end{example}

\begin{proposition}[{\cite{S00}}]\label{prop:flow_pol_ehrhart}
	The amount of integer points in the~$\fpol(\bfi)$ is~$K_G(\bfi)$. Moreover, the Ehrhart polynomial of the flow polytope~$\fpol(\bfi)$ is given by~$L_{\fpol(\bfi)}(t) = K_G(t,0,\ldots,0,-t)$.
\end{proposition}

Although Propositions~\ref{prop:flow_polytope_volume_full} is defined for the basic netflow~$\bfi=(1,0,\ldots,0,-1)$, the general case can be seen to decompose nicely using the Kostant partition formula.

\begin{definition}\label{def:dominance_order}
	Let~$\mathbf{b}$ and~$\mathbf{c}$ be two weak compositions of an integer~$N>0$. We say that~$\mathbf{b}$ \defn{dominates}~$\mathbf{c}$ if~$\sum_{i=1}^kb_i\geq\sum_{i=1}^kc_i$. We denote the \defn{dominance order}\index{dominance order} by~$\mathbf{b}\succeq\mathbf{c}$. Let $\bbinom{m}{k}=\binom{m+k-1}{k}$ be the number of multisets of $[m]$ of size $k$.
\end{definition}

\begin{theorem}[{\cite[Thm.39]{BV08}}]\label{thm:Lidskii_formulas}
	Consider a graph~$G$ with vertices~$\{v_0,\ldots,v_{n}\}$ and a netflow~$\bfa\in\ZZ^{n+1}$. Let~$d_j=\indeg_j(G)-1$ (resp.~$o_j=\outdeg_j(G)-1$) and~$m:=|E|$. The volume and Kostant partition functions of~$\fpol(\bfa)$ decompose as \begin{equation*}
		\begin{split}
			\vol\big(\fpol(\bfa)\big) &= \sum_{\mathbf{j}}\gbinom{m-n}{j_0,\ldots,j_{n-1}}a_0^{j_0}\cdots a_{n-1}^{j_{n-1}}K_G(j_0-o_0,\ldots,j_{n-1}-o_{n-1},0), \\
			K_{\fpol}(\bfa) &= \sum_{\mathbf{j}}\gbinom{a_0+o_0}{j_1}\cdots \gbinom{a_{n-1}+o_{n-1}}{j_{n-1}}K_{G}(j_0-o_0,\ldots,j_{n-1}-o_{n-1},0) \\
			&= \sum_{\mathbf{j}}\bbinom{a_0-d_0}{j_0}\cdots \bbinom{a_{n-1}-d_{n-1}}{j_{n-1}}K_{G}(j_0-o_0,\ldots,j_{n-1}-o_{n-1},0)
		\end{split}
	\end{equation*} where the sums go over all weak compositions~$\mathbf{j}$ of~$m-n$ such that~$\sum_{i=0}^{k} j_i\geq \sum_{i=0}^{k} o_i$ for all~$i\in[0,n-1]$.
	These are known as the \defn{Baldoni–Vergne–Lidskii formulas}.
\end{theorem}

\begin{remark}\label{rem:BV_formulas_summands_description}
	The Baldoni–Vergne–Lidskii formulas describe a subdivision technique on~$\fpol(\bfa)$ (see Subsection~\ref{ssec:ps_subdivisions}). In~\cite{MM19} it was shown that the composition~$\mathbf{j}$ of each summand represents a type of cell of said subdivision. In this way the Kostant partition function describes the number of times that type of cell appears in the subdivision and the binomial coefficients give its volume.
\end{remark}

\begin{example}\label{ex:flow_polytope_volumes} Some interesting volumes of flow polytopes include:
	\begin{itemize}
		\item $\vol\big(\cF_{K_{n+1}}(\bfi)\big)= \prod_{i=1}^{n-2}C_i$ where~$C_i$ is the~$i$-th Catalan number.
		\item $\vol\big(\cF_{Zig_{n+1}}(\bfi)\big)= E_{n-1}$ where~$E_n$ is the number of alternating permutations in~$\fS_n$.
		\item $\vol\big(\cF_{\car_{n+1}}(\bfi)\big)= C_{n-2}$ where~$\car_{n+1}$ is the caracol graph and~$C_{n-2}$ is a Catalan number.
	\end{itemize}
\end{example}

The following is a geometrical result from which the geometrical reader is invited to derive extensions of Propositions~\ref{prop:flow_polytope_volume_full} and~\ref{prop:flow_pol_ehrhart} to any netflow~$\bfa$.

\begin{proposition}[{\cite[\S 3.4]{BV08},\cite[Prop.2.1]{MM19}}]\label{prop:flow_poltope_minkowski_decomposition}
	Let~$G$ be a graph on~$\{v_0,\ldots,v_n\}$ and~$\bfa\in\ZZ^{n+1}$ a netflow. Then \begin{equation*}
		\fpol(\bfa) = \sum_{i=0}^{n-1}a_i \fpol(\mathbf{e_i}-\mathbf{e_n}).
	\end{equation*}
\end{proposition}

We finish with a crucial relation between flow polytopes and Cayley embeddings.

\begin{definition}
	Let~$G$ be a graph on~$\{v_0,\ldots,v_n\}$. The graph \defn{$G^*$} is the graph obtained from~$G$ by adding a vertex~$v^*$ and edges~$(v^*,v_i)$ for~$i\in[0,n-1]$. See Figure~\ref{fig:flow_pols_are_cayley_embs} for an example.
\end{definition}

\begin{proposition}[{\cite[Prop.7.2]{MM19}}]\label{prop:flow_pols_and_cayley_embs}
	The flow polytope~$\fpol[G*](\mathbf{e_0}-\mathbf{e_{n+1}})$ is the Cayley embedding~$\cC\big(\fpol(\mathbf{e_0}-\mathbf{e_{n}}),\ldots,\fpol(\mathbf{e_{n-1}}-\mathbf{e_{n}})\big)$.

	\begin{figure}[h!]
		\centering
		\includegraphics[scale=0.9]{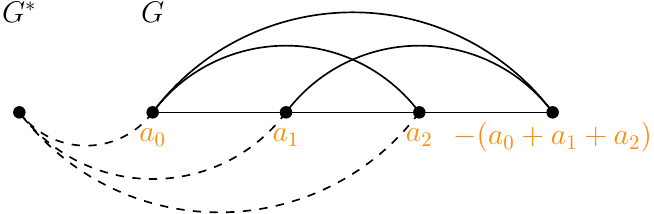}
		\caption[The construction of the graph~$G^*$ whose flow polytope is the Cayley embedding of flow polytopes from~$G$.]{ The construction of the graph~$G^*$ whose flow polytope is the Cayley embedding of flow polytopes from~$G$. Figure based on~\cite{MM19}.}\label{fig:flow_pols_are_cayley_embs}
	\end{figure}
\end{proposition}

\begin{proposition}\label{prop:flow_pols_ARE_cayley_embs}
	Let~$G$ be a graph on~$\{v_0,\ldots,v_n\}$. The flow polytope~$\fpol(\mathbf{e_0}-\mathbf{e_n})$ is the Cayley embedding of the flow polytopes~$\fpol(\mathbf{e_i}-\mathbf{e_n})$ for the~$i\in[n]$ such that~$G$ has an edge~$(v_0,v_i)$.
\end{proposition}

We now move on to the main property of flow polytopes that is of use for our work which is how they can be subdivided into smaller polytopes. We present the techniques of two such subdivisions coming from~\cite{DKK12} and~\cite{S00}.

\subsection{DKK Triangulations}\label{ssec:DKK_triangulation}

For the first subdivision let us define a way to compare routes by inducing a partial order on routes from an order on in-coming and out-going edges of the vertices of~$G$. We base our presentation on~\cite{DKK12}.

\begin{definition}\label{def:framing}
	Let~$G$ be a graph on~$\{v_0,\ldots,v_n\}$. A \defn{framing}\index{graph!framing} of~$G$ is a choice of linear orders~$\preceq_{\cI_i}$ and~$\preceq_{\cO_i}$ on the sets of incoming and outgoing edges for each vertex~$v_i$ where~$i\in[n-1]$. When~$G$ is endowed with such a framing~$\preceq$, we say that~$G$ is \defn{framed}\index{graph!framed}. For an edge~$e=(v_i,v_j)$ we denote by~$\cI(e)$ (resp.~$\cO(e)$) the position of the edge~$e$ in the framing order~$\preceq_{\cI_i}$ (resp.~$\preceq_{\cO_i}$).
\end{definition}

\begin{definition}\label{def:comparing_routes}
	Let~$P,Q$ be two routes of~$G$ that contain vertices~$v_i$ and~$v_j$. We denote by \defn{$Pv_i$} the prefix of~$P$ that ends at~$v_i$, \defn{$v_iP$} the suffix of~$P$ that starts at~$v_i$ and \defn{$v_iPv_j$} the subroute of~$P$ that starts at~$v_i$ and ends at~$v_j$.

	Let~$\prec$ be a framing of~$G$ and consider~$v_iP,v_iQ$ (res.~$Pv_i,Qv_i$) a pair of paths between~$v_i$ and~$v_n$ (resp.\ between~$v_0$ and~$v_i$) which coincide from~$v_i$ until a minimal vertex~$v_j$ (resp.\ from a maximal vertex~$v_j$ until~$v_i$) (possibly~$i=j$). Denote by~$e_p$ and~$e_Q$ the corresponding edges in~$P$ and~$Q$ with starting (resp.\ ending) vertex~$v_j$. We say that \defn{$v_iP\preceq v_iQ$} if~$e_P\preceq_{\cO_j} e_Q$ (resp.\ \defn{$Pv_i\preceq Qv_i$} if~$e_P\preceq_{\cI_j} e_Q$).
\end{definition}

\begin{definition}\label{def:coherent_routes}
	Consider~$G$ to be a graph with framing~$\prec$. Let~$P$ and~$Q$ be routes of~$G$ that share a subroute between the vertices~$v_i$ and~$v_j$. We say that~$P$ and~$Q$ are \defn{conflicting routes}\index{graph!route!conflicting} at~$[v_i,v_j]$ (possibly~$v_i=v_j$) if the initial paths~$Pv_i$ and~$Qv_i$ are ordered different relative to the paths~$v_iP$ and~$v_iQ$ with respect to~$\prec$. In the case that~$P$ and~$Q$ have no conflict in any common subroute we say that~$P$ and~$Q$ are \defn{coherent routes}\index{graph!route!coherent}.
\end{definition}

See Figure~\ref{fig:paths_coherent_not_coherent} for an example of coherent and conflicting routes on a framed graph.

\begin{figure}[h!]
	\centering
	\includegraphics[scale=1.3]{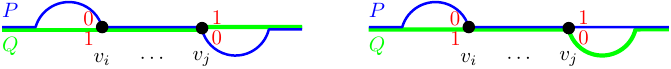}
	\caption[A framed graph with a pair of coherent and conflicting routes.]{ A framed graph with a pair of coherent (left) and conflicting (right) routes. Only~$1$ route is bolded.}\label{fig:paths_coherent_not_coherent}
\end{figure}

\begin{remark}\label{rem:extreme_framings_dont_matter}
	Notice that one could define framing orders for~$v_0$ (resp.~$v_n$) but since this vertex has no incoming (resp.\ outgoing) edges, checking for coherence is trivially true and thus not necessary. For completeness, we say that~$\cO_1(e)=0$ (resp.~$\cI_n(e)=0$) for any starting (resp.\ final) edge. 
\end{remark}

\begin{definition}\label{def:cliques}
	Let~$(G,\preceq)$ be a framed graph. We call a set of mutually coherent routes of~$G$ a \defn{clique}\index{graph!clique}. We denote by \defn{$\cliques$} the set of cliques of~$(G, \preceq)$, and \defn{$\maxcliques$} the set of maximal collection of cliques under inclusion. If a route~$P$ is coherent with all other routes of~$G$ we say that~$P$ is \defn{exceptional}\index{graph!route!exceptional}.
\end{definition}

\begin{remark}\label{rem:order_routes_properties}
	Notice that the coherence relation between routes is reflexive and symmetric but not transitive. Take all the routes given in Figure~\ref{fig:paths_coherent_not_transitive} for a particular framed graph. The maximal cliques in this case are~$\{R_1,R_2,R_3,R_5\}$ and~$\{R_1,R_2,R_4,R_5\}$. Notice that although~$R_1$ is coherent with both~$R_3$ and~$R_4$, they are not coherent as they have a conflict in the inner vertex. The routes~$R_1,R_2$, and~$R_5$ are all exceptional routes.

	\begin{figure}[h]
		\centering
		\includegraphics[scale=1.3]{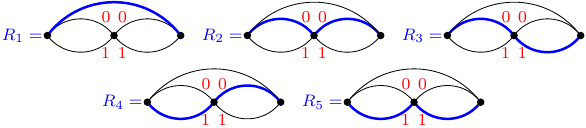}
		\caption[All possible routes of a framed graph.]{ All possible routes (bolded in blue) of a framed graph.}\label{fig:paths_coherent_not_transitive}
	\end{figure}
\end{remark}

\begin{definition}\label{def:cliques_to_triangles}
	Given the flow polytope~$\fpol(\bfi)$ and a clique~$C\in\cliques[G]$, we define the simplex \defn{$\Delta_C$} corresponding to~$C$ as \begin{equation*}
		\Delta_C = \conv\Big(\big(\mathds{1}_R(e)\big)_{e\in E} \,:\, R \text{ is a route in } C\Big).
	\end{equation*}
\end{definition}

\begin{proposition}[{\cite[Thm.1 \& 2]{DKK12}}]\label{prop:DKK_triangulation}
	Given a graph~$G$ with framing~$\preceq$, the set of simplices~$\{\Delta_C\,:\,C\in \maxcliques[G]\}$ form a regular triangulation of~$\fpol(\bfi)$. We refer to this triangulation as the  \defn{DKK triangulation}\index{flow polytope!triangulations!DKK} and denote it by \defn{$\triangDKK$}.
\end{proposition}

For the regularity of~$\triangDKK$, the authors in~\cite{DKK12} gave the following condition for a height function to be admissible where~$P+Q$ denotes the union of the edges of the routes~$P$ and~$Q$.

\begin{proposition}[{\cite[Lem.2]{DKK12}}]\label{prop:DKKlem2_original}
	Let~$(G, \preceq)$ be a framed graph. A function~$\height$ from the routes of~$G$ to~$\RR$ is an admissible height function of~$\triangDKK$ if for any two non-coherent routes~$P$ and~$Q$ there exist routes~$P'$ and~$Q'$ such that~$P+Q=P'+Q'$ and
	\begin{equation}
		\height(P)+\height(Q)>\height(P')+\height(Q').
	\end{equation}
\end{proposition}

We refine the condition on the non-coherence to make it necessary and sufficient. For that we need the following definition.

\begin{definition}\label{def:resolvents}
	Let~$(G,\preceq)$ be a framed graph and~$P,Q$ a pair of conflicting routes at the subroutes~$[v_{i_1},v_{i_1}']$,~$\dots$,~$[v_{i_k},v_{i_k}']$, where~$i_1\leq i_1' < i_2\leq i_2' \ldots <i_k\leq i_k'$. We call the \defn{resolvents}\index{graph!route!resolvents} of~$P$ and~$Q$ the paths~$P'$ and~$Q'$ defined as \begin{equation*}
		\begin{split}
			P'&:=Pv_{i_1}Qv_{i_2}Pv_{i_3}\cdots,\\
			Q'&:=Qv_{i_1}Pv_{i_2}Qv_{i_3}\cdots
		\end{split}
	\end{equation*} where~$Pv_{i_1}Qv_{i_2}Pv_{i_3}\cdots$ denotes the concatenation of the subroutes~$Pv_{i_1},v_{i_1}Qv_{i_2},v_{i_2}Pv_{i_3},\ldots$ finishing with~$v_{i_k}P$ or~$v_{i_k}Q$ depending on the parity of~$k$.

	We say that~$P$ and~$Q$ are in \defn{minimal conflict}\index{graph!route!minimal conflict} if they are in conflict in exactly one subroute~$[v_i,v_j]$ and the edges of~$P$ and~$Q$ that end at~$v_i$ (resp.\ start at~$v_j$) are adjacent for the total order~$\preceq_{\cI_i}$ (resp.~$\preceq_{\cO_j}$).

	See Figure~\ref{fig:resolvents} for an example of the resolvents of a minimal conflict.
\end{definition}

\begin{figure}[h!]
	\centering
	\includegraphics{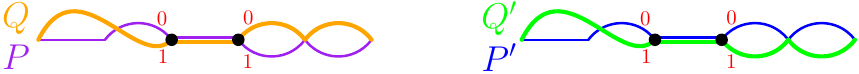}
	\caption[Two routes with a minimal conflict and their resolvents.]{ Two routes with a minimal conflict (left) and their resolvents (right). Only~$1$ route is bolded.}\label{fig:resolvents}
\end{figure}

\begin{lemma}[{\cite[Lem.5.4]{GMPTY23}}]\label{lem:DKKlem2_us}
	Let~$(G, \preceq)$ be a framed graph.
	A function~$\height$ from the routes of~$G$ to~$\RR$ is an admissible height function of~$\triangDKK$ if and only if for any two non-coherent routes~$P$ and~$Q$ with resolvents~$P'$ and~$Q'$ we have
	\begin{equation}
		\height(P)+\height(Q)>\height(P')+\height(Q').
	\end{equation}
\end{lemma}

\begin{proof}
	The proof of Proposition~\ref{prop:DKKlem2_original} works the same when~$P'$ and~$Q'$ are the resolvents of~$P$ and~$Q$ as the resolvents are always coherent not only between themselves but also with~$P$ and~$Q$. We thus omit this direction. Suppose that~$\height$ is an admissible height function for~$\triangDKK[G]$ and consider~$P$ and~$Q$ two non-coherent routes with resolvents~$P'$ and~$Q'$. As~$P'$ and~$Q'$ are coherent, they are part of at least one clique and~$\mathbf{p'}=(\mathds{1}_{P'}(e))_{e\in E}$ and~$\mathbf{q'}=(\mathds{1}_{Q'}(e))_{e\in E}$ are the vertices of an edge of~$\triangDKK[G]$. Letting~$\mathbf{p}$ and~$\mathbf{q}$ be the corresponding points to the routes~$P$ and~$Q$, we have that the point~$\mathbf(c)=\frac{1}{2}(\mathbf{p'}+\mathbf{q'})=\frac{1}{2}(\mathbf{p}+\mathbf{q})$ is in the edge~$[\mathbf{p'},\mathbf{q'}]$. As this edge must be lifted to a lower face of the lift of~$\fpol(\bfi)$ via the admissible function~$\height$ we have that~$\height(P)+\height(Q)>\height(P')+\height(Q')$.
\end{proof}

Using minimal conflicts we can make this statement stronger.

\begin{lemma}[{\cite[Lem.5.5]{GMPTY23}}]\label{lem:DKKlem2_us_pro}
	Let~$(G, \preceq)$ be a framed graph.
	A function~$\height$ from the routes of~$G$ to~$\RR$ is an admissible height function of~$\triangDKK$ if and only if for any two non-coherent routes~$P$ and~$Q$ in minimal conflict with resolvents~$P'$ and~$Q'$ we have
	\begin{equation}
		\height(P)+\height(Q)>\height(P')+\height(Q').
	\end{equation}
\end{lemma}

\begin{proof}
	Let~$\height$ be a function from the routes of~$G$ to~$\RR$ such that for any minimal conflict between two routes~$P$ and~$Q$ with resolvents~$P'$ and~$Q'$, we have~$\height(P)+\height(Q)>\height(P')+\height(Q')$. From Lemma~\ref{lem:DKKlem2_us} we only need to show that for any two conflicting routes~$P$ and~$Q$, there exist routes~$P'$ and~$Q'$ such that~$P+Q=P'+Q'$ and~$\height(P)+\height(Q)>\height(P')+\height(Q')$. We proceed by induction on the number of conflicts between~$P$ and~$Q$.

	First, suppose that~$P$ and~$Q$ are conflicting at exactly one subroute~$[v_i,v_j]$. Let~$e_1:=e_P$ and~$e_k:=e_Q$ be the respective edges of~$P$ and~$Q$ that end at~$v_i$ and~$e_i$ for~$i\in[2,k-1]$ be all the other edges of~$G$ such that we have~$\cI(e_1)\lessdot \cI(e_2)\lessdot\cdots\lessdot\cI(e_k)$. Letting~$R_1:=Pv_i$ and~$R_k:=Qv_i$, we define the partial routes~$R_i$ (resp.~$S_i$) for~$i\in[2,k-1]$ from right to left and from~$1$ to~$k-1$. The starting step is the ending edge which is already determined. At step~$i$ we can choose any other edge unless we arrive at a vertex common to a previously built partial route. In this case we must choose the same edges as in this partial route. This gives~$Pv_i=R_1\prec R_2 \prec \ldots \prec R_k=Qv_i$ and a similar construction gives~$v_jQ=S_1\prec S_2 \prec \ldots \prec S_l=v_jP$.

	With this for any~$x\in[k-1]$ and~$y\in [l-1]$ we have that the routes~$R_xv_iPv_jS_{y+1}$ and~$R_{x+1}v_iPv_jS_{y}$ are in minimal conflict, with resolvents~$R_xv_iPv_jS_{y}$ and~$R_{x+1}v_iPv_jS_{y+1}$. Our assumption on the height function~$\height$ implies the inequality:
	\begin{equation}
		\height(R_xv_iPv_jS_{y+1})+\height(R_{x+1}v_iPv_jS_{y}) > \height(R_xv_iPv_jS_{y}) +\height(R_{x+1}v_iPv_jS_{y+1}).
	\end{equation}
	Adding all these inequalities for all~$x\in [k-1]$ and~$y\in[l-1]$ we have that all terms of the form~$\height(R_{x}v_iPv_jS_{y})$ cancel out by pairs except for the pairs~$(x,y)\in \{(1,1), (k,t), (1,t), (k,1)\}$. This gives us the inequality
	\begin{equation*}
		\height(P)+\height(Q)=\height(R_1v_iPv_jS_t)+\height(R_kv_iPv_jS_{1}) > \height(R_1v_iPv_jS_{1}) +\height(R_{k}v_iPv_jS_{t})=\height(P')+\height(Q'),
	\end{equation*}
	where~$P'$ and~$Q'$ are the resolvents of~$P$ and~$Q$ as we wished.

	For the induction step suppose that~$\height$ satisfies~$\height(P)+\height(Q)>\height(P')+\height(Q')$ where~$P$ and~$Q$ are conflicting routes~with at most~$n$ conflicts and resolvents~$P', Q'$. Let~$P$ and~$Q$ be conflicting routes with~$n+1$ conflicts at the subroutes~$[x_1, y_1], \ldots, [x_{n+1}, y_{n+1}]$.
	Since the routes~$P$ and~$Px_1Q$ (resp.~$Q$ and~$Qx_1P$) have~$n$ conflicts and their resolvents are~$Px_1P'=P'$ and~$Px_1Q'$ ($Qx_1P'$ and~$Q'x_1Q=Q'$), the induction hypothesis gives us the inequalities
	\begin{equation*}
		\begin{split}
			\height(P)+\height(Px_1Q) &> \height(P')+\height(Px_1Q'),\\
			\height(Q)+\height(Qx_1P) &> \height(Qx_1P')+\height(Q').
		\end{split}
	\end{equation*}
	Now notice that the routes~$P$ and~$Qx_1P'$ only have one conflict and their resolvents are~$P'$ and~$Qx_1P$. Therefore, we also have the inequalities
	\begin{equation*}
		\begin{split}
			\height(P)+\height(Qx_1P') &> \height(P')+\height(Qx_1P),\\
			\height(Q)+\height(Px_1Q') &> \height(Px_1Q)+\height(Q').
		\end{split}
	\end{equation*}
	Adding up these four inequalities yields
	\begin{equation*}
		\height(P)+\height(Q) > \height(P')+\height(Q'). \qedhere
	\end{equation*}
\end{proof}

After obtaining a condition for the height function~$\height$ to be admissible on~$\triangDKK[G]$, the authors of~\cite{DKK12} constructed an explicit height function as follows.

\begin{definition}\label{def:DKK_height_function}
	Let~$P$ be a route in a framed graph~$(G,\preceq)$ given by the edges~$e_{i_1},\ldots,e_{i_k}$ and~$\varepsilon>0$ sufficiently small. The height \defn{$\height_\varepsilon$} of~$P$ is defined as \begin{equation*}
		\height_\varepsilon(P) := \sum_{1\leq a<c\leq k}\varepsilon^{c-a}{(\cI(e_a)+\cO(e_c))}^2.
	\end{equation*}
\end{definition}

\begin{proposition}[{\cite[Lem.3]{DKK12}}]\label{prop:DKKlem3_original}
	Let~$(G,\preceq)$ be a framed graph. The height function~$\height_\varepsilon(R)$ is an admissible height function for the triangulation~$\triangDKK[G]$ of~$\fpol(\bfi)$.
\end{proposition}

\subsection{Postnikov-Stanley Subdivisions}\label{ssec:ps_subdivisions}

Another way to subdivide flow polytopes consists on dividing them into two polytopes that are integrally equivalent to other flow polytopes. We refer the curious reader (out of respect like other flow polytope papers do) to~\cite{P1014} and~\cite{S00} for the birthing place of this subdivision. For an actual detailed view we recommend~\cite{MM13} and~\cite{MM19}. For this subsection we follow~\cite{MM19}. 

\begin{definition}\label{def:reduction_rule}
	Let~$G$ be a graph on vertices~$\{v_0,\ldots,v_n\}$ with edges~$(v_i,v_j),(v_j,v_k)\in E(G)$. The \defn{basic reduction}\index{graph!basic reduction} is the creation of the graphs~$G_1$ and~$G_2$ on vertices~$\{v_0,\ldots,v_n\}$ and edges \begin{equation*}\label{eq:post_stanley_reductions}\begin{split}
		E(G_1) & := E(G)\setminus\{(v_j,v_k)\}\cup \{(v_i,v_k)\}, \\
		E(G_2) & := E(G) \setminus\{(v_i,v_j)\} \cup \{(v_i,v_k)\}.
	\end{split}
\end{equation*}
\end{definition}

\begin{proposition}[{\cite{P1014},\cite{S00},\cite[Prop.3.1]{MM19}}]\label{reduction proposition}
	Let~$G$ be a graph on~$\{v_0, \ldots, v_n\}$ with a pair of edges~$e_1,e_2$ on which the basic reduction can be done. Then we have
	\begin{equation*}
		\fpol[G](\bfa)=\cP_1 \cup \cP_2 \quad\text{ and }\quad \mathcal \cP_1^{\circ} \cap \cP_2^{\circ} = \varnothing,
	\end{equation*}
	where~$\cP_i$ is integrally equivalent to~$\fpol[G_i](\bfa)$ and~$\cP^{\circ}$ denotes the interior of the polytope~$\cP$.
\end{proposition}

Figure~\ref{fig:subdiv} contains an illustration of the basic reduction.

The repeated use of basic reductions (called a \defn{compounded reduction}\index{graph!compounded reduction}) on~$G$ together with some extra conditions on when one discards a resulting graph yields a triangulation of~$\fpol(\bfa)$ called a \defn{Postnikov-Stanley triangulation}\index{flow polytope!triangulations!Postnikov-Stanley} Indeed, at the end of a compounded reduction one obtains multi-graphs with edges of the form~$(v_i,v_n)$ for~$i\in[0,n-1]$. Like in Example~\ref{ex:famous_flow_polytope}, the flow polytopes of these graphs are simplices of the highest dimension.

\begin{figure}[h!]
	\centering
	\includegraphics[scale=1.2]{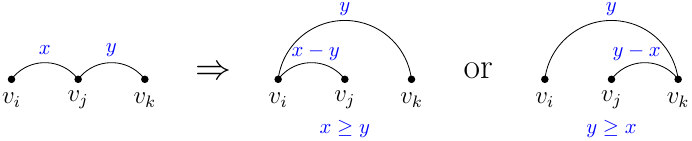}
	\caption[The basic reduction of flow polytopes.]{ The basic reduction of flow polytopes. Figure based on~\cite{MM19}.}\label{fig:subdiv}
\end{figure}

The choosing on the vertices or the pairs of edges taken for each step of the compounded reductions can change the final triangulation. In~\cite{MMS19} it was shown that Postnikov-Stanley triangulations coincided in certain cases with DKK triangulations by doing a compounded reduction following the framing of~$G$. The triangulations obtained like this are called \defn{framed Postnikov-Stanley triangulations}\index{flow polytope!triangulations!framed Postnikov-Stanley}. In our interest, this includes the case of~$\fpol(\bfi)$ which we now present.

\begin{definition}
	Let~$(G,\preceq)$ be a framed graph with netflow~$\bfd$. We define the function \defn{$\Omega_{G,\preceq}$} between the maximal cliques~$\maxcliques[G]$ and the integer flows of~$\fpol(\bfd)$ as \begin{equation*}
		\Omega_{G,\preceq}(C):={\big(n_C(e)-1\big)}_{e\in E(G)}
	\end{equation*}
	where~$n_C(v_i,v_j):=|\{P\in C\,:\, (v_i,v_j) \text{ is in the prefix } Pv_j\}|$.
\end{definition}

\begin{proposition}[{\cite[Thm 7.8]{MMS19}}]\label{prop:bij_cliques_intflows}
	Given a framed graph~$(G,\preceq)$, the map~$\Omega_{G,\preceq}$ is a bijection between maximal cliques in~$\maxcliques$ and integer flows in~$\fpol(\bfd)$.
\end{proposition}

As a direct consequence of Proposition~\ref{prop:bij_cliques_intflows} we get the following result.

\begin{proposition}\label{prop:volume_intflows}
	For a graph~$G$ and netflow~$\bfd:=(0,d_1,\ldots,d_{n-1},-\sum_i d_i)$ where~$d_i=\indeg_i(G)-1$, we have that \begin{equation*}
		\vol\Big(\fpol(\bfi)\Big)=\Big|\fpol(\bfd)\Big|.	
	\end{equation*}
\end{proposition}

\section{Tropical Geometry}\label{sec:tropical_geometry}

We now move on to describe the bases of tropical geometry and an application of this context to the Cayley trick. We follow heavily~\cite{J17} and~\cite{J21}.

\begin{definition}\label{def:tropical_semiring}
	The \defn{tropical semiring}\index{tropical!semiring} of the min-plus algebra is the tuple~$(\TT,\odot,\oplus)$ formed by~$\TT:=\RR\cup\{\infty\}$ together with the~$\min$ operation as the tropical addition~$\oplus$ and the usual addition~$+$ as the tropical multiplication~$\odot$. The additive identity is~$\infty$ and the multiplicative identity is~$0$.
\end{definition}

\begin{definition}\label{def:tropical_polynomials}
	An \defn{$n$-variate tropical polynomial} is a linear combination of tropical monomials with possibly negative exponents and appear as \begin{equation*}
		F(\mathbf{x})=\bigoplus_{\mathbf{m}\in I}c_\mathbf{m}\odot \mathbf{x}^{\mathbf{m}}=\min_{\mathbf{m}\in I}\big(c_\mathbf{m}+\langle\mathbf{m},\mathbf{x}\rangle\big)
	\end{equation*} where~$I$ is a finite subset of~$\ZZ^n$ and~$c_m\in\TT$. The \defn{support}\index{tropical!support} of a polynomial is the set~$\supp(F)=\{\mathbf{m}\in I\,:\, c_\mathbf{m}\neq\infty\}$. The \defn{degree}\index{tropical!degree} of a monomial~$x_1^{m_1}\cdots x_n^{m_n}$ is~$\sum_{i=1}^nm_i$ and the \defn{degree} of a polynomial~$F$ is the maximal degree of its monomials.
\end{definition}

\begin{remark}\label{rem:tropical_polynomials_support_height_function_point_configuration}
	Notice that the support of a tropical polynomial can be seen as a point configuration on~$\ZZ^n$ equipped with a height function given by the coefficients~$c_m$. In the reverse direction, any lattice point configuration together with a lifting function determines a tropical polynomial. Thus, any interest of heightened point configurations can make use of tropical polynomials.
\end{remark}

\begin{figure}[h!]
	\centering
	\includegraphics[scale=1.5]{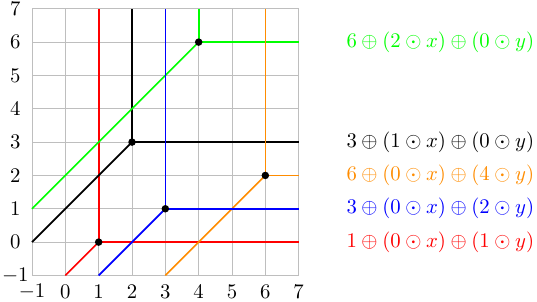}
	\caption{The tropical hypersurfaces of several tropical polynomials.}
	\label{fig:tropical_polynomials}
\end{figure}

In~$\RR$ hypersurfaces of varying degree come from seeing where polynomials vanish. In the tropical setting of~$\TT$ this is recontextualized by looking at where the minimum of a tropical polynomial is attained at least twice.

\begin{definition}\label{def:tropical_hypersurface}
	The \defn{tropical hypersurface} determined by~$F$ is the set \begin{equation*}
		\cT(F):= \left\{ \mathbf{x}\in \RR^n\, :\, \text{the minimum of~$F(\mathbf{x})$ is attained at least twice}\right\}.
	\end{equation*}
\end{definition}

Several tropical hypersurfaces are illustrated in Figure~\ref{fig:tropical_polynomials}.

\begin{proposition}\label{prop:tropical_hypersurface_multiplication}
	For tropical polynomials~$F_1,\ldots,F_k$ we have that \begin{equation*}
		\cT\Big(\bigodot_{i\in[k]} F_i\Big)=\bigcup_{i\in[k]}\cT(F_i).
	\end{equation*}
\end{proposition}

\subsection{Polyhedral Constructions}

Tropical geometry is of particular use to us since it has a deep relation with polyhedral geometry. In this section we describe two polyhedral constructions we need as tools called the dome and the Newton polytope of a tropical polynomial. We begin with the following crucial remark.

\begin{remark}\label{rem:tropical_hypersurface_are_complexes}
	Notice that as~$F(\mathbf{x})$ consists only of linear combinations and minima,~$\cT(F)$ is an~${n-1}$-polyhedral complex.
\end{remark}

\begin{definition}\label{def:tropical_dome}
	For a tropical polynomial~$F$, its \defn{dome}\index{tropical!dome} is the unbounded~$n+1$-polyhedron \begin{equation*}\begin{split}
			\cD(F)&:=\left\{(\mathbf{x},s)\in\RR^{n+1}\,:\, \mathbf{x}\in \RR^n,s\in\RR,s\leq F(\mathbf{x})\right\}\\
			&=\bigcap_{\mathbf{x}\in \supp(F)}\left\{(\mathbf{p},s)\in\RR^{n+1}\,:\, s\leq c_\mathbf{m}+ \langle \mathbf{m},\mathbf{x}\rangle\right\}.
		\end{split}
	\end{equation*}

\end{definition}

The dome helps us as it recovers in the polyhedral setting the geometry of the tropical hypersurface.

\begin{proposition}[{\cite[Cor. 1.6]{J21}}]\label{prop:tropical_dome_tropical_hypersurface}
	The tropical hypersurface~$\cT(F)$ is the image of the~$n-1$-skeleton of its dome~$\cD(F)$ under the projection that forgets the last coordinate.
\end{proposition}

In other words, if~$F$ is an~$n$-variate polynomial,~$\cT(F)$ is a~$(n-1)$-dimensional polyhedral complex and the connected components of its complement are projections of the relative interiors of facets of the dome~$\cD(F)$. In particular, each facet of the dome corresponds to a tropical monomial~$x_1^{a_1}\cdots x_n^{a_n}$ and the edges correspond to when the minimum is shared between two monomials. See Figure~\ref{fig:tropical_normal_complex} (left) for an example.

\begin{figure}[h!]
	\centering
	\includegraphics[scale=1.5]{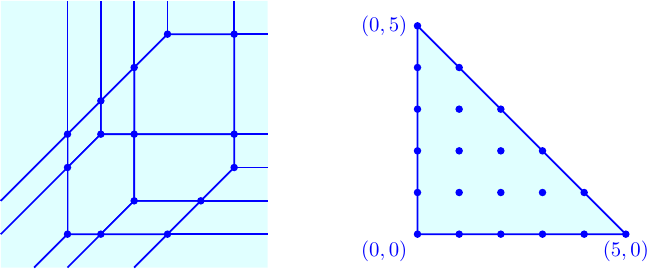}
	\caption[The orthogonal projection of the facets of~$\cD(F)$ and~$\cN(F)$.]{ The orthogonal projection of the facets of~$\cD(F)$ (left) and~$\cN(F)$ (right) where~$F$ is the product of the tropical polynomials in Figure~\ref{fig:tropical_polynomials}.}
	\label{fig:tropical_normal_complex}
\end{figure}

The second polytope related to tropical polynomials is the Newton polytope associated to the support points of~$F$. See Figure~\ref{fig:tropical_normal_complex} (right) for the Newton polytope of our running example.

\begin{definition}\label{def:tropical_newton_polytope}
	Given a tropical polynomial~$F(x)$, the convex hull of its support is called the \defn{Newton polytope}\index{tropical!Newton polytope} of~$F$. That is, \begin{equation*}
		\cN(F):=\conv\left(\mathbf{m}\in\ZZ^d \,:\, \mathbf{m}\in\supp(F)\right).
	\end{equation*}
	The \defn{extended Newton polytope} corresponds to the Minkowski sum
	\begin{equation*}
		\begin{split}
			\widetilde{\cN}(F):=&\conv\left((\mathbf{m},r)\in\ZZ^d\times\RR \,:\, \mathbf{m}\in\supp(F),r\geq c_\mathbf{m}\right)\\
			=&\conv\left((\mathbf{m},c_{\mathbf{m}})\in\ZZ^d\times\RR \,:\, \mathbf{m}\in\supp(F)\right)+\cone(\mathbf{e_{n+1}}).
		\end{split}
	\end{equation*}
	Following Remark~\ref{rem:tropical_polynomials_support_height_function_point_configuration}, consider~$\supp(F)$ as a point configuration. The projection downwards of the faces of~$\widetilde{N}(F)$ is a regular subdivision of the Newton polytope~$\cN(F)$ called the \defn{dual subdivision}\index{tropical!dual subdivision} and denoted as \defn{$\cS(F)$}.
\end{definition}

\begin{remark}\label{rem:newton_polytope_multiplication}
	Notice that the product of tropical monomials is
	\begin{equation*}
		\begin{split}
			F(\mathbf{x})& = \bigodot_{i\in[k]}F_i(\mathbf{x}) = \bigodot_{i\in[k]}\bigoplus_{\mathbf{m}^i\in I_i}c_{\mathbf{m}^i}\odot \mathbf{x}^{\mathbf{m}^i}\\
			& =\bigoplus_{\mathbf{M}\in \cI}\bigodot_{i\in[k]}c_{\mathbf{m}^i}\odot \mathbf{x}^{\mathbf{m}^i}=\bigoplus_{\mathbf{M}\in \cI}c_{\mathbf{m}^1}+\cdots+c_{\mathbf{m}^k}\odot \mathbf{x}^{\mathbf{m}^1+\cdots+\mathbf{m}^k}\\
		\end{split}
	\end{equation*} where~$I_i\subset\in \ZZ^n$,~$\cI=I_1\times\cdots\times I_k$ and~$\mathbf{M}=(\mathbf{m}^1,\ldots,\mathbf{m}^k)$. Thus, for the product of tropical polynomials~$F=\bigodot_{i\in[k]}F_i$, the corresponding Newton polytope is \begin{equation*}
		\cN(F)=\cN(F_1)+\cdots+\cN(F_k).
	\end{equation*} Moreover, abusing the notation of Minkowski sums, we have that~$\supp(F)=\supp(F_1)+\cdots+\supp(F_k)$.
\end{remark}

It turns out that the dome and the Newton polytope of tropical polynomials are actually dual constructions. Moreover, our constructions~$\cT(F)$,~$\cD(F)$,~$\cN(F)$ and~$\cS(F)$ are related in the following way.

\begin{proposition}[{\cite[Thm.1.13]{J21}}]\label{prop:bijections_dome_newton_polytope}
	Let~$F$ be an~$n$-variate tropical polynomial.
	\begin{itemize}
		\item The faces of the dome~$\cD(F)$ have an inclusion-reversing bijection with the bounded faces of the extended Newton polytope~$\widetilde{\cN}(F)$.
		\item The orthogonal projection of the proper faces of~$\widetilde{\cN}(F)$ give the cells of the dual subdivision~$\cS(F)$.
		\item The~$k$-dimensional cells of~$\cT(F)$ are in bijection with the~$(n-k)$-dimensional cells of~$\cS(F)$ for~$0\leq k < n$.
	\end{itemize}
\end{proposition}

In particular, take notice that the bijection of Proposition~\ref{prop:bijections_dome_newton_polytope} sends a vertex~$\mathbf{m}\in\cN(F)$ to the region~$\left\{\mathbf{x}\in \RR^n \, |\,  c_{\mathbf{m}} + \langle \mathbf{m}, \mathbf{x} \rangle = \min_{m'\in I} \Big( c_{\mathbf{m}'} + \langle \mathbf{m}', \mathbf{x} \rangle\Big) \right\}$, and a cell of~$\cS$ to the intersection of the regions corresponding to its vertices. See Figure~\ref{fig:tropical_normal_complex_with_polynomials} for an example. Due to the dual nature of Proposition~\ref{prop:bijections_dome_newton_polytope} we say that~$S(\cF)$ is the \defn{tropical dual} of~$\cT(F)$.

\begin{figure}[h!]
	\centering
	\includegraphics[scale=0.8]{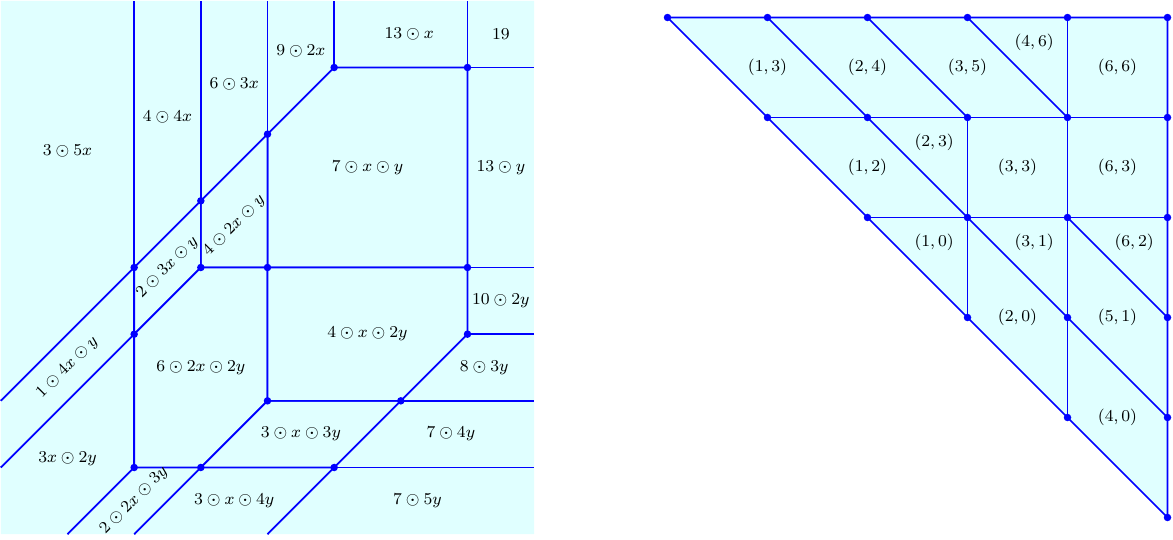}
	\caption[The orthogonal projection of~$\cD(F)$ and the dual subdivision~$\cS(F)$.]{ The orthogonal projection of~$\cD(F)$ (left) with regions labeled by the respective dominant monomial and the (flipped) dual subdivision~$\cS(F)$ with cells labeled by their respective exponent in~$\supp(F)$ (right) where~$F$ is the multiplication of the tropical polynomials in Figure~\ref{fig:tropical_polynomials}.}
	\label{fig:tropical_normal_complex_with_polynomials}
\end{figure}

We go a step further and restrict Proposition~\ref{prop:bijections_dome_newton_polytope} to the bounded cells of~$\cT(F)$.

\begin{lemma}[{\cite[Lem.5.2]{GMPTY23}}]\label{lem:tropical_dual_interior}
	The bijection of Proposition~\ref{prop:bijections_dome_newton_polytope} restricts to a bijection between the bounded cells of~$\cT(F)$ and the interior cells of~$\cS(F)$.
\end{lemma}

\begin{proof}
	Since~$\cS(F)$ is a regular subdivision of a polytope, it is a pure complex. Thus, we abuse the notation and call the~$(d-1)$-dimensional cells the facets of~$\cS(F)$.
	We first show the Lemma for the bounded edges of~$\cT(F)$ and the interior facets of~$\cS(F)$ in a fashion similar to the proof of~\cite[Theorem 1.13]{J21}.
	Let~$\mathbf{e}$ be an edge of~$\cT(F)$,~$H$ its corresponding facet in~$\cS$ via the bijection, and~$\widetilde{H}$ its lift to~$\widetilde{\cN}(F)$. We begin by characterizing these objects.

	A point~$\mathbf{x}\in \mathbf{e}$ exists if and only if the minimum~$F(\mathbf{x})$ equals~$c_{\mathbf{m}}+\langle \mathbf{m},\mathbf{x}\rangle$ for each~$\mathbf{m}\in H$. This means that taking~$(\mathbf{x},1):=(x_1,\ldots,x_n,1)$, we have that the linear form~$(\mathbf{x},1)$ attains its minimum on~$\widetilde{H}$ and the vector~$-(\mathbf{x},1)$ is in the normal cone of~$\widetilde{H}$. For~$H$, and~$\widetilde{H}$, notice that since~$\widetilde{\cN}(F)$ is Minkowski sum of a polytope and the ray~$\cone(\mathbf{e_{n+1}})$, a vector is normal to an unbounded face of~$\widetilde{\cN}(F)$ if and only if its last coordinate is~$0$.

	Suppose that~$\mathbf{e}$ is unbounded, that is,~$\mathbf{e}=\mathbf{w}+\cone(\mathbf{v})$ for some~$\mathbf{v},\mathbf{w}\in \RR^d$. From before we have that for all~$\lambda\in\RR$,~$-(\mathbf{w}+\lambda\mathbf{v},1)$ is in the normal cone of~$\widetilde{H}$. Taking the limit of~$\lambda\to\infty$ of~$-\left(\frac{1}{\lambda}\mathbf{w}+\mathbf{v}, \frac{1}{\lambda}\right)$ tells us that~$-(\mathbf{v}, 0)$ is also in the normal cone of~$\widetilde{H}$. Thus,~$H$ is in the boundary of~$\cS$.

	Reciprocally, if~$H$ is in the boundary of~$\cS$, it means that the normal cone of~$\widetilde{H}$ in~$\widetilde{\cN}(F)$ is a~$2$-dimensional cone whose extremal rays can be written~$-\cone\big((\mathbf{v},0)\big)$ and~$-\cone\big((\mathbf{w}, 1)\big)$, for some~$\mathbf{v},\mathbf{w}\in \RR^d$. Since for any~$\lambda\in \RR_{>0}$, the vector~$-(\mathbf{w}+\lambda \mathbf{v}, 1)$ is in this cone, for all~$\lambda\in \RR_{>0}$ we have that the point~$\mathbf{w}+\lambda \mathbf{v}$ belongs to the edge~$\mathbf{e}$ in~$\cT(F)$. Thus, this edge is unbounded.

	Having that the bijection restricts to a bijection between the bounded edges of~$\cT(F)$ and the interior facets of~$\cS$, we can extend easily the bijection to other faces. Notice that any cell of~$\cS$ is either maximal in dimension and is sent to a vertex of~$\cT(F)$, or it is an intersection of (possibly 1) facets of~$\cS$. A non-maximal cell of~$\cS$ is interior if and only if it is included only in interior facets of~$\cS$. Thus, by the inclusion-reversing bijection, it is sent to a cell of~$\cT(F)$ containing only bounded edges. Reciprocally, a non-bounded cell of~$\cT(F)$ contains a non-bounded edge, so it is sent to a boundary cell of~$\cS$. We get the desired result.
\end{proof}

\subsection{An Application of the Cayley Trick}

Consider a finite family of point configurations~$\cA_1,\ldots,\cA_k$ with associated height functions~$h_i:\cA_1\to\RR$. Following Remark~\ref{rem:tropical_polynomials_support_height_function_point_configuration} these point configurations correspond to tropical polynomials of the form~$F_i(\mathbf{x})=\bigoplus_{\bfa^i\in \cA_i}h_i(\bfa^i)\odot\mathbf{x}^{\bfa^i}$. Denote their multiplication by~$F=\bigodot_{i\in[k]F_i}$. Taking their corresponding Newton polytopes~$\cN(F_1),\ldots,\cN(F_k)$, we consider their Cayley embedding \begin{equation*}
	\cC\big(\cN(F_1),\ldots,\cN(F_k)\big):=\conv\Big(\{\mathbf{e_1}\}\times \cN(F_1), \ldots, \{\mathbf{e_k}\}\times \cN(F_k)\Big)
\end{equation*} and endow it with the height function on its integer points given by~$h_F\big((\mathbf{e_i},\bfa^i)\big)=h_i(\bfa_i)$. Let \defn{$\Sigma(F)$} be the regular subdivision of~$\cC\big(\cN(F_1),\ldots,\cN(F_k)\big)$ induced by~$h_F$ following the construction denoted in Definition~\ref{def:subdivison_from_height_function}.

In this context the Cayley trick (Proposition~\ref{prop:cayley_trick}) tells us that~$\Sigma_F$ corresponds to a mixed subdivision of~$\cN(F_1)+\cdots+\cN(F_k)=:\cN(F)$. The following proposition tells us that this mixed subdivision is one that we have already encountered.

\begin{proposition}[{\cite[Cor.4.9]{J21}}]\label{prop:cayley_trick_tropical}
	Let~$F=\bigodot_{i\in[k]F_i}$ be a multiplication of tropical polynomials. The mixed subdivision of~$\cN(F)$ corresponding to the subdivision~$\Sigma_F$ of the Cayley embedding~$\cC(\cN(F_1),\ldots,\cN(F_k))$ via the Cayley trick is the dual subdivision~$\cS(F)$.
\end{proposition}

We now reformulate Proposition~\ref{prop:bijections_dome_newton_polytope} using the context we have obtained from Proposition~\ref{prop:cayley_trick_tropical}.

\begin{proposition}\label{prop:arrangement_tropical_hypersurfaces}
	Let~$F=F_1\odot\cdots\odot F_k$ be a product of tropical polynomials. The tropical dual of the polyhedral complex~$\cT(F)$ (i.e. the union~$\bigcup_{i\in[k]}\cT(F_i)$) is the mixed subdivision~$\cS(F)$ obtained via the Cayley trick.
\end{proposition}

%% file: includes/contenu/chap_sorder_realizations.tex

\chapter{Realizing the~\texorpdfstring{$s$}{}-Permutahedron via Flow Polytopes}\label{chap:sorder_realizations}

\addcontentsline{lof}{chapter}{\protect\numberline{\thechapter}Realizing the~\texorpdfstring{$s$}{}-Permutahedron via Flow Polytopes}

We devote this chapter to answer Conjecture~\ref{conj:s-permutahedron} of Ceballos and Pons~\cite[{Conj.1}]{CP19} about realizing the~$s$-permutahedron~$\PSPerm$ as a polyhedral subdivision of a zonotope in the case where~$s$ is free of zeros. We make use of the combinatorial and geometrical toolbox of flows on graphs and tropical geometry described in Chapter~\ref{chap:sorder_Flows}. This chapter is based directly on~\cite{GMPTY23}. In what follows we present three realizations of~$\PSPerm$.

\section{The Flow Polytope Realization}\label{sec:realiz_flow_potyope}

Let us start by defining a graph that encodes the~$s$-weak order. Following the tradition of~\cite{BGHHKMY19} and~\cite{BGMY23} where the respective Caracol~$\car_{n}$ and~$s$-Caracol~$\car(s)$ graphs were defined, we define the~$s$-Oruga graph denoted~$\oru(s)$. Oruga and caracol are respectively the Spanish words for caterpillar and snail. These names come from the embedding of these graphs when~$s=(1,\ldots,1)$ as in Figure~\ref{fig:oru_car_mar}. Examples of this embedding for~$\car(s)$ and~$\oru(s)$ are shown in Figure~\ref{fig:s_car_oru}.

\begin{figure}[h!]
	\centering
	\includegraphics[scale=1.25]{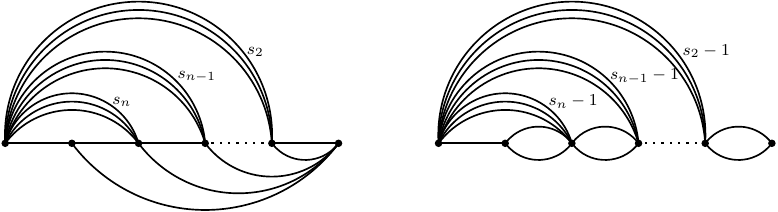}
	\caption[The~$s$-caracol~$\car(s)$ and~$s$-oruga~$\oru(s)$ graphs.]{ The~$s$-caracol~$\car(s)$ (left) and~$s$-oruga~$\oru(s)$ (right) graphs.}
	\label{fig:s_car_oru}
\end{figure}

\begin{definition} \label{def:Gs}
	Let~$s=(s_1, \ldots, s_n)$ be a composition, and for convenience of notation fix~$s_{n+1}=2$. The \defn{$s$-oruga graph}\index{$s$-oruga graph} denoted \defn{$\oru(s)$} is the graph on vertices~$\{v_{-1},v_0, \ldots, v_n\}$ such that for
	\begin{itemize}
		\itemsep0em
		\item $i\in [n+1]$, there are~$s_i-1$ \defn{source edges}\index{$s$-oruga graph!source edges}~$(v_{-1}, v_{n+1-i})$ labeled~$e^i_1,\ldots,e^i_{s_{i}-1}$,
		\item $i\in [n]$, there are two edges~$(v_{n+1-(i+1)}, v_{n+1-i})$ called \defn{bump}\index{$s$-oruga graph!bump} and \defn{dip}\index{$s$-oruga graph!dip} labeled~$e^{i}_0$ and~$e^{i}_{s_{i}}$.
	\end{itemize}

	For the rest of our work we endow~$\oru(s)$ with the following framing for: \begin{itemize}
		\itemsep0em
		\item the incoming edges of~$v_{n+1-i}$ are ordered~$e^i_j \prec_{\cI_{n+1-i}} e^i_k$ for~$0\leq j < k \leq s_{i}$,
		\item the outgoing edges of~$v_{n+1-i}$ are ordered~$e^{i-1}_0\prec_{\cO_{n+1-i}} e^{i-1}_{s_{i-1}}$.
	\end{itemize}

	When~$s=(1,\ldots,1)$, following Remark~\ref{rem:vertex_v_m1} we contract the edge~$(v_{-1},v_0)$ and call the resulting graph the \defn{oruga graph}\index{$s$-oruga graph} denoted \defn{$\oru_n$}.
\end{definition}

Figure~\ref{s_oru_edges_framings_labels} shows examples of this construction. We choose to draw the graph~$\oru(s)$ in such a way that the framing of the incoming and outgoing edges at each inner vertex can be read ``from top to bottom''.

\begin{figure}[h!]
	\centering
	\includegraphics[scale=1.25]{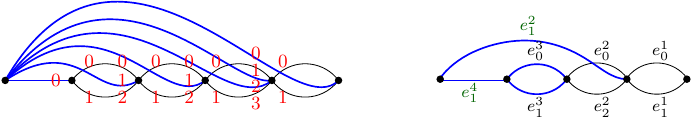}
	\caption[The framed graphs~$\oru((2,3,2,2))$ and~$\oru((1,2,1))$.]{ The graph~$\oru(s)$ for~$s=(2,3,2,2)$ with framing shown in red (left) and the graph~$\oru(s)$ for~$s=(1,2,1)$ with edge labels (right). Source edges are bolded in blue.}
	\label{s_oru_edges_framings_labels}
\end{figure}

\begin{remark}\label{rem:oru_dimension}
	Given a composition~$s=(s_1,\ldots,s_n)$, we denote \defn{$|s|$} the sum of its entries.~$\fpol[\oru(s)](\bfa)$ lives in dimension \begin{equation*}
		\Big|E\Big(\oru(s)\Big)\Big|=\left(\sum_{i=1}^{n+1}s_i-(n+1)+2n\right)=|s|+n+1
	\end{equation*} and following Proposition~\ref{def:flow_polytope_dimension} we have that \begin{equation*}
		\dimension\Big(\fpol[\oru(s)](\bfa)\Big)=|E|-|V|+1=|s|+n+1-(n+2)+1=\sum_{i=1}^n s_i=|s|.
	\end{equation*}
\end{remark}

\begin{remark}\label{rem:vertex_v_m1}
	Although the graph~$\oru(s)$ starts with vertex~$v_{-1}$ instead of~$v_0$, all the technology of Section~\ref{sec:flows_on_graphs} can be applied to it. To see this, we can either contract the edge~$e_1^{n+1}$ to obtain a graph~$\oru(s)^\#$ such that~$\fpol[\oru(s)^\#](\bfa)$ is integrally equivalent to~$\fpol[\oru(s)](\bfa)$, or simply relabel the vertices with~$[0,n+1]$. The resulting graph has flows and routes directly in bijection with the flows and routes of~$\oru(s)$. Similar to Remark~\ref{rem:multi_edges}, this is a particular case of~\cite[Cor.2.13]{GHMY21}.
\end{remark}

We now start relating the combinatorics of~$\oru(s)$ to the combinatorics of the~$s$-weak order.

\begin{theorem}\label{thm:bij_simplices_permutations}
	Let~$s$ be a composition and~$\bfd:=(0,d_0,d_1,\ldots,d_{n-1},-\sum_{i=0}^{n-1} d_i)$ where~$d_i:=\indeg_i(\oru(s))$. The set of Stirling~$s$-permutations is in bijection with the set of integer~$\bfd$-flows of~$\oru(s)$.
\end{theorem}
\begin{proof}
	We begin by highlighting the fact that~$\bfd=(0,0,s_n, s_{n-1},\ldots, s_2,-\sum_{i=2}^n s_i)$.
	Notice that any integer~$\bfd$-flow~$f$ on~$\oru(s)$ has zero flow on every source edge. Therefore,~$f$ is characterized by the fact that the sum of the flows passing by any pair of bump and dip edges satisfies~$f(e_0^i) + f(e_{s_i}^{i}) = s_n + \cdots + s_{i+1}$ for all~$i\in[n-1]$. This means that it suffices to describe the flow in the bump edges~$e_0^i$ for all~$i\in [n-1]$ to describe an integer~$\bfd$-flow on~$\oru(s)$.

	Given a Stirling~$s$-permutation~$w$, let~$f(e_0^i)$ be the number letters strictly greater than~$i$ that occur before the~$i$-block~$B_i$ in~$w$.
	As for each~$j>i$ there are at most~$s_j$ such repetitions of~$j$, this quantity satisfies~$0\leq f_{e_0^i} \leq s_n+\cdots+s_{i+1}$. Meaning that it defines an integer~$\bfd$-flow on~$\oru(s)$.

	Conversely, given a~$\bfd$-flow on~$\oru(s)$, we can build a Stirling~$s$-permutation from the flow on the bumps~$e_0^i$ for~$i\in [n-1]$ via an insertion algorithm in the following way. For step~$0$ we begin with the~$s_n$ consecutive copies of~$n$. At step~$i$, among the~$s_n+\cdots+s_{n-i+1}+1$ possible positions between letters, insert the~$s_{n-i}$ consecutive copies of~$n-i$ in the~$f(e_0^{n-i})$-th position. After step~$n-1$ we obtain a permutation of the word~$1^{s_1}2^{s_2}\cdots n^{s_n}$. This permutation is~$121$-avoiding as all values have been placed in descending order and by blocks.
\end{proof}

In Figure~\ref{fig:insertion_algorithm} we illustrate the bijection of Proposition~\ref{thm:bij_simplices_permutations} including the insertion algorithm.

\begin{figure}[!ht]
	\centering
	\includegraphics[scale=0.95]{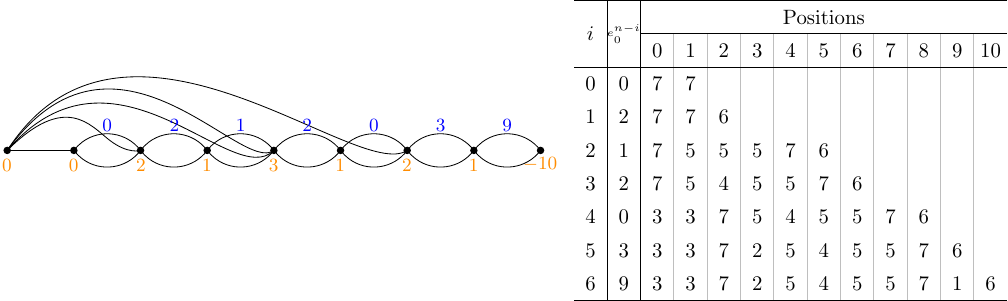}
	\caption[An integer~$\bfd$-flow of~$\oru((1,1,2,1,3,1,2))$ and the insertion algorithm giving the corresponding Stirling~$s$-permutation~$w=33725455716$]{ An integer~$\bfd$-flow of~$\oru((1,1,2,1,3,1,2))$ (left) (only the flow passing by the bumps is shown in blue) and the steps of the insertion algorithm that output the corresponding to the Stirling~$s$-permutation~$w=33725455716$ (right).
	}
	\label{fig:insertion_algorithm}
\end{figure}

Due to Proposition~\ref{prop:volume_intflows}, as the normalized volume of the flow polytope~$\mathcal{F}_{\oru(s)}(\bfi)$ is the number of integer~$\bfd$-flows on~$\oru(s)$, then we obtain the following enumerative corollary.

\begin{corollary} \label{cor:volume is number of trees}
	Given a composition~$s$, then the volume of~$\mathcal{F}_{\oru(s)}$ is the number of~$s$-decreasing trees and the number of Stirling~$s$-permutations.
	\[\vol\Big(\fpol[\oru(s)](\bfi)\Big)
		= |\cT_s|
		= |\cW_s|
		= \prod_{i=1}^{n-1}\left(1+s_{n-i+1}+s_{n-i+2}+\cdots + s_n\right).
	\]
\end{corollary}

\begin{remark}\label{rem:bij_simplices_trees}
	It is also possible to give an explicit correspondence between~$s$-decreasing trees and integer~$\bfd$-flows of~$\oru(s)$ even when~$s$ is a weak composition.

	Given an integer~$\bfd$-flow~$f$ of~$\oru(s)$, we construct an~$s$-decreasing tree via induction as follows. At step~$0$ start with a tree with root the node~$n$ and~$s_n+1$ leaves. At step~$i$ for~$i \in [ n-1]$, we have a decreasing tree with labeled nodes~$n$ to~$n+1-i$, and~$1+\sum_{k=n+1-i}^{n} s_{k}$ leaves that we momentarily label from~$0$ to~$\sum_{k=n+1-i}^{n} s_{k}$ along the counterclockwise walk of the tree. Graft to the leaf labeled~$f(e^{n-i}_0)$ the labeled node~$n-i$ with~$s_{n-i}+1$ leaves. After step~$n-1$ we obtain a decreasing tree with each node labeled~$i$ having~$s_i$ children.
	Conversely, any~$s$-decreasing tree can be built in this way as seen in Proposition~\ref{prop:num_s_trees}. Therefore, any~$s$-decreasing tree is associated to a choice of integers~$f(e^i_0)\in [0, \sum_{k=n+1-i}^{n} s_{k}]$ for all~$i\in[ n-1]$. That is, a~$\bfd$-flow of~$\oru(s)$.

	See Figure~\ref{fig:s_insertion_algorithm_tree} for an example of the bijection.
\end{remark}

\begin{figure}[!ht]
	\centering
	\includegraphics[scale=1.15]{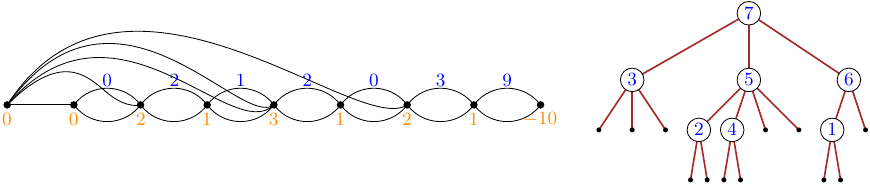}
	\caption[An integer~$\bfd$-flow of~$\oru((1,1,2,1,3,1,2))$ together with its~$(1,1,2,1,3,1,2)$-decreasing tree.]{ An integer~$\bfd$-flow of~$\oru((1,1,2,1,3,1,2))$ (left) (only the flow passing by the bumps is shown in blue) and its corresponding~$(1,1,2,1,3,1,2)$-decreasing tree (right).
	}
	\label{fig:s_insertion_algorithm_tree}
\end{figure}

The routes of~$\oru(s)$ play a key role in what follows, so we denote them as \defn{$R(k, t, \delta)$}\index{$s$-oruga graph!route}. Intuitively this notation comes from the fact that every route of~$\oru(s)$ starts from~$v_{-1}$, lands in a vertex~$v_{n+1-k}$ via a source edge labeled~$e_{t}^k$ and then follows~$k-1$ edges that are either bumps or dips meaning a binary choice~$\delta$. Figure~\ref{fig:s-oruga_graph_route2} contains an example of this. More formally, we give the following definition.

\begin{definition}\label{def:routes_Gs}
	For~$k\in[n+1]$,~$t\in [ s_k-1]$, and~$\delta=(\delta_1, \ldots, \delta_{k-1})\in \{0, 1\}^{k-1}$, we denote by \defn{$R(k, t, \delta)$} the sequence of edges~$(e^{k}_{t_k}, \, e^{k-1}_{t_{k-1}}, \, \ldots, \, e^1_{t_1})$
	where \begin{itemize}
		\itemsep0em
		\item $t_k:=t$,
		\item for all~$j\in [ k-1]$,~$t_j :=\delta_j s_j$.
	\end{itemize}
\end{definition}

\begin{figure}[ht!]
	\centering
	\includegraphics[scale=1.25]{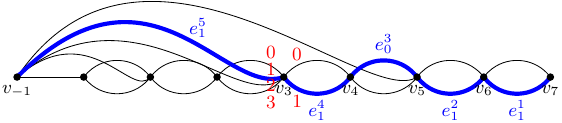}
	\caption[The route~$R(5,1,(1,0,1,1))$ of~$\oru((1,1,2,1,3,1,2))$ together with the framing orders~$\cI_3$ and~$\cO_3$.]{ The route~$R(5,1,(1,0,1,1))$ of~$\oru((1,1,2,1,3,1,2))$ (bolded in blue) together with the framing orders~$\cI_3$ and~$\cO_3$.}
	\label{fig:s-oruga_graph_route2}
\end{figure}

Wanting to use the technology of Subsection~\ref{ssec:DKK_triangulation} and obtain~$\triangDKK[\oru(s)]$, we describe the maximal cliques of routes of~$\oru(s)$ via Stirling~$s$-permutations.

\begin{definition}\label{def:Lw}
	Let~$s=(s_1,\ldots,s_n)$ be a composition and~$w$ a Stirling~$s$-permutation. Consider~$u$ to be a (possibly empty) prefix of~$w$.
	For all~$a\in [n]$, we denote by~$t_a$ (or \defn{$t_a(u)$} if~$u$ is not clear from the context) the number of occurrences of~$a$ in~$u$, and by~$c$ the smallest value in~$[n]$ such that~$0<t_c<s_{c}$. If there is no such value, we set~$c=n+1$ and~$t_{n+1}=1$. Notice that the minimality of~$c$ implies that either~$t_a=0$ or~$t_a=s_a$ for all~$a<c$ giving us a~$\delta\in[c-1]$. We denote \defn{$\rpre{u}$} the route~$R(c,t_c,\delta)=(e^{c}_{t_c}, \, e^{c-1}_{t_{c-1}}, \, \ldots, \, e^1_{t_1})$.

	For~$i\in [|s|]$, we call \defn{$w_i$} the~$i$-th letter of~$w$, and for~$i\in [0, |s|]$ we denote by \defn{$\prefix{w}{i}$} the prefix of~$w$ of length~$i$, with~$\prefix{w}{0}:=\emptyset$. We denote by \defn{$\Delta_w$} the set of routes $\{\rpre{\prefix{w}{i}}\,|\, i\in[0,|s|]\}$ and abusing notation identify it with the simplex whose vertices are the indicator vectors of these routes.
\end{definition}

\begin{remark}\label{rem:oru_exceptional_routes}
	Notice that the exceptional routes of~$(\oru(s),\preceq)$ are the two routes that are respectively formed only by bumps~$R(n+1,1,(0)^n)=(e^{n+1}_1, e^n_0, \ldots, e^1_0)$ or by dips~$R(n+1,1,(1)^n)=(e^{n+1}_1, e^n_1, \ldots, e^1_1)$. Following Definition~\ref{def:Lw} we can see that~$\Delta_w$ always contains these two routes as~$\rpre{\prefix{w}{0}}=R(n+1, 1, (0)^n)$ and~$\rpre{\prefix{w}{|s|}}=R(n+1, 1, (1)^n)$. See Figure~\ref{fig:l_clique_3221} for an example of~$\Delta_w$.
\end{remark}

\begin{figure}[h!]
	\centering
	\includegraphics{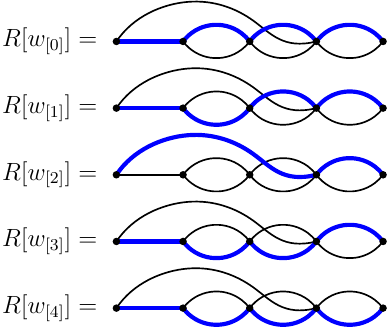}
	\caption[The maximal clique corresponding to the Stirling~$(1,2,1)$-permutation~$w=3221$.]{ The maximal clique~$\Delta_w=\{\rpre{\prefix{w}{0}},\ldots,\rpre{\prefix{w}{4}}\}$ corresponding to the Stirling~$(1,2,1)$-permutation~$w=3221$. Each route is bolded in blue.}
	\label{fig:l_clique_3221}
\end{figure}

We are now ready to characterize~$\triangDKK[\oru(s)]$ using the~$s$-weak order.

\begin{lemma}\label{lem:max_clique}
	The maximal simplices of~$\triangDKK[\oru(s)]$ are exactly the simplices~$\Delta_{w}$ where~$w$ ranges over all Stirling~$s$-permutations.
\end{lemma}

\begin{proof}
	Recall from Proposition~\ref{prop:DKK_triangulation} that the maximal simplices of~$\triangDKK[\oru(s)]$ are the simplices~$\Delta_C$, where~$C$ is a maximal clique of coherent routes of~$(\oru(s), \preceq)$. We claim that for a Stirling~$s$-permutation~$w$, the simplex~$\Delta_w$ is a clique of coherent routes of~$(\oru(s), \preceq)$.
	Let~$1\leq i < j\leq [|s|]$ corresponding to two routes~$\rpre{\prefix{w}{i}}$ and~$\rpre{\prefix{w}{j}}$ in~$\Delta_w$.
	As~$i<j$, we have the inequalities~$t_a(\prefix{w}{i})\leq t_a(\prefix{w}{j})$ for all~$a\in [n]$. This means that for any common vertex~$v_{n+1-a}$ of~$\rpre{\prefix{w}{i}}$ and~$\rpre{\prefix{w}{j}}$, we have that the incoming (resp.\ outgoing) edge of~$\rpre{\prefix{w}{i}}$ is smaller than the incoming (resp.\ outgoing) edge of~$\rpre{\prefix{w}{j}}$ in~$\preceq$.
	Thus, the routes~$\rpre{\prefix{w}{i}}$ and~$\rpre{\prefix{w}{j}}$ are coherent, and~$\Delta_w$ is a clique of coherent routes. Now, since~$\Delta_w$ has~$|s|+1=\text{dim}(\fpol[\oru(s)])+1$ elements, it is a maximal clique and corresponds to a maximal simplex in~$\triangDKK[\oru(s)]$.

	Now let us see that assigning~$w$ to~$\Delta_w$ is injective. Let~$w'$ be a Stirling~$s$-permutation such that~$w'\neq w$ and~$i\in [|s|-1]$ be the minimal index such that~$w_i \neq w'_i$, meaning that~$R[w_{[j]}]=R[w'_{[j]}]$ for all~$j<i$. Without loss of generality we suppose that~$a:=w_i<w'_i:=b$. We claim that~$\rpre{\prefix{w'}{i}}$ cannot belong to~$\Delta_w$. Notice that~$a$ (resp.~$b$) is the minimal value~$c$ in~$w_{[i]}$ (resp.~$w'_{[i]}$) such that~$0<t_c<s_c$. If this were not the case then we would have that~$w$ (resp.~$w'$) contains the pattern~$121$. Thus,~$\rpre{\prefix{w}{i}}$ contains the edge~$e^{a}_{t_{a}(\prefix{w}{i})}$ while~$\rpre{\prefix{w'}{i}}$ contains~$e^{a}_{t_{a}(\prefix{w'}{i})}=e^{a}_{t_{a}(\prefix{w}{i})-1}$ as~$a<b$. Meaning that~$R[w_{[i]}]\neq R[w'_{[i]}]$. Finally, for any~$j>i$ we have that~$t_{a}(\prefix{w}{j})\geq t_{a}(\prefix{w}{i})$, so~$e^{a}_{t_{a}(\prefix{w}{i})-1}\notin\rpre{\prefix{w}{j}}$. Therefore,~$\rpre{\prefix{w'}{i}}\notin\Delta_w$ and the map~$w \mapsto \Delta_w$ is injective.

	Through the chain of bijections from maximal simplices of~$\triangDKK[\oru(s)]$, to~$\bfd$-flows of~$\oru(s)$ (Proposition~\ref{prop:bij_cliques_intflows}), to~$s$-decreasing trees (Remark~\ref{rem:bij_simplices_trees}), to Stirling~$s$-permutations (Proposition~\ref{prop:s_tree_to_s_permutation}) we get that this injection is a bijection.
\end{proof}

Combining previous results, we now have the following commutative diagram of bijections:

\begin{center}
	\begin{tikzpicture}
		\begin{scope}[xscale=2]
			\node[](n1) at (0,0) {$\cT_s$};
			\node[](n2) at (1.5,0) {$\cW_s$};
			\node[](n3) at (3,0) {$\mathcal{F}_{\oru(s)}^{\mathbb{Z}}(\bfd)$};
			\node[](n4) at (5.5,0) {$\maxcliques[\oru(s)]$};

			\draw[-stealth] (n1)--(n2);
			\draw[-stealth] (n2)--(n1);
			\draw[-stealth] (n2)--(n3);
			\draw[-stealth] (n3)--(n2);
			\draw[-stealth] (n3)--(n4);
			\draw[-stealth] (n4)--(n3);
			\draw[-stealth] (n1) to[out=50,in=-210] (n3);
			\draw[-stealth] (n3) to[out=-210,in=50] (n1);
			\draw[-stealth] (n2) to[out=-45,in=200] (n4);
			\draw[-stealth] (n4) to[out=200,in=-45] (n2);

			\node[] at (.75,.2) {\tiny Prop.~\ref{prop:s_tree_to_s_permutation}};
			\node[] at (2.075,.2) {\tiny Prop~\ref{thm:bij_simplices_permutations}};
			\node[] at (3.95,.2) {\tiny Prop~\ref{prop:bij_cliques_intflows}};
			\node[] at (1.3,1.) {\tiny Rem.~\ref{rem:bij_simplices_trees}};
			\node[] at (3,-0.6) {\tiny Lem.~\ref{lem:max_clique}};

		\end{scope}
	\end{tikzpicture}
\end{center}

\subsection{The~\texorpdfstring{$1$}{}-Skeleton of the~\texorpdfstring{$s$}{}-Weak Order}

\begin{theorem}\label{thm:cover_relations}
	Let~$s=(s_1, \ldots, s_n)$ be a composition and~$w$ and~$w'$ be two Stirling~$s$-permutations.
	There is a cover relation between~$w$ and~$w'$ in the~$s$-weak order if and only if the simplices~$\Delta_{w}$ and~$\Delta_{{w'}}$ are adjacent in~$\triangDKK[\oru(s)]$.
\end{theorem}

\begin{proof}
	Suppose that~$w'$ is obtained from~$w$ by a transposition along the ascent~$(a,c)$. From Corollary~\ref{prop:cover_relations_multiperm} we have that~$w=u_1 B_a c u_2$ and~$w'=u_1 c B_a u_2$ where~$u_1$ and~$u_2$ are words in~$[n]$ and~$B_a$ is the~$a$-block of the permutations.
	We denote by~$\ell(u)$ the length of a word~$u$. Notice that for all~$i\in [0, \ell(u_1)]\cup[\ell(u_1)+\ell(B_a)+1, |s|]$, the routes~$\rpre{\prefix{w}{i}}$ and~$\rpre{\prefix{w'}{i}}$ are equal as the corresponding prefixes consist of the same amounts of the same values.

	For the indices~$i\in[\ell(u_1)+1, \ell(u_1)+\ell(B_a)-1]$ we claim that the routes~$\rpre{\prefix{w}{i}}$ and~$\rpre{\prefix{w'}{i+1}}$ are equal as well. Notice that for these~$i$ we have that~$t_b(\prefix{w}{i})=t_{b}(\prefix{w'}{i+1})$ for all~$b\in [n]\setminus \{c\}$. Also, since we are reading the substring~$B_a$ with these indices, we have that~$0<t_a(\prefix{w}{i})<s_a$. Using that~$a<c$, we get that the value~$t_c(\prefix{w}{i})$ (respectively~$t_c(\prefix{w'}{i+1})$) does not play a role in the determination of the edges for the route~$\rpre{\prefix{w}{i}}$ (respectively~$\rpre{\prefix{w'}{i+1}}$). In this way, the vertices of~$\Delta_w$ and~$\Delta_{w'}$ differ only in one element. Namely,~$\rpre{\prefix{w}{\ell(u_1)+\ell(B_a)}}\in \Delta_w$ corresponding to the prefix~$u_1B_a$ of~$w$ and~$\rpre{\prefix{w'}{\ell(u_1)+1}}\in \Delta_{w'}$ corresponding to the prefix~$u_1c$ of~$w'$. As simplices this means that~$\Delta_{w'}$ and~$\Delta_{w}$ share a common facet in~$\triangDKK[\oru(s)]$.

	Reciprocally, suppose that~$\Delta_w$ and~$\Delta_{w'}$ are adjacent and thus differ only in one vertex (i.e.\ one route).
	We denote~$u_1$ the longest common prefix of~$w$ and~$w'$ and~$a:=w_{\ell(u_1)+1}$,~$c:=w'_{\ell(u_1)+1}$. Suppose without loss of generality that~$a<c$. This gives us that~$a\notin u_1$ as if~$a\in u_1$ then~$w'$ would contain the subword~$aca$ and thus the pattern~$121$. Therefore,~$u_1B_a$ is a prefix of~$w$ and the route~$\rpre{\prefix{w}{\ell(u_1)+\ell(B_a)}}$ (resp.~$\rpre{\prefix{w'}{\ell(u_1)+1}}$) is in~$\Delta_w$ (resp.~$\Delta_{w'}$) but not in~$\Delta_{w'}$ (resp.~$\Delta_{w}$).

	Thus, the only possibility that~$\Delta_w$ and~$\Delta_{w'}$ differ only on these elements is that~$w=u_1B_acu_2$ and~$w'=u_1cB_au_2$, where~$u_2$ is their longest common suffix.
	This means that there is an~$s$-tree rotation along the ascent~$(a,c)$ between~$w$ and~$w'$ and Corollary~\ref{prop:cover_relations_multiperm} gives us that~$w\lessdot w'$.
\end{proof}

\begin{remark}\label{rem:interior_facet_ascent}
	In this context, we say that the common facet of~$\Delta_{w}$ and~$\Delta_{{w'}}$ is associated to the transposition of~$w$ along~$(a,c)$. Notice that such a facet lies in the interior of~$\fpol[\oru(s)]$ since it separates two interior maximal simplices of~$\triangDKK[\oru(s)]$.
\end{remark}

Figure~\ref{fig:adjacency_graph_121} shows the graph dual to the DKK triangulation of~$\fpol[\oru(s)](\bfi)$ for~$s=(1,2,1)$, which corresponds to the (unoriented) Hasse diagram of the~$(1,2,1)$-weak order. Notice that in this Figure we omit the routes~$\rpre{\prefix{w}{0}}$ and~$\rpre{\prefix{w}{|s|}}$ since both appear in~$\Delta_{w}$ for every~$w \in \mathcal{W}_{(1,2,1)}$.

\begin{figure}[h!]
	\centering
	\includegraphics[scale=0.74]{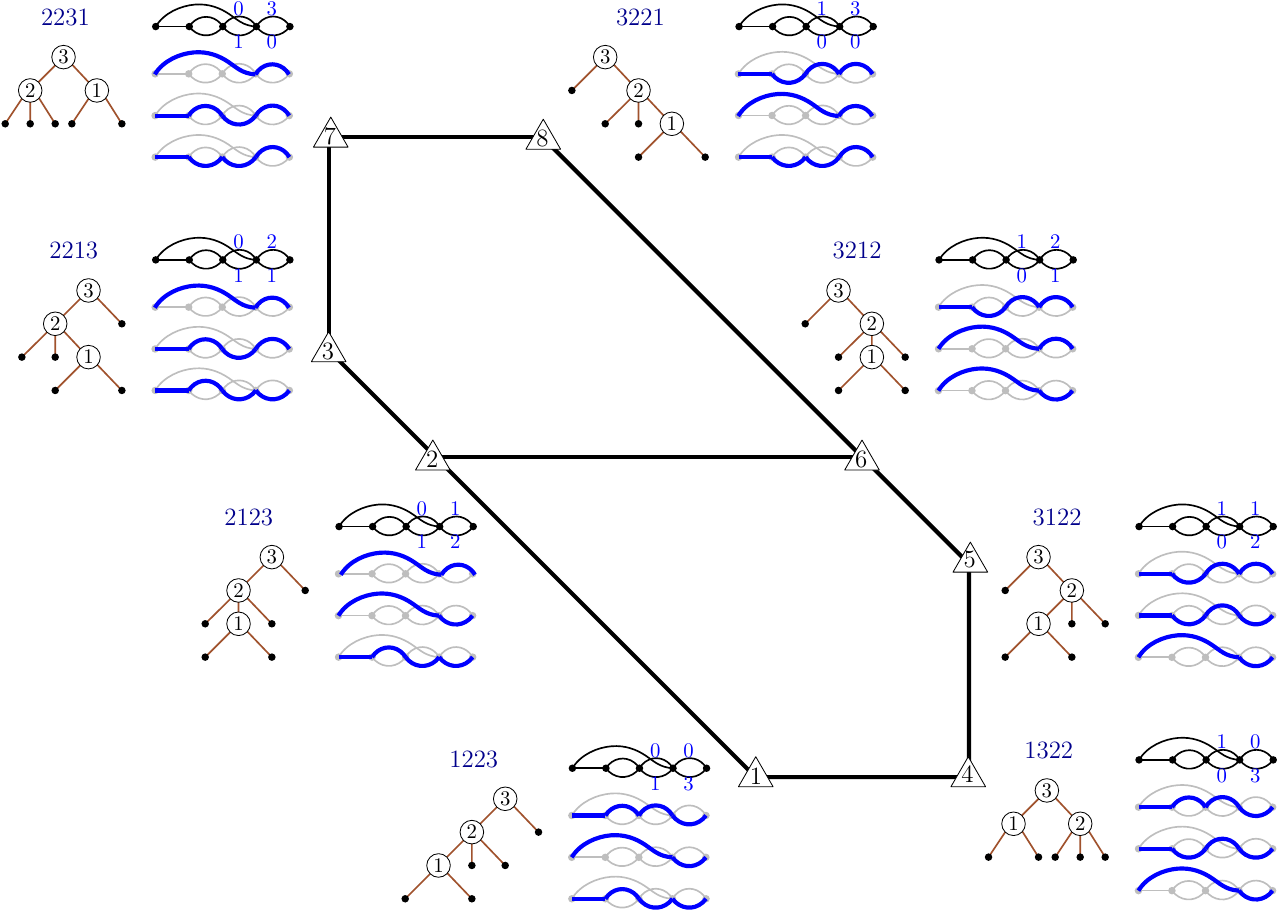}
	\caption[The dual realization of the~$(1,2,1)$-permutahedron with vertices indexed by~$s$-decreasing trees, Stirling~$s$-permutations, maximal cliques of routes, and integer flows.]{ The dual realization of the~$(1,2,1)$-permutahedron with vertices indexed by~$s$-decreasing trees, Stirling~$s$-permutations, maximal cliques of routes (omitting~$R(n+1, 1, (0)^n)$ and~$R(n+1, 1, (1)^n)$), and integer flows.}
	\label{fig:adjacency_graph_121}
\end{figure}

\subsection{Higher Faces of the~\texorpdfstring{$s$}{}-Permutahedron}

We now show that not only the vertices and edges but all the faces of the~$s$-permutahedron are also encoded in the triangulation~$\triangDKK[\oru(s)].$ We begin with some technical results.

\begin{lemma}\label{lem:routes_invsets}
	Let~$w$ be a Stirling~$s$-permutation and~$R=R(k,t,\delta)$ a route of~$\oru(s)$.~$R$ is a vertex of~$\Delta_w$ if and only if the inversion set of~$w$ satisfies the following inequalities:

	\begin{enumerate}
		\itemsep0em
		\item $|(k,i)_w|\geq t$ for all~$1\leq i < k$ such that~$\delta_i=0$,
		\item $|(k,i)_w|\leq t$ for all~$1\leq i < k$ such that~$\delta_i=1$,
		\item $|(j,i)_w|=0$ for all~$1\leq i < j <k$ such that~$(\delta_i, \delta_j)=(1,0)$,
		\item $|(j,i)_w|=s_j$ for all~$1\leq i < j <k$ such that~$(\delta_i, \delta_j)=(0,1)$.
	\end{enumerate}
\end{lemma}

We say that the route~$R$ \defn{implies}\index{graph!route!implies} these inequalities on inversion sets.

\begin{proof}
	($\Rightarrow$) Suppose that~$R(k,t,\delta)$ is a vertex of~$\Delta_w$. It means that~$R(k,t,\delta)=\rpre{\prefix{w}{r}}$ for a certain~$r\in[0, |s|]$. This signifies that~$t$ corresponds to the number of occurrences of~$k$ in~$\prefix{w}{r}$, and for all~$1\leq i < k$, the number of occurrences of~$i$ in~$\prefix{w}{r}$ is~$0$ (resp.~$s_i$) if~$\delta_i=0$ (resp.~$\delta_i=1$). This gives the announced inequalities on the inversion set of~$w$.

	($\Leftarrow$) Reciprocally, suppose that~$\inv(w)$ satisfies these inequalities. Then there is a prefix~$\prefix{w}{r}$ of~$w$ that contains no occurrence of~$i$ for~$i$ such that~$\delta_i=0$, all~$s_i$ occurrences of~$i$ for~$i$ such that~$\delta_i=1$, and exactly~$t$ occurrences of~$k$. Then this prefix is exactly associated to the route~$\rpre{\prefix{w}{r}}=R(k,t,\delta)$ following Definition~\ref{def:Lw}.
\end{proof}

\begin{definition}\label{def:simplex_of_faces}
	Let~$(w, A)$ be a face of~$\PSPerm$. We define \defn{$\Delta_{(w, A)}$} as the following intersection of facets of~$\Delta_w$ following Remark~\ref{rem:interior_facet_ascent}: \begin{equation*}
		\Delta_{(w,A)}:=\bigcap_{(a,c)\in A} \left\{\Delta_w\cap \Delta_{w'} \, : \, w' \text{ is the transposition of~$w$ along~$(a,c)$} \right\},
	\end{equation*}
	and~$\Delta_{(w,A)}:=\Delta_w$ if~$A=\emptyset$.
\end{definition}

\begin{remark}\label{rem:routes_from_delta_w_in_delta_wA}
	Notice that the~$|A|$ routes that are in~$\Delta_w\setminus{\Delta_{(w,A)}}$ correspond to the prefixes of~$w$ that end precisely between two letters of~$w$ that form an ascent in~$A$.
\end{remark}

\begin{lemma}\label{lem:Delta_(T,A)}
	Let~$(w, A)$ be a face of~$\PSPerm$ and~$w'$ be a Stirling~$s$-permutation. Then,~$w'$ is in the interval~$[w, w+A]$ if and only if~$\Delta_{(w,A)}\subseteq \Delta_{w'}$.
\end{lemma}

\begin{proof}
	Recall that Theorem~\ref{thm:transitiveclosure} describes~$\inv(w+A)$ and that~$w'\in [w, w+A]$ if and only if~$\inv(w')$ satisfies~$|(c,a)_w|\leq |(c,a)_{w'}|\leq |(c,a)_{w+A}|$ for all~$1\leq a < c \leq n$. We proceed to show that these inequalities are exactly the ones implied by the routes that are vertices of~$\Delta_{(w,A)}$, in the sense of Lemma~\ref{lem:routes_invsets}.

	Let~$(a,c)$ be a pair with~$|(c,a)_w|=:t$.
	For each inequality of Lemma~\ref{lem:routes_invsets} we describe a route~$R\in\Delta_{(w,A)}$ such that~$R\in\Delta_{w'}$ that implies it.
	\begin{enumerate}
		\itemsep0em
		\item Consider the inequality~$|(c,a)_{w'}|\geq t$. We can take the route that corresponds to the prefix of~$w$ that ends following the~$t$-th occurrence of~$c$ and that does not correspond to an ascent in~$A$. Such ending point necessarily lies before the~$a$-block since~$a<c$.
		\item Consider the inequality~$|(c,a)_{w'}|\leq t$ if and only if~$|(c,a)_{w+A}|=t$ (i.e.\ the pair~$(a,c)$ is not~$A$-dependent). This inequality is only implied by routes that contain the edges~$e^c_t$ and~$e^a_{s_a}$. Such routes in~$\Delta_w$ correspond to prefixed of~$w$ that end between the last occurrence of~$a$ and the~$t$-th occurrence of~$c$ and that are not inside a~$b$-block~$B_b$ for any~$b<c$. The pair~$(a,c)$ is~$A$-dependent exactly when all such breaks are in~$A$, so the corresponding routes are removed in~$\Delta_{(w,A)}$.
		\item Consider the inequality~$|(c,a)_{w'}|\leq t+1$ if~$t+1< s_c$ and~$|(c,a)_{w+A}|=t+1$ (i.e.~$(a,c)$ is an~$A$-dependent pair).
		      We can take the route that corresponds to the prefix of~$w$ that ends immediately after the~$(t+1)$-th occurrence of~$c$. Since~$t+1<s_c$, other occurrences of~$c$ appear afterwards. Therefore, this break is not an ascent. Notice that if~$t+1=s_c$ there is no need to check that~$|(c,a)_{w'}|\leq s_c$. \qedhere
	\end{enumerate}
\end{proof}

Lemma~\ref{lem:Delta_(T,A)} also gives us the following alternative characterization of~$\Delta_{(w,A)}$.

\begin{corollary}\label{cor:Delta_T_a_alt}
	$\Delta_{(w,A)}=\bigcap_{w' \in [w, w+A]} \Delta_{w'}$.
\end{corollary}

\begin{lemma}\label{cor:interiorsimplicesDKK}
	Let~$C$ be a clique of routes of~$(\oru(s), \preceq)$ that contains the exceptional routes and at least one route that starts with~$e$ for each source edge~$e$ that is not~$(v_{-1}, v_0)$, then~$\Delta_C$ is in the interior of~$\triangDKK[\oru(s)]$.
\end{lemma}

\begin{proof}
	Suppose that~$\Delta_C$ is a boundary simplex of~$\triangDKK[\oru(s)]$. Since it is contained in a facet that is in the boundary of~$\triangDKK[\oru(s)]$, this facet corresponds to a clique of the form~$\Delta_w\setminus R$, where~$w$ is a Stirling~$s$-permutation and~$R$ is a route of~$\Delta_w$ that does not correspond to an ascent nor a descent of~$w$. Hence, either~$R$ is an exceptional route or it corresponds to a prefix~$\prefix{w}{i}$ such that~$w_i=w_{i+1}$. In this case,~$w_i$ is the~$t$-th occurrence of some~$c$ in~$w$ with~$t<s_c$. In this scenario~$R$ is the only route of~$\Delta_w$ that starts with the edge~$e^c_t$. As~$C\subseteq \Delta_w\setminus R$, this means that~$C$ does not satisfy the condition of the lemma.
\end{proof}

\begin{corollary}\label{cor:Delta(w,A)_interior}
	Let~$w$ be a Stirling~$s$-permutation and~$A$ a subset of its ascents. Then~$\Delta_{(w,A)}$ is an interior simplex of~$\triangDKK[\oru(s)]$.
\end{corollary}

\begin{proof}
	Due to Lemma~\ref{cor:interiorsimplicesDKK} it is sufficient to show that~$\Delta_{(w,A)}$ contains the exceptional routes and at least one route that starts with~$e$ for each source edge~$e$ that is not~$(v_{-1}, v_0)$. First, it is clear that~$\Delta_{(w,A)}$ contains the exceptional routes because of Corollary~\ref{cor:Delta_T_a_alt}. Now let~$c\in[n]$ and~$t\in[s_{c}-1]$. In this situation the prefix of~$w$ that ends with the~$t$-th occurrence of~$c$ corresponds to a route~$R$ that contains the edge source edge~$e^c_t$. Moreover, there cannot be an ascent of~$w$ after this prefix since there are still occurrences of~$c$ afterwards. Thus, the route~$R$ is not removed from~$\Delta_w$ to~$\Delta_{(w,A)}$.
\end{proof}

\begin{theorem}\label{thm:bij_interiorfacesDKK_facessperm}
	The application~$(w,A) \mapsto \Delta_{(w,A)}$ induces a poset isomorphism between the face poset of~$\PSPerm$ and the set of interior simplices of~$\triangDKK[\oru(s)]$ ordered by reverse inclusion.
\end{theorem}

\begin{proof}
	The fact that all~$\Delta_{(w,A)}$ are interior simplices of~$\triangDKK[\oru(s)]$ is stated in Corollary~\ref{cor:Delta(w,A)_interior}. Lemma~\ref{lem:Delta_(T,A)} gives us the injectivity of this map. We now show that this map is surjective. Let~$F$ be an interior simplex of~$\triangDKK[\oru(s)]$ and~$w$ be a Stirling~$s$-permutation that is minimal in the~$s$-weak order with respect to the condition that~$F\subseteq \Delta_w$. Notice that such~$w$ is unique by the polygonality of~$\PSPerm$. Then,~$F$ is an intersection of facets of~$\Delta_w$. These facets correspond to certain transpositions that can be applied on~$w$. We denote~$A$ the set of ascents corresponding to these transpositions. The minimality of~$w$ implies that all elements in~$A$ are ascents of~$w$ giving us that~$F=\Delta_{(w,A)}$.

	Finally, let~$w, w'$ be Stirling~$s$-permutations and~$A,A'$ the subsets of their respective ascents. Lemma~\ref{lem:Delta_(T,A)} implies that~$[w, w+A]\subseteq [w', w'+A']$ if and only if~$\Delta_{(w', A')} \subseteq \Delta_{(w,A)}$, which proves that the map is a poset isomorphism.
\end{proof}

Just as how the atoms of the face poset of~$\PSPerm$ have a characterization as the maximal cliques of~$\triangDKK[\oru(s)]$, the coatoms of the face poset also have an explicit characterization in terms of cliques.

\begin{corollary}\label{cor:maximal_faces_perm}
A simplex~$\Delta_C$ of~$\triangDKK[\oru(s)]$ corresponds to a maximal interior face of~$\PSPerm$ if and only if~$C$ is a clique of size~$|s|-n+2$ that contains the exceptional routes and exactly one route that starts with each source edge that is different from~$(v_{-1},v_0)$.
\end{corollary}

\begin{proof}
We first highlight that for each~$i\in [n]$,~$\oru(s)$ has~$s_i-1$ source edges, implying that there are~$\sum_{i=1}^n (s_i-1) = |s|-n$ source edges different from~$(v_{-1},v_0)$. Together with the exceptional routes this gives us that~$C$ has size~$|s|-n+2$. This together with Lemma~\ref{cor:interiorsimplicesDKK} tells us that a clique~$C$ with the above stated properties corresponds with a maximal interior face of~$\PSPerm\subset\RR^{|s|+n+1}$.

Conversely, let~$(w, A)$ be a facet of~$\PSPerm$ and denote by~$C:= \Delta_{(w,A)}$ and by~$N \subset [0,|s|]$ the set of non-ascent positions in~$w$, so that~$C = \cup_{j\in N} \rpre{\prefix{w}{j}}$. Since~$w$ is~$121$-avoiding, if it has~$n-1$ ascents, then the ascents are of the form~$(i,c_i)$ where~$i < c_i$ for each~$i\in [n-1]$.
Moreover, it is the~$s_i$-th occurrence of~$i$ in~$w$ which produces an ascent pair in~$w$. Thus,~$N\setminus{\{0,|s|\}}$ indexes the first~$s_i-1$ occurrences of~$i$ in~$w$ for all~$i\in[n-1]$ and~$|N\setminus{\{0,|s|\}}|=|s|-n$. Now let~$j\in N \setminus{\{0,|s|\}}$ be a non-ascent position of~$w$ so that~$\rpre{\prefix{w}{j}} \in C$. If~$w_j$ is the~$k$-th occurrence of the value~$a$ in~$w$ for some~$k\in[s_{a}-1]$, then as before, the route~$\rpre{\prefix{w}{j}}$ contains the proper source edge~$e^{a}_k$.
Since~$|N\setminus{\{0,|s|\}}| = |s|-n$, then~$C$ satisfies the properties of the statement.
\end{proof}

Notice that Theorem~\ref{thm:bij_interiorfacesDKK_facessperm} does not suffice to answer Conjecture~\ref{conj:s-permutahedron} as this construction lives in an ambient space of dimension~$|s|+n+1>n$, and it does not have explicit geometrical coordinates. We now fix the first of these issues using the Cayley trick (Proposition~\ref{prop:cayley_trick}).

\section{The Sum of Cubes Realization}\label{sec:realiz_sum_of_cubes}

To apply the Cayley trick to our triangulation~$\triangDKK[\oru(s)]$, we follow Proposition~\ref{prop:flow_pols_ARE_cayley_embs} to describe~$\fpol[\oru(s)](\bfi)$ as the Cayley embedding of some lower-dimensional polytopes.

\begin{definition}\label{rem:}
	Let~$p:=\sum_{i=1}^{n+1}(s_i-1)$. We parametrize the space~$\RR^{|s|+n+1}$ of edges of~$\oru(s)$ as~\defn{$\RR^p\times\RR^{2n}$} where~$\RR^p$ (resp.~$\RR^{2n}$) corresponds to the space of source edges (resp.\ bumps and dips) of~$\oru(s)$.
\end{definition}

Notice that in the context of~$\RR^p\times\RR^{2n}$, any integer point~$f$ of~$\fpol[\oru(s)]$ (i.e.\ an integer~$\bfi$-flow of~$\oru(s)$) we have that~$f(e_0^i)+f(e_{s_i}^i)$ is determined by the coordinates~$f(e_t^k)$ for~$k\in[i+1, n+1]$,~$t\in [s_k-1]$ corresponding to source edges. Therefore,~$\fpol[\oru(s)]$ is affinely equivalent to its projection on to the space~$\RR^p\times \RR^{n}$ where~$\RR^{n}$ corresponds to the space of bumps~$e^i_0$ for~$i\in [n]$.

With this in hand, for a route~$R(k,t,\delta)$ of~$\oru(s)$ with~$k\in[n+1]$,~$t\in[s_k-1]$ and~$\delta\in \{0,1\}^{k-1}$, we have that the indicator vector~$\mathds{1}_{R(k,t,\delta)}$  (Definition~\ref{def:routes_Gs}) is~$e^{k}_{t}\times \sum_{i\in[k-1], \, \delta_i=0} e^i_0.$ Denoting by~\defn{$\square_{k-1}$} the~$(k-1)$-dimensional hypercubes with vertices~$\{0,1\}^{k-1}\times \{0\}^{n-k+1}$ embedded in~$\RR^{n}$, we get the following Lemma.

\begin{lemma}\label{lem:oru_IS_cayley_embedding}
	\begin{equation*}
		\fpol[\oru(s)](\bfi)=\cC(\square_n,\underbrace{\square_{n-1},\ldots,\square_{n-1}}_\text{$(s_{n-1}-1)$ times},\ldots,\underbrace{\square_{0},\ldots,\square_{0}}_\text{$(s_{1}-1)$ times})
	\end{equation*}
\end{lemma}

\begin{definition}\label{def:square_mxied_subdiv}
	Given a composition~$s=(s_1,\ldots,s_n)$, we denote \defn{$\subdivCay$} the fine mixed subdivision of the Minkowski sum of hypercubes~$\square_{n}+\sum_{i=1}^{n} (s_i-1)\square_{i-1} \subseteq \RR^n$ obtained by intersecting the triangulation~$\triangDKK[\oru(s)]$ (projected onto~$\RR^p\times \RR^n$) with the subspace~${\left\{\frac{1}{p}\right\}}^p\times \RR^{n}$ as in Remark~\ref{rem:cayley_trick_proof}.
\end{definition}

Applying the Cayley trick (Proposition~\ref{prop:cayley_trick}) together with the isomorphism between the face poset of~$\PSPerm$ and the interior simplices of~$\triangDKK[\oru(s)]$ (Theorem~\ref{thm:bij_interiorfacesDKK_facessperm}) gives us the following theorem.

\begin{theorem}\label{thm:bij_mixed_subdiv}
	The face poset of the~$s$-permutahedron~$\PSPerm$ is isomorphic to the set of interior cells of~$\subdivCay$ ordered by reverse inclusion. Moreover, Stirling~$s$-permutations are in bijection with the maximal cells of~$\subdivCay$.
\end{theorem}

Figure~\ref{fig:comparing_triang_flow_cube} contains an example of~$\subdivCay$ and the dual graph formed by the adjacency of its interior cells with edges oriented perpendicular to each inner wall.

\begin{figure}[h!]
	\centering
	\includegraphics[scale=1]{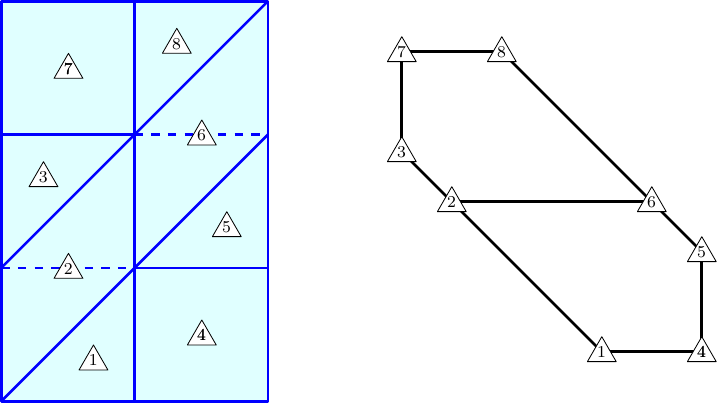}
	\caption[$\subdivCay$ and its dual graph giving~$\PSPerm$ for~$s=(1,2,1)$.]{ $\subdivCay[(1,2,1)]$ (left) and its dual graph (right).}
	\label{fig:comparing_triang_flow_cube}
\end{figure}

\begin{remark}\label{rem:dim_n_minus_1}
	Recall from Remark~\ref{rem:vertex_v_m1} that the graph~$\oru(s)^\#$ obtained from~$\oru(s)$ by contracting the edge~$(v_{-1},v_0)$ gives a flow polytope integrally equivalent to~$\fpol[oru(s)](\bfi)$. In our situation of summing cubes, we can use a different parametrization of the space where~$\fpol[\oru(s)](\bfi)$ lives by considering the cube~$\square_n$ (coming from the source edge~$e_1^{n+1}$) as the Cayley embedding of two hypercubes~$\square_{n-1}$ (making the bump~$e_0^{n}$ and dip~$e_{s_n}^{n}$ into source edges). We could equivalently intersect~$\RR^n$ with the hyperplane~$x_n=\frac{1}{2}$. This allows us to lower the dimension further to~$n-1$ and obtain a fine mixed subdivision of the Minkowski sum of hypercubes~$(s_n +1)\square_{n-1}+\sum_{i=1}^{n-1} (s_i-1)\square_{i-1}$. We use this representation for the figures.
\end{remark}

Figure~\ref{fig:mixed_cell_3221} shows how given the Stirling~$(1,2,1)$-permutation~$w=3221$, its maximal clique~$\Delta_{3221}$ determines the construction of its mixed cell via the Cayley trick. Said mixed cell is highlighted in~\ref{fig:mixed_subdivision}.
Both figures follow the coordinate system~$(e^2_0, e^3_0)$.

\begin{figure}[h!]
	\centering
	\begin{subfigure}[b]{0.5\textwidth}
		\centering
		\includegraphics[scale=0.85]{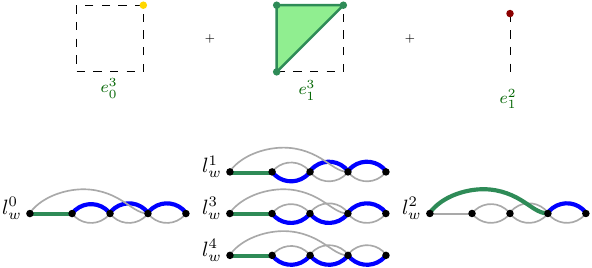}
		\caption{}
		\label{fig:mixed_cell_3221}
	\end{subfigure}
	\qquad\qquad
	\begin{subfigure}[b]{0.3\textwidth}
		\centering
		\includegraphics[scale=0.8]{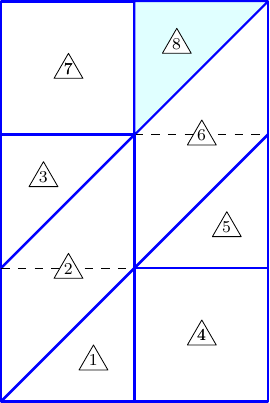}
		\caption{}
		\label{fig:mixed_subdivision}
	\end{subfigure}
	\caption[(a) Summands of the Minkowski cell corresponding to~$w=3221$ with their corresponding routes in the clique~$\Delta_w$. (b) Mixed subdivision of~$2\square_2+\square_1$ corresponding dually to the~$(1,2,1)$-permutahedron.]{(a) Summands of the Minkowski cell corresponding to~$w=3221$ with their corresponding routes in the clique~$\Delta_w$. (b) Mixed subdivision of~$2\square_2+\square_1$ corresponding dually to the~$(1,2,1)$-permutahedron. The cells are numbered according to Figure~\ref{fig:adjacency_graph_121}. The highlighted cell in blue corresponds to~$w=3221$ as obtained in Figure~\ref{fig:mixed_cell_3221}.}
	\label{fig:comparing flows PS and gen PS}
\end{figure}

Via Theorem~\ref{thm:bij_mixed_subdiv} and Remark~\ref{rem:dim_n_minus_1} we have obtained a realization of~$\PSPerm$ of the correct dimension as in Conjecture~\ref{conj:s-permutahedron}. Still, we want to have explicit coordinates to further prove properties of this polytopal complex. We solve this now using a tropical context.

\section{The Tropical Realization}\label{sec:realiz_tropical}

Recall that the routes of~$\oru(s)$ are described in Definition~\ref{def:routes_Gs} and consider the DKK triangulation~$\triangDKK[\oru(s)]$ constructed in Section~\ref{sec:realiz_flow_potyope}. Using Definition~\ref{def:DKK_height_function} and Proposition~\ref{prop:DKKlem3_original} we have the following lemma.

\begin{lemma}\label{lem:epsilonheight}
	Let~$s$ be a composition and~$\varepsilon>0$ a sufficiently small real number.
	Consider~$\height_{\varepsilon}$ to be the function that associates to a route~$R=R(k, t, \delta)$ of~$\oru(s)$ the quantity
	\begin{equation}\label{eq:epsilonheight}
		\height_{\varepsilon}(R)=-\sum_{k\geq c > a \geq 1} \varepsilon^{c-a} (t_c+\delta_a)^2,
	\end{equation}
	where~$t_c=
		\begin{cases}
			0   & \text{ if } \delta_c=0, \\
			s_c & \text{ if } \delta_c=1,
		\end{cases}$ for all~$c\in [k-1]$.

	Then~$\height_{\varepsilon}$ is an admissible height function for~$\triangDKK[\oru(s)]$.
\end{lemma}

Before describing explicit coordinates for~$\PSPerm$, we take a moment to give an explicit bound for~$\varepsilon$ such that Lemma~\ref{lem:epsilonheight} holds.

\begin{theorem}\label{thm:epsilonheight}
	For Lemma~\ref{lem:epsilonheight}, it is enough to take~$\varepsilon<\frac{1}{n(1+\sum_{j=2}^n (2s_j+1))}$.
\end{theorem}

\begin{proof}
	Let~$P=R(k,t,\delta)$ and~$Q=R(k', t', \delta')$ be two routes of~$\oru(s)$ that are in minimal conflict at a common route~$[v_{n+1-y}, v_{n-x}]$. Without loss of generality suppose that~$Pv_{n+1-y}\prec Qv_{n+1-y}$ giving us that~$\delta_x=1$ and~$\delta'_x=0$. Taking~$P'$ and~$Q'$ as they resolvents consider the quantity~$H:=\height_{\varepsilon}(P)+\height_{\varepsilon}(Q)-\height_{\varepsilon}(P')-\height_{\varepsilon}(Q')$. As~$\height_\varepsilon$ is a height function, Lemma~\ref{lem:DKKlem2_us_pro} tells us that~$H>0$. Therefore, to determine a bound for~$\varepsilon$ we compute~$H$ for the only three cases which can occur between~$P$ and~$Q$ as they are adjacent in~$\preceq_{\cI_{n+1-y}}$ and~$\preceq_{\cO_{n+1-x}}$. Figure~\ref{fig:s_proof_help} contains a visual aid corresponding to each of these cases.

	\begin{figure}[h!]
		\centering
		\includegraphics[scale=1]{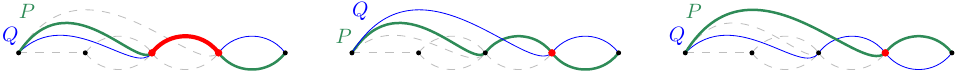}
		\caption[The possible cases of minimal conflict between routes.]{ The possible cases of minimal conflict between routes. Only one route is bolded in green, and the conflict area is heavily bolded in red.}
		\label{fig:s_proof_help}
	\end{figure}

	\vspace{-1cm}

	\begin{itemize}
		\item If~$k=k'=y$,~$t\in [s_y-2]$, and~$t'=t+1$ we have that in~$\height_{\varepsilon}(P)+\height_{\varepsilon}(Q)-\height_{\varepsilon}(P')-\height_{\varepsilon}(Q')$, all pairs~$(a,c)$ in formula~\ref{eq:epsilonheight} cancel out either with~$\height_{\varepsilon}(P)-\height_{\varepsilon}(Q')$ or~$\height_{\varepsilon}(Q)-\height_{\varepsilon}(P')$ except for~$(a,c)=(x,y)$. Thus,~$H$ becomes
		      \begin{align*}
			      H & = \height_{\varepsilon}(P)+\height_{\varepsilon}(Q)-\height_{\varepsilon}(P')-\height_{\varepsilon}(Q') \\
			        & = -\varepsilon^{y-x}\Big((t+1)^2+((t+1)+0)^2 - (t+0)^2 - ((t+1)+1)^2 \Big)                              \\
			        & = 2\ \varepsilon^{y-x} >0.
		      \end{align*}

		\item If~$k>k'=y$,~$\delta_y=0$, and~$t'=1$, then as in the previous case where~$(x,y)$ did not cancel out, here~$(a,y)$ for~$x\geq a$ do not cancel out. Now, in~$\height_\varepsilon(P)-\height_\varepsilon(Q)$ all pairs~$(a,c)$ cancel out except those for~$k\geq x>y$ and~$x\geq a$ whose terms come solely from~$\height_\varepsilon(P)$. Similarly, for~$\height_\varepsilon(Q)-\height_\varepsilon(P)$ all pairs~$(a,c)$ cancel out except those for~$k\geq x>y$ and~$x\geq a$ whose term~$t_c$ comes from~$P'$ and~$\delta'_a$ comes from~$Q$. Thus, the pairs that do not cancel out in~$H$ are all pairs~$(a,c)$ for~$k\geq c\geq y$ and~$x\geq a$ and~$H$ becomes \begin{align*}
			      H & =
			      - \sum_{x\geq a} \varepsilon^{y-a} \Big( \delta_a^2 +(1+\delta'_a)^2 - {\delta'_a}^{2} - (1+\delta_a)^2\Big)
			      - \sum_{\substack{k\geq c > y                                                                            \\ x\geq a}} \varepsilon^{c-a} \Big( (t_c+\delta_a)^2-(t_c+\delta'_a)^2 \Big)                    \\
			        & = 2\ \varepsilon^{y-x}
			      - 2\ \sum_{x> a} \varepsilon^{y-a} ( {\delta'_a} - \delta_a)
			      - \sum_{\substack{k\geq c > y                                                                            \\ x\geq a}} \varepsilon^{c-a} \Big(2\ t_c(\delta_a-\delta'_a)+\delta_a^2-{\delta'_a}^{2} \Big) \\
			        & \geq 2\ \varepsilon^{y-x}
			      - 2\ \sum_{x> a} \varepsilon^{y-a}
			      - \sum_{\substack{k\geq c > y                                                                            \\ x\geq a}} \varepsilon^{c-a} (2 s_c+1)                                                         \\
			        & \geq 2\ \varepsilon^{y-x}
			      - 2\ \varepsilon^{y-x+1}\Big(x-1 + x \sum_{k\geq c > y} (2 s_c+1) \Big)                                  \\
			        & \geq 2\ \varepsilon^{y-x}\Big(1 - \varepsilon\Big(y-2+ (y-1) \sum_{k\geq c > y} (2 s_c+1)\Big)\Big).
		      \end{align*}
		      Thus, taking~$\varepsilon<\frac{1}{n(1+\sum_{j=2}^n (2s_j+1))}$, gives us~$H>0$.

		\item If~$k'>k=y$,~$t=s_y-1$, and~$\delta'_y=1$, then as in the previous case the pairs that do not cancel out are all pairs~$(a,c)$ for~$k\geq c\geq y$ and~$x\geq a$. This gives us \begin{align*}
			      H & =
			      - \sum_{x\geq a} \varepsilon^{y-a} \Big( (s_y-1+\delta_a)^2 +(s_y+\delta'_a)^2 - (s_y-1+\delta'_a)^2 - (s_y+\delta_a)^2\Big) \\
			        & \phantom{=}
			      - \sum_{k\geq c > y, \ x\geq a} \varepsilon^{c-a} \Big( (t'_c+\delta'_a)^2-(t'_c+\delta_a)^2 \Big)                           \\
			        & = 2\ \varepsilon^{y-x} +2\sum_{x>a} \varepsilon^{y-a}(\delta'_a-\delta_a)
			      - \sum_{k\geq c > y, \ x\geq a} \varepsilon^{c-a} \Big(2\ t'_c(\delta'_a - \delta_a)+{\delta'_a}^2-\delta_a^2\Big).
		      \end{align*}
		      The final computations are similar to the previous case and so, we omit them. \qedhere
	\end{itemize}
\end{proof}

\begin{remark}\label{rem:disclosure_epislons}
	We disclose to the reader that the bound of Proposition~\ref{thm:epsilonheight} is not sharp at all. However, the proof of Proposition~\ref{thm:epsilonheight}, shows that for the cases~$s=(1,\ldots,1,k)$ for~$k\in\ZZ_{>0}$ the value~$\varepsilon$ can be any positive number. For the case~$k=1$ this is justified in since any two routes in minimal conflict of~$\oru(1,\ldots,1)$ is the first one where~$\varepsilon$ just needed to be positive. For the case~$k>1$ notice that any Stirling~$(1,\ldots,1,k)$ permutation of length~$n$ can be obtained from a Stirling~$(1,\ldots,1)$ permutation of length~$n+k$. Passing this idea through our realizations gives us that~$\PSPerm[(1,\ldots,1,k)]$ can be thought as a projection of a~$\PSPerm[(1,\ldots,1)]$ of bigger dimension.
\end{remark}

\subsection{Coordinates for the~\texorpdfstring{$s$}{}-Permutahedron}

We move on to use the tropical technology of Section~\ref{sec:tropical_geometry} to obtain explicit points for~$\PSPerm$. For the remainder of this chapter~$\height$ is an admissible height function for~$\triangDKK[\oru(s)]$.

Since in Section~\ref{sec:realiz_sum_of_cubes} we used the Cayley trick on the triangulation~$\triangDKK[\oru(s)]$ to obtain the mixed subdivision~$\subdivCay$, the following theorem directly follows from Proposition~\ref{prop:arrangement_tropical_hypersurfaces}.

\begin{theorem}\label{thm:arr_trop_hypersurfaces_s-perm}
	The tropical dual of the mixed subdivision~$\subdivCay$ is the polyhedral complex of cells induced by the arrangement of tropical hypersurfaces \begin{equation*}
		\cH_{s}(\height):=\left\{\cT(F^k_{t}) \, :\, k\in [2, n+1], \, t\in [s_k-1]\right\},
	\end{equation*} where~$F^k_{t}(\mathbf{x}):=\bigoplus_{} \height(R(k, t, \delta)) \odot \mathbf{x}^{\delta} = \min \left\{\height(R(k, t, \delta)) + \sum_{i\in [k-1]} \delta_i x_i \, :\, \delta\in\{0,1\}^{k-1}  \right\}$.
\end{theorem}

Notice that in this description of the tropical polynomials we avoid using the notation~$\langle \delta,\mathbf{x}\rangle$ as~$\mathbf{x}\in\RR^{n}$ while~$\delta\in\{0,1\}^{k-1}$.

\begin{definition}\label{def:s-perm_geom}
	We denote by \defn{$\PSPerm(\height)$} the polyhedral complex of bounded cells induced by the arrangement~$\cH_{s}(\height)$.
\end{definition}

\begin{theorem}\label{thm:bij_trop_arr}
	The face poset of the geometric polyhedral complex~$\PSPerm(\height)$ is isomorphic to the face poset of the combinatorial~$s$-permutahedron~$\PSPerm$.
\end{theorem}

\begin{proof}
	In~Theorem~\ref{thm:bij_mixed_subdiv} we saw that the face poset of~$\PSPerm$ is anti-isomorphic to the face poset of interior cells of the mixed subdivision~$\subdivCay$. From Lemma~\ref{lem:tropical_dual_interior} and Theorem~\ref{thm:arr_trop_hypersurfaces_s-perm} we get that this poset is isomorphic to the poset of bounded cells of~$\cH_{s}(\height)$, which is precisely the face poset of~$\PSPerm(\height)$.
\end{proof}

Figure~\ref{fig:tropical_realizations} shows some examples of the~$1$-skeleton of such realizations of the~$s$-permutahedron using the height function~$\height_\varepsilon$ of Lemma~\ref{lem:epsilonheight}. Figure~\ref{fig:tropical_pol_decomposed} shows~$\PSPerm[(1,1,1,2)](\height_1)$ with its maximal cells unglued. More examples with interactivity can be found in this \href{https://sites.google.com/view/danieltamayo22/gallery-of-s-permutahedra}{website}\footnote{{https://sites.google.com/view/danieltamayo22/gallery-of-s-permutahedra}}.

\begin{figure}[h!]
	\centering
	\includegraphics[scale=0.5]{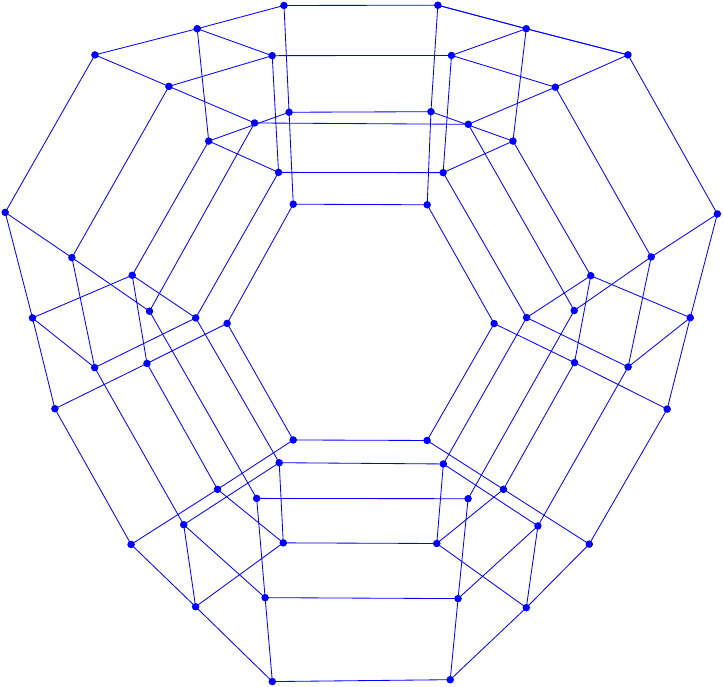}
	\hspace{2cm}
	\includegraphics[scale=0.58]{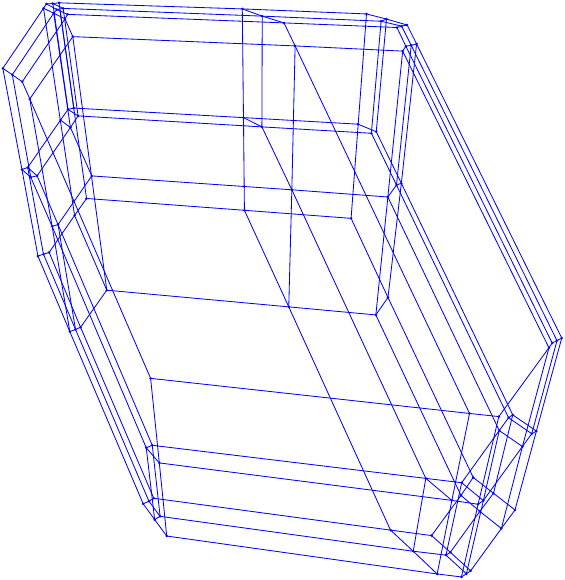}
	\caption[Examples of the~$1$-skeleton of~$\PSPerm(\height_\varepsilon)$ via its tropical realization.]{ The~$1$-skeletons of $\PSPerm[(1,1,1,2)](\height_1)$ (left) and~$\PSPerm[(1,2,2,2)](\height_{0.2})$ (right) via their tropical realization and the height function~$\height_\varepsilon$.}
	\label{fig:tropical_realizations}
\end{figure}

\begin{figure}[h!]
	\centering
	\includegraphics[scale=0.5]{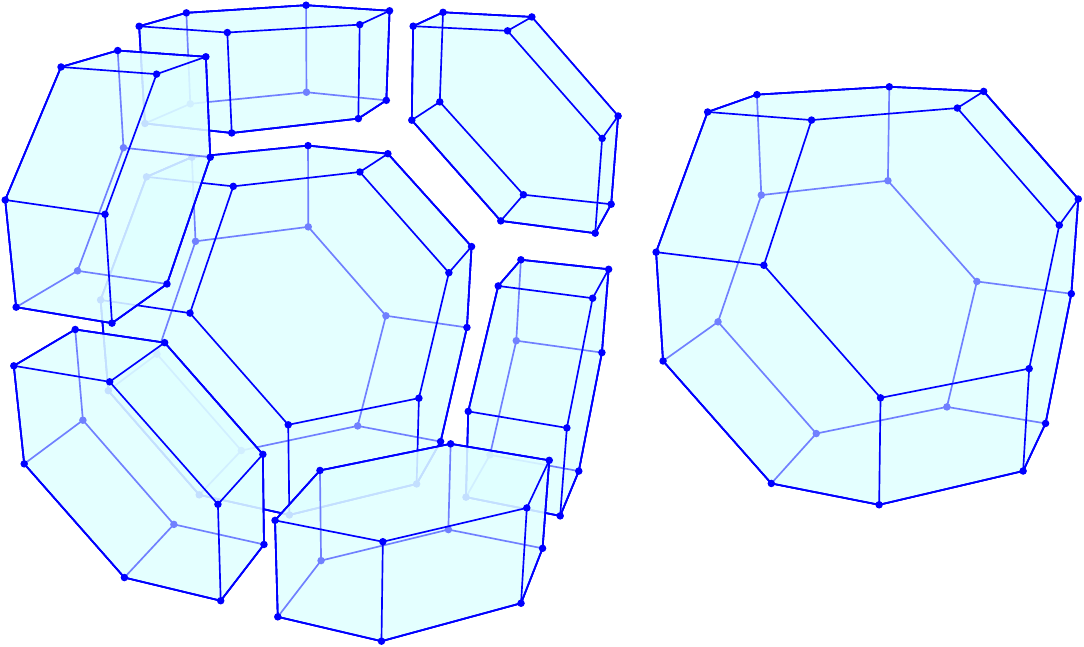}
	\caption[$\PSPerm(\height_\varepsilon)$ with its maximal cells distanced.]{ $\PSPerm[(1,1,1,2)](\height_1)$ with its maximal cells distanced. One cell has been moved further for a better view.}
	\label{fig:tropical_pol_decomposed}
\end{figure}

With this polytopal complex in hand, we can describe the explicit coordinates of the vertices of~$\PSPerm(\height)$ as follows.

\begin{definition}\label{def:length_prefix_occurence}
	Let~$w\in\cW_s$ be a Stirling~$s$-permutation,~$a\in [n]$ and~$t\in [s_a]$. We denote~\defn{$\ell(a^t)$} the length of the prefix of~$w$ that precedes the~$t$-th occurrence of~$a$. As in Definition~\ref{def:Lw}, this prefix corresponds to the route~$\rpre{\prefix{w}{\ell(a^t)}}$ in the maximal clique of coherent routes~$\Delta_w$.
\end{definition}

\begin{theorem}\label{thm:vertices}
	The vertices of~$\PSPerm(\height)$ are in bijection with Stirling~$s$-permutations.
	Moreover, the vertex~$\mathbf{v}(w)=(\mathbf{v}(w)_a)_{a\in[n]}$ associated to a Stirling~$s$-permutation~$w$ has coordinates
	\begin{equation}\label{eq:coordinates}
		\mathbf{v}(w)_a = \sum_{t=1}^{s_a} \left(\height(\rpre{\prefix{w}{\ell(a^t)}})-\height(\rpre{\prefix{w}{\ell(a^t)+1}})\right).
	\end{equation}
\end{theorem}

\begin{proof}
	We obtain the bijection between vertices of~$\PSPerm(\height)$ and Stirling~$s$-permutations directly as a consequence of Theorem~\ref{thm:bij_trop_arr}.

	Let~$w$ be a Stirling~$s$-permutation. Its simplex~$\Delta_w$ is associated via Theorem~\ref{thm:arr_trop_hypersurfaces_s-perm} to the intersection of all regions of the form
	\begin{equation}
		\Bigg\{ \mathbf{x}\in \RR^n \, : \, \height(R(k, t, \delta)) + \sum_{a\in [c-1]} \delta_a x_a = \min_{\theta \in \{0,1\}^{k-1}} \Big\{ \height(R(k, t, \theta)) + \sum_{a\in [k-1]} \theta_a x_a\Big\}\Bigg\},
	\end{equation}\label{eq:coordinates_regions}
	where~$R(c,t,\delta)$ is a route in the clique~$\Delta_w$.
	It follows from the bijection between vertices and Stirling~$s$-permutations that this intersection is a single point that we denote~$\mathbf{v}$.
	Let us show that~$\mathbf{v}$ has the coordinates stated in the theorem.
	Let~$a\in [n]$. Both routes~$\rpre{\prefix{w}{\ell(a^1)}}$ and~$\rpre{\prefix{w}{\ell(a^{s_a})+1}}$ are of the form~$R(c,t, \delta)$ and~$R(c,t,\delta')$ respectively, where~$c$ is the smallest letter such that the~$a$-block~$B_a$ is contained in the~$c$-block~$B_c$ in~$w$, and~$t$ denotes the number of occurrences of~$c$ that precedes~$B_a$. If the~$B_a$ is contained in no other block we set~$c=n+1$ and~$t=1$ (i.e. the routes start with the source edge~$(v_{-1},v_0)$). The indicator vectors~$\delta$ and~$\delta'$ satisfy that~$\delta'-\delta$ is the indicator vector of the letters~$b\leq a$ such that the~$B_b$ is contained in~$B_a$ in~$w$.
	The fact that both routes belong to~$\Delta_w$ implies that~$\height(\rpre{\prefix{w}{\ell(a^1)}})+ \sum_{b\in [c-1]} \delta_b \mathbf{v}_b = \height(\rpre{\prefix{w}{\ell(a^{s_a})+1}})+ \sum_{b\in [c-1]} \delta'_b \mathbf{v}_b$, thus
	\begin{equation}\label{eq:coordinates_induction}
		\sum_{b\in [c-1]} (\delta'_b-\delta_b) \mathbf{v}_b=\sum_{\substack{b \in [a] \\B_b\subseteq B_a}} \mathbf{v}_b = \height(\rpre{\prefix{w}{\ell(a^1)}}) - \height(\rpre{\prefix{w}{\ell(a^{s_a})+1}}).
	\end{equation}
	We finish with an induction on~$a\in[n]$. If~$a=1$, then all terms of Equation~\ref{eq:coordinates} cancel out except~$\height(\rpre{\prefix{w}{\ell(a^1)}})$ and~$-\height(\rpre{\prefix{w}{\ell(a^{s_a})}})$ which is exactly Equation~\ref{eq:coordinates_induction} for~$\mathbf{v}_1$. Suppose we have the vertex description of all coordinates~$\mathbf{v}_b$ for~$b<a$. From Equation~\ref{eq:coordinates_induction} and the induction hypothesis we get that \begin{equation*}
		\mathbf{v}_a=\height(\rpre{\prefix{w}{\ell(a^1)}}) - \height(\rpre{\prefix{w}{\ell(a^{s_a})+1}})-\sum_{\substack{b \in [a-1]\\B_b\subset B_a}} \sum_{t=1}^{s_b} \left(\height(\rpre{\prefix{w}{\ell(b^t)}})-\height(\rpre{\prefix{w}{\ell(b^t)+1}})\right).
	\end{equation*}
	In this case all terms on the right cancel whenever the value at position~$\ell(b^t)+1\neq a$ giving us terms of the form~$-\height(\rpre{\prefix{w}{\ell(a^r)}})$ for~$r\in [2, s_a]$ or~$\height(\rpre{\prefix{w}{\ell(a^r)+1}})$ for~$r\in [s_a-1]$.
\end{proof}

As all maximal cliques~$\Delta_w$ contain the exceptional routes~$R(n+1, 1, (0)^n)$ and~$R(n+1, 1, (1)^n)$ (see Remark~\ref{rem:oru_exceptional_routes}), we obtain the following corollary.

\begin{corollary}\label{cor:s_containging_hyperplane}
	The~$s$-permutahedron~$\PSPerm(\height)$ is contained in the hyperplane
	\begin{equation}
		\left\{\mathbf{x}\in \RR^n \, :\, \sum_{i=1}^n x_i = \height\Big(R(n+1, 1, (0)^n)\Big)-\height\Big(R(n+1, 1, (1)^n)\Big)\right\}.
	\end{equation}
\end{corollary}

As a next step, we show that~$\PSPerm(\height)$ is a generalized permutahedron as in Definition~\ref{def:generalized_permutahedra}.

\begin{theorem}\label{thm:edges}
	Let~$1\leq a < c\leq n$ and~$w$ and~$w'$ be Stirling~$s$-permutations of the form~$u_1B_acu_2$ and~$u_1cB_au_2$ respectively, where~$B_a$ is the~$a$-block of~$w$ and~$w'$. The edge of~$\PSPerm(\height)$ corresponding to the transposition between~$w$ and~$w'$ is
	\begin{equation}\label{eq:edges}
		\mathbf{v}(w')-\mathbf{v}(w)= \Big(\height(\rpre{u_1c})+\height(\rpre{u_1B_a}) - \height(\rpre{u_1})-\height(\rpre{u_1B_ac}) \Big)(\mathbf{e}_a-\mathbf{e}_c).
	\end{equation}
\end{theorem}

\begin{proof}
	Denote~$t:=|(c,a)_w|+1$ meaning that the transposition from~$w$ to~$w'$ exchanges the~$a$-block with the~$t$-th occurrence of~$c$. Let us describe~$\mathbf{v}(w')-\mathbf{v}(w)$ via the expression of the explicit coordinates we obtained in Theorem~\ref{thm:vertices}. Following the construction of routes~$R$ from prefixes~$u$, notice that any route ending inside~$u_1$ or inside~$u_2$ appears in~$\Delta_w\cap\Delta_{w'}$ as it has the same letters the same amount of times meaning that it is cancelled out in~$\mathbf{v}(w')-\mathbf{v}(w)$. Also, all routes ending inside~$B_a$ also appear in~$\mathbf{v}(w')-\mathbf{v}(w)$ as in this case~$c$ is not minimal enough to play the role of defining these routes. Thus, they are also cancelled. The remaining routes that we have are~$u_1$,~$u_1c$,~$u_1B_a$ and~$u_1B_ac$, which gives the same route as~$u_1cB_a$.
	Therefore, we have that
	\begin{align*}
		\mathbf{v}(w')-\mathbf{v}(w) & = \left(\mathbf{v}(w')_a-\mathbf{v}(w)_a \right)\mathbf{e}_a + \left(\mathbf{v}(w')_c-\mathbf{v}(w)_c \right)\mathbf{e}_c                                                                                 \\
		                             & =  \left(\height(\rpre{\prefix{w'}{\ell(a^1)}})-\height(\rpre{\prefix{w'}{\ell(a^{s_a})+1}}) - \height(\rpre{\prefix{w}{\ell(a^1)}}) + \height(\rpre{\prefix{w}{\ell(a^{s_a})+1}}) \right)\mathbf{e}_a    \\
		                             & \phantom{=} + \left(\height(\rpre{\prefix{w'}{\ell(c^t)}})-\height(\rpre{\prefix{w'}{\ell(c^t)+1}}) - \height(\rpre{\prefix{w}{\ell(c^t)}}) + \height(\rpre{\prefix{w}{\ell(c^t)+1}}) \right)\mathbf{e}_c \\
		                             & = \left(\height(\rpre{u_1c})-\height(\rpre{u_1cB_a}) - \height(\rpre{u_1}) + \height(\rpre{u_1B_a}) \right)\mathbf{e}_a                                                                                   \\
		                             & \phantom{=} + \left(\height(\rpre{u_1})-\height(\rpre{u_1c}) - \height(\rpre{u_1B_a}) + \height(\rpre{u_1B_ac}) \right)\mathbf{e}_c                                                                       \\
		                             & = \left(\height(\rpre{u_1c})+\height(\rpre{u_1B_a}) - \height(\rpre{u_1})-\height(\rpre{u_1B_ac}) \right)(\mathbf{e}_a-\mathbf{e}_c). \qedhere
	\end{align*}
\end{proof}

\begin{lemma}
	For any strictly decreasing sequence of real numbers~$\kappa_1>\cdots > \kappa_n$, the direction~$\sum_{i=1}^n \kappa_i \mathbf{e}_i$ orients the edges of~$\PSPerm(\height)$ according to the~$s$-weak order covering relations.
\end{lemma}

\begin{proof}
	Using Theorem~\ref{thm:edges}, we know that the edges of~$\PSPerm(\height)$ have direction~$(\height(\rpre{u_1c})+\height(\rpre{u_1B_a}) - \height(\rpre{u_1})-\height(\rpre{u_1B_ac}) )(\mathbf{e_a}-\mathbf{e_c})$. As the routes~$\rpre{u_1B_a}$ and~$\rpre{u_1c}$ are in minimal conflict at~$[v_{n+1-c}, v_{n+1-a}]$ and~$\rpre{u_1}$ and~$\rpre{u_1B_ac}$ are their resolvents, Lemma~\ref{lem:DKKlem2_us_pro} tells us that~$\height(\rpre{u_1c})+\height(\rpre{u_1B_a}) - \height(\rpre{u_1})-\height(\rpre{u_1B_ac}) >0$.
\end{proof}

\begin{lemma}\label{lem:support}
	The support~$\supp(\PSPerm(\height))$ (i.e.\ the union of faces of~$\PSPerm(\height)$) is a polytope combinatorially isomorphic to the~$(n-1)$-dimensional permutahedron. More precisely it can be described as
	\begin{enumerate}
		\itemsep0em
		\item the convex hull of the vertices~$\mathbf{v}(w^ {\sigma})$ where~$\sigma\in\fS_n$ and~$w^{\sigma}$ is the Stirling~$s$-permutation
		      \begin{equation*}
			      w^{\sigma} = \underbrace{\sigma(1)\ldots \sigma(1)}_{s_{\sigma(1)}\text{ times}}\ldots \underbrace{\sigma(n)\ldots\sigma(n)}_{s_{\sigma(n)}\text{ times}},
		      \end{equation*}
		\item the intersection of the inequalities
		      \begin{equation}\label{eq:halfspace1}
			      \langle\delta,\mathbf{x}\rangle \geq \height(R(n+1, 1, (0)^n))-\height(R(n+1, 1, \delta)),
		      \end{equation}
		      \begin{equation}\label{eq:halfspace2}
			      \langle\mathbf{1}-\delta,\mathbf{x}\rangle \leq \height(R(n+1, 1, \delta))-\height(R(n+1, 1, (1)^n)),
		      \end{equation}
		      for all~$\delta\in \{0,1\}^n$.
	\end{enumerate}
\end{lemma}

\begin{proof} We prove each statement separately.
	\begin{enumerate}
		\item Let~$\sigma\in\fS_n$ and consider the linear functional~$f(\mathbf{x})=\sum_{a\in[n]} \sigma(a) x_a$. We claim that among all faces of~$\PSPerm(\height)$,~$f$ is maximized on~$\mathbf{v}(w^{\sigma})$. To see this let~$w$ be a Stirling~$s$-permutation and~$1\leq a<c\leq n$. Notice the following cases.
		      \begin{itemize}
			      \item If~$w$ contains an ascent~$(a,c)$ such that~$\sigma(a)>\sigma(c)$, then~$f$ is increasing along the edge of direction~$\mathbf{e}_a-\mathbf{e}_c$ corresponding to the transposition of~$w$ along the ascent~$(a,c)$.
			      \item If~$w$ contains a descent~$(a,c)$ such that~$\sigma(a)<\sigma(c)$, then~$f$ is increasing along the edge of direction~$\mathbf{e_{c}}-\mathbf{e_{a}}$ corresponding to the transposition of~$w$ along the descent~$(a,c)$.
			      \item Otherwise,~$w=w^{\sigma}$.
		      \end{itemize}
		      This shows that the vertices of~$\supp(\PSPerm(\height))$ have the same normal cones as the~$(n-1)$-permutahedron (embedded in~$\RR^n$), hence its normal fan is the braid fan.
		\item From Remark~\ref{rem:oru_exceptional_routes} we know that that all cliques~$\Delta_w$ contain the routes~$R(n+1, 1, (0)^n)$ and~$R(n+1, 1, (1)^n)$. This implies that all vertices of~$\PSPerm(\height)$ are contained in the following intersection of regions given by Equation~\ref{eq:coordinates_regions}:
		      \begin{align*}
			      \Big\{ \mathbf{x}\in \RR^n \, : \, \height(R(n+1, 1, (0)^n)) & = \height(R(n+1, 1, (1)^n)) + \langle\mathbf{1},\mathbf{x}\rangle                                                         \\
			                                                                   & = \min_{\delta \in \{0,1\}^{n}} \big\{ \height(R(n+1, 1, \delta)) + \langle\mathbf{\delta},\mathbf{x}\rangle\big\}\Big\}.
		      \end{align*}
		      This intersection is precisely the one obtained from the half-spaces defined in Equation~\ref{eq:halfspace1} and Equation~\ref{eq:halfspace2}.

		      Moreover, notice that since each~$a$-block~$B_a$ in~$w^{\sigma}$ is a consecutive repetition of the letter~$a$, we have that the vertex~$\mathbf{v}(w^{\sigma})$ has coordinates~$\mathbf{v}(w^\sigma)_a=\height(R[w^\sigma_{[\ell(a^1)]}])-\height(R[w^\sigma_{[\ell(a^{s_a})+1]}])$. Thus, letting~$I:=\{i\in [n]\, :\, \delta_i=1\}$ for~$\delta\in\{0,1\}^n$, Equation~\ref{eq:halfspace1} and Equation~\ref{eq:halfspace2} achieve equality with~$\mathbf{v}(w^{\sigma})$ exactly when~$\{\sigma(1), \ldots, \sigma(|I|)\}=I$. Meaning that these inequalities define the facets of~$\supp(\PSPerm(\height))$. \qedhere
	\end{enumerate}
\end{proof}

\begin{remark}
	With similar arguments we can see that the restriction of the~$s$-weak order to a face of~$\supp(\PSPerm(\height))$, associated to an ordered partition, corresponds to a product of~$s'$-weak orders, one for each part of the ordered partition.
\end{remark}

\begin{remark}\label{rem:conjecture_done}
	Since the zonotope~$\sum_{1\leq i<j\leq n}s_j[\mathbf{e}_i,\mathbf{e}_j]$ can be seen to be combinatorially isomorphic to the~$(n-1)$-dimensional permutahedron, with Lemma~\ref{lem:support} we have finished answering Conjecture~\ref{conj:s-permutahedron} in the case where~$s$ is a composition.
\end{remark}

We finish by refining Remark~\ref{rem:conjecture_done} for the case where~$\height$ is given by Lemma~\ref{lem:epsilonheight}.

\begin{theorem}\label{thm:s_realization_zonotope}
	Let~$\varepsilon>0$ be a small enough real number such that~$\height_{\varepsilon}$ is an admissible height function for~$\triangDKK[\oru(s)]$. Then the support~$\supp(\PSPerm(\height_{\varepsilon}))$ is a translation of the zonotope~$2\sum_{1\leq a < c\le n} s_c\varepsilon^{c-a}[\mathbf{e}_a, \mathbf{e}_c].$
\end{theorem}

\begin{proof}
	Lemma~\ref{lem:support} and Theorem~\ref{thm:edges} tell us that the edges of~$\supp(\PSPerm(\height_{\varepsilon}))$ are of the form~$[\mathbf{v}(w^{\sigma}), \mathbf{v}(w^{\sigma'})]$, where~$\sigma$ and~$\sigma'$ are permutations of~$[n]$ related by a transposition along an ascent~$(a,c)$. Using the DKK height function~$\height_\varepsilon$ from Equation~\ref{eq:epsilonheight} into Equation~\ref{eq:edges} with the modification that the letter~$c$ is replaced by~$s_c$ occurrences of~$c$ (or a repeated use of Equation~\ref{eq:edges}), we see that the only terms that do not cancel out are those involving the pair~$(a,c)$ giving us:
	\begin{align*}
		\mathbf{v}(w^{\sigma'})- \mathbf{v}(w^{\sigma}) & = \Big(\height(\rpre{u_1B_c})+\height(\rpre{u_1B_a}) - \height(\rpre{u_1})-\height(\rpre{u_1B_aB_c}) \Big)(\mathbf{e}_a-\mathbf{e}_c) \\
		                                                & = -\varepsilon^{c-a}\Big( (s_c+1)^2 + (0+1)^2 - (0+0)^2 - (s_c+1)^2\Big)(\mathbf{e}_a - \mathbf{e}_c)                                 \\
		                                                & = 2\ s_c \ \varepsilon^{c-a}(\mathbf{e}_a - \mathbf{e}_c).
	\end{align*}
	Thus, all edges with a same direction also have the same length. Since~$\supp(\PSPerm(\height)[\height_{\varepsilon}])$ is combinatorially equivalent to a permutahedron, it follows that it is a zonotope.
\end{proof}

\section{Enumerative Consequences}\label{sec:enum_consequences}

We finish this chapter showing some enumerative consequences for the elements in the~$s$-weak order by calculating volume and the lattice points of~$\fpol[\oru(s)]$ via the Baldoni–Vergne–Lidskii formulas (Theorem~\ref{thm:Lidskii_formulas}).

\begin{corollary}\label{cor:identitise s-trees}
	Let~$s$ be a (weak) composition. The number of elements in the~$s$-weak order decomposes as
	\begin{align*}
		\prod_{i=1}^{n-1}\Bigl(1+\sum_{r=n-i+1}^n s_r\Bigr)
		 & =
		\sum_{\bf j}  \binom{s_n+1}{j_1}\binom{s_{n-1}+1}{j_2} \cdots \binom{s_2+1}{j_{n-1}} \cdot \prod_{i=1}^{n-1} (j_1+\cdots + j_i-i+1)            \\
		 & \,=\, \sum_{\bf j} \bbinom{s_n+1}{j_1}\bbinom{s_{n-1}-1}{j_2} \cdots \bbinom{s_2-1}{j_{n-1}} \cdot \prod_{i=1}^{n-1} (j_1+\cdots + j_i-i+1)
	\end{align*}
	where the sums range over weak compositions~${\bf j}$ of~$n-1$ such that~$\mathbf{j}\succeq (1,1,\ldots,1)$.
\end{corollary}

We give two proofs. The first proof is geometrical using the flow polytope~$\fpol[\oru_n]$ while the second is combinatorial using~$s$-weak order families.

\begin{proof}[First proof.]
	Let~$s$ be a weak composition. Notice that taking the oruga graph~$\oru_n$, through Remark~\ref{rem:bij_simplices_trees} we can see that the flow polytope~$\fpol[\oru_n](s_n,s_{n-1},\ldots,s_2,-\sum_is_i)$ is integrally equivalent to the product of segments~$\prod_{i=1}^{n-1}[0,\sum_{r=n-i+1}^n s_r]$.

	Using the Baldoni–Vergne–Lidskii formula for the volume and considering that~$\oru_n$ has shifted outdegrees~$o_i=1$ for~$i\in[n-1]$ and shifted indegrees~$d_1=-1$ and~$d_i=1$ for~$i\in[2,n]$, we have that
	\begin{align*}
		\Big|\fpol[\oruga]^{\mathbb{Z}}\Big(s_n,s_{n-1},\ldots,s_2,-\sum_i s_i\Big)\Big| & =  \sum_{\bf j}  \binom{s_n+1}{j_1}\binom{s_{n-1}+1}{j_2} \cdots \binom{s_2+1}{j_{n-1}} |\mathcal{F}^{\mathbb{Z}}_{\oruga}({\bf j}-{\bf 1})|,   \\
		                                                                                 & =  \sum_{\bf j} \bbinom{s_n+1}{j_1}\bbinom{s_{n-1}-1}{j_2} \cdots \bbinom{s_2-1}{j_{n-1}} |\mathcal{F}^{\mathbb{Z}}_{\oruga}({\bf j}-{\bf 1})|,
	\end{align*}
	where the sums range over compositions~${\bf j}=(j_1,\ldots,j_{n-1})$ of~$n-1$ that are~$\succeq (1,1,\ldots,1)$. Now, let us count the integer flows in~$\mathcal{F}^{\mathbb{Z}}_{\oruga}(\mathbf{j}-\mathbf{1})$. For such an integer flow, the incoming flow at vertex~$i\in[n-1]$ is~$j_1+\cdots+j_{i-1}-(i-1)$. Since the netflow on vertex~$i$ is~$j_i-1$, the outgoing flow equals~$j_1+\cdots+j_{i}-i$. Moreover, there are~$j_1+\cdots+j_i-i+1$ possible outgoing integer flows on the two edges~$(i,i+1)$ for which the choice of flow is independent. Thus, we obtain as wished
	\begin{equation*}
		\vol\Big(\mathcal{F}_{\oruga}(\mathbf{j}-\mathbf{1})\Big) = \prod_{i=1}^{n-1} (j_1+\cdots + j_i-i+1). \qedhere
	\end{equation*}
\end{proof}

\begin{proof}[Second proof]
	We prove the equality of the (LHS) with each expression on the (RHS) separately.

	The first formula using binomials, can be obtained similarly to the proof given in Proposition~\ref{prop:num_s_trees} enumerating~$s$-decreasing trees. For step~$0$ we begin with the node labeled~$n$. At step~$1$ we choose which of its~$s_n+1$ children become nodes (as opposed to leaves). This gives a coefficient~$\binom{s_n+1}{j_1}$, where~$j_1\in[\min(s_n+1,n-1)]$. Among these~$j_1$ nodes, one carries the label~$n-1$ with~$s_{n-1}+1$ children. Right before step~$i$, we have a partial~$s$-decreasing tree with~$i$ nodes labeled by~$[n+1-i,n]$,and~$j_1 + \cdots + j_{i-1} - (i-1)$ unlabeled nodes, where~$j_k$ is the number of non-empty subtrees of the node~$n+1-k$, whose positions were chosen at the step~$k$. At step~$i$ we choose~$j_i$ from the~$s_{n+1-i}+1$ new children to become new nodes. At this point we get~$n_{i}:=j_1 + \cdots + j_{i} - (i-1)$ nodes without labels. To ensure that~$n_i>0$, we have that constraint to verify that the~$j_k$ chosen so far satisfy that~$\sum_{k=1}^i j_i \geq i$. Afterwards we choose one of these~$n_i$ nodes to carry the label~$n-i$. We stop after step~$n-1$ where we have obtained an~$s$-decreasing tree.

	The second formula that utilizes multiset binomials can be obtained in the following way of building a Stirling~$s$-permutation. Let us say that an~$a$-block~$B_a$ covers a~$b$-block~$B_b$ in~$w$ if~$a$ is the smallest letter such that the~$B_b\subset B_a$. At step~$0$ we begin with the sequence of~$s_n$ consecutive occurrences of~$n$ and choose a number~$j_1$ of blocks to be covered by~$B_n$ or appear before or after it in the final multipermutation~$w$. Notice that there are~$\bbinom{s_n+1}{j_1}$ ways to arrange these blocks among the occurrences of~$n$. Choose one of these~$j-1$ blocks to be the sequence of~$s_{(n-1)}$ consecutive occurrences of~$n-1$. At the beginning of step~$i$, we have a partial Stirling~$s$-permutation that contains all occurrences of the letters in~$[n+1-i,n]$ and~$j_1 + \cdots + j_{i-1} - (i-1)$ unlabeled blocks, where~$j_k$ is the number of blocks covered by the~$(n+1-k)$-block, whose positions were chosen at the step~$k$. At step~$i$ choose the number~$j_i$ of blocks that to be covered by the~$(n+1-i)$-block and one among the~$\bbinom{s_{n+1-i}-1}{j_i}$ ways to arrange them between the first and the last occurrence of~$n+1-i$. Choose one of the~$n_{i}:=j_1 + \cdots + j_{i} - (i-1)$ unlabeled block to be the~$(n-i)$-block.	After step~$n-1$ we have inserted~$B_1$ and finished constructing a Stirling~$s$-permutation.
\end{proof}

\begin{remark} \label{rem: connection volume and lattice point oruga}
	When~$s$ is a composition, Corollary~\ref{cor:volume is number of trees}, tells us that the RHS of Corollary~\ref{cor:identitise s-trees} gives the volume of~$\fpol[\oru(s)](\bfi)$. This together with Proposition~\ref{prop:volume_intflows} gives us that
	\begin{align}
		\vol \Big(\fpol[\oru(s)](1,0,\ldots,0,-1)\Big) & = \Big|\fpol[\oru(s)]^{\mathbb{Z}}(0,s_n, s_{n-1},\ldots, s_2,-\sum_i s_i)\Big| \notag
		\\
		                                               & = \Big|\fpol[\oruga]^{\mathbb{Z}}(s_n,s_{n-1},\ldots,s_2,-\sum_i s_i)\Big|,\label{eq: vol Gs as int flows in Ps}
	\end{align}
	where the second equality follows from our Remark~\ref{rem:vertex_v_m1}. Thus, the formulas of Corollary~\ref{cor:volume is number of trees} as decomposition formulas for the volume of~$\fpol[\oru(s)](\bfi)$ as well. This is the approach followed in~\cite{KMS21} to prove geometrically the Lidskii decomposition of the Kostant partition formula given in Theorem~\ref{thm:Lidskii_formulas}.
\end{remark}

\begin{remark}
	Notice that the Lidskii formulas hold for zero also when~$s$ is a weak composition. Moreover, each term of the RHS of the first Lidskii formula is nonnegative for~$s_i\geq 0$, however terms of the RHS of the second formula can be negative as in Example~\ref{ex: case n=3 lidskii} for~$s=(1,0,1)$.
\end{remark}

\begin{example} \label{ex: case n=3 lidskii}
	For~$n=3$, Corollary~\ref{cor:identitise s-trees} yields
	\begin{align}
		(1+s_3)(1+s_2+s_3) & = \binom{s_{3}+1}{1}\binom{s_{2}+1}{1}+ \binom{s_{3}+1}{2}\cdot 2, \label{eq: 1st lidskii n=3} \\
		                   & =\binom{s_{3}+1}{1}\binom{s_2-1}{1}+\binom{s_{3}+2}{2}\cdot 2. \label{eq: 2nd lidskii n=3}
	\end{align}
\end{example}

We finish with some perspectives stemming out from our work.

\begin{perspective}\label{pers:s_associahedra}
	Ceballos and Pons also conjectured (\cite[Conjecture 2]{CP19}) that there exists a geometric realization of~$\PSPerm$ (when~$s$ is a strict composition) such that the~$s$-associahedron can be obtained from it by removing certain facets.
	Our realizations seem very promising for providing a geometric relation between~$s$-permutahedra and~$s$-associahedra, but this is still work in progress.
\end{perspective}

\begin{perspective}\label{pers:other_framings}
	Since in this chapter our objective was to answer Conjecture~\ref{conj:s-permutahedron}, we concentrated our efforts on studying the framing of~$\oru(s)$ described in Definition~\ref{def:Gs} to obtain the triangulation~$\triangDKK[\oru(s)]$ which encoded the~$s$-weak order. It is natural to ask about other DKK triangulations of~$\fpol[\oru(s)]$ coming from other framings and the poset structures that arise from them. In the case of the caracol graph~$\car(s)$, Bell et al.~\cite{BGMY23} studied the length framing and the planar framing whose corresponding DKK triangulation correspondingly realized in their dual graph the~$s$-Tamari lattice and the principal order ideal~$I(\nu)$ in {\em Young's lattice}, where~$\nu_i=1+s_{n-i+1}+s_{n-i+2}+\cdots+s_n$.
\end{perspective}

\begin{perspective}\label{pers:graphs_for_zeroes}
	Although the~$s$-weak order of Ceballos and Pons is defined for weak compositions, all of our realizations are defined uniquely for compositions as~$\oru(s)$ requires that~$s_i>0$ for~$i\in[n+1]$. Moreover, even though Corollary~\ref{cor:volume is number of trees} does not hold for weak compositions, Remark~\ref{rem:bij_simplices_trees} and the first proof of Corollary~\ref{cor:identitise s-trees} let us see that there is a connection. Computationally we have found certain variations of~$\oru(s)$ to describe the~$s$-weak order for certain families of weak compositions. This is still work in progress.
\end{perspective}

%% file: includes/contenu/chap_sorder_quotients.tex

\chapter{Recovering Permutrees with Flow Polytopes}\label{chap:sorder_quotients}

\addcontentsline{lof}{chapter}{\protect\numberline{\thechapter}Recovering Permutrees with Flow Polytopes}

In this chapter we apply our machinery of flow polytopes to give another answer to Perspective~\ref{pers:Coxeter_permutrees}. That is, we recover permutree lattices from the dual triangulation of a framed graph. This chapter comes from ongoing work~\cite{GMPTY2X}.

Our inspiration for the following came from looking at~$\car_n$ and~$\oru_n$ in Figure~\ref{fig:oru_car_mar} and deducing a way to create other similar graphs.
\begin{figure}[h!]
	\centering
	\includegraphics[scale=0.9]{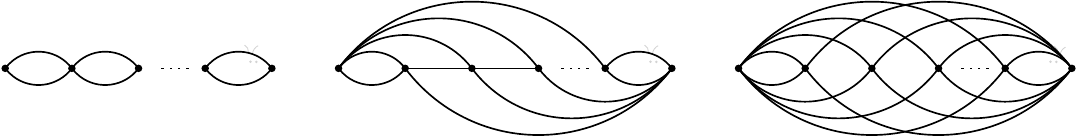}
	\caption[The oruga~$\oru_{n}$, caracol~$\car_n$, and mariposa~$\mar_n$ graphs.]{ The oruga~$\oru_{n}$ (left), caracol~$\car_n$ (middle), and mariposa~$\mar_n$ (right) graphs.}
	\label{fig:oru_car_mar}
\end{figure}

\section{M-Moves}\label{sec:m_moves}

Taking~$\oru_n$ as a starting point, we provide an operation on the edges of~$\oru_n$ to create new graphs akin to~$\car_n$.

\begin{definition}\label{def:m_move}
	Let~$G$ be a graph on~$\{v_0,\ldots,v_{n}\}$. The graph obtained by the \defn{M-move}\index{M-move} applied on an edge~$(v_{i},v_{i+1})$ is the graph~$M(G)$ on~$\{v_0,\ldots,v_{n}\}$ with edge set~$E(M(G)):= (E(G)\setminus\{(v_i,v_{i+1})\})\cup\{(v_0,v_{i+1}),(v_i,v_{n})\}$.

	Endow~$\oru_n$ with the framing~$e_0^i\preceq_{\cI_{n+1-i}} e_1^i$ (resp.~$e_0^{i-1}\preceq_{\cO_{n+1-i}} e_1^{i-1}$) for~$i\in[1,\ldots,n-1]$. The graph~$M(\oru_n)$ resulting from the M-move on the edge~$e_0^{i}$ (resp.~$e_1^{i}$) has the new modified framing~$\preceq'$ where the new edges inherit the framing of the removed edge. Figure~\ref{fig:m_moves} contains examples of the M-moves with the normal embedding of~$\oru_n$. This framing is called the \defn{inherited framing}\index{graph!framing!inherited} of~$M(\oru_n)$.

	We call the resulting graph after doing all possible M-moves on~$\oru_n$ the \defn{mariposa graph}\index{mariposa graph} denoted~$\mar_n$. Mariposa is the Spanish word for butterfly and as before, this name comes from the embedding of~$\mar_n$ shown in Figure~\ref{fig:oru_car_mar}.
\end{definition}

\begin{figure}[h!]
	\centering
	\includegraphics[scale=1.2]{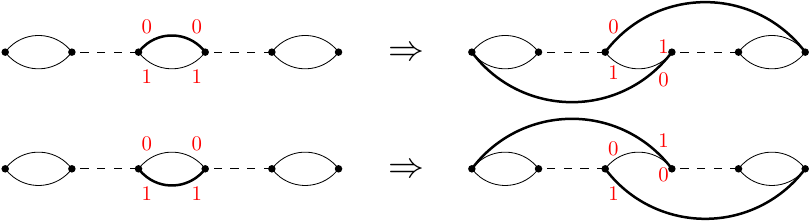}
	\caption[The M-moves on~$\oru_n$.]{ The resulting framed graph~$M(\oru_n)$ after applying an M-move~$M$ on the edge~$e^i_0$ (top) (resp.~$e^i_1$ (bottom)) of~$\oru_n$. The edges affected by the~$M$ are bolded.}
	\label{fig:m_moves}
\end{figure}

Notice that any sequence of M-moves gives a graph that can be considered to be between~$\oru_n$ and~$\mar_n$. Moreover, all M-moves are independent of each other as they manipulate disjoint sets of edges. As we can apply M-moves to none, one, or both of the edges between the vertices~$v_{n+1-i}$ and~$v_{n+1-i}$, the following definition is in order.

\begin{definition}\label{def:delta_bichos}
	Let~$\delta\in\nonee\cdot\{\nonee,\downn,\upp,\uppdownn\}^{n-2}\cdot\nonee$ be a permutree decoration. Apply on~$\oru_n$ the sequence of M-moves on the edge~$e_0^i$ if~$\delta_i\in\{\upp,\uppdownn\}$ (resp.~$e_1^i$ if~$\delta_i\in\{\downn,\uppdownn\}$). We call the resulting graph the \defn{$\delta$-bicho graph}\index{permutree!bicho graph} and denote it by~$\bic_\delta$. Bicho is the closest Spanish word for critter that encompasses caterpillars, snails, and butterflies at the same time.
\end{definition}

\begin{remark}\label{rem:all_bichos_graphs}
	With Definition~\ref{def:delta_bichos} it is clear that~$M(\bic_\delta)=\bic_{\delta'}$ for a collection~$M$ of M-moves if and only if~$\delta$ refines~$\delta'$. In this context we say that~$M(\delta)=\delta'$. Similar to Figure~\ref{fig:fibersPermutreeCongruences}, this gives us Figure~\ref{fig:permutree_bichos} containing all possible~$\delta$-bicho graphs for decorations~$\delta\in\nonee\cdot\{\nonee,\downn,\upp,\uppdownn\}^2\cdot\nonee$.
\end{remark}

\begin{figure}[h!]
	\centering
	\includegraphics[scale=0.87]{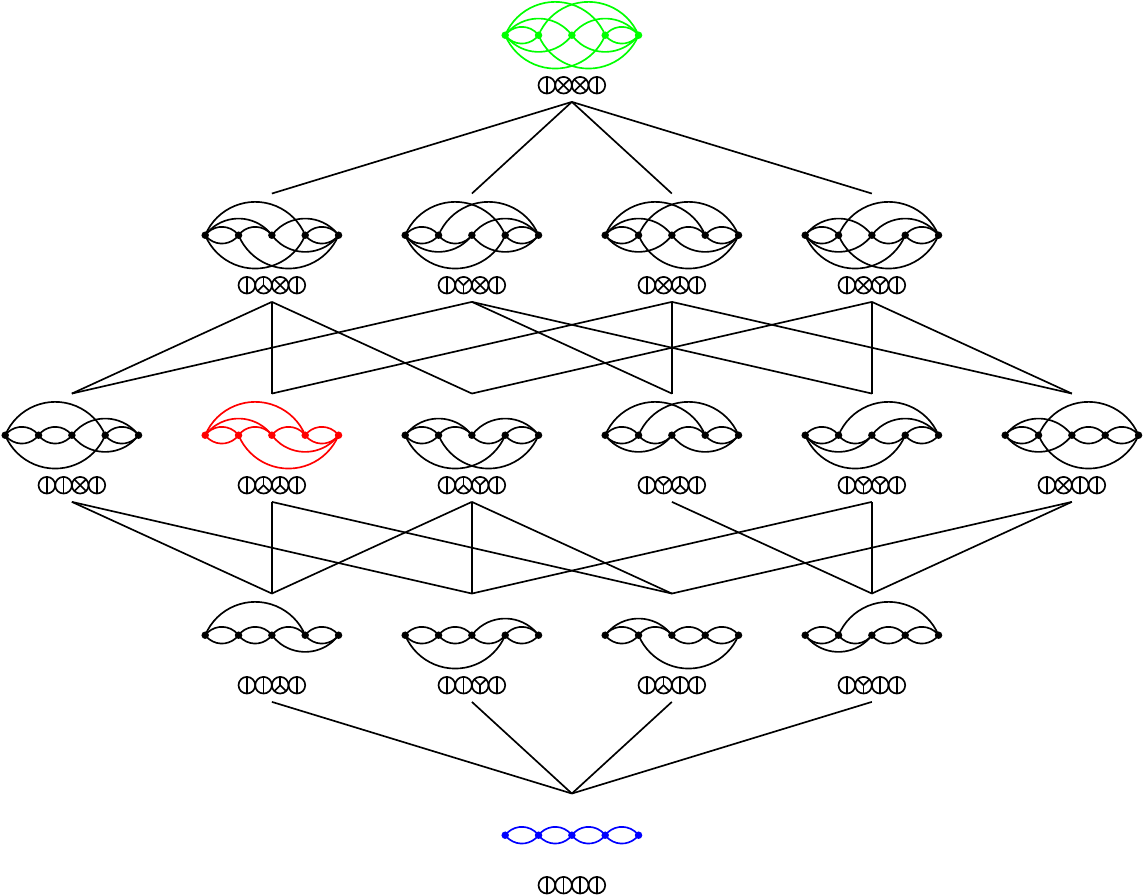}
	\caption[All~$\nonee\cdot\{\nonee,\downn,\upp,\uppdownn\}^2\cdot\nonee$-bicho graphs.]{ All~$\nonee\cdot\{\nonee,\downn,\upp,\uppdownn\}^2\cdot\nonee$-bicho graphs. The bottom graph is the {\color{blue} oruga graph}, the second graph in the middle is the {\color{red} caracol graph}, and the top graph is the {\color{green} mariposa graph}.}\label{fig:permutree_bichos}
\end{figure}

With our~$\delta$-bicho graphs in hand, we can study their corresponding flow polytopes~$\fpol[\bic_\delta](\bfa)$. 

\begin{remark}\label{rem:num_edges_dim_bicho}
	Since we obtain~$\bic_\delta$ from~$\oru_n$ via M-moves, we get that its polytope lives in dimension \begin{equation*}
		\Big|E\Big(\bic_\delta\Big)\Big|=2n+\Big|\Big\{i\in[n]:\delta_i\in\{\downn,\upp\}\Big\}\Big|+2\Big|\Big\{i\in[n]:\delta_i\in\{\uppdownn\}\Big\}\Big|
	\end{equation*} and \begin{equation*}
		\dimension\Big(\fpol[\bic_\delta](\bfa)\Big)=n+\Big|\Big\{i\in[n]:\delta_i\in\{\downn,\upp\}\Big\}\Big|+2\Big|\Big\{i\in[n]:\delta_i\in\{\uppdownn\}\Big\}\Big|.
	\end{equation*}
\end{remark}

Similar to the case of~$\oru(s)$, we denote routes of~$\bic_\delta$ as \defn{$R(k_1, t_1, \theta,k_2,t_2)$}\index{permutree!bicho graph!route}. Intuitively this notation comes from the fact that every route of~$\bic_\delta$ starts from~$v_{0}$, lands in a vertex~$v_{n+1-k_1}$ via a source bump or dip, follows~$k_1-k_2-1$ edges that can be either bumps or dips (depending on~$\delta$) and finally jumps from a vertex~$v_{n+1-k_2}$ to~$v_n$ by a sink bump or dip. In a more formal fashion, we give the following definition.

\begin{definition}\label{routes}
	Let~$\delta\in\nonee\cdot\{\nonee,\downn,\upp,\uppdownn\}^{n-2}\cdot\nonee$ and denote the sequences of sets \begin{flalign*}
		&&\Theta=&\Big(\emptyset,\Theta_2,\ldots,\Theta_{n-1},\emptyset\Big) && \text{ where } && \begin{cases}0\in\Theta_i \text{ if } \delta_i\in\{\nonee,\downn\}, \\
			1\in\Theta_i \text{ if } \delta_i\in\{\nonee,\upp\},
		\end{cases}&&\\
		&&\Omega^{\inEdge}=&\Big(\{0,1\},\Omega^{\inEdge}_2,\ldots,\Omega^{\inEdge}_{n-1},\emptyset\Big) && \text{ where } && \begin{cases}0\in\Omega^{\inEdge}_i \text{ if } \delta_i\in\{\downn,\uppdownn\}, \\
			1\in\Omega^{\inEdge}_i \text{ if } \delta_i\in\{\upp,\uppdownn\},
		\end{cases}&&\\
		&&\Omega^{\outEdge}=&\Big(\emptyset,\Omega^{\outEdge}_2,\ldots,\Omega^{\outEdge}_{n-1},\{0,1\}\Big) && \text{ where } && \begin{cases}0\in\Omega^{\outEdge}_i \text{ if } \delta_i\in\{\upp,\uppdownn\}, \\
			1\in\Omega^{\outEdge}_i \text{ if } \delta_i\in\{\downn,\uppdownn\}.
		\end{cases}&&
	\end{flalign*}

	We denote the routes of~$\bic_\delta$ as~\defn{$R(k_1,t_1,\theta,k_2,t_2)$} where~$2\leq k_2<k_1\leq n$,~$t_1\in\Omega^{\inEdge}_{n+1-k_1}$,~$t_2\in\Omega^{\outEdge}_{n+1-(k_2-1)}$, and~$\theta\in\prod_{i=n+1-(k_1-1)}^{n+1-k_2} \Theta_i$.
\end{definition}

\begin{example}\label{ex:bicho_route}
	Consider the route~$R(6,1,(1,0,1),3,0)$ of~$\bic_{\nonee\upp\nonee\downn\nonee\uppdownn\nonee}$ depicted in Figure~\ref{fig:s_bicho_route}.
	In this case~$1\leq 3<6\leq 6$,~$1\in\Omega^{\inEdge}_2$,~$0\in\Omega^{\outEdge}_6$, and~$\delta\in\Theta_3\times\Theta_4\times\Theta_5$ where \begin{align*}
		\Theta            & =\emptyset\times\{1\}\times\{0,1\}\times\{0\}\times\{0,1\}\times\emptyset\times\emptyset,                \\
		\Omega^{\inEdge}  & =\{0,1\}\times\{1\}\times\emptyset\times\{0\}\times\emptyset\times\{0,1\}\times\emptyset, \\
		\Omega^{\outEdge} & =\emptyset\times\{0\}\times\emptyset\times\{1\}\times\emptyset\times\{0,1\}\times\{0,1\}.
	\end{align*}
\end{example}

\begin{figure}[ht!]
	\centering
	\includegraphics[scale=1.5]{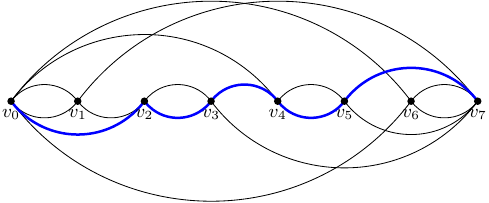}
	\caption[The route~$R(6,1,(1,0,1),3,0)$ of~$\bic_{\nonee\upp\nonee\downn\nonee\uppdownn\nonee}$.]{ The route~$R(6,1,(1,0,1),3,0)$ of~$\bic_{\nonee\upp\nonee\downn\nonee\uppdownn\nonee}$ bolded in blue.}
	\label{fig:s_bicho_route}
\end{figure}

\begin{remark}\label{rem:m_move_routes}
	Given a route~$R:=R(k_1,t_1,(\theta_{n+1-(k_1-1)},\ldots,\theta_{n+1-(k_2)}),k_2,t_2)$ in~$\bic_{\delta}$, the action of doing an M-move~$M$ on an edge~$e_k^i=(v_{n+1-(i+1)},v_{n+1-{i}})$ of~$\bic_{\delta}$ either transforms~$R$ into a pair of routes~$M(R):=\{R_1,R_2\}$ where
	\begin{flalign*}
		R_1&:=R(i,\theta_{n+1-i},(\theta_{n+1-(i-1)},\ldots,\theta_{n+1-k_2}),k_2,t_2)\\
		R_2&:=R(k_1,t_1,(\theta_{n+1-(k_1-1)},\ldots,\theta_{n+1-(i+1)}),i+1,\theta_{n+1-i})	
	\end{flalign*}
	if~$k_2\leq i\leq k_1-1$ and~$k=\theta_{n+1-i}$ or and does not affect~$R$ at all and~$M(R)=\{R_0\}$ where~$R_0:=R$. The routes~$R_0,R_1,R_2$ are all shown in the graph~$\bic_{M(\delta)}$. See Figure~\ref{fig:m_move_route} for an example of this construction in~$\oru_5$.
\end{remark}

\begin{figure}[h!]
	\centering
	\includegraphics[scale=1]{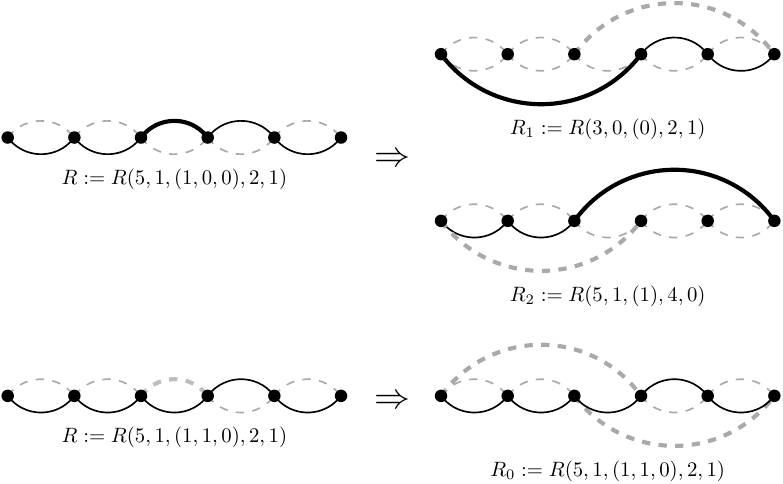}
	\caption[An M-move applied on two routes of~$\oru_n$.]{ The M-move corresponding to the bump~$(v_2,v_3)$ applied on two routes (left) of~$\oru_5$ and their responding routes after the M-move (right). The edges affected by the M-move are bolded.}
	\label{fig:m_move_route}
\end{figure}

\section{Enumerating Integer~\texorpdfstring{$\bfd$}{}-Flows in~\texorpdfstring{$\bic_\delta$}{}}\label{sec:bic_flows}

In this section we draw several parallels between the enumeration of permutrees done in~\cite[Sec.2.5]{PP18} and the enumeration of integer~$\bfd$-flows in~$\bic_\delta$.

\begin{remark}\label{rem:d_flows}
	Recall that~$\bfd:=(d_0,d_1,\ldots,d_{n-1},-\sum_{i=0}^{n-1}d_i)$ which in our current context is defined as~$d_i=\indeg_i(\bic_\delta)-1$. Moreover, notice that for any permutree decoration~$\delta$ we have that this vector is always~$\bfd=(0,1\ldots,1,-n+1)$.
\end{remark}

\begin{lemma}\label{lem:downn_upp_equivariance}
	Let	$\delta$ be a permutree decoration and~$\delta'$ be any decoration such that~$\delta^{-1}(\nonee)={\delta'}^{-1}(\nonee)$ and~$\delta^{-1}(\uppdownn)={\delta'}^{-1}(\uppdownn)$. Then~$|\fpol[\bic_\delta]^{\ZZ}(\bfd)|=|\fpol[\bic_{\delta'}]^{\ZZ}(\bfd)|$.
\end{lemma}

\begin{proof}
	Suppose that~$\delta$ and~$\delta'$ are permutree decorations with the stated property. Notice that the definition of M-moves and~$\bic_\delta$ is symmetric between choosing the decorations~$\downn$ and~$\upp$ except for the defined framing. Therefore, the graphs~$\bic_\delta$ and~$\bic_{\delta'}$ are isomorphic. Since the counting of integer~$\bfd$-flows depends only on the underlying graph structure, we get our result.
\end{proof}

\begin{lemma}\label{lem:uppdownn_decomposition}
	Suppose that a permutree decoration~$\delta$ decomposes as~$\delta=\delta'\uppdownn\delta''$. Then~$|\fpol[\bic_\delta]^{\ZZ}(\bfd)|=|\fpol[\bic_{\delta'\nonee}]^{\ZZ}(\bfd)|\cdot|\fpol[\bic_{\nonee\delta''}]^{\ZZ}(\bfd)|$.
\end{lemma}

\begin{proof}
	From Definition~\ref{def:delta_bichos} it follows that an M-move corresponding to a decoration~$\delta_i=\uppdownn$ partitions the set of routes of~$\bic_\delta$ into routes of the form~$R(k_1',t_1',\theta',k_2',t_2')$ where~$i+1\leq k_2'<k_1'\leq n$ and~$R(k_1'',t_1'',\theta'',k_2'',t_2'')$ where~$2\leq k_2''<k_1''\leq i$. This partitions the edges of~$\bic_\delta$ into two sets depending on the type of routes that contain them.  In this context any~$\bfd$-flow on~$\bic_\delta$ gets divided into the two sections of the graph. Since the sections form a partition, the flows are independent. As the subgraphs corresponding to this partition are isomorphic to~$\bic_{\delta'\nonee}$ and~$\bic_{\nonee\delta''}$ after a relabeling of the vertices the result follows.
\end{proof}

\begin{remark}\label{rem:bases_of_recursion_bic_flows}
	For~$n=1$ the graph~$\bic_\delta$ is isomorphic just two vertices with two edges no matter the decoration. In this case~$\bfd=(0,0)$. As such, in this case there is only~$1$ integer~$\bfd$-flow which is the constant flow~$f_0$ assigning to every edge flow~$0$. Similarly, for~$n=0$, we have that~$\bic_\delta$ is a point with~$\bfd=(0)$ and the same constant flow~$f_0$ is the only integer~$\bfd$-flow. 
\end{remark}

\begin{theorem}\label{thm:bij_d_flows_bichos_permutrees}
	For any decoration~$\delta\in\{\nonee,\downn\}^n$, the sets of integer~$\bfd$-flows of~$\bic_\delta$ and~$\delta$-permutrees are in bijection.
\end{theorem}

\begin{proof}
	Given~$i\in[n]$ denote respectively by~$e^{n+1-i}_0:=(v_{i-1},v_{i})$ and~$e^{n+1-i}_1:=\begin{cases}
		(v_{i-1},v_{i}) \text{ if } \delta_i=\nonee, \\
		(v_{i-1},v_{n}) \text{ if } \delta_i=\downn
	\end{cases}$ the bump edges of~$\bic_\delta$ and the dip edges of~$\bic_\delta$ that can receive any flow when the netflow is given by~$\bfd$. In this context notice that a~$\bfd$-flow~$f$ is uniquely determined by the values~$f(e^{n+1-i}_0)$ as all source edges have~$0$ flow, and all dip edges have flow that satisfies the recursive relations \begin{equation*}
		f(e^{n+1-i}_1)=\begin{cases}
			f(e^{n+1-(i-1)}_0)+f(e^{n+1-(i-1)}_1)+1-f(e^{n+1-i}_0) \text{ if } \delta_{i-1}=\nonee, \\
			f(e^{n+1-(i-1)}_0)+1-f(e^{n+1-i}_0) \hspace{2.705cm} \text{ if } \delta_{i-1}=\downn.
		\end{cases}
	\end{equation*} 
	
	Take~$f$ an integer~$\bfd$-flow of~$\bic_\delta$. Let us construct a~$\delta$-permutree (with nodes relabeled~$\{v'_1,\ldots,v'_n\}$ to avoid confusion). We do this starting from the integers~$f(e^i)$ and proceeding like in the proofs of Lemma~\ref{lem:permutree_inversion_sets} and Theorem~\ref{thm:cubic_property_convec_cube} using the insertion algorithm as follows. Begin by taking an~$n\times n$ grid. At step~$1$, since~$f(e^{n+1-1})=0$, place the vertex~$v'_1$ anywhere in the first column. For convenience place it at position~$(1,1)$. At step~$i$, place the vertex~$v'_i$ in position~$(i,i-f(e^{n+1-i}))$ moving all previous vertices that are at the same or greater height upwards by~$1$ unit. After step~$n$ we obtain a permutation table. Decorating this table with~$\delta$ and applying the permutree insertion algorithm gives us a~$\delta$-permutree~$T_f$ such that~$f_{T_f}=f$. See Figure~\ref{fig:permutree_d_flows_insertion_algorithm} for an example.

	\begin{figure}[h!]
		\centering
		\includegraphics[scale=0.76]{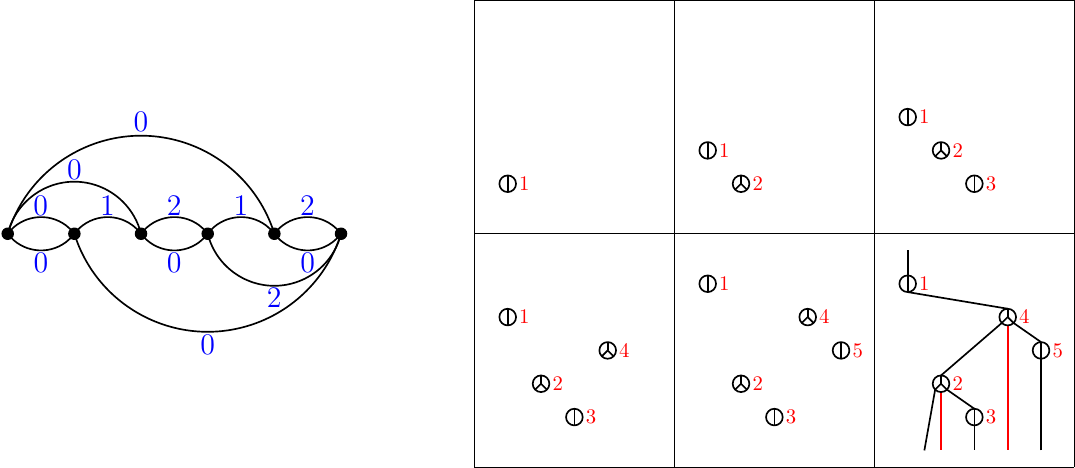}
		\caption[The construction of a~$\nonee\downn\nonee\downn\nonee$-permutree from an integer~$\bfd$-flow on~$\bic_{\nonee\downn\nonee\downn\nonee}$.]{ The construction of a~$\nonee\downn\nonee\downn\nonee$-permutree (right) from the integer~$\bfd$-flow characterized by the flow~$(0,1,2,1,2)$ on the bumps of~$\bic_{\nonee\downn\nonee\downn\nonee}$ (left).}
		\label{fig:permutree_d_flows_insertion_algorithm}
	\end{figure}

	Conversely, let~$T\in\cPT(\delta)$ be a~$\delta$-permutree. We define the function~$f_T(e^{n+1-i})=\big|\{j\in[i-1]\,:\, i\to j\}\big|$. That is, the amount of vertices in $LA_i$ (resp.~$A_i$) with label smaller than~$i$. We claim that this function is compatible with the netflow~$\bfd$. Indeed, notice that although the definition tells us that~$f_T(e^{n+1-i})\leq i-1$ for~$i\in[n]$, we actually have that~$f_T(e^{n+1-1})=0$ and~$f_T(e^{n+1-i})\leq f_T(e^{n+1-(i-1)})+1$ for~$i\in[1,n]$. This occurs as~$\delta\in\{\nonee,\downn\}$ implies that either~$i-1\to i$ and~$f_T(e^{n+1-i})\leq f_T(e^{n+1-(i-1)})$ or~$i\to i-1$ and~$f_T(e^{n+1-i})=f_T(e^{n+1-(i-1)})+1$. So we know that the amount of flow is the correct one and respects the netflow~$\bfd$. Thus,~$f_T$ can be extended to a unique integer~$\bfd$-flow of~$\bic_\delta$.
\end{proof}

\begin{corollary}\label{cor:vol_bic_bfi_permutrees}
	Given a permutree decoration~$\delta$, we have that \begin{equation*}
		\big|\cPT(\delta)\big|=\big|\fpol[\bic_\delta]^{\ZZ}(\bfd)\big|=\vol\Big(\fpol[\bic_\delta](\bfi)\Big).
	\end{equation*}
\end{corollary}

\begin{proof}
	Due to Lemma~\ref{lem:uppdownn_decomposition} and~\ref{lem:downn_upp_equivariance},~$\big|\fpol[\bic_\delta]^{\ZZ}(\bfd)\big|$ can be decomposed as a product of~$\big|\fpol[\bic_{\delta'}]^{\ZZ}(\bfd)\big|$ where~$\delta'_i\in\{\nonee,\downn\}$. As these integer~$\bfd$-flows are in bijection with~$\delta'$-permutrees through Theorem~\ref{thm:bij_d_flows_bichos_permutrees} and~$\delta$-permutrees follow the same decomposition given in Proposition~\ref{prop:permutree_recursion}, we have that~$|\fpol[\bic_\delta]^{\ZZ}(\bfd)|=|\cPT(\delta)|$. We finish by applying Theorem~\ref{prop:volume_intflows}.
\end{proof}

\section{Permutrees and Cliques of~\texorpdfstring{$\bic_\delta$}{}}\label{sec:permutree_cliques}

As before, the routes of the framed graph~$\bic_\delta$ give us the triangulation~$\triangDKK[\bic_\delta]$ of~$\fpol[\bic_\delta](\bfi)$. We proceed to show how simplices of these triangulations relate with each other.

\begin{definition}\label{def:MC}
	Let~$C$ be a collection of routes of~$\bic_\delta$ and~$M$ a sequence of M-moves on a subset of edges of~$\bic_\delta$. We denote by \defn{$M(C)$} the set~$\bigcup_{R\in C}M(R)$ following Remark~\ref{rem:m_move_routes}.
\end{definition}

\begin{lemma}\label{lem:C_coherent_MC_coherent}
	If~$C$ is a set of coherent routes of~$\bic_\delta$, then~$M(C)$ is a set of coherent routes of~$\bic_{M(\delta)}$.
\end{lemma}

\begin{proof}
	It is enough to prove the lemma for the case when~$M$ is an M-move on a single edge~$e^i_{k}=(v_{n+1-(i+1)},v_{n+1-i})$. In this case~$\delta\lessdot\delta'$ and~$\bic_{\delta'}$ has the same framing of~$\bic_{\delta}$ except for~$\cO_{n+1-(i+1)}$ and~$\cI_{n+1-i}$. Let~$P$ and~$Q$ be two routes in~$C$. We have the three following cases. \begin{itemize}
		\itemsep0em
		\item If both~$P$ and~$Q$ contain~$e^i_{k}$ then~$M(P)=\{P_1,P_2\}$ and~$M(Q)=\{Q_1,Q_2\}$. Let us see that~$\{P_1,P_2,Q_1,Q_2\}$ is a clique. \begin{itemize}
			\itemsep0em
			\item The pairs of routes~$\{P_1,P_2\}$,~$\{P_1,Q_2\}$,~$(Q_1,P_2)$ and~$\{Q_1,Q_2\}$ are all cliques as each pair of routes has no common vertices except~$v_0$ and~$v_n$. Thus, no possible conflict.
			\item For~$\{P_1,Q_1\}$ notice that their suffixes satisfy~$v_{n+1-i}P_1=v_{n+1-i}P$ and~$v_{n+1-i}Q_1=v_{n+1-i}Q$ and that both their source edges are~$P_1v_{n+1-i}=Q_1v_{n+1-i}$. Therefore, the edges on which~$P_1$ and~$Q_1$ differ are the same as the edges on which~$P$ and~$Q$ differ. As~$P$ and~$Q$ are coherent,~$P_1$ and~$Q_1$ are coherent. Similar for~$\{P_2,Q_2\}$.
		\end{itemize}
		Thus, all routes in~$\{P_1,P_2,Q_1,Q_2\}$ are pairwise coherent.
		\item If~$P$ contains~$e^i_{k}$ but~$Q$ does not, then~$M(P)=\{P_1,P_2\}$ and~$M(Q)=\{Q_0\}$. \begin{itemize}
			\itemsep0em
			\item If~$Q$ contains~$e^i_{1-k}$ then notice that as~$Q_0=Q$,~$v_{n+1-i}P_1=v_{n+1-i}P$ and~$P_2v_{n+1-(i+1)}=Pv_{n+1-(i+1)}$, and their framing is inherited from~$P$ and~$Q$, we have that no conflict can arise.
			\item If~$Q$ does not contain~$e^i_{1-k}$ then the common edges of~$P_1$ and~$Q_0$ (resp.~$P_2$ and~$Q_0$) are common edges of~$P$ and~$Q$ meaning that no conflict occurs.
		\end{itemize}
		Thus,~$P_1$,~$P_2$ and~$Q_0$ are pairwise coherent. The argument is similar with the roles of~$P$ and~$Q$ reversed.
		\item If neither~$P$ nor~$Q$ contain~$e^i_{k}$ then~$M(P)=\{P_0\}$ and~$M(Q)=\{Q_0\}$. As~$P=P_0$,~$Q=Q_0$ and their framings are the same as before, we have that~$P_0$ and~$Q_0$ are coherent.
	\end{itemize}
	As in all cases~$M(P)\cup M(Q)$ is formed by pairwise coherent routes, we have that~$M(C)=\bigcup_{R\in C}M(R)$ is a clique of~$\bic_{\delta'}$.
\end{proof}

\begin{remark}\label{rem:M_exceptional}
	For the exceptional routes~$R^d=R(n,1,(1)^{n-2},2,1)$ and~$R^b=R(n,1,(0)^{n-2},2,1)$ of~$\oru_n=\bic_{\nonee^n}$, we have that~$M(R^d)\cup M(R^b)$ are the exceptional routes of~$\bic_{M\big(\nonee^n\big)}$. 
\end{remark}

\begin{lemma}\label{lem:cliques_through_M_moves}
	Let~$\delta,\delta'$ be permutree decorations such that~$\delta$ refines~$\delta'$ and~$M$ be the collection of M-moves transforming~$\bic_{\delta}$ to~$\bic_{\delta'}$. Then every~$C'\in\maxcliques[\bic_{\delta'}]$ is of the form~$M(C)$ for some~$C\in\maxcliques[\bic_{\delta}]$.
\end{lemma}

\begin{proof}
	Suppose that~$\delta\lessdot\delta'$ in the refinement order of permutree decorations. In this case, as an M-move replaces 1 edge by 2 edges, we have that~$|E(\bic_{\delta'})|=|E(\bic_{\delta})|+1$. Since the number of vertices after an M-move is the same, we also obtain that~$\dimension(\fpol[\bic_{\delta'}](\bfd))=\dimension(\fpol[\bic_{\delta}](\bfd))+1$. Suppose that~$C$ is a maximal clique of coherent routes of~$\bic_{\delta}$. Lemma~\ref{lem:C_coherent_MC_coherent} tells us that~$M(C)$ is a clique of coherent routes of~$\bic_{\delta'}$. We proceed to show that~$M(C)$ is maximal, that is,~$|M(C)|=|C|+1$. 

	Let~$e^i_k$ denote the edge on which~$M$ is applied. Consider~$P,Q$ to be routes of~$\bic_\delta$ containing the edge~$e^i_k$ and let~$z$ denote the number of such routes. Following Remark~\ref{rem:m_move_routes} we have that if~$v_{n+1-i}P=v_{n+1-i}Q$ then~$P_1=Q_1$ (resp.\ if~$Pv_{n+1-(i+1)}=Qv_{n+1-(i+1)}$ then~$P_2=Q_2$). Thus, each path from~$v_0$ to~$v_{n+1-(i+1)}$ and each path from~$v_{n+1-i}$ to~$v_n$ appearing in a route of~$C$ containing the edge~$e^i_k$ contributes a route to~$|M(C)|$. The fact that each route containing~$e^i_k$ contains~$2$ such paths together with how coherent paths interact after an M-move shown in the proof of Lemma~\ref{def:MC} gives us that there is a total of~$z+1$ such paths in~$C$. That is, there are~$z+1$ routes in~$M(C)$ contributed from routes containing~$e^i_k$. Any route~$R$ that does not contain~$e^i_k$ contributes just~$1$ to the cardinality of~$|M(C)|$ as~$M(R)=\{R_0\}$. Since there are~$|C|-z$ such routes we are done.
	
	The general case follows from the fact that all M-moves affect different edges coming from~$\oru_n$.
\end{proof}

Notice that Lemma~\ref{lem:cliques_through_M_moves} agrees with the fact that~$\dim(\oru_n)=n$ and~$\dim(\mar_n)=3n-4$ as between~$\nonee^n$ and~$\uppdownn^n$ there are~$2(n-2)$ coarsening covering relations in the order of permutree decorations.

\begin{theorem}\label{thm:permutree_to_clique}
	Let~$\delta$ be a permutree decoration. The set of~$\delta$-permutrees~$\cPT(\delta)$ is in bijection with the set of maximal cliques of coherent routes~$\maxcliques[\bic_\delta]$.
\end{theorem}

\begin{proof}
	Recall from Lemma~\ref{lem:permutree_cardinality_edges} that the number of edges of a~$\delta$-permutree is precisely~$n+1+|\{i\in[n]:\delta_i\in\{\downn,\upp\}\}|+2|\{i\in[n]:\delta_i\in\{\uppdownn\}\}|$. This coincides with the number of elements of a maximal clique of coherent routes of~$\bic_\delta$ being~$1$ more than the dimension calculated in Remark~\ref{rem:num_edges_dim_bicho}.

	Let~$M$ denote the collection of M-moves such that~$M(\nonee^n)=\delta$. We now proceed to label the edges of a~$\delta$-permutree~$T$ with routes in~$\bic_\delta$ by doing an enhanced version of the permutree insertion algorithm (see Definition~\ref{def:decorated_permutation}). 
	
	Consider a decorated table corresponding to~$T$. Below the table place the all-dip route~$R^{d1}_1:=R(n,1,(1)^{n-2},2,1)$. Now, in increasing order for~$i\in[2,n-1]$, if~$\delta_i\in\{\downn,\uppdownn\}$ (resp.\ if~$\delta_i\in\{\upp,\uppdownn\}$) draw a red wall below (resp.\ above) the coordinate decorated with~$\delta_i$ and apply an M-move~$M_i$ on the edge~$e^{n+1-i}_1$ of~$R^{di}_1$ as per Definition~\ref{def:delta_bichos} to obtain~$M_i(R^{di})=\{R^{d(i+1)}_1,R^{d(i+1)}_2\}$. At the bottom of the left (resp.\ right) zone created by the red wall place~$R^{d(i+1)}_2$ (resp.~$R^{d(i+1)}_1$). This gives at the bottom of each zone delimited by the red walls an exceptional route of~$\bic_\delta$ as given by~$M(R^d)$.
	
	For step~$0$, from the bottom of each zone generate a string labeled with the route of that zone. At step~$i$ extend these labeled strings by a unit of~$1$. If the height of a subset of labeled strings is the same as the height of a decorated coordinate and are not separated by a red wall, the decoration at the coordinate catches the labeled strings and then releases a new set of labeled strings through the following rules. \begin{itemize}
		\itemsep0em
		\item If~$\delta_i=\nonee$, then change~$e^{n+1-i}_1$ by~$e^{n+1-i}_0$ in the entering route.
		\item If~$\delta_i=\downn$, then undo the M-move on~$e^{n+1-i}_1$ from the two entering routes to get a single new route and change its edge~$e^{n+1-i}_1$ by~$e^{n+1-i}_0$.
		\item If~$\delta_i=\upp$, then change~$e^{n+1-i}_1$ by~$e^{n+1-i}_0$ in the entering route and apply the M-move of~$e^{n+1-i}_0$.
		\item If~$\delta_i=\uppdownn$, apply rule~$\downn$ and then rule~$\upp$.
	\end{itemize}
	The rules are depicted in Figure~\ref{fig:route_decoration_rules}.  The algorithm ends when the strings have lengths~$n+1$. See Figure~\ref{fig:permutree_to_clique} for an example of a~$\delta$-permutree obtained through this altered insertion algorithm.

	\begin{figure}[h!]
		\centering
		\includegraphics[scale=1.25]{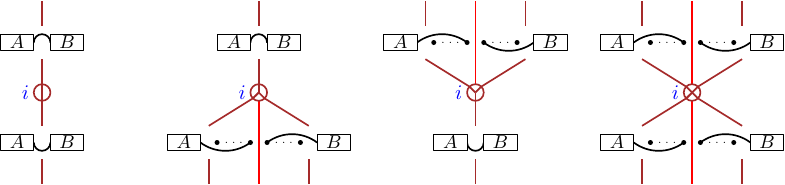}
		\caption[The release and catching rules of the modified permutree insertion algorithm.]{ The release and catching rules of the modified permutree insertion algorithm.~$A$ (resp.~$B$) represents a path from~$v_0$ to~$v_{n+1-(i+1)}$ (resp.\ from~$v_{n+1-i}$ to~$v_{n}$).}\label{fig:route_decoration_rules}
	\end{figure}

	Denote by~$\cR(T)$ the collection of routes that label the edges of~$T$. We claim that~$\cR(T)$ is a clique of coherent routes of~$\bic_\delta$. Moreover, we claim that~$\cR(T)=M(C)$ where~$C$ is the clique of routes that labels the~$\nonee^n$-permutree with same table as~$T$. Suppose that~$\nonee^n\lessdot\delta$ and the general case follows through any sequence of coarsening relations between~$\nonee^n$ and~$\delta$. Following the proof of Lemma~\ref{lem:max_clique}, the maximal clique~$C$ in~$\oru((1,\ldots,1))=\oru_n$ is endowed with a total order on the routes dictated by its corresponding permutation. As each of the routes~$R\in C$ is transformed to~$M(R)$, these covering relations are translated into the covering relations represented by the rules~$\downn$,~$\upp$, and~$\uppdownn$. Following the cases in the proof of Lemma~\ref{lem:C_coherent_MC_coherent} we get that~$\bigcup_{R\in C} M(R)=:M(C)$ and~$\cR(T)=M(C)$ which is a set of coherent routes according to Lemma~\ref{lem:cliques_through_M_moves}.
	
	Since~$\cR(T)$ has the same cardinality as the edges of~$T$, it follows that~$\cR(T)$ is a maximal clique of coherent routes of~$\bic_\delta$. As M-moves are independent between themselves and the rules in Figure~\ref{fig:route_decoration_rules} are also independent as they interact with different edges of~$T$, it follows that this is an injection from~$\delta$-permutrees to maximal cliques of routes of~$\bic_\delta$. Corollary~\ref{cor:vol_bic_bfi_permutrees} tells us that these sets have the same cardinality, so we obtain a bijection.
\end{proof}

\begin{figure}[h!]
	\centering
	\includegraphics[scale=1.1]{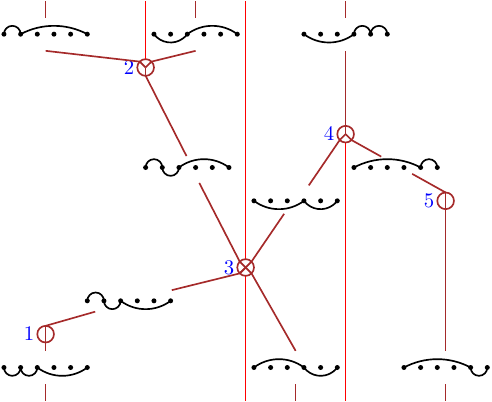}
	\caption{A~$\nonee\upp\uppdownn\downn\nonee$-permutree with edges labeled by its maximal clique of flows.}\label{fig:permutree_to_clique}
\end{figure}

\begin{remark}
	The routes at the top of the zones of a~$\delta$-permutree~$T$ are exactly the exceptional routes~$M(R^b)$ ordered from left to right via the M-moves on~$e^i_0$ if~$\delta_i\in\{\upp,\uppdownn\}$. Thus, we can also describe the modified permutree insertion algorithm via strings going down.
\end{remark}

\begin{corollary}\label{cor:rotation_to_adjacency}
	Let~$T_1,T_2$ be different~$\delta$-permutrees. The simplices~$\Delta_{T_1}$ and~$\Delta_{T_2}$ are adjacent in~$\triangDKK[\bic_\delta]$ if and only if there is a covering relation between~$T_1$ and~$T_2$ in the~$\delta$-permutree rotation order.
\end{corollary}

\begin{proof}
	Notice that an edge~$i\to j$ with~$i<j$ in a~$\delta$-permutree is equivalent to doing rule~$\delta_i$ followed by rule~$\delta_j$. Since these rules act on disjoint sets of routes in~$\bic_\delta$, we have that the rules satisfy~$\delta_i\circ\delta_j=\delta_j\circ\delta_i$ applied to the maximal route in~$RD_i$. 
	
	Suppose now that there is an~$ij$-rotation from~$T_1$ to~$T_2$ with~$i<j$. As the~$ij$-rotation keeps the structure of~$T_1$ intact except for the edge~$i\to j$ changed to~$j\to i$, we have by Theorem~\ref{thm:permutree_to_clique} and the start of our proof that this is equivalent to~$\Delta_{T_1}$ and~$\Delta_{T_2}$ differing exactly in~$1$ route. This route corresponds to applying rule~$\delta_j$ before~$\delta_i$ in the modified permutree insertion algorithm.
\end{proof}

In Figure~\ref{fig:permutree_IXYI_flow} we present the~$3$ initial covering relations of the~$\nonee\uppdownn\upp\nonee$-permutree rotation lattice.

\begin{figure}[h!]
	\centering
	\includegraphics[scale=0.95]{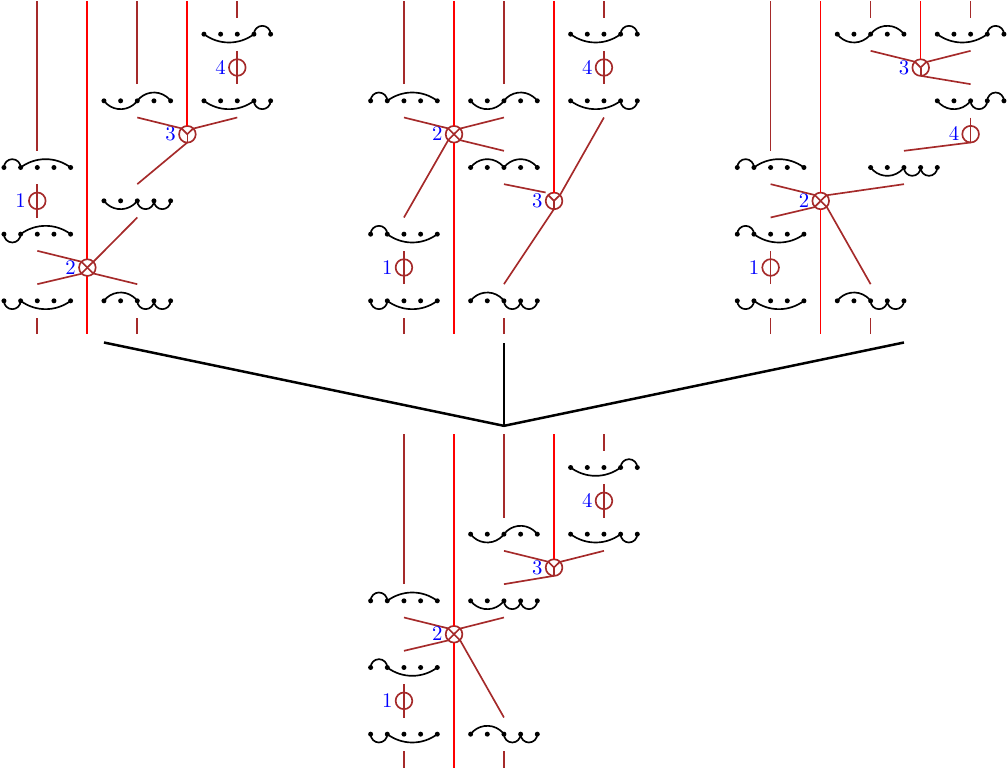}
	\caption{The rank~$1$ covering relations of the~$\nonee\uppdownn\upp\nonee$-permutree rotation lattice with each permutree having edges labeled by their maximal clique of flows.}\label{fig:permutree_IXYI_flow}
\end{figure}

We finish this thesis with several perspectives that would be interesting to study stemming from the work of this last chapter.

\section{Perspectives and Conjectures}\label{sec:perspectives_permutree_flows}

With Theorem~\ref{thm:bij_d_flows_bichos_permutrees} we opted to define a direct bijection between~$\delta$-permutrees and integer~$\bfd$-flows of~$\bic_\delta$ when~$\delta\in\{\nonee,\downn\}$ instead of proving the full decomposition shown in Proposition~\ref{prop:permutree_recursion} for permutrees of said decorations. This leads us to the following conjecture.

\begin{conjecture}\label{conj:nonee_downn_flows_2}
	Let~$\delta\in\{\nonee,\downn\}$. The number of~$\bfd$-flows of~$\bic_\delta$ satisfies the recursive formula \begin{equation*}
		\Big|\fpol[\bic_\delta]^{\ZZ}(\bfd)\Big|=\sum_{\substack{i\in\delta^{-1}\big(\downn\big)\\J\subseteq\delta^{-1}\big(\nonee\big)}}\Big|\fpol[\bic_{\delta_{[i-1]\setminus{J}}}]^{\ZZ}(\bfd)\Big|\Big|\fpol[\bic_{\delta_{[i+1,n]\setminus{J}}}]^{\ZZ}(\bfd)\Big||J|!.
	\end{equation*}
\end{conjecture}

Lemmas~\ref{lem:downn_upp_equivariance} and~\ref{lem:uppdownn_decomposition} together with Conjecture~\ref{conj:nonee_downn_flows_2} would give us the following general recursive decomposition of the number of~$\bfd$-flows of~$\bic_\delta$ as follows.

\begin{conjecture}\label{conj:bicho_recursion}
	For any permutree decoration~$\delta$, the number of~$\bfd$-flows of~$\bic_\delta$ satisfies the recursive formula \begin{equation*}
		\Big|\fpol[\bic_\delta]^{\ZZ}(\bfd)\Big|=\prod_{k\in[\ell]}\sum_{\substack{i\in[b_{k-1},b_k]\cap\delta^{-1}\big(\downn\big)\\J\subseteq[b_{k-1},b_k]\cap\delta^{-1}\big(\nonee\big)}}\Big|\fpol[\bic_{\delta_{[b_{k-1},i-1]\setminus{J}}}]^{\ZZ}(\bfd)\Big|\Big|\fpol[\bic_{\delta_{[i+1,b_k]\setminus{J}}}]^{\ZZ}(\bfd)\Big||J|!
	\end{equation*} where $\{b_0<\cdots<b_\ell\}=\{0,n\}\cup\delta^{-1}(\uppdownn)$.
\end{conjecture}

\begin{perspective}
	The proof given in Theorem~\ref{thm:bij_d_flows_bichos_permutrees} also raises several interesting questions. A concrete one being if it is possible to extend this bijection from the case~$\delta\in\{\nonee,\downn\}$ to all possible permutree decorations. A more abstract one is that the proof makes use of an insertion algorithm construction based on a vector like in the proofs of Lemma~\ref{lem:permutree_inversion_sets} and Theorem~\ref{thm:cubic_property_convec_cube}. As these proofs characterized certain properties of permutrees linked to inversion and cubic vectors respectively, we wonder if there is some other vector for permutrees at play here. At first glance it would appear to be the bracket vector of~$\delta^{\hspace{0.05cm}\updownarrow\hspace{-0.23cm}\leftrightarrow}$-permutrees where~$\delta^{\hspace{0.05cm}\updownarrow\hspace{-0.23cm}\leftrightarrow}$ denotes the decoration obtained after flipping all decorations and reversing the order of~$\delta$.
\end{perspective}

\begin{perspective}\label{pers:other_realizations}
	Corollary~\ref{cor:rotation_to_adjacency} gives us a realization of the~$1$-skeleton of the permutreehedron~$\PPT(\delta)$ as the dual poset of the maximal interior faces of~$\triangDKK[\bic_\delta]$ with the inherited framing of~$\oru_n$. It would be nice to extend this to get a full realization of~$\PPT(\delta)$ as the dual of this triangulation. Furthermore, Proposition~\ref{prop:flow_pols_ARE_cayley_embs} tells us that the Cayley trick might give us an interesting family of sums of polytopes somewhere between summations of cubes and summations of simplices. Finally, applying the tropical toolbox described in Subsection~\ref{sec:tropical_geometry} together with the DKK framing should give us explicit coordinates and a new explicit realization of~$\PPT(\delta)$.
\end{perspective}

\begin{perspective}\label{pers:s_m_moves}
	In this chapter we concentrated on the M-moves applied onto~$\oru_n$ since the lattice structure of~$\delta$-permutrees was already given in~\cite{PP18}. It would be interesting to study M-moves on~$\oru(s)$. The source edges created by the M-move correspond to adding a 1 to a coordinate of~$s$, but it is unclear what role the new sink edges play. The hope is that M-moves define lattice congruences of the~$s$-weak order. If so, this could give a path for permutrees to find their place in the~$s$ and~$\nu$ world.

	Perspective~\ref{pers:other_realizations} of obtaining explicit polytopes for these lattices applies to this case as well to obtain~\defn{$s$-permutreehedra}. This is of particular interest as if~$\PPT_s(\delta)$ are obtained by removing facets from~$\PSPerm$, we would have a positive answer to the conjecture~\cite[Conj.2]{CP19} about obtaining the~$s$-associahedron by removing certain facets from the~$s$-permutahedron when~$s$ is free of zeros.
\end{perspective}

\begin{perspective}\label{pers:s_permutrees_join_irreds}
	Permutrees only define a certain subset of lattice congruences on the weak order. To study the whole set of lattice congruences of a lattice one needs to understand the join irreducible elements of the lattice as done in~\cite{PS17} for the weak order. In our case that would mean extending the arc diagram characterization of join irreducibles of permutations to~$s$-arc diagrams for Stirling~$s$-permutations or even~$s$-decreasing permutrees for the most general case.
\end{perspective}

\begin{perspective}\label{pers:s_permutrees_combinatorics}
	Moreover, it would be interesting to define and understand~\defn{$(s,\delta)$-permutrees} and develop their combinatorics as in~\cite{PP18} for permutrees. That is, finding~\defn{$(s,\delta)$-permutree congruences}, an~\defn{$s$-insertion algorithm},~\defn{$s$-permutreehedra}, and even~\defn{$(s,\delta)$-permutree Hopf algebras}.
\end{perspective}

\begin{perspective}\label{pers:other_types}
	In~\cite{MM13} the concepts of graphs, flows, and the Kostant partition functions were given not only for type~$A$ but also for types~$B$ and~$D$. Since the definition of~$\oru_n$ seems to have a certain connection with the Coxeter graph of type~$A$, it may be worth it to look for similar graphs for types~$B$ and~$D$ and specify M-moves on them. This could also help generalize the Cambrian phenomenon found in all Coxeter types to the context of permutrees.
\end{perspective}